\documentclass[11pt]{amsart}
\usepackage{amscd,amssymb,longtable}
\usepackage[matrix,arrow,curve]{xy}

\theoremstyle{definition}
\newtheorem{theorem}[equation]{Theorem}
\newtheorem*{theorem*}{Theorem}
\newtheorem{lemma}[equation]{Lemma}
\newtheorem{corollary}[equation]{Corollary}
\newtheorem{conjecture}[equation]{Conjecture}
\newtheorem{example}[equation]{Example}
\newtheorem{definition}[equation]{Definition}
\newtheorem*{definition*}{Definition}

\theoremstyle{remark}
\newtheorem{remark}[equation]{Remark}

\makeatletter\@addtoreset{equation}{section} \makeatother

\def \O {\mathcal{O}}
\def \X {\mathcal{X}}
\def \LCS {\mathrm{LCS}}
\def \P {\mathbb{P}}
\def \Q {\mathbb{Q}}

\def \Z {\mathbb{Z}}
\def \F {\mathbb{F}}

\def \qlineq {\sim_{\Q}}
\def \dbar {{\overline\partial}}
\def \ddbar {{\partial\overline\partial}}

\def \Supp {\mathrm{Supp}\,}
\def \mult {\mathrm{mult}}

\def \Aut {\mathrm{Aut}}

\def \Sing {\mathrm{Sing}\,}
\def \lct {\mathrm{lct}}
\def \diam {\mathrm{diam}}

\def \ge {\geqslant}
\def \le {\leqslant}
\def \kappa {\varkappa}
\def \eps {\varepsilon}

\def\nlb {\nolinebreak}

\author{Ivan Cheltsov and Constantin Shramov\\ \\ With an appendix by Jean-Pierre Demailly}

\title{Log canonical thresholds of smooth Fano threefolds}

\thanks{The first author was supported by
the grants NSF DMS-0701465 and EPSRC EP/E048412/1,
the~second~author was supported by the grants RFFI
No.~08-01-00395-a, N.Sh.-1987.2008.1 and EPSRC EP/E048412/1.}

\begin{document}

\begin{abstract}
Complex singularity exponent is a local invariant of a holomorphic
function defined by the square-integrability of fractional powers
of the function. Log canonical thresholds of effective
$\mathbb{Q}$-divisors on normal algebraic varieties are algebraic
counterparts of complex singularity exponents. For a Fano variety,
these invariants have global analogues. In the former case, it is
so-called $\alpha$-invariant of Tian. In the latter case, it is a
global log canonical threshold of the Fano variety, which is the
infimum of log canonical thresholds of all effective
$\mathbb{Q}$-divisors numerically equivalent to the anticanonical
divisor. The appendix to this paper contains the proof that the
global log canonical threshold of a smooth Fano variety coincides
with its $\alpha$-invariant. The purpose of this paper is to
compute global log canonical thresholds of smooth Fano threefolds
(altogether there are 105 deformation families of such
threefolds). We compute global log canonical thresholds of every
smooth threefold in 64 deformation families, and we compute global
log canonical thresholds of general threefold in 20 deformation
families. For 14 deformation families, we find some bounds for
global log canonical thresholds.
\end{abstract}

\maketitle

\tableofcontents

\section{Introduction}
\label{section:intro}

The multiplicity of a polynomial $\phi\in\mathbb{C}[z_1,\ldots,
z_n]$ in the origin $O\in \mathbb{C}^n$ is the number
$$
\mathrm{min}\left\{m\in\mathbb{Z}_{\ge 0}\ \Big|\ \frac{\partial^m  \phi\big(z_{1},\ldots,z_{n}\big)}{\partial^{m_1} z_1\partial^{m_2} z_2\ldots\partial^{m_n} z_n}\big(O\big)\ne 0 \right\}\in\mathbb{Z}_{\ge 0}\cup\big\{+\infty\big\}.%
$$

There is a similar but more subtle invariant
$$
c_{0}\big(\phi\big)=\mathrm{sup}\left\{\eps\in\mathbb{Q}\ \Big|\
\text{the function}\  \frac{1}{\ \ \big|\phi\big|^{2\eps}}~\text{is
locally integrable near $O\in\mathbb{C}^n$}\right\}\in\mathbb{Q}_{\ge 0}
\cup\{+\infty\},%
$$
which is called the complex singularity exponent of the polynomial
$\phi$ at the point $O$.

\begin{example}
\label{remark:sum-of-powers} Let $m_{1},\ldots,m_{n}$ be positive
integers. Then
$$
\mathrm{min}\left(1,\ \sum_{i=1}^{n} \frac{1}{m_{i}}\right)=c_{0}\left(\sum_{i=1}^{n}z_{i}^{m_{i}}\right)\geqslant c_{0}\left(\prod_{i=1}^{n}z_{i}^{m_{i}}\right)=\mathrm{min}\left(\frac{1}{m_{1}},\frac{1}{m_{2}},\ldots,\frac{1}{m_{n}}\right).%
$$
\end{example}

Let $X$ be a variety\footnote{All varieties are assumed to be
complex, algebraic, projective and normal if the opposite is not
stated explicitly.} with at most log canonical singularities (see
\cite{Ko97}), let $Z\subseteq X$ be a closed subvariety, and let
$D$ be an effective $\mathbb{Q}$-Cartier $\mathbb{Q}$-divisor on
the variety $X$. Then the number
$$
\mathrm{lct}_{Z}\big(X,D\big)=\mathrm{sup}\left\{\lambda\in\mathbb{Q}\
\Big|\ \text{the log pair}\
 \big(X, \lambda D\big)\ \text{is log canonical along}~Z\right\}\in\mathbb{Q}\cup\big\{+\infty\big\}%
$$
is called a log canonical threshold of the divisor $D$ along $Z$.
It follows from \cite{Ko97} that
$$
\mathrm{lct}_{O}\Big(\mathbb{C}^n,\big(\phi=0\big)\Big)=c_0\big(\phi\big),
$$
so that $\mathrm{lct}_{Z}(X,D)$ is an algebraic
counterpart of the number $c_{0}(\phi)$. One has
$$
\mathrm{lct}_{X}\big(X,D\big)=\mathrm{inf}\left\{\mathrm{lct}_P\big(X,D\big)\
\Big\vert\ P\in
 X\right\}=\mathrm{sup}\left\{\lambda\in\mathbb{Q}\ \Big|\ \text{the log pair}\ \big(X, \lambda
 D\big)\
 \text{is log canonical}\right\},%
$$
and, for simplicity, we put
$\mathrm{lct}(X,D)=\mathrm{lct}_X(X,D)$.

Suppose that $X$ is a Fano variety with at most log terminal
singularities (see \cite{IsPr99}).

\begin{definition}
\label{definition:threshold} Global log canonical threshold of the
Fano variety $X$ is the number
$$
\mathrm{lct}\big(X\big)=\mathrm{inf}\left\{\mathrm{lct}\big(X,D\big)\ \Big\vert\ D\
\text{is an effective $\mathbb{Q}$-divisor on $X$ such that}\
D\ \sim_{\mathbb{Q}} -K_{X}\right\}\geqslant 0.%
$$
\end{definition}

The number $\mathrm{lct}(X)$ is an algebraic counterpart of the
$\alpha$-invariant introduced in \cite{Ti87}. One~has
$$
\mathrm{lct}\big(X\big)=\mathrm{sup}\left\{\eps\in\mathbb{Q}\ \left|\ %
\aligned
&\text{the log pair}\ \left(X, \frac{\eps}{n}D\right)\ \text{is log canonical for}\\
&\text{every divisor}\ D\in\big|-nK_{X}\big|, n\in\mathbb{Z}_{>0}\\
\endaligned\right.\right\}.%
$$

Recall that every Fano variety $X$ is rationally connected (see
\cite{Zh06}). Thus, the group $\mathrm{Pic}(X)$ is torsion free.
Then
$$
\mathrm{lct}\big(X\big)=\mathrm{sup}\left\{\lambda\in\mathbb{Q}\ \left|\ %
\aligned
&\text{the log pair}\ \Big(X, \lambda D\Big)\ \text{is log canonical}\\
&\text{for every effective $\mathbb{Q}$-divisor}\ D\qlineq -K_{X}\\
\endaligned\right.\right\}.
$$

\begin{example}
\label{example:hypersurface-index-big} Let $X$ be a smooth
hypersurface in $\mathbb{P}^{n}$ of degree $m<n$. Then
$$
\mathrm{lct}\big(X\big)=\frac{1}{n+1-m}%
$$
as shown in \cite{Ch01b}. In
particular, the equality $\mathrm{lct}(\mathbb{P}^{n})=1/(n+1)$ holds.
\end{example}

\begin{example}
\label{example:Hwang} Let $X$ be a rational homogeneous space such
that $-K_{X}\sim rD$ and
$$
\mathrm{Pic}\big(X\big)=\mathbb{Z}\big[D\big],
$$
where $D$ is an ample divisor and $r\in\mathbb{Z}_{>0}$. Then
$\mathrm{lct}(X)=1/r$ (see \cite{Hw06b}).
\end{example}

In general, the number $\mathrm{lct}(X)$ depends on small
deformations of the variety $X$.

\begin{example}
\label{example:sextic-double-solid-explicit} Let $X$ be a smooth
hypersurface in $\mathbb{P}(1,1,1,1,3)$ of degree $6$. Then
$$
\mathrm{lct}\big(X\big)\in\left\{\frac{5}{6},\frac{43}{50},\frac{13}{15},\frac{33}{38},\frac{7}{8},\frac{33}{38},\frac{8}{9}, \frac{9}{10},\frac{11}{12},\frac{13}{14},\frac{15}{16},\frac{17}{18},\frac{19}{20},\frac{21}{22},\frac{29}{30},1\right\}%
$$
by \cite{Pu04d} and \cite{ChPaWo}, and all these values are attained.
\end{example}

\begin{example}
\label{example:Cheltsov-Park} Let $X$ be a smooth hypersurface in
$\mathbb{P}^{n}$ of degree $n\geqslant 2$. Then the inequalities
$$
1\geqslant\mathrm{lct}\big(X\big)\geqslant \frac{n-1}{n}%
$$
hold (see \cite{Ch01b}). Then it follows from \cite{Pu04d} and \cite{ChPaWo}
that
$$
\mathrm{lct}\big(X\big)\geqslant\left\{%
\aligned
&1\ \text{if}\ n\geqslant 6,\\%
&22/25\ \text{if}\ n=5,\\%
&16/21\ \text{if}\ n=4,\\%
&3/4\ \text{if}\ n=3,\\%
\endaligned\right.%
$$
whenver $X$ is general. But $\mathrm{lct}(X)=1-1/n$ if $X$ contains a cone of
dimension $n-2$.
\end{example}

\begin{example}
\label{example:IHES} Let $X$ be a quasismooth hypersurface in
$\mathbb{P}(1,a_{1},a_{2},a_{3},a_{4})$ of
degree~$\sum_{i=1}^{4}a_{i}$ such that $X$ has at most terminal
singularities (see \cite{Ko97}), where $a_{1}\leqslant
a_{2}\leqslant a_{3}\leqslant a_{4}$. Then
$$
-K_{X}\sim\mathcal{O}_{\mathbb{P}(1,\,a_{1},\,a_{2},\,a_{3},\,a_{4})}\big(1\big)\Big\vert_{X},
$$
and there are $95$ possibilities for the quadruple
$(a_{1},a_{2},a_{3},a_{4})$ (see \cite{IF00}, \cite{JoKo01}). Then
$$
1\geqslant\mathrm{lct}\big(X\big)\geqslant\left\{%
\aligned
&16/21\ \text{if}\ a_{1}=a_{2}=a_{3}=a_{4}=1,\\%
&7/9\ \text{if}\ (a_{1},a_{2},a_{3},a_{4})=(1,1,1,2),\\%
&4/5\ \text{if}\ (a_{1},a_{2},a_{3},a_{4})=(1,1,2,2),\\%
&6/7\ \text{if}\ (a_{1},a_{2},a_{3},a_{4})=(1,1,2,3),\\%
&1\ \text{in the remaining cases},\\%
\endaligned\right.%
$$
if $X$ is general (see  \cite{Ch07a}, \cite{ChPaWo},
\cite{Ch08d}). The global log canonical threshold of the
hypersurface
$$
w^{2}=t^{3}+z^{9}+y^{18}+x^{18}\subset\mathbb{P}\big(1,1,2,6,9\big)\cong\mathrm{Proj}\Big(\mathbb{C}\big[x,y,z,t,w\big]\Big)
$$
is equal to $17/18$ (see  \cite{Ch07a}), where
$\mathrm{wt}(x)=\mathrm{wt}(y)=1$, $\mathrm{wt}(z)=2$,
$\mathrm{wt}(t)=6$, $\mathrm{wt}(w)=9$.%
\end{example}

\begin{example}
\label{example:WPS} It follows from Lemma~\ref{lemma:toric} that
$$
\mathrm{lct}\Big(\mathbb{P}\big(a_{0},a_{1},\ldots,a_{n}\big)\Big)=\frac{a_{0}}{\sum_{i=0}^{n}a_{i}},%
$$
where  $\mathbb{P}(a_{0},a_{1},\ldots,a_{n})$ is well-formed (see
\cite{IF00}), and $a_{0}\leqslant a_{1}\leqslant\ldots\leqslant
a_{n}$.
\end{example}

\begin{example}
\label{example:double-cover-big-index} Let $X$ be a smooth
hypersurface in $\mathbb{P}(1^{n+1},d)$ of degree $2d$.
Then
$$
\mathrm{lct}\big(X\big)=\frac{1}{n+1-d}%
$$
in the case when the inequalities $2\leqslant d\leqslant n-1$ hold
(see Proposition~20 in \cite{Ch08a}).
\end{example}

\begin{example}
\label{example:del-Pezzos} Let $X$ be smooth surface del Pezzo. It
follows from \cite{Ch07b} that
$$
\mathrm{lct}\big(X\big)=\left\{%
\aligned
&1\ \text{if}\ K_{X}^{2}=1\ \text{and}\ |-K_{X}|\ \text{contains no cuspidal curves},\\%
&5/6\ \text{if}\ K_{X}^{2}=1\ \text{and}\ |-K_{X}|\ \text{contains a cuspidal curve},\\%
&5/6\ \text{if}\ K_{X}^{2}=2\ \text{and}\ |-K_{X}|\ \text{contains no tacnodal curves},\\%
&3/4\ \text{if}\ K_{X}^{2}=2\ \text{and}\ |-K_{X}|\ \text{contains a tacnodal curve},\\%
&3/4\ \text{if}\ X\ \text{is a cubic in}\ \mathbb{P}^{3}\ \text{with no Eckardt points},\\%
&2/3\ \text{if either}\ X\ \text{is a cubic in}\ \mathbb{P}^{3}\
\text{with an Eckardt point, or
 $K_{X}^{2}=4$},\\%
&1/2\ \text{if}\ X\cong\mathbb{P}^{1}\times\mathbb{P}^{1}\ \text{or}\ K_{X}^{2}\in\{5,6\},\\%
&1/3\ \text{in the remaining cases}.\\%
\endaligned\right.%
$$
\end{example}

It would be interesting to compute global log canonical thresholds
of del Pezzo surfaces with at most canonical singularities that
are of Picard rank one, which has been classified in \cite{Fur86}.

\begin{example}
\label{example:singular-cubics} Let $X$ be a singular cubic
surface in $\mathbb{P}^{3}$ such that $X$ has at most canonical
singularities. The singularities of the surface $X$ are classified
in \cite{BW79}. It follows from \cite{Ch07c} that
$$
\mathrm{lct}\big(X\big)=\left\{%
\aligned
&2/3\ \text{if}\ \mathrm{Sing}\big(X\big)=\big\{\mathbb{A}_{1}\big\},\\%
&1/3\ \text{if}\ \mathrm{Sing}\big(X\big)\supseteq\big\{\mathbb{A}_{4}\big\},\
\mathrm{Sing}\big(X\big)=\big\{\mathbb{D}_{4}\big\}\ \text{or}\ \mathrm{Sing}\big(X\big)\supseteq\big\{\mathbb{A}_{2},\mathbb{A}_{2}\big\},\\%
&1/4\ \text{if}\ \mathrm{Sing}\big(X\big)\supseteq\big\{\mathbb{A}_{5}\big\}\
\text{or}\ \mathrm{Sing}\big(X\big)=\big\{\mathbb{D}_{5}\big\},\\%
&1/6\ \text{if}\ \mathrm{Sing}\big(X\big)=\big\{\mathbb{E}_{6}\big\},\\%
&1/2\ \text{in the remaining cases}.\\%
\endaligned\right.%
$$
\end{example}

It is unknown whether $\mathrm{lct}(X)\in\mathbb{Q}$ or not\footnote{It is even
unknown whether $\mathrm{lct}(X)\in\mathbb{Q}$ or not if $X$ is a del Pezzo
surfaces with log terminal singularities.} (cf. Question~1 in \cite{Ti90b}).

\begin{conjecture}
\label{conjecture:stabilization} There is an effective
$\mathbb{Q}$-divisor $D\qlineq -K_{X}$ on the variety $X$ such that
$$
\mathrm{lct}\big(X\big)=\mathrm{lct}\big(X,D\big)\in\mathbb{Q}.%
$$
\end{conjecture}

Let $G\subset\mathrm{Aut}(X)$ be an arbitrary subgroup.

\begin{definition}
\label{definition:G-threshold} Global $G$-invariant log canonical threshold of
the Fano variety $X$ is
$$
\mathrm{lct}\big(X, G\big)=\mathrm{sup}\left\{\eps\in\mathbb{Q}\ \left|\ %
\aligned
&\text{the log pair}\ \left(X, \frac{\eps}{n}\mathcal{D}\right)\ \text{has log canonical singularities for every}\\
&\text{$G$-invariant linear system}\
\mathcal{D}\subset\big|-nK_{X}\big|, n\in\mathbb{Z}_{>0}\\
\endaligned\right.\right\}.%
$$
\end{definition}

\begin{remark}
\label{remark:G-extension} To define thresholds $\mathrm{lct}(X)$
and $\mathrm{lct}(X,G)$, we only need to assume that
$|-nK_X|\neq\varnothing$ for some $n\gg 0$. This property is
shared by many varieties (toric varieties, weak Fano varieties),
but all known applications are related to the case when $-K_{X}$
is ample and $G$ is compact.
\end{remark}

In the case when the Fano variety $X$ is smooth and $G$ is compact,
the equality
$$
\mathrm{lct}\big(X,G\big)=\alpha_{G}\big(X\big),
$$
holds (see Appendix~\ref{section:alpha}), where $\alpha_{G}(X)$ is
the $\alpha$-invariant introduced in \cite{Ti87}. It is clear that
$$
\mathrm{lct}\big(X, G\big)=\mathrm{sup}\left\{\lambda\in\mathbb{Q}\ \left|\ %
\aligned
&\text{the log pair}\ \Big(X, \lambda D\Big)\ \text{has log canonical singularities}\\
&\text{for every $G$-invariant effective $\mathbb{Q}$-divisor}\ D\qlineq -K_{X}\\
\endaligned\right.\right\}%
$$
in the case when $|G|<+\infty$. Note that
$0\leqslant\mathrm{lct}(X)\leqslant\mathrm{lct}(X,G)\in\mathbb{R}\cup\{+\infty\}$.

\begin{example}
\label{example:quintic} Let $X$ be a smooth del Pezzo surface such
that $K_{X}^{2}=5$. Then
\begin{itemize}
\item the isomorphism $\mathrm{Aut}(X)\cong\mathrm{S}_{5}$ holds (see \cite{DoIs06}),%
\item the equalities $\mathrm{lct}(X,\mathrm{S}_{5})=\mathrm{lct}(X,\mathrm{A}_{5})=2$ hold (see~\cite{Ch07b}).%
\end{itemize}
\end{example}

\begin{example}
\label{example:Fermat-cubic} Let $X$ be the cubic surface in
$\mathbb{P}^{3}$ given by the equation
$$
x^{3}+y^{3}+z^{3}+t^{3}=0\subset\mathbb{P}^{3}\cong\mathrm{Proj}\Big(\mathbb{C}\big[x,y,z,t\big]\Big),
$$
and let $G=\mathrm{Aut}(X)\cong
\mathbb{Z}_{3}^{3}\rtimes\mathrm{S}_{4}$. Then $\mathrm{lct}(X,
G)=4$ by~\cite{Ch07b}.
\end{example}

The following result was proved in \cite{Ti87}, \cite{Na90},
\cite{DeKo01} (see Appendix~\ref{section:alpha}).

\begin{theorem}
\label{theorem:KE} Suppose that $X$ has at most quotient
singularities, and $G\subset\mathrm{Aut}(X)$ is a compact subgroup.
Assume that the inequality
$$
\mathrm{lct}\big(X,G\big)>\frac{\mathrm{dim}\big(X\big)}{\mathrm{dim}\big(X\big)+1}%
$$
holds. Then $X$ admits an orbifold K\"ahler--Einstein metric.
\end{theorem}

Theorem~\ref{theorem:KE} has various applications (see~\cite{Na90},
and Example~\ref{example:skew-blow-up}; cf Examples~\ref{example:Cheltsov-Park},~\ref{example:IHES}).

\begin{example}
\label{example:skew-blow-up} Let $X$ be a blow up of
$\mathbb{P}^{3}$ along a disjoint union of two lines,
let~$G$~be~a~maximal compact subgroup in $\mathrm{Aut}(X)$. Then
the inequality $\mathrm{lct}(X,G)\geqslant 1$ holds by
\cite{Na90}. But $\mathrm{lct}(X)=1/3$ by
Theorem~\ref{theorem:main}.
\end{example}

If a variety with quotient singularities admits an orbifold
K\"ahler--Einstein metric, then either its canonical divisor is
numerically trivial, or its canonical divisor is ample (variety of
general type), or its canonical divisor is antiample (Fano
variety). Every variety with at most quotient singularities that
has numerically trivial or ample canonical divisor always admits
an orbifold K\"ahler--Einstein metric (see \cite{Aub78},
\cite{Yau78}, \cite{Yau96}).

There are several known obstructions for the Fano variety $X$ to admit a
K\"ahler--Einstein~met\-ric. For example, if the variety $X$ is smooth, then it
does not admit a K\"ahler--Einstein metric~if
\begin{itemize}
\item either the group $\mathrm{Aut}(X)$ is not reductive (see~\cite{Mat57}),%
\item or the tangent bundle of $X$ is not polystable with respect to $-K_{X}$ (see  \cite{Lub83}),%
\item or the Futaki character of holomorphic vector fields on $X$ does not vanish~(see~\cite{Fu83}).%
\end{itemize}

\begin{example}
\label{example:V7} The following varieties admit no
K\"ahler--Einstein metrics:
\begin{itemize}
\item a blow up of $\mathbb{P}^{2}$ in one or two points (see~\cite{Mat57}),%
\item a smooth Fano threefold
$\mathbb{P}(\mathcal{O}_{\mathbb{P}^{2}}\oplus
 \mathcal{O}_{\mathbb{P}^{2}}(1))$ (see  \cite{Stef96}),%
\item a smooth Fano fourfold
$$
\mathbb{P}\Big(\alpha^{*}\big(\mathcal{O}_{\mathbb{P}^{1}}(1)\big)\oplus
 \beta^{*}\big(\mathcal{O}_{\mathbb{P}^{2}}(1)\big)\Big),%
$$
where $\alpha\colon
\mathbb{P}^{1}\times\mathbb{P}^{2}\to\mathbb{P}^{1}$ and
$\beta\colon \mathbb{P}^{1}\times\mathbb{P}^{2}\to\mathbb{P}^{2}$
are natural projections (see \cite{Fu83}).
\end{itemize}
\end{example}

The problem of existence of K\"ahler--Einstein metrics on smooth
toric Fano varieties is completely solved. Namely, the following
result holds (see \cite{Mab87}, \cite{BaSe99}, \cite{WaZhu04},
\cite{Ni90}).

\begin{theorem}
\label{theorem:toric} If $X$ is smooth and toric, then the
following conditions are equivalent:
\begin{itemize}
\item the variety $X$ admits a K\"ahler--Einstein metric;%
\item the Futaki character of holomorphic vector field of $X$ vanishes;%
\item the baricenter of the reflexive polytope of $X$ is zero.%
\end{itemize}
\end{theorem}

It should be pointed out that the assertion of
Theorem~\ref{theorem:KE} gives only a sufficient condition~for the
existence of a K\"ahler--Einstein metric on $X$.

\begin{example}
\label{example:cubic-Eckardt-point} Let $X$ be a general cubic
surface in $\mathbb{P}^{3}$ that has an Eckardt point (see
Definition~\ref{definition:Eckardt-point}). Then
$\mathrm{Aut}(X)\cong\mathbb{Z}_{2}$ (see \cite{DoIs06}) and
$\mathrm{lct}(X,\ \mathrm{Aut}(X))=\mathrm{lct}(X)=2/3$
by~\cite{Ch07b}. But every smooth del Pezzo surface~that has a
reductive automorphism groups admits a K\"ahler--Einstein metric
(see \cite{Ti90}).
\end{example}

\begin{example}[{cf. Example~\ref{example:double-cover-big-index}}]
\label{example:Nadel-Tian-Arezzo-Ghigi-Pirola} Let $X$ be a
general hypersurface in $\mathbb{P}(1^{5},3)$ of degree $6$. Then
$\mathrm{Aut}(X)\cong\mathbb{Z}_{2}$ (see \cite{MaMon64}) and
$\mathrm{lct}(X,\
\mathrm{Aut}(X))=\mathrm{lct}(X)=1/2$
by~\cite{Ch08a}. But $X$ admit a K\"ahler--Einstein metric (see
\cite{ArGha06}).
\end{example}

The numbers $\mathrm{lct}(X)$ and  $\mathrm{lct}(X, G)$ play an
important role in birational geometry.

\begin{example}
\label{example:fiberwise} Let $V$ and $\bar{V}$ be varieties with
at most terminal and $\mathbb{Q}$-factorial singularities, and let
$Z$ be a smooth curve. Suppose that there is a commutative diagram
$$
\xymatrix{
&V\ar@{->}[d]_{\pi}\ar@{-->}[rr]^{\rho}&&\ \bar{V}\ar@{->}[d]^{\bar{\pi}}&\\%
&Z\ar@{=}[rr]&&Z&}
$$
such that $\pi$ and $\bar{\pi}$ are flat morphisms, and $\rho$ is
a birational map that induces an isomorphism
$$
V\setminus X\cong\bar{V}\setminus\bar{X},%
$$
where $X$ and $\bar{X}$ are scheme fibers of $\pi$ and $\bar{\pi}$
over a point $O\in Z$, respectively. Suppose that
\begin{itemize}
\item the fibers $X$ and $\bar{X}$ are irreducible and reduced,
\item the divisors $-K_{V}$ and $-K_{\bar{V}}$ are $\pi$-ample and
$\bar{\pi}$-ample,
 respectively,%
\item the varieties $X$ and $\bar{X}$ have at most log terminal singularities,%
\end{itemize}
and $\rho$ is not an isomorphism. Then it follows from \cite{Pa01}
and \cite{Ch07c} that
$\mathrm{lct}(X)+\mathrm{lct}(\bar{X})\leqslant 1$, where $X$ and
$\bar{X}$ are Fano varieties by the adjunction formula.
\end{example}

In general, the inequality in Example~\ref{example:fiberwise} is
sharp.

\begin{example}
\label{example:dP3-D4} Let $\pi\colon V\to Z$ be a surjective flat
morphism from a smooth threefold $V$ to a smooth curve $Z$ such that
the divisor $-K_{V}$ is $\pi$-ample,
let $X$ be a scheme fiber of the morphism $\pi$ over a point $O\in
Z$ such that $X$ is a smooth cubic surface in $\mathbb{P}^{3}$
that has an Eckardt poin $P\in X$ (cf.
Definition~\ref{definition:Eckardt-point}), let
$L_{1},L_{2},L_{3}\subset X$ be the lines that pass through the
point $P$. Then it follows from \cite{Co96} that there is a
commutative diagram
$$
\xymatrix{
&&U\ar@{->}[dl]_{\alpha}\ar@{-->}[rr]^{\psi}&&\bar{U}\ar@{->}[dr]^{\beta}\\%
&V\ar@{->}[dr]_{\pi}\ar@{-->}[rrrr]^{\rho}&&&&\ \bar{V}\ar@{->}[dl]^{\bar{\pi}}&\\%
&&Z\ar@{=}[rr]&&Z&}
$$
such that $\alpha$ is a blow up of the point $P$, the map $\psi$
is an antiflip in the proper transforms of the~curves
$L_{1},L_{2},L_{3}$, and $\beta$~is~a~contraction of the~proper
transform of the fiber $X$. Then
\begin{itemize}
\item the birational map $\rho$ is not an isomorphism,%
\item the threefold $\bar{V}$ has terminal and $\mathbb{Q}$-factorial singularities,%
\item the divisor $-K_{\bar{V}}$ is a Cartier $\bar{\pi}$-ample divisor,%
\item the map $\rho$ induces an isomorphism
$V\setminus X\cong\bar{V}\setminus\bar{X}$,
where $\bar{X}$ is a scheme fiber of $\bar{\pi}$ over the point $O$,%
\item the surface $\bar{X}$ is a cubic surface with a singular point of
type $\mathbb{D}_{4}$.
\end{itemize}
The latter assertion implies that $\mathrm{lct}(X)+\mathrm{lct}(\bar{X})=1$ (see
Examples~\ref{example:del-Pezzos}~and~\ref{example:singular-cubics}).
\end{example}

Global log canonical thresholds can be used to prove that some
higher-dimensional Fano varieties are non-rational.

\begin{definition}
\label{definition:birational-superrigidity} The variety $X$ is
said to be birationally superrigid if the following conditions
hold:
\begin{itemize}
\item  $\mathrm{rk}\,\mathrm{Pic}(X)=1$;%
\item the variety $X$ has terminal $\mathbb{Q}$-factorial singularities;%
\item there is no rational dominant map $\rho\colon X\dasharrow Y$ with rationally connected fibers such that $0\ne \mathrm{dim}(Y)<\mathrm{dim}(X)$;%
\item there is no birational map $\rho\colon X\dasharrow Y$ onto a
variety $Y$ with terminal $\mathbb{Q}$-factorial singularities
such that $\mathrm{rk}\,\mathrm{Pic}(Y)=1$ holds,%
\item groups $\mathrm{Bir}(X)$ and $\mathrm{Aut}(X)$ coincide.
\end{itemize}
\end{definition}

The following result is known as the Noether--Fano inequality (see
\cite{Ch05umn}).

\begin{theorem}
\label{theorem:superrigidity} The variety $X$ is birationally superrigid
if and only if $\mathrm{rk}\,\mathrm{Pic}(X)=1$,
the variety $X$ has terminal $\mathbb{Q}$-factorial singularities, and
for every linear system $\mathcal{M}$ on the variety $X$
that does not have fixed components, the log pair $(X, \lambda
\mathcal{M})$ has canonical singularities, where
$K_{X}+\lambda\mathcal{M}\qlineq 0$.
\end{theorem}

\begin{proof}
Because one part of the required assertion is well-known (see
\cite{Ch05umn}), we prove only another part of the required
assertion. Suppose that the variety  $X$ is birationally
superrigid, but there is a linear system $\mathcal{M}$ on $X$ such
that $\mathcal{M}$ has no fixed components and the~singularities
of $(X,\lambda \mathcal{M})$ are not canonical,
where $K_{X}+\lambda\mathcal{M}\qlineq 0$.%

Let  $\pi\colon V\to X$ be a birational morphism such that the
variety $V$ is smooth, and the proper transform of $\mathcal{M}$
on the variety $V$ has no base points. Let $\mathcal{B}$ be the
proper transform of the linear system $\mathcal{M}$ on the variety
$V$. Then
$$
K_{V}+\lambda\mathcal{B}\qlineq\pi^{*}\Big(K_{X}+\lambda\mathcal{M}\Big)+\sum_{i=1}^{r}a_{i}E_{i}\qlineq\sum_{i=1}^{r}a_{i}E_{i},
$$
where $E_{i}$ is an exceptional divisor of $\pi$, and
$a_{i}\in\mathbb{Q}$.

It follows from \cite{BCHM06} that there is a commutative diagram
$$
\xymatrix{
&V\ar@{-->}[dl]_{\rho}\ar@{->}[dr]^{\pi}\\%
U\ar@{->}[rr]_{\phi}&&X}
$$
such that $\rho$ is a birational map, the morphism $\phi$ is
birational, the divisor
$$
K_{U}+\lambda\rho\big(\mathcal{B}\big)\qlineq\phi^{*}\Big(K_{X}+\lambda\mathcal{M}\Big)+\sum_{i=1}^{r}a_{i}\rho\big(E_{i}\big)\qlineq\sum_{i=1}^{r}a_{i}\rho\big(E_{i}\big)
$$
is $\phi$-nef, the variety $U$ is $\mathbb{Q}$-factorial, the log
pair $(U,\lambda\rho(\mathcal{B})$ has terminal singularities.

The morphism $\phi$ is not an isomorphism. It follows from
\cite[1.1]{Sho93} that
$$
a_{i}>0\Longrightarrow\mathrm{dim}\Big(\rho\big(E_{i}\big)\Big)\leqslant\mathrm{dim}\big(X\big)-2,
$$
and it follows from the construction of the map $\rho$ that there
is $k\in\{1,\ldots,r\}$ such that the inequality $a_{k}<0$ holds and
the subvariety $\rho(E_{k})\subset U$ is a divisor,
because the singularities of the log pair $(X,\lambda\mathcal{M})$
are not canonical.

The divisor $K_{U}+\lambda\rho(\mathcal{B})$ is not
pseudo-effective. Then it follows from \cite{BCHM06} that there is
a diagram
$$
\xymatrix{
U\ar@{->}[d]_{\phi}\ar@{-->}[rr]^{\psi}&&Y\ar@{->}[d]^{\tau}\\%
X&&Z}
$$
such that $\psi$ is a birational map, the morphism $\tau$ is a
Mori fibred space (see \cite{KoMo98}), and the divisor
$-(K_{Y}+\lambda(\psi\circ\rho)(\mathcal{B}))$ is $\tau$-ample.
The variety $Y$ has terminal $\mathbb{Q}$-factorial singularities,
and $\mathrm{rk}\,\mathrm{Pic}(Y/Z)=1$. Then the birational map
$\psi\circ\rho\circ\pi^{-1}$ is not an isomorphism, because
$K_{X}+\lambda\mathcal{M}\qlineq 0$, but a general fiber of the
morphism $\tau$ is rationally connected (see \cite{Zh06}), which
contradicts the assumption that $X$ is birationally superrigid.
\end{proof}

Birationally superrigid Fano varieties are non-rational. In
particular, if the variety $X$~is~birationally superrigid, then
$\dim(X)\neq 2$.

\begin{example}
\label{example:Pukhlikov-rigid} General hypersurface in
$\mathbb{P}^{n}$ and $\mathbb{P}(1^{n+1}, n)$ of degree
$n\geqslant 4$ and $2n\geqslant 6$ are birationally superrigid
(see \cite{Pu98a}, \cite{Pu04d}), respectively.
\end{example}

The following result is proved in \cite{Pu04d}.

\begin{theorem}
\label{theorem:Pukhlikov} Let $X_{1},X_{2},\ldots,X_{r}$ be
birationally superrigid Fano varieties such that
$\mathrm{lct}(X_{1})\geqslant 1, \mathrm{lct}(X_{2})\geqslant
1,\ldots, \mathrm{lct}(X_{r})\geqslant 1$. Then
\begin{itemize}
\item the variety $X_{1}\times\ldots\times X_{r}$ is non-rational
and
$$
\mathrm{Bir}\Big(X_{1}\times\ldots\times
X_{r}\Big)=\mathrm{Aut}\Big(X_{1}\times\ldots\times
 X_{r}\Big),
$$
\item for every rational dominant map
$$
\rho\colon X_{1}\times\ldots\times X_{r}\dasharrow Y,
$$
whose general fiber is rationally connected, there is a
commutative diagram
$$
\xymatrix{
X_{1}\times\ldots\times X_{r}\ar@{->}[d]_{\pi}\ar@{-->}[rrrrd]^{\rho}\\
X_{i_{1}}\times\ldots\times X_{i_{k}}\ar@{-->}[rrrr]_{\xi}&&&&Y}%
$$
for some $\{i_{1},\ldots,i_{k}\}\subseteq\{1,\ldots,r\}$, where
$\xi$ is a birational map, and $\pi$ is a projection.
\end{itemize}
\end{theorem}

Varieties satisfying all hypotheses of
Theorem~\ref{theorem:Pukhlikov}
exist~(see also Examples~\ref{example:Cheltsov-Park},~\ref{example:Pukhlikov-rigid}).

\begin{example}
\label{example:Cheltsov-Park-Won} Let $X$ be a hypersurface
that is given by
$$
w^{2}=x^6+y^6+z^6+t^6+x^2y^2zt\subset\mathbb{P}\big(1,1,1,1,3\big)\cong\mathrm{Proj}\Big(\mathbb{C}\big[x,y,z,t,w\big]\Big),
$$
where
$\mathrm{wt}(x)=\mathrm{wt}(y)=\mathrm{wt}(z)=\mathrm{wt}(t)=1$
and $\mathrm{wt}(w)=3$. Then
\begin{itemize}
\item the threefold $X$ is smooth~and~bira\-ti\-onally superrigid~(see~\cite{Is80b}),%
\item it follows from \cite{ChPaWo} that the equality
$\mathrm{lct}(X)=1$ holds.%
\end{itemize}
\end{example}

Suppose, in addition, that the subgroup $G\subset\mathrm{Aut}(X)$
is finite.

\begin{definition}
\label{definition:G-birational-superrigidity} A Fano variety $X$ is
$G$-birationally superrigid if
\begin{itemize}
\item the $G$-invariant subgroup of the group $\mathrm{Cl}(X)$ is isomorphic to $\mathbb{Z}$,%
\item the variety $X$ has terminal singularities,%
\item there is no $G$-equivariant rational map $\rho\colon X\dasharrow Y$ with rationally connected fibers such that  $0\ne \mathrm{dim}(Y)<\mathrm{dim}(X)$,%
\item there is no
$G$-equivariant non-biregular birational  map $\rho\colon
X\dasharrow Y$ onto a variety $Y$ with terminal singularities such
that the $G$-invariant subgroup of the group $\mathrm{Cl}(Y)$ is
isomorphic to $\mathbb{Z}$.
\end{itemize}
\end{definition}

Arguing as in the proof of Theorem~\ref{theorem:superrigidity}, we
obtain the following result.

\begin{theorem}
\label{theorem:G-superrigidity} The variety  $X$ is $G$-birationally
superrigid if and only if the $G$-invariant subgroup of the
group $\mathrm{Cl}(X)$ is isomorphic to $\mathbb{Z}$, the variety $X$
has terminal singularities,
for every $G$-invariant linear system $\mathcal{M}$ on $X$
that~has~no~fixed~components, the log pair $(X,
\lambda\mathcal{M})$ is canonical, where
$K_{X}+\lambda\mathcal{M}\qlineq 0$.
\end{theorem}

The proof of Theorem~\ref{theorem:Pukhlikov} implies the
following result (see \cite{Ch07b}).

\begin{theorem}
\label{theorem:G-Pukhlikov} Let $X_{i}$ be a Fano variety, and let
$G_{i}\subset\mathrm{Aut}(X_{i})$ be a finite subgroup such that
the variety $X_{i}$ is $G_{i}$-birationally superrigid, and
the inequality $\mathrm{lct}(X_{i},G_{i})\geqslant 1$ holds for
$i=1,\ldots,r$.
Then
\begin{itemize}
\item there is no $G_{1}\times\ldots\times G_{r}$-equivariant birational map $\rho\colon X_{1}\times\ldots\times X_{r}\dasharrow\mathbb{P}^{n}$;%
\item every $G_{1}\times\ldots\times G_{r}$-equivariant birational automorphism of $X_{1}\times\ldots\times X_{r}$ is biregular;%
\item for every $G_{1}\times\ldots\times G_{r}$-equivariant rational dominant
map
$$
\rho\colon X_{1}\times\ldots\times X_{r}\dasharrow Y,
$$
whose general fiber is
rationally connected, there a commutative diagram
$$
\xymatrix{
X_{1}\times\ldots\times X_{r}\ar@{->}[d]_{\pi}\ar@{-->}[rrrrd]^{\rho}\\
X_{i_{1}}\times\ldots\times X_{i_{k}}\ar@{-->}[rrrr]_{\xi}&&&&Y}%
$$
where $\xi$ is a birational~map, $\pi$ is a natural projection,
and $\{i_{1},\ldots,i_{k}\}\subseteq\{1,\ldots,r\}$.
\end{itemize}
\end{theorem}

Varieties satisfying all hypotheses of
Theorem~\ref{theorem:G-Pukhlikov} do exist (see
Example~\ref{example:Fermat-cubic}).

\begin{example}
\label{example:Valentiner} The simple group $\mathrm{A}_{6}$ is a
group of automorphisms of the sextic
$$
10x^{3}y^{3}+9zx^{5}+9zy^{5}+27z^{6}=45x^{2}y^{2}z^{2}+135xyz^{4}\subset\mathbb{P}^{2}\cong\mathrm{Proj}\Big(\mathbb{C}\big[x,y,z\big]\Big),
$$
which induces an embedding
$\mathrm{A}_{6}\subset\mathrm{Aut}(\mathbb{P}^{2})$. Then
$\mathbb{P}^{2}$ is $\mathrm{A}_{6}$-birationally~super\-ri\-gid
and $\mathrm{lct}(\mathbb{P}^{2},\mathrm{A}_{6})=2$ (see
\cite{Ch07b}). Thus, there is an induced embedding
$$
\mathrm{A}_{6}\times
\mathrm{A}_{6}\cong\Omega\subset\mathrm{Bir}\big(\mathbb{P}^{4}\big)
$$
such that $\Omega$ is not conjugate to any subgroup in
$\mathrm{Aut}(\mathbb{P}^{4})$ by Theorem~\ref{theorem:G-Pukhlikov}.
\end{example}

Let us consider Fano varieties that are close to being
birationally superrigid.

\begin{definition}
\label{definition:birational-rigidity} The Fano variety $X$ is birationally rigid  if%
\begin{itemize}
\item the equality $\mathrm{rk}\,\mathrm{Pic}(X)=1$ holds,
\item the variety $X$ has $\mathbb{Q}$-factorial and terminal singularities,%
\item there is no rational map $\rho\colon X\dasharrow Y$ with
rationally connected fibers such that $0\ne \mathrm{dim}(Y)<\mathrm{dim}(X)$;%
\item there is no birational map
$\rho\colon X\dasharrow Y$ onto a variety $Y\not\cong X$ with
terminal $\mathbb{Q}$-factorial singularities such that
$\mathrm{rk}\,\mathrm{Pic}(Y)=1$.
\end{itemize}
\end{definition}

Arguing as in the proof of Theorem~\ref{theorem:superrigidity}, we
obtain the following result.

\begin{theorem}
\label{theorem:rigidity} The variety  $X$ is birationally rigid
if and only if $\mathrm{rk}\,\mathrm{Pic}(X)=1$,
the variety $X$ has terminal $\mathbb{Q}$-factorial singularities, and
for every linear system $\mathcal{M}$ on
$X$ that does not have fixed components, there is a birational
automorphism $\xi\in\mathrm{Bir}(X)$ such that the log pair
$(X,\ \lambda\,\xi(\mathcal{M}))$
has canonical singularities, where
$K_{X}+\lambda\,\xi(\mathcal{M})\qlineq 0$.
\end{theorem}

Birationally rigid Fano varieties are non-rational~(see
\cite{Ch05umn}).

\begin{definition}
\label{definition:untwisting} Suppose that $X$ is birationally
rigid. A subset
$\Gamma\subset\mathrm{Bir}(X)$~untwists~all~maximal singularities
if for every linear system $\mathcal{M}$~on the variety $X$ that
has no fixed components, there is a birational automorphism
$\xi\in\Gamma$ such that the log pair
$$
\Big(X,\ \lambda\,\xi\big(\mathcal{M}\big)\Big)
$$
has canonical singularities, where $\lambda$ is a rational number
such that $K_{X}+\lambda\,\xi(\mathcal{M})\qlineq 0$.
\end{definition}

If $X$ is birationally rigid and there is
$\Gamma\subset\mathrm{Bir}(X)$ that untwists all maximal
singularities, then the group $\mathrm{Bir}(X)$ is generated by
$\Gamma$ and $\mathrm{Aut}(X)$.%

\begin{definition}
\label{definition:birational-rigidity-universal} The variety $X$
is universally birationally rigid if for any variety $U$, the
variety
$$
X\otimes\mathrm{Spec}\Big(\mathbb{C}\big(U\big)\Big)
$$
is birationally rigid over a field of rational functions
$\mathbb{C}(U)$ of the variety $U$.
\end{definition}

If $X$ is defined over a perfect field, then
Definition~\ref{definition:birational-rigidity} still makes sense
(see \cite{Ma67}, \cite{DoIs06}).

\begin{example}
\label{example:double-quadric} Let $X$ be a smooth Fano threefold
such that there is a double cover
$$
\pi\colon X\longrightarrow Q\subset\mathbb{P}^{3},
$$
where $Q$ is a quadric threefold, and $\pi$ is branched in a
surface $S\subset Q$ of degree $8$. Put
$$
\mathcal{C}=\Big\{C\subset X\ \Big\vert\ C\ \text{is a smooth
curve such that}\ -K_{X}\cdot
 C=1\Big\},%
$$
then $\mathcal{C}$ is a one-dimensional family. For every curve
$C\in\mathcal{C}$ there is a commutative diagram
$$
\xymatrix{
X\ar@{-->}[d]_{\psi_{C}}\ar@{->}[rr]^{\pi}&&Q\ar@{-->}[d]^{\phi_{C}}\\
\mathbb{P}^{2}\ar@{=}[rr]&&\mathbb{P}^{2}}%
$$
where $\phi_{C}$ is a projection from the line $\pi(C)$. General
fiber of the map $\psi_{C}$ is an elliptic curve, the map
$\psi_{C}$ induces an elliptic fibration with a section and an
involution $\tau_{C}\in\mathrm{Bir}(X)$. Then
$$
\psi_{C}\in\mathrm{Aut}\big(X\big)\iff C\subset S,
$$
and $S$ contains no curves in $\mathcal{C}$ if $X$ is general. It
follows from \cite{Is80b} that there is an exact sequence
$$
1\longrightarrow\Gamma\longrightarrow\mathrm{Bir}\big(X\big)\longrightarrow\mathrm{Aut}\big(X\big)\longrightarrow
 1,%
$$
where $\Gamma$ is a free product of subgroups that are generated
by  non-biregular birational involutions constructed above. Hence
the Fano variety $X$ is universally birationally rigid (see
\cite{Is80b}).
\end{example}

Birationally superrigid Fano varieties are universally
birationally rigid.

\begin{definition}
\label{definition:untwisting-universal}  Suppose that $X$ is
universally birationally rigid. A subset
$\Gamma\subset\mathrm{Bir}(X)$ universally untwists all maximal
singularities if for every variety $U$ the induced subset
$$
\Gamma\subset\mathrm{Bir}\big(X\big)\subseteq\mathrm{Bir}\left(X\otimes\mathrm{Spec}\Big(\mathbb{C}\big(U\big)\Big)\right)
$$
untwists all maximal singularities on
$X\otimes\mathrm{Spec}(\mathbb{C}(U))$.
\end{definition}

An identity map universally untwists all maximal singularities if
$X$ is birationally~superrigid.

\begin{remark}
\label{remark:Kollar} Suppose that $X$~is~bira\-ti\-onally rigid,
and $\mathrm{dim}(X)\ne 1$. Let $\Gamma\subseteq\mathrm{Bir}(X)$
be a subset. It follows from \cite{Ko08} that the following
conditions are equivalent:
\begin{itemize}
\item the subset $\Gamma$ universally untwists all maximal singularities;%
\item the subset $\Gamma$ untwists all maximal singularities, and $\mathrm{Bir}(X)$ is countable.%
\end{itemize}
\end{remark}

\begin{example}
\label{example:CPR} In the assumptions of
Example~\ref{example:IHES},~suppose~that~$X$~is~general.~Then
\begin{itemize}
\item the hypersurface $X$ is universally birationally rigid (see \cite{CPR}),%
\item there are involutions
$\tau_{1}, \ldots, \tau_{k}\in\mathrm{Bir}(X)$ such that that the
sequence of groups
$$
1\longrightarrow\big\langle\tau_{1}, \ldots, \tau_{k}\big\rangle\longrightarrow\mathrm{Bir}\big(X\big)\longrightarrow\mathrm{Aut}\big(X\big)\longrightarrow 1%
$$
is exact (see \cite{CPR}, \cite{ChPa05}), where $\langle\tau_{1}, \ldots,
\tau_{k}\rangle$ is a subgroup generated by $\tau_{1},\ldots,\tau_{k}$,%
\item the subgroup $\langle\tau_{1},\ldots,\tau_{k}\rangle$ universally untwists all maximal  singularities (see \cite{CPR}).%
\end{itemize}
All relations between the involutions $\tau_{1},\ldots,\tau_{k}$
are found in \cite{ChPa05}.
\end{example}

Let  $X_{1},\ldots,X_{r}$ be Fano varieties that have at most
$\mathbb{Q}$-factorial and terminal singularities, let
$$
\pi_{i}\colon X_{1}\times\ldots\times X_{i-1}\times X_{i}\times
X_{i+1}\times\ldots\times X_{r}\longrightarrow
X_{1}\times\ldots\times X_{i-1}\times \widehat{X_{i}}\times
X_{i+1}\times\ldots\times X_{r}
$$
be a natural projection, and let $\X_{i}$ be a scheme general
fiber of the projection $\pi_{i}$, which is defined over
$\mathbb{C}(X_{1}\times\ldots\times X_{i-1}\times
\widehat{X_{i}}\times X_{i+1}\times\ldots\times X_{r})$. Suppose
that
$\mathrm{rk}\,\mathrm{Pic}(X_{1})=\ldots=\mathrm{rk}\,\mathrm{Pic}(X_{r})=1$.

\begin{remark}
\label{remark:generic-fiber} There are natural embeddings of
groups
$$
\prod_{i=1}^{r}\mathrm{Bir}\big(X_{i}\big)\subseteq
\Big\langle\mathrm{Bir}\big(\X_{1}\big), \ldots, \mathrm{Bir}\big(\X_{r}\big)
\Big\rangle\subseteq\mathrm{Bir}\Big(X_{1}\times\ldots\times X_{r}\Big).%
$$
\end{remark}

The following generalization of Theorem~\ref{theorem:Pukhlikov}
holds (see \cite{Ch07a}).

\begin{theorem}
\label{theorem:Cheltsov} Suppose that $X_{1},X_{2},\ldots,X_{r}$ are
universally birationally rigid. Then
\begin{itemize}
\item the variety $X_{1}\times\ldots\times X_{r}$ is non-rational
and
$$
\mathrm{Bir}\Big(X_{1}\times\ldots\times X_{r}\Big)=
\Big\langle\mathrm{Bir}\big(\X_{1}\big), \ldots, \mathrm{Bir}\big(\X_{r}\big),
\mathrm{Aut}\Big(X_{1}\times\ldots\times X_{r}\Big)\Big\rangle,%
$$%
\item for every rational dominant map $\rho\colon
X_{1}\times\ldots\times X_{r}\dasharrow Y$, whose general fiber is
rationally connected, there is a subset
$\{i_{1},\ldots,i_{k}\}\subseteq\{1,\ldots,r\}$ and a commutative
diagram
$$
\xymatrix{ X_{1}\times\ldots\times
X_{r}\ar@{->}[d]_{\pi}\ar@{-->}[rr]^{\sigma}&&X_{1}\times\ldots\times
 X_{r}\ar@{-->}[rrd]^{\rho}\\
X_{i_{1}}\times\ldots\times X_{i_{k}}\ar@{-->}[rrrr]_{\xi}&&&&Y}%
$$
where
$\xi$ and $\sigma$  are birational maps, and $\pi$ is a projection,%
\end{itemize}
in the case when the inequalities $\mathrm{lct}(X_{1})\geqslant 1,
\mathrm{lct}(X_{2})\geqslant 1, \ldots,\mathrm{lct}(X_{r})\geqslant 1$ hold.
\end{theorem}

\begin{corollary}
\label{corollary:Cheltsov} Suppose that there are subgroups
$\Gamma_{1}\subseteq\mathrm{Bir}(X_{1}),\ldots,\Gamma_{r}\subseteq\mathrm{Bir}(X_{r})$
that universally untwists all maximal singularities, and
$\mathrm{lct}(X_{1})\geqslant 1, \mathrm{lct}(X_{2})\geqslant 1,
\ldots,\mathrm{lct}(X_{r})\geqslant 1$. Then
$$
\mathrm{Bir}\Big(X_{1}\times\ldots\times X_{r}\Big)=\Big<\prod_{i=1}^{r}\Gamma_{i},\ \mathrm{Aut}\Big(X_{1}\times\ldots\times X_{r}\Big)\Big>.%
$$
\end{corollary}

Thus, the following example is implied by
Examples~\ref{example:IHES} and \ref{example:CPR}.

\begin{example}[{cf Example~\ref{example:CPR}}]
\label{example:41-41} Let $X$ be a general hypersurface in
$\mathbb{P}(1,1,4,5,10)$ of degree~$20$. The sequence
$$
1\longrightarrow\prod_{i=1}^{m}\Big(\mathbb{Z}_{2}\ast\mathbb{Z}_{2}\Big)\longrightarrow\mathrm{Bir}\Big(\underbrace{X\times\ldots\times X}_{m\ \mathrm{times}}\Big)\longrightarrow\mathrm{S}_{m}\longrightarrow 1%
$$
is exact, where $\mathrm{S}_{m}$ is a permutation group, and
$\mathbb{Z}_{2}\ast\mathbb{Z}_{2}$ is the infinite dihedral group.
\end{example}

Suppose now that $X$ is a smooth Fano threefold (see
\cite{IsPr99}). The threefold $X$ lies in $105$ deformation
families (see~\cite{Is77}, \cite{Is78}, \cite{MoMu81},
\cite{MoMu83}, \cite{MoMu03}). Let
$$
\gimel\big(X\big)\in\Big\{1.1,1.2,\ldots,1.17,2.1,\ldots,2.36,3.1,\ldots,3.31,4.1,\ldots,4.13,5.1,\ldots,5.7,5.8\Big\}
$$
be the ordinal number of the deformation type of the threefold $X$
in the notation of Table~\ref{table:Fanos}. The main purpose of
this paper is to prove the following result.

\begin{theorem}
\label{theorem:main} The following assertions hold:
\begin{itemize}
\item $\mathrm{lct}(X)=1/5$, if $\gimel(X)\in\{2.36, 3.29\}$;%
\item $\mathrm{lct}(X)=1/4$, if $\gimel(X)\in\{1.17,\allowbreak 2.28,\allowbreak 2.30,\allowbreak 2.33,\allowbreak 2.35,\allowbreak 3.23,\allowbreak 3.26,\allowbreak 3.30, 4.12\}$;%
\item $\mathrm{lct}(X)=1/3$, if $\gimel(X)\in\{1.16,\allowbreak
2.29,\allowbreak 2.31,\allowbreak 2.34,\allowbreak 3.9,\allowbreak
3.18, \allowbreak 3.19, \allowbreak 3.20, \allowbreak 3.21,
\allowbreak 3.22, \allowbreak 3.24,\allowbreak 3.25, \allowbreak
3.28,\allowbreak 3.31,\allowbreak 4.4,
\allowbreak 4.8, \allowbreak 4.9, \allowbreak 4.10, \allowbreak 4.11,\allowbreak 5.1,\allowbreak 5.2\}$;%
\item $\mathrm{lct}(X)=3/7$, if $\gimel(X)=4.5$;%
\item $\mathrm{lct}(X)=1/2$, if $\gimel(X)\in\{1.11,\allowbreak
1.12,\allowbreak 1.13,\allowbreak 1.14,\allowbreak
1.15,\allowbreak 2.1,\allowbreak 2.3,\allowbreak 2.18,
2.25,\allowbreak 2.27,\allowbreak 2.32,\allowbreak 3.4,
\allowbreak 3.10,\allowbreak 3.11,\allowbreak 3.12,\allowbreak
3.14,\allowbreak 3.15, \allowbreak 3.16,\allowbreak
3.17,\allowbreak 3.24,\allowbreak 3.27,\allowbreak 4.1,\allowbreak
4.2,\allowbreak 4.3,\allowbreak 4.6,\allowbreak 4.7,\allowbreak
5.3,
\allowbreak 5.4,\allowbreak 5.5,\allowbreak 5.6, 5.7, 5.8\}$;%
\item if $X$ is a general threefold in its deformation family,
then
\begin{itemize}
\item $\mathrm{lct}(X)=1/3$ if $\gimel(X)=2.23$,%
\item $\mathrm{lct}(X)=1/2$ if $\gimel(X)\in\{2.5,\allowbreak
2.8,\allowbreak 2.10,\allowbreak 2.11,\allowbreak 2.14,\allowbreak
2.15,\allowbreak 2.19, \allowbreak 2.24,\allowbreak
2.26,\allowbreak 3.2,\allowbreak 3.5,\allowbreak 3.6,\allowbreak
3.7, \allowbreak 3.8,\allowbreak 4.13\}$,%
\item $\mathrm{lct}(X)=2/3$ if $\gimel(X)=3.3$,%
\item $\mathrm{lct}(X)=3/4$ if $\gimel(X)\in\{2.4, 3.1\}$,%
\item $\mathrm{lct}(X)=1$ if $\gimel(X)=1.1$.
\end{itemize}
\end{itemize}
\end{theorem}

The generality condition in Theorem~\ref{theorem:main} can not be
omitted in many cases.

\begin{example}
\label{example:4-13} Suppose that $\gimel(X)=4.13$. Note that this deformation
type was omitted in~\cite{MoMu81}, and it has been discovered only twenty years
later (see \cite{MoMu03}). There is a birational morphism
$$
\alpha\colon X\longrightarrow \mathbb{P}^1\times\mathbb{P}^1\times\mathbb{P}^1
$$
that contracts a surface $E\subset X$ to a curve $C$ such that
$C\cdot F_{1}=C\cdot F_{2}=1$ and $C\cdot F_{3}=3$, where
$$
F_{1}\cong F_{2}\cong F_{3}\cong\P^1\times\P^1
$$
are fibers of three different projections
$\mathbb{P}^1\times\mathbb{P}^1\times\mathbb{P}^1\to\mathbb{P}^{1}$,
respectively. Then
$$
\mathrm{lct}\big(X\big)=1/2
$$
by Theorem~\ref{theorem:main} if $X$ is general. There is a
surface $G\in|F_{1}+F_{2}|$ such that $C\subset G$. Then
$$
-K_{X}\sim 2\bar{G}+E+\bar{F}_{3},
$$
where $\bar{F}_{3}\subset X\supset\bar{G}$ are proper transforms of $F_{3}$ and
$G$, respectively. Then $\mathrm{lct}(X)\leqslant 1/2$.~But
$$
\mathrm{lct}\big(X\big)\leqslant\mathrm{lct}\Big(X,\ 2\bar{G}+E+\bar{F}_{3}\Big)\leqslant 4/9<1/2%
$$
in the case when the intersection $F_{3}\cap C$ consists of a
single point.
\end{example}

We organize the paper in the following way. In
Sections~\ref{section:preliminaries}, \ref{section:cubic-surfaces}
and \ref{section:del-Pezzo} we consider auxiliary results that are
used in the proof of Theorem~\ref{theorem:main}.  In
Section~\ref{section:toric}, we compute global log canonical
thresholds of toric Fano varieties. In
Section~\ref{section:del-Pezzos}, we prove
Theorem~\ref{theorem:main} for smooth Fano threefolds of index
$2$, i.e., for
$$
\gimel\big(X\big)\in\Big\{1.11, 1.12, 1.13, 1.14, 1.15, 2.32, 2.35,
3.27\Big\}.
$$
In Section~\ref{section:rho-2}, we prove Theorem~\ref{theorem:main} in the
case when $\mathrm{rk}\,\mathrm{Pic}(X)=2$.
In Section~\ref{section:rho-3}, we prove Theorem~\ref{theorem:main} in the
case when $\mathrm{rk}\,\mathrm{Pic}(X)=3$.
In Section~\ref{section:rho-4}, we prove Theorem~\ref{theorem:main} in the
case when $\mathrm{rk}\,\mathrm{Pic}(X)\geqslant 4$.
In Section~\ref{section:bounds}, we find upper bounds for
$\mathrm{lct}(X)$ in the case when
$$
\gimel\big(X\big)\in\Big\{1.8,1.9,1.10,2.2, 2.7,
2.9,2.12,2.13,2.16,2.17, 2.20,2.21,2.22,3.13\Big\}.
$$
In Appendix~\ref{section:alpha}, written by J.-P.\,Demailly, the
relation between global log canonical thresholds of smooth Fano
varieties and the $\alpha$-invariants of smooth Fano varieties
introduced by G.\,Tian~in~\cite{Ti87} for the study of the
existence of K\"ahler--Einstein metrics is established. In
Appendix~\ref{section:appendix}, we put Table~\ref{table:Fanos}
that contains the list of all smooth Fano threefolds together with
the known values and bounds for their global log canonical
thresholds.

We use a standard notation $D_1\sim D_2$ (resp., $D_1\qlineq D_2$)
for the linearly equivalent (resp., $\Q$-linearly equivalent)
divisors (resp., $\Q$-divisors). If a divisor (resp., a
$\Q$-divisor) $D$ is linearly equivalent to a line bundle
$\mathcal{L}$ (resp., is $\Q$-linearly equivalent to a divisor
that is linearly equivalent to $\mathcal{L}$), we write
$D\sim\mathcal{L}$ (resp., $D\qlineq\mathcal{L}$). Recall that
$\Q$-linear  equivalence coincides with numerical equivalence in
the case of Fano varieties.

The projectivisation $\P_{Y}(\mathcal{E})$ of a vector bundle $\mathcal{E}$
on a variety $Y$ is the variety of hyperplanes in the fibers of
$\mathcal{E}$.
The symbol $\mathbb{F}_{n}$ denotes the
Hirzebruch surface
$\mathbb{P}(\mathcal{O}_{\mathbb{P}^1}\oplus\mathcal{O}_{\mathbb{P}^1}(n))$.

We always refer to a smooth Fano threefold $X$ using the ordinal
number $\gimel(X)$ introduced in Table~\ref{table:Fanos}.

We are very grateful to J.-P.\,Demailly for writing
Appendix~\ref{section:alpha}, and to C.\,Boyer, A.\,Iliev,
P.\,Jahnke, A.-S.\,Kaloghiros, A.\,G.\,Kuznetsov, J.\,Park and
Yu.\,G.\,Prokhorov for useful discussions. The first author would
like to express his gratitude to IHES (Bures-sur-Yvette, France)
and MPIM (Bonn, Germany) for hospitality.


\section{Preliminaries}
\label{section:preliminaries}

Let $X$ be a variety with log terminal singularities. Let us
consider a $\mathbb{Q}$-divisor
$$
B_{X}=\sum_{i=1}^{r}a_{i}B_{i},
$$
where $B_{i}$ is a prime Weil divisor on the variety $X$, and
$a_{i}$ is an arbitrary non-negative rational number. Suppose that
$B_{X}$ is a $\mathbb{Q}$-Cartier divisor such that
$B_{i}\neq B_{j}$ for $i\neq j$.

Let $\pi\colon\bar{X}\to X$ be a birational morphism such that
$\bar{X}$ is smooth. Put
$$
B_{\bar{X}}=\sum_{i=1}^{r}a_{i}\bar{B}_{i},
$$
where $\bar{B}_{i}$ is a proper transform of the divisor $B_{i}$
on the variety $\bar{X}$. Then
$$
K_{\bar{X}}+B_{\bar{X}}\qlineq\pi^{*}\Big(K_{X}+B_{X}\Big)+\sum_{i=1}^{n}c_{i}E_{i},
$$
where $c_{i}\in\mathbb{Q}$, and $E_{i}$ is an exceptional divisor
of the morphism $\pi$. Suppose that
$$
\left(\bigcup_{i=1}^{r}\bar{B}_{i}\right)\bigcup\left(\bigcup_{i=1}^{n}E_{i}\right)
$$
is a divisor with simple normal crossing. Put
$$
B^{\bar{X}}=B_{\bar{X}}-\sum_{i=1}^{n}c_{i}E_{i}.
$$

\begin{definition}
\label{definition:log canonical-singularities} The singularities
of $(X, B_{X})$ are log canonical (resp., log terminal) if
\begin{itemize}
\item the inequality $a_{i}\leqslant 1$ holds (resp., the inequality $a_{i}<1$
holds),%
\item the inequality $c_{j}\geqslant -1$ holds (resp., the inequality
$c_{j}>-1$ holds),%
\end{itemize}
for every $i=1,\ldots,r$ and $j=1,\ldots,n$.
\end{definition}

One can show that Definition~\ref{definition:log
canonical-singularities} does not depend on the choice of the
morphism $\pi$. Put
$$
\mathrm{LCS}\Big(X, B_{X}\Big)=\left(\bigcup_{a_{i}\geqslant
 1}B_{i}\right)\bigcup\left(\bigcup_{c_{i}\leqslant -1}\pi\big(E_{i}\big)\right)\subsetneq X,%
$$
then $\mathrm{LCS}(X,B_{X})$ is called the locus of log canonical
singularities of the log pair $(X, B_{X})$.

\begin{definition}
\label{definition:log canonical-center} A proper irreducible
subvariety $Y\subsetneq X$ is said to be a center of log canonical
singularities of the log pair $(X, B_{X})$ if one of the following
conditions is satisfied:
\begin{itemize}
\item either the inequality $a_{i}\geqslant 1$ holds and $Y=B_{i}$,%
\item or the inequality $c_{i}\leqslant -1$ holds and $Y=\pi(E_{i})$
for some choice of the birational morphism $\pi\colon\bar{X}\to X$.
\end{itemize}
\end{definition}

Let $\mathbb{LCS}(X, B_{X})$ be the set of all centers of log
canonical singularities of $(X, B_{X})$. Then
$$
Y\in\mathbb{LCS}\Big(X, B_{X}\Big)\Longrightarrow Y\subseteq\mathrm{LCS}\Big(X, B_{X}\Big)%
$$
and $\mathbb{LCS}(X, B_{X})=\varnothing\iff\mathrm{LCS}(X,
B_{X})=\varnothing$ $\iff$ the log pair $(X,B_{X})$ is log
terminal.

\begin{remark}
\label{remark:hyperplane-reduction} Let $\mathcal{H}$ be a linear
system on $X$ that has  no base points, let $H$ be a sufficiently
general divisor in the linear system $\mathcal{H}$, and let
$Y\subsetneq X$ be an irreducible subvariety. Put
$$
Y\Big\vert_{H}=\sum_{i=1}^{m}Z_{i},
$$
where $Z_{i}\subset H$ is an irreducible subvariety. It follows
from Definition~\ref{definition:log canonical-center} (cf.
Theorem~\ref{theorem:adjunction})~that
$$
Y\in\mathbb{LCS}\Big(X,\ B_{X}\Big)\iff \Big\{Z_{1},\ldots,Z_{m}\Big\}\subseteq \mathbb{LCS}\Big(H,\ B_{X}\Big\vert_{H}\Big).%
$$
\end{remark}

\begin{example}
\label{example:smooth-point-and-log-pull-back} Let $\alpha\colon
V\to X$ be a blow up of a smooth point $O\in X$. Then
$$
B_{V}\qlineq\alpha^{*}\big(B_{X}\big)-\mathrm{mult}_{O}\big(B_{X}\big)E,
$$
where $\mathrm{mult}_{O}(B_{X})\in\mathbb{Q}$, and $E$ is the
exceptional divisor of the blow up $\alpha$. Then
$$
\mathrm{mult}_{O}\big(B_{X}\big)>1
$$
if the log pair $(X,B_{X})$ is not log canonical at the point $O$.
Put
$$
B^{V}=B_{V}+\Big(\mathrm{mult}_{O}\big(B_{X}\big)-\mathrm{dim}\big(X\big)+1\Big)E,
$$
and suppose that
$\mathrm{mult}_{O}(B_{X})\geqslant\mathrm{dim}(X)-1$. Then $O\in
\mathbb{LCS}(X, B_{X})$ if and only if
\begin{itemize}
\item either $E\in\mathbb{LCS}(V, B^{V}), \mbox{ i.\,e. }
\mathrm{mult}_{O}(B_{X})\geqslant\mathrm{dim}(X)$,%
\item or there is a subvariety $Z\subsetneq E$ such that $Z\in\mathbb{LCS}(V, B^{V})$.%
\end{itemize}
\end{example}

The locus $\mathrm{LCS}(X,B_{X})\subset X$ can be equipped with a
scheme structure (see \cite{Na90}, \cite{Sho93}). Put
$$
\mathcal{I}\Big(X, B_{X}\Big)=\pi_{*}\Big(\sum_{i=1}^{n}\lceil
c_{i}\rceil
 E_{i}-\sum_{i=1}^{r}\lfloor a_{i}\rfloor \bar{B}_{i}\Big),%
$$
and let $\mathcal{L}(X, B_{X})$ be a subscheme that corresponds to
the ideal sheaf $\mathcal{I}(X, B_{X})$.

\begin{definition}
\label{definition:log canonical-subscheme} For the log pair
$(X,B_{X})$, we say that
\begin{itemize}
\item the subscheme $\mathcal{L}(X, B_{X})$ is the subscheme of
log canonical singularities of
 $(X,B_{X})$,%
\item the ideal sheaf $\mathcal{I}(X, B_{X})$ is the multiplier ideal sheaf of $(X,B_{X})$.%
\end{itemize}
\end{definition}

It follows from the construction of the subscheme $\mathcal{L}(X,
B_{X})$ that
$$
\mathrm{Supp}\Big(\mathcal{L}\big(X,
B_{X}\big)\Big)=\mathrm{LCS}\Big(X, B_{X}\Big)\subset X.
$$

The following result is the Nadel--Shokurov vanishing theorem (see
\cite{Sho93},~\cite[Theorem~9.4.8]{La04}).

\begin{theorem}
\label{theorem:Shokurov-vanishing} Let $H$ be a nef and big
$\mathbb{Q}$-divisor on $X$ such that
$$
K_{X}+B_{X}+H\qlineq D
$$
for some Cartier divisor $D$ on the variety $X$. Then for every
$i\geqslant 1$
$$
H^{i}\Big(X,\ \mathcal{I}\big(X, B_{X}\big)\otimes D\Big)=0.
$$
\end{theorem}

For every Cartier divisor $D$ on the variety $X$, let us consider
the exact sequence of sheaves
$$
0\longrightarrow\mathcal{I}\big(X, B_{X}\big)\otimes
 D\longrightarrow\mathcal{O}_{X}\big(D\big)\longrightarrow\mathcal{O}_{\mathcal{L}(X,\,
 B_{X})}\big(D\big)\longrightarrow 0,%
$$
and let us consider the corresponding exact sequence of cohomology
groups
$$
H^{0}\Big(\mathcal{O}_{X}\big(D\big)\Big)\longrightarrow
H^{0}\Big(\mathcal{O}_{\mathcal{L}(X,\,
 B_{X})}\big(D\big)\Big)\longrightarrow H^{1}\Big(\mathcal{I}\big(X, B_{X}\big)\otimes D\Big).%
$$

\begin{theorem}
\label{theorem:connectedness} Suppose that $-(K_{X}+B_{X})$ is nef
and big. Then $\mathrm{LCS}(X, B_{X})$ is connected.
\end{theorem}

\begin{proof}
Put $D=0$. Then it follows from
Theorem~\ref{theorem:Shokurov-vanishing} that the sequence
$$
\mathbb{C}=H^{0}\Big(\mathcal{O}_{X}\Big)\longrightarrow
H^{0}\Big(\mathcal{O}_{\mathcal{L}(X,\,
 B_{X})}\Big)\longrightarrow H^{1}\Big(\mathcal{I}\big(X, B_{X}\big)\Big)=0%
$$
is exact. Thus, the
locus
$$
\mathrm{LCS}\Big(X,\ B_{X}\Big)=\mathrm{Supp}\Big(\mathcal{L}\big(X,\ B_{X}\big)\Big)%
$$
is connected.
\end{proof}

Let us consider few elementary applications of
Theorem~\ref{theorem:connectedness} (cf.
Example~\ref{example:del-Pezzos}).

\begin{lemma}
\label{lemma:plane} Suppose that $\mathrm{LCS}(X,
B_{X})\ne\varnothing$, where $X\cong\mathbb{P}^{n}$, and
$$
B_{X}\qlineq -\lambda K_{X}
$$
for some rational number $0<\lambda<n/(n+1)$. Then
\begin{itemize}
\item the inequality $\mathrm{dim}(\mathrm{LCS}(X, B_{X}))\ge 1$
holds,
\item the subscheme $\mathcal{L}(X,B_{X})$ does not contain isolated zero-dimensional components.%
\end{itemize}
\end{lemma}

\begin{proof}
Suppose that there is a point $O\in X$ such that
$$
\mathrm{LCS}\Big(X, \lambda B_{X}\Big)=O\cup\Sigma,
$$
where $\Sigma\subset X$ is a possibly empty subset such that
$O\not\in X$.

Let $H$ be a general line in $X\cong\mathbb{P}^{2}$. Then
the locus
$$
\mathrm{LCS}\Big(X, \lambda B_{X}+H\Big)=O\cup H\cup\Sigma
$$
is disconnected. But the divisor $-(K_{X}+\lambda B_{X}+H)$ is
ample, which contradicts Theorem~\ref{theorem:connectedness}.
\end{proof}

\begin{lemma}
\label{lemma:P3} Suppose that $\mathrm{LCS}(X,
B_{X})\ne\varnothing$, where $X\cong\mathbb{P}^{3}$, and
$$
B_{X}\qlineq -\lambda K_{X}
$$
for some rational number $0<\lambda<1/2$. Then
$\mathbb{LCS}(X, B_{X})$ contains a surface.
\end{lemma}

\begin{proof}
Suppose that $\mathbb{LCS}(X, B_{X})$ contains no surfaces. Let
$S\subset\mathbb{P}^{3}$ be a general plane. The~locus
$$
\LCS\Big(\P^3,\ B_{X}+S\Big)
$$
is connected by Theorem~\ref{theorem:connectedness}. Then $(S,B_{X}\vert_{S})$
is not log terminal by Remark~\ref{remark:hyperplane-reduction}. But the locus
$$
\LCS\Big(S,\ B_{X}\Big\vert_{S}\Big)
$$
consists of finitely many points, which is impossible by
Lemma~\ref{lemma:plane}.
\end{proof}

\begin{lemma}
\label{lemma:quadric-curves} Suppose that $\mathrm{LCS}(X,
B_{X})\ne\varnothing$, where $X$ is a smooth quadric threefold in~$\P^4$,~and
$$
B_{X}\qlineq -\lambda K_{X}
$$
for some rational number $0<\lambda<1/2$. Then
$\mathbb{LCS}(X, B_{X})$ contains a surface.
\end{lemma}

\begin{proof}
Let $L\subset X$ be a general line, let $P_1\in L\ni P_2$ be two general
points, let $H_1$ and $H_2$ be the hyperplane sections of $X$ that are tangent
to $X$ at the points $P_1$ and $P_2$, respectively. Then
$$
\LCS\left(X,\ \lambda B_X+\frac{3}{4}\Big(H_1+H_2\Big)\right)=\LCS\Big(X,\ \lambda B_X\Big)\cup L%
$$
is disconnected, which is impossible by
Theorem~\ref{theorem:connectedness}.
\end{proof}

\begin{remark}
One can prove Lemmas~\ref{lemma:P3},~\ref{lemma:quadric-curves} and
\ref{lemma:P1xP2} using the following trick. Suppose that
$$
B_{X}\qlineq -\lambda K_{X}
$$
for some $\lambda\in\mathbb{Q}$ such that $0<\lambda<1/2$, where
$X$ is either $\mathbb{P}^{3}$, or $\P^1\times\P^2$, or a smooth
quadric threefold, and the set $\mathbb{LCS}(X, B_{X})$ contains
no surfaces. Then
$$
\LCS\Big(X,\ B_{X}\Big)\subseteq\Sigma,
$$
where $\Sigma\subset X$ is a (possibly reducible) curve. For a
general $\phi\in\mathrm{Aut}(X)$ we have
$$
\phi\big(\Sigma\big)\cap \Sigma=\varnothing,
$$
which implies that $\LCS(X, \phi(B_{X}))\cap\LCS(X,
B_{X})=\varnothing$. But
$$
\LCS\Big(X,\ \phi\big(B_{X}\big)+B_{X}\Big)=\LCS\Big(X,\ \phi\big(B_{X}\big)\Big)\bigcup\LCS\Big(X,\ B_{X}\Big)%
$$
whenever $\phi$ is sufficiently general. The latter contradicts
Theorem~\ref{theorem:connectedness} since $\lambda<1/2$.
\end{remark}

\begin{lemma}
\label{lemma:V7} Suppose that $\mathrm{LCS}(X,
B_{X})\ne\varnothing$, where $X$ is a blow up of $\P^3$ in one
point,~and
$$
B_{X}\qlineq -\lambda K_{X}
$$
for some rational number $0<\lambda<1/2$. Then
$\mathbb{LCS}(X, B_{X})$ contains a surface.
\end{lemma}

\begin{proof}
Suppose that the set $\mathbb{LCS}(X, B_{X})$ contains no
surfaces. Let
$$
\alpha\colon X\longrightarrow\P^3
$$
be the blow up of a point, and let $E$ be the exceptional divisor
of $\alpha$. In the case when
$$
\LCS\big(X,\ \lambda B_{X}\Big)\not\subseteq E,
$$
we can apply Lemma~\ref{lemma:P3} to the pair
$(\mathbb{P}^{3},\alpha(B_{X}))$ to get a contradiction. Hence
$\LCS(X, B_{X})\subseteq E$.

Let $H\subset\P^3$ a general hyperplane, and let $H_1\subset\P^3\supset H_2$ be
general hyperplanes that pass through $\alpha(E)$. Denote by $\bar{H}$,
$\bar{H}_1$ and $\bar{H}_2$ the proper transforms of these planes on $X$. Then
$$
\LCS\left(X,\
B_{X}+\frac{1}{2}\Big(\bar{H}_1+\bar{H}_2+2\bar{H}\Big)\right)
$$
is disconnected, which is impossible by
Theorem~\ref{theorem:connectedness}.
\end{proof}

\begin{lemma}
\label{lemma:singular-quadric-threefold} Suppose that $X$ is a cone
in $\mathbb{P}^{4}$ over a smooth quadric surface, and
$$
B_{X}\qlineq -\lambda K_{X}
$$
for some rational number $0<\lambda<1/3$. Then $\mathbb{LCS}(X,
B_{X})=\varnothing$.
\end{lemma}

\begin{proof}
Suppose that $\mathbb{LCS}(X, B_{X})\ne \varnothing$. Let
$S\subset X$ be a general hyperplane section. Then
$$
\LCS\Big(S,\ B_{X}\Big\vert_{S}\Big)=\varnothing,
$$
because $S\cong\mathbb{P}^{1}\times\mathbb{P}^{1}$ and
$\mathrm{lct}(\mathbb{P}^{1}\times\mathbb{P}^{1})=1/2$ (see
Example~\ref{example:del-Pezzos}).

One has $|\mathrm{LCS}(X, B_{X})|<+\infty$ by
Remark~\ref{remark:hyperplane-reduction}. Then the~locus
$$
\LCS\Big(X,\ B_{X}+S\Big)
$$
is disconnected, which contradicts
Theorem~\ref{theorem:connectedness}.
\end{proof}

The following result is a corollary
Theorem~\ref{theorem:Shokurov-vanishing}
(see~\cite[Theorem~4.1]{Na90}).

\begin{lemma}\label{lemma:rational-tree}
Suppose that $-(K_{X}+B_{X})$ is nef and big and
$\mathrm{dim}(\mathrm{LCS}(X, B_{X}))=1$. Then
\begin{itemize}
\item the locus $\mathrm{LCS}(X, B_X)$ is a connected union of smooth rational curves,%
\item every two irreducible components of the locus
$\mathrm{LCS}(X, B_X)$ meet in at most one
 point,%
\item every intersecting irreducible components of the locus
$\mathrm{LCS}(X, B_X)$ meet
 transversally,%
\item no three irreducible components of the locus $\mathrm{LCS}(X, B_X)$ meet in one point,%
\item the locus $\mathrm{LCS}(X, B_X)$ does not contain a cycle of smooth rational curves.%
\end{itemize}
\end{lemma}

\begin{proof}
Arguing as in the proof of Theorem~\ref{theorem:connectedness}, we
see that the locus $\mathrm{LCS}(X, B_X)$ is a connected tree of
smooth rational curves with simple normal crossings.
\end{proof}

To consider another application of
Theorem~\ref{theorem:connectedness}, we need the following result
(see \cite{Pu98a}).

\begin{lemma}
\label{lemma:Cheltsov-Pukhlikov}%
Suppose that $X$ is a hypersurface in $\P^m$, and
$$
B_{X}\qlineq\mathcal{O}_{\mathbb{P}^{m}}(1)\Big\vert_{X}.
$$
Let $S\subsetneq X$ be an
irreducible subvariety such that $\mathrm{dim}(S)\geqslant k$.
Then
$$
\mathrm{mult}_{S}\big(B_{X}\big)\leqslant 1.
$$
\end{lemma}

\begin{proof}
See~\cite{Pu98a}.
\end{proof}

Let us consider another  simple application of
Theorem~\ref{theorem:connectedness} and
Lemma~\ref{lemma:Cheltsov-Pukhlikov}.

\begin{lemma}
\label{lemma:singular-cubic-threefold} Let $X$ be a cubic
hypersurface in $\mathbb{P}^{4}$ such that
$|\mathrm{Sing}(X)|<+\infty$. Suppose that
$$
B_{X}\qlineq -K_{X},
$$
and there is a positive rational number $\lambda<1/2$ such that
$\mathrm{LCS}(X,\lambda B_{X})\ne\varnothing$. Then
$$
\mathrm{LCS}\Big(X,\ \lambda B_{X}\Big)=L,
$$
where $L$ is a line in $X\subset\mathbb{P}^{4}$ such that $L\cap\mathrm{Sing}(X)\ne\varnothing$.%
\end{lemma}

\begin{proof}
Let $S$ be a general hyperplane section of $X$. Then
$$
S\cup\mathrm{LCS}\Big(X,\ \lambda B_{X}\Big)\subseteq \mathrm{LCS}\Big(X,\ \lambda B_{X}+S\Big),%
$$
which implies that $\dim(\mathrm{LCS}(X,\lambda B_{X}))\ge 1$ by
Theorem~\ref{theorem:connectedness}. Then
$$
\mathrm{LCS}\Big(S,\ \lambda B_{X}\Big\vert_{S}\Big)\ne\varnothing
$$
by Remark~\ref{remark:hyperplane-reduction}. But
$|\mathrm{LCS}(S,\ \lambda B_{X}\vert_{S})|<+\infty$ by
Lemma~\ref{lemma:Cheltsov-Pukhlikov}. There is a point $O\in S$
such that
$$
\mathrm{LCS}\Big(S,\ \lambda B_{X}\Big\vert_{S}\Big)=O%
$$
by Theorem~\ref{theorem:connectedness}. Therefore, there is a line $L\subset X$
such that $\mathrm{LCS}(X,\lambda B_{X})=L$ by
Remark~\ref{remark:hyperplane-reduction}.

Arguing as in the proof of Lemma~\ref{lemma:Cheltsov-Pukhlikov}, we see that
$L\cap\mathrm{Sing}(X)\ne\varnothing$.
\end{proof}

Similar to Lemma~\ref{lemma:singular-cubic-threefold}, one can
prove the following result.

\begin{lemma}
\label{lemma:double-covers} Suppose that there is a double cover $\tau\colon
X\to\mathbb{P}^{3}$ branched over an irreducible reduced quartic surface
$R\subset\mathbb{P}^{3}$ that has at most ordinary double points, the
equivalence
$$
B_{X}\qlineq -\lambda K_{X}
$$
holds and $\mathrm{LCS}(X,B_{X})\ne\varnothing$, where
$\lambda<1/2$. Then
$\mathrm{Sing}(X)\ne\varnothing$ and
$$
\mathrm{LCS}\Big(X,\ B_{X}\Big)=L,
$$
where $L$ is an irreducible curve on $X$ such that $-K_{X}\cdot
L=2$ and $L\cap\mathrm{Sing}(X)\ne\varnothing$.
\end{lemma}

\begin{proof}
We have $-K_{X}\sim 2H$, where $H$ is a Cartier divisor on $X$
such that
$$
H\sim\tau^{*}\Big(\mathcal{O}_{\mathbb{P}^{3}}\big(1\big)\Big).
$$
The variety $X$ is a Fano threefold, and $H^{3}=2$. Then
$$
\mathrm{LCS}\Big(X,\ B_{X}+H\Big)
$$
must be connected by Theorem~\ref{theorem:connectedness}.
Thus, there is a curve
$$
C\in\mathbb{LCS}\Big(X,\ B_{X}\Big),
$$
which implies that $\mathrm{mult}_{C}(B_{X})\geqslant
1/\lambda>2$.

Let $S$ be a general surface in $|H|$. Put $B_{S}=B_{X}\vert_{S}$. Then
$$
-K_{S}\sim H\big\vert_{S}\qlineq \frac{1}{\lambda}B_{S},
$$
but the log pair $(S, B_{S})$ is not log canonical in every point
of the intersection $S\cap\mathrm{LCS}(X,B_{X})$.

The surface $H$ is a smooth surface in $\mathbb{P}(1,1,1,2)$ of degree $4$.

Let $P$ be any point in $S\cap\mathrm{LCS}(X,B_{X})$. Then there
is a birational morphism
$$
\rho\colon S\longrightarrow\bar{S}
$$
such that $\bar{S}$ is a cubic surface in $\mathbb{P}^{3}$ and $\rho$ is an
isomorphism in a neighborhood of $P$. Then
$$
\Big(\bar{S},\ \rho\big(B_{S}\big)\Big)
$$
is not log terminal at the point $\rho(P)$. Thus, we have
$\mathrm{LCS}(\bar{S},\rho(B_{S}))\ne\varnothing$. But
$$
\frac{1}{\lambda}\rho(B_{S})\qlineq-K_{\bar{S}}\sim\mathcal{O}_{\mathbb{P}^{3}}\big(1\big)\Big\vert_{\bar{S}},
$$
which implies that $\mathrm{LCS}(\bar{S},\rho(B_{S}))$ consists of one point by
Lemma~\ref{lemma:Cheltsov-Pukhlikov} and Theorem~\ref{theorem:connectedness}.
Then
$$
P=S\cap C=S\cap\mathrm{LCS}\Big(X,B_{X}\Big)
$$
if the point $P$ is sufficiently general. Therefore, we see that
$$
\mathrm{LCS}\Big(X,B_{X}\Big)=C,
$$
the curve $C$ is irreducible and $-K_{X}\cdot C=2$. Then
$\tau(C)\subset\mathbb{P}^{3}$ is a line.

Suppose that $C\cap\mathrm{Sing}(X)=\varnothing$. Let us derive a
contradiction.

Suppose that $\tau(C)\subset R$. Take a general point $O\in C$.
Let
$$
\tau\big(O\big)\in \Pi\subset\mathbb{P}^{3}
$$
be a plane that is tangent to $R$ at the point $\tau(O)$. Arguing
as in the proof of Lemma~\ref{lemma:Cheltsov-Pukhlikov}, we see
that $R\vert_{\Pi}$ is reduced along $\tau(C)$, because
$\tau(C)\cap\mathrm{Sing}(R)=\varnothing$. Fix a general line
$$
\Gamma\subset\Pi\subset\mathbb{P}^{3}
$$
such that $\tau(O)\in\Gamma$. Let $\bar{\Gamma}\subset X$ be an
irreducible curve such that $\tau(\bar{\Gamma})=\Gamma$. Then
$$
\bar{\Gamma}\not\subseteq\mathrm{Supp}\Big(B_{X}\Big),
$$
because $\Gamma$ spans a dense subset in $\mathbb{P}^{3}$ when we
vary the point $O\in C$ and the line $\Gamma\subset\Pi$.
Note that $H\cdot\bar{\Gamma}$ equals either $1$ or $2$, and
$\mult_O(\bar{\Gamma})=2$ in the case when $H\cdot\bar{Gamma}=2$. Hence
$$
H\cdot\bar{\Gamma}>2\lambda H\cdot\bar{\Gamma}=\bar{\Gamma}\cdot B_{X}
\geqslant\mathrm{mult}_{O}\big(\bar{\Gamma}\big)\mathrm{mult}_{C}
\big(B_{V}\big)\geqslant H\cdot\bar{\Gamma},%
$$
which is a contradiction. Thus, we see that $\tau(C)\not\subset
R$.

There is an irreducible reduced curve $\bar{C}\subset X$ such that
$$
\tau\big(\bar{C}\big)=\tau\big(C\big)\subset\mathbb{P}^{3}
$$
and $\bar{C}\ne C$. Let $Y$ be a general surface in $|H|$ that
passes through the curves $\bar{C}$ and $C$. Then $Y$ is smooth, because
$C\cap\mathrm{Sing}(X)=\varnothing$, and
$$
\bar{C}\cdot\bar{C}=C\cdot C=-2
$$
on the surface $Y$.

By construction, we have $Y\not\subset\mathrm{Supp}(B_{X})$. Put
$B_{Y}=B_{X}\vert_{Y}$. Then
$$
B_{Y}=\mathrm{mult}_{\bar{C}}\big(B_{X}\big)\bar{C}+\mathrm{mult}_{C}\big(B_{X}\big)C+\Delta%
$$
where $\Delta$ is an effective $\mathbb{Q}$-divisor on the surface
$Y$ such that
$\bar{C}\not\subset\mathrm{Supp}(\Delta)\not\supset C$. But
$$
B_{Y}\qlineq 2\lambda\Big(\bar{C}+C\Big),
$$
which implies, in particular, that
$$
\Big(2\lambda-\mathrm{mult}_{C}\big(B_{X}\big)\Big)C\cdot C=\Big(\mathrm{mult}_{\bar{C}}\big(B_{X}\big)-2\lambda\Big)\bar{C}\cdot C+\Delta\cdot C\geqslant \Big(\mathrm{mult}_{\bar{C}}\big(B_{X}\big)-2\lambda\Big)\bar{C}\cdot C\geqslant 0,%
$$
because $\Delta\cdot C\geqslant 0$ and $\bar{C}\cdot C\geqslant 0$. Then
$\mathrm{mult}_{\bar{C}}(B_{X})\geqslant 2\lambda$, because $C\cdot C<0$. Thus,
we have
$$
-\Delta\qlineq\Big(\mathrm{mult}_{\bar{C}}\big(B_{X}\big)-2\lambda\Big)\bar{C}+\Big(\mathrm{mult}_{C}\big(B_{X}\big)-2\lambda\Big)C%
$$
which is impossible, because $\mathrm{mult}_{C}(B_{X})>2\lambda$ and $Y$ is
projective.
\end{proof}

One can generalize Theorem~\ref{theorem:connectedness} in the
following way (see \cite[Lemma~5.7]{Sho93}).

\begin{theorem}
\label{theorem:log-adjunction-connectedness-theorem} Let
$\psi\colon X\to Z$ be a morphism. Then the set
$$
\mathrm{LCS}\Big(\bar{X},\ B^{\bar{X}}\Big)
$$
is connected in a neighborhood of every fiber of the morphism
$\psi\circ\pi\colon X\to Z$ in the case when
\begin{itemize}
\item the morphism $\psi$ is surjective and has connected fibers,%
\item the divisor $-(K_{X}+B_{X})$ is nef and big with respect to $\psi$.%
\end{itemize}
\end{theorem}

Let us consider one important application of
Theorem~\ref{theorem:log-adjunction-connectedness-theorem} (see
\cite[Theorem~5.50]{KoMo98}).

\begin{theorem}
\label{theorem:adjunction} Suppose that $B_{1}$ is a Cartier
divisor, $a_{1}=1$, and $B_{1}$ has at most log terminal
singularities. Then the following assertions are equivalent:
\begin{itemize}
\item the log pair $(X, B_{X})$ is log canonical in  a neighborhood of the divisor $B_{1}$;%
\item the singularities of the log pair
$(B_{1},\sum_{i=2}^{r}a_{i}B_{i}\vert_{B_{1}})$ are log
 canonical.%
\end{itemize}
\end{theorem}

The simplest application of Theorem~\ref{theorem:adjunction} is a non-obvious
result (see \cite[Corollary~5.57]{KoMo98}).

\begin{lemma}
\label{lemma:adjunction} Suppose that $\mathrm{dim}(X)=2$ and
$a_{1}\leqslant 1$. Then
$$
\Big(\sum_{i=2}^{r}a_{i}B_{i}\Big)\cdot B_{1}>1
$$
whenever $(X,B_{X})$ is not log canonical at some point $O\in B_{1}$
such that $O\not\in\mathrm{Sing}(X)\cup\mathrm{Sing}(B_{1})$.
\end{lemma}

\begin{proof}
Suppose that  $(X,B_{X})$ is not log canonical in~a point $O\in B_{1}$. By
Theorem~\ref{theorem:adjunction}, we have
$$
\Big(\sum_{i=2}^{r}a_{i}B_{i}\Big)\cdot B_{1}\geqslant
\mathrm{mult}_{O}\Big(\sum_{i=2}^{r}a_{i}B_{i}\Big\vert_{B_{1}}\Big)>1
$$
if $O\not\in\mathrm{Sing}(X)\cup\mathrm{Sing}(B_{1})$, because $(X,
B_{1}+\sum_{i=2}^{r}a_{i}B_{i})$ is not log canonical at the point $O$.
\end{proof}

Let us consider another application of Theorem~\ref{theorem:adjunction} (cf.
Lemma~\ref{lemma:lct-product}).

\begin{lemma}
\label{lemma:lct-P1-product} Suppose that $X$ is a Fano variety
with log terminal singularities. Then
$$
\mathrm{lct}\Big(\mathbb{P}^{1}\times
X\Big)=\mathrm{min}\left(\frac{1}{2},\
 \mathrm{lct}\big(X\big)\right).%
$$
\end{lemma}

\begin{proof}
The inequalities $1/2\geqslant\mathrm{lct}(V\times
U)\leqslant\mathrm{lct}(X)$ are obvious. Suppose that
$$
\mathrm{lct}\Big(\mathbb{P}^{1}\times
X\Big)<\mathrm{min}\left(\frac{1}{2},\
 \mathrm{lct}\big(X\big)\right),%
$$
and let us show that this assumption leads to a contradiction.

There is an effective $\mathbb{Q}$-divisor $D\qlineq
-K_{\mathbb{P}^{1}\times X}$ such that the log pair
$$
\Big(\mathbb{P}^{1}\times X,\ \lambda D\Big)
$$
is not log canonical in some point $P\in \mathbb{P}^{1}\times X$,
where $\lambda<\min(1/2, \mathrm{lct}(X))$.

Let $F$ be a fiber of the projection $\mathbb{P}^{1}\times
X\to\mathbb{P}^{1}$ such that $P\in F$. Then
$$
D=\mu F+\Omega,
$$
where $\Omega$ is an effective $\mathbb{Q}$-divisor on
$\mathbb{P}^{1}\times X$ such that
$F\not\subset\mathrm{Supp}(\Omega)$.

Let $L$ be a general fiber of the projection $\mathbb{P}^{1}\times
X\to X$. Then
$$
2=D\cdot L=\mu+\Omega\cdot L\geqslant\mu,
$$
which implies that the log pair $(\mathbb{P}^{1}\times X,
F+\lambda\Omega)$ is not log canonical at the point $P$. Then
$$
\Big(F,\ \lambda\Omega\Big\vert_{F}\Big)
$$
is not log canonical at the point $P$ by
Theorem~\ref{theorem:adjunction}. But
$$
\Omega\Big\vert_{F}\qlineq D\Big\vert_{F}\qlineq -K_{F},
$$
which is impossible, because $X\cong F$ and
$\lambda<\mathrm{lct}(X)$.
\end{proof}

Let $P$ be a point in $X$. Let us consider an effective  divisor
$$
\Delta=\sum_{i=1}^{r}\eps_{i}B_{i}\qlineq B_{X},%
$$
where $\eps_{i}$ is a non-negative rational number. Suppose
that
\begin{itemize}
\item the divisor $\Delta$ is a $\mathbb{Q}$-Cartier divisor,%
\item the equivalence $\Delta\qlineq B_{X}$ holds,%
\item the log pair  $(X, \Delta)$ is log canonical in the point $P\in X$.%
\end{itemize}

\begin{remark}
\label{remark:convexity} Suppose that $(X, B_{X})$ is not log
canonical in the point $P\in X$. Put
$$
\alpha=\mathrm{min}\left\{\frac{a_{i}}{\eps_{i}}\ \Big\vert\ \eps_{i}\ne 0\right\},%
$$
where $\alpha$ is well defined, because there is $\eps_{i}\ne
0$. Then $\alpha<1$, the log pair
$$
\left(X,\ \sum_{i=1}^{r}\frac{a_{i}-\alpha\eps_{i}}{1-\alpha}B_{i}\right)%
$$
is not log canonical in the point $P\in X$, the equivalence
$$
\sum_{i=1}^{r}\frac{a_{i}-\alpha\eps_{i}}{1-\alpha}B_{i}\qlineq B_{X}\qlineq\Delta%
$$
holds, and at least one irreducible component of the divisor
$\mathrm{Supp}(\Delta)$ is not contained in
$$
\mathrm{Supp}\left(\sum_{i=1}^{r}\frac{a_{i}-\alpha\eps_{i}}{1-\alpha}B_{i}\right).
$$
\end{remark}

The assertion of Remark~\ref{remark:convexity} is obvious.
Nevertheless it is very useful.

\begin{lemma}
\label{lemma:elliptic-times-P1} Suppose that $X\cong C_1\times C_2$, where
$C_1$ and $C_{2}$ are smooth curves, suppose that
$$
B_{X}\qlineq\lambda E+\mu F
$$
where $E\cong C_1$ and $F\cong C_2$ are curves on the surface $X$ such that
$$
E\cdot E=F\cdot F=0
$$
and $E\cdot F=1$, and $\lambda$ and $\mu$ are non-negative rational numbers.
Then
\begin{itemize}
\item the pair $(X, B_X)$ is log terminal if $\lambda<1$ and $\mu<1$,%
\item the pair $(X, B_X)$ is log canonical if $\lambda\leqslant 1$
and $\mu\leqslant 1$.
\end{itemize}
\end{lemma}
\begin{proof}
Suppose that $\lambda, \mu<1$, but $(X, B_{X})$ is not log terminal at some
point $P\in X$. Then
$$
\mathrm{mult}_{P}\big(B_{X}\big)\geqslant 1,
$$
and we may assume that $E\not\subset\mathrm{Supp}(B_{X})$ or
$F\not\subset\mathrm{Supp}(B_{X})$ by
Remark~\ref{remark:convexity}. But
$$
E\cdot B_{X}=\mu,\ F\cdot B_{X}=\lambda,
$$
which immediately leads to a contradiction, because
$\mathrm{mult}_{P}(B_{X})\geqslant 1$.
\end{proof}

Let $[B_{X}]$ be a class of $\mathbb{Q}$-rational equivalence of
the divisor $B_{X}$. Put
$$
\mathrm{lct}\Big(X,
\big[B_{X}\big]\Big)=\mathrm{inf}\left\{\mathrm{lct}\big(X,D\big)\
\Big\vert\
 D\ \text{is an effective $\mathbb{Q}$-divisor on $X$ such that}\ D\sim_{\mathbb{Q}}
 B_{X}\right\}\geqslant 0,%
$$
and put $\mathrm{lct}(X,[B_{X}])=+\infty$ if $B_{X}=0$. Note that
$B_{X}$ is an effective by assumption.

\begin{remark}
\label{remark:lct-class} The equality
$\mathrm{lct}(X,[-K_{X}])=\mathrm{lct}(X)$ holds
(see Definition~\ref{definition:threshold}).
\end{remark}

Arguing as in the proof of Lemma~\ref{lemma:lct-P1-product}, we obtain the
following result.

\begin{lemma}
\label{lemma:Hwang} Suppose that there is a surjective morphism with connected
fibers
$$
\phi\colon X\longrightarrow Z
$$
such that $\dim(Z)=1$. Let $F$ be a fiber of $\phi$ that has log
terminal singularities. Then either
$$
\mathrm{lct}_{F}\Big(X,
B_{X}\Big)\geqslant\mathrm{lct}\Big(F,\Big[B_{X}\big\vert_{F}\Big]\Big),
$$
or there is a positive rational number
$\eps<\mathrm{lct}(F,[B_{X}\vert_{F}])$ such that
$F\subseteq\mathrm{LCS}(X, \eps B_{X})$.
\end{lemma}

\begin{proof}
Suppose that $\mathrm{lct}_{F}(X, B_{X})<\mathrm{lct}(F,[B_{X}\vert_{F}])$.
Then there is a rational number
$$
\eps<\mathrm{lct}\Big(F,\Big[B_{X}\big\vert_{F}\Big]\Big)
$$
such that the log pair $(X,\eps B_{X})$ is not log canonical at some point $P\in
F$. Put
$$
B_{X}=\mu F+\Omega,
$$
where $\Omega$ is an
effective $\mathbb{Q}$-divisor on $X$ such that
$F\not\subset\mathrm{Supp}(\Omega)$.

We may assume that $\eps\mu\leqslant 1$. Then $(X, F+\eps\Omega)$ is
not canonical at the point $P$. Then
$$
\Big(F,\ \eps\Omega\Big\vert_{F}\Big)
$$
is not log canonical at $P$ by Theorem~\ref{theorem:adjunction}. But
$\Omega\vert_{F}\sim_{\mathbb{Q}} B_{X}\vert_{F}$, which is a contradiction.
\end{proof}

Let us show how to apply Lemma~\ref{lemma:Hwang}.

\begin{lemma}
\label{lemma:2-29-singular} Let $Q\subset\mathbb{P}^{4}$ be a cone
over a smooth quadric surface, and let $\alpha\colon X\to Q$ be a
blow up along a smooth conic $C\subset
Q\setminus\mathrm{Sing}(Q)$. Then $\mathrm{lct}(X)=1/3$.
\end{lemma}

\begin{proof}
Let $H$ be a general hyperplane section of
$Q\subset\mathbb{P}^{4}$ that contains $C$, and let $\bar{H}$ be a
proper transform of the surface $H$ on the threefold $X$. Then
$$
-K_{X}\sim 3\bar{H}+2E,
$$
where $E$ is the exceptional divisor of $\alpha$. In particular,
the inequality $\mathrm{lct}(X)\leqslant 1/3$ holds.

We suppose that $\mathrm{lct}(X)<1/3$. Then there exists an
effective $\mathbb{Q}$-divisor $D\qlineq -K_{X}$ such that the log
pair $(X,\lambda D)$ is not log canonical for some positive
rational number $\lambda<1/3$.

There is a commutative diagram
$$
\xymatrix{
&X\ar@{->}[dl]_{\alpha}\ar@{->}[rd]^{\beta}&\\%
Q\ar@{-->}[rr]_{\psi}&&\mathbb{P}^{1}}
$$
where $\beta$ is a morphism given by the linear system
$|\bar{H}|$, and $\psi$ is a projection from the two-dimensional
linear subspace that contains the conic $C$.

Suppose that $\mathbb{LCS}(X,\lambda D)$ contains a surface
$M\subset X$. Then
$$
D=\mu M+\Omega,
$$
where $\mu\geqslant 1/\lambda$, and $\Omega$ is an effective
$\mathbb{Q}$-divisor such that
$M\not\subset\mathrm{Supp}(\Omega)$.

Let $F$ be a general fiber of $\beta$. Then
$F\cong\mathbb{P}^{1}\times\mathbb{P}^{1}$ and
$$
D\Big\vert_{F}=\mu M\Big\vert_{F}+\Omega\Big\vert_{F}\qlineq -K_{F},%
$$
which immediately implies that $M$ is a fiber of the morphism
$\beta$. But
$$
\alpha\big(D\big)=\mu\alpha\big(M\big)+\alpha\big(\Omega\big)\qlineq-K_{Q}\sim 3\alpha\big(M\big),%
$$
which is impossible, because $\mu\geqslant 1/\lambda>3$. Thus, the
set $\mathbb{LCS}(X,\lambda D)$ contains no surfaces.

There is a fiber $S$ of the morphism $\beta$ such that
$$
S\ne S\cap\mathrm{LCS}\Big(X,\ \lambda D\Big)\ne\varnothing,
$$
which implies that $S$ is singular by Lemma~\ref{lemma:Hwang},
because $\mathrm{lct}(\mathbb{P}^{1}\times\mathbb{P}^{1})=1/2$.

Thus, the surface $S$ is an irreducible quadric cone in
$\mathbb{P}^{3}$. Then
$$
\mathrm{LCS}\Big(X,\ \lambda D\Big)\subseteq S
$$
by Theorem~\ref{theorem:connectedness}.  We may assume that either
$S\not\subset\mathrm{Supp}(D)$ or $E\not\subset\mathrm{Supp}(D)$
by  Remark~\ref{remark:convexity}, because
$$
\left(X,\ S+\frac{2}{3}E\right)
$$
has log canonical singularities, and the equivalence $3S+2E\qlineq
D$ holds.

Put $\Gamma=E\cap S$. The curve $\Gamma$ is an irreducible conic
in $S$. Then
$$
\mathrm{LCS}\Big(X,\ \lambda D\Big)\subseteq \Gamma
$$
by Lemma~\ref{lemma:singular-quadric-threefold}. Intersecting $D$
with a general ruling of the cone $S\subset\mathbb{P}^{3}$, and
intersecting $D$ with a general fiber of the projection $E\to C$,
we see that
$$
\Gamma\not\subseteq\mathrm{LCS}\Big(X,\lambda D\Big),
$$
which implies that $\mathrm{LCS}(X,\lambda D)$ consists of a
single point $O\in\Gamma$ by Theorem~\ref{theorem:connectedness}.

Let $R$ be a general (not passing through $O$) surface in
$|\alpha^{*}(H)|$. Then
$$
\mathrm{LCS}\left(X,\ \lambda D+\frac{1}{2}\Big(\bar{H}+2R\Big)\right)=R\cup O,%
$$
which is impossible by Theorem~\ref{theorem:connectedness}, since
$-K_{X}\sim\bar{H}+2R\qlineq D$ and $\lambda<1/3$.
\end{proof}

The following generalization of Lemma~\ref{lemma:Hwang} follows
from \cite[Proposition~5.19]{Vi95} (cf. \cite{Hw06b}).

\begin{theorem}
\label{theorem:Hwang} Let $\phi\colon X\to Z$ be a surjective flat
morphism with connected fibers such that $Z$ has rational
singularities, and all scheme fibers of $\phi$ has at most
canonical Gorenstein singularities. Let $F$ be a scheme fiber of
$\phi$. Then either $\mathrm{lct}_{F}(X,
B_{X})\geqslant\mathrm{lct}(F,[B_{X}\vert_{F}])$, or there is a
positive rational number $\eps<\mathrm{lct}(F,[B_{X}\vert_{F}])$
such that $F\subseteq\nlb\mathrm{LCS}(X, \eps B_{X})$.
\end{theorem}

Let us consider two elementary applications of
Theorem~\ref{theorem:Hwang}.

\begin{lemma}
\label{lemma:P1xP2} Suppose that $\mathrm{LCS}(X, B_{X})\ne
\varnothing$, where $X\cong\P^1\times\P^2$ and
$$
B_{X}\qlineq -\lambda K_{X}
$$
for some rational number $0<\lambda<1/2$. Then
$\mathbb{LCS}(X, B_{X})$ contains a surface.
\end{lemma}

\begin{proof}
Suppose that $\mathbb{LCS}(X, B_{X})$ contains no surfaces. By
Theorems~\ref{theorem:connectedness} and \ref{theorem:Hwang}, we
have
$$
\LCS\Big(X,\ B_{X}\Big)=F,
$$
where $F$ is a fiber of the natural projection $\pi_{2}\colon
X\to\P^2$. Let $S$ be a general surface in
$$
\Big|\pi_{1}^{*}\Big(\mathcal{O}_{\mathbb{P}^{2}}\big(1\big)\Big)\Big|,
$$
let $M_{1}$ and $M_{2}$ be general fibers of the natural
projection $\pi_{1}\colon X\to\P^1$. Then the locus
$$
\LCS\left(X,\ \lambda
D+\frac{1}{2}\Big(M_1+M_2+3S\Big)\right)=F\cup S
$$
is disconnected,  which is impossible by
Theorem~\ref{theorem:connectedness}.
\end{proof}

\begin{lemma}
\label{lemma:lct-product} Let $V$ and $U$ be Fano varieties with
at most canonical Gorenstein singularities. Then
$$
\mathrm{lct}\Big(V\times
U\Big)=\mathrm{min}\Big(\mathrm{lct}\big(V\big),\
 \mathrm{lct}\big(U\big)\Big).%
$$
\end{lemma}

\begin{proof}
The inequalities $\mathrm{lct}(U)\geqslant\mathrm{lct}(V\times
U)\leqslant\mathrm{lct}(U)$ are obvious. Suppose that
$$
\mathrm{lct}\Big(V\times
U\Big)<\mathrm{min}\Big(\mathrm{lct}\big(V\big),\
 \mathrm{lct}\big(U\big)\Big),%
$$
and let us show that this assumption leads to a contradiction.

There is an effective $\mathbb{Q}$-divisor $D\qlineq -K_{V\times
U}$ such that the log pair
$$
\Big(V\times U,\ \lambda D\Big)
$$
is not log canonical in some point $P\in V\times U$, where
$\lambda<\min(\lct(V), \lct(U))$.

Let us identify $V$ with a fiber of the projection $V\times U\to
U$ that contains the point $P$. The inequalities
$$
\mathrm{lct}\big(V\big)>\lambda>\mathrm{lct}_{V}\Big(V\times U,\ D\Big)\geqslant\mathrm{lct}\Big(V,\Big[D\big\vert_{V}\Big]\Big)=\mathrm{lct}\Big(V,\big[-K_{V}\big]\Big)=\mathrm{lct}\big(V\big)%
$$
are inconsistent. So, it follows from Theorem~\ref{theorem:Hwang},
that the log pair $(V\times U, \lambda D)$ is not log canonical in
every point of $V\subset V\times U$.

Let us identify $U$ with a general fiber of the projection
$V\times U\to V$. Then
$$
D\Big\vert_{U}\qlineq -K_{U},
$$
and $(U,\lambda D\vert_{U})$ is not log canonical in $U\cap V$ by
Remark~\ref{remark:hyperplane-reduction} (applied $\dim(V)$ times
here), which contradicts the inequality $\lambda<\mathrm{lct}(U)$.
\end{proof}

We believe that the assertion of Lemma~\ref{lemma:lct-product}
holds for log terminal varieties (cf.
Lemma~\ref{lemma:lct-P1-product}).


\section{Cubic surfaces}
\label{section:cubic-surfaces}

Let $X$ be a cubic surface in $\mathbb{P}^{3}$ that has at most
one ordinary double point.

\begin{definition}
\label{definition:Eckardt-point} A point $O\in X$ is said to be an
Eckardt point if $O\not\in\mathrm{Sing}(X)$ and
$$
O=L_{1}\cap L_{2}\cap L_{3},
$$
where $L_{1},L_{2},L_{3}$ are different lines on the surface
$X\subset\mathbb{P}^{3}$.
\end{definition}

General cubic surfaces have no Eckardt points. It follows from
Example~\ref{example:del-Pezzos} and \ref{example:singular-cubics}
that
$$
\mathrm{lct}\big(X\big)=\left\{%
\aligned
&3/4\ \text{when}\ X\ \text{has no Eckardt points and}\ \mathrm{Sing}\big(X\big)=\varnothing\,\\%
&2/3\ \text{when}\ X\ \text{has an Eckardt point or}\ \mathrm{Sing}\big(X\big)\ne\varnothing,\\%
\endaligned\right.%
$$

Let $D$ be an effective $\mathbb{Q}$-divisor on $X$ such that
$D\qlineq -K_{X}$, and let $\omega$ be a positive rational number
such that $\omega<3/4$. In this section we prove the following
result (cf.  \cite{Ch01b}, \cite{Ch07b}).

\begin{theorem}
\label{theorem:cubic-surfaces} Suppose that $(X,\omega D)$ is not
log canonical. Then
$$
\mathrm{LCS}\Big(X,\ \omega D\Big)=O,
$$
where $O\in X$ is either a singular point or an Eckardt point.
\end{theorem}

Suppose that $(X,\omega D)$ is not log canonical. Let $P$ be a
point in $\mathrm{LCS}(X,\omega D)$. Suppose that
\begin{itemize}
\item neither $P=\mathrm{Sing}(X)$,%
\item nor $P$ is an Eckardt point.
\end{itemize}

\begin{lemma}
\label{lemma:dp3-point} One has $\mathrm{LCS}(X,\omega D)=P$.
\end{lemma}

\begin{proof}
Suppose that $\mathrm{LCS}(X,\omega D)\ne P$. Then there is a
curve $C\subset X$ such that
$$
P\in C\subseteq\mathrm{LCS}\Big(X,\omega D\Big)
$$
by Theorem~\ref{theorem:connectedness}. Then there is an effective
$\mathbb{Q}$-divisor $\Omega$ on $X$ such that
$C\not\subset\mathrm{Supp}(\Omega)$ and
$$
D=\mu C+\Omega,
$$
where $\mu\geqslant 1/\omega$. Let
$H$ be a general hyperplane section of $X$. Then
$$
3=H\cdot D=\mu H\cdot C+H\cdot \Omega\geqslant
\mu\mathrm{deg}\big(C\big),
$$
which implies that either $\mathrm{deg}(C)=1$, or
$\mathrm{deg}(C)=2$.

Suppose that $\mathrm{deg}(C)=1$. Let $Z$ be a general conic on
$X$ such that $-K_{X}\sim C+Z$. Then
$$
2=Z\cdot D=\mu Z\cdot C+Z\cdot \Omega\geqslant \mu Z\cdot C=\left\{%
\aligned
&2\mu\ \text{if}\ C\cap\mathrm{Sing}\big(X\big)=\varnothing\,\\%
&3\mu/2\ \text{if}\ C\cap\mathrm{Sing}\big(X\big)\ne\varnothing,\\%
\endaligned\right.%
$$
which implies that $\mu\leqslant 4/3$. But $\mu\geqslant
1/\omega>4/3$, which gives a contradiction.

We see that $\mathrm{deg}(C)=2$. Let $L$ be a line on $X$ such
that $-K_{X}\sim C+L$. Then
$$
D=\mu C+\lambda L+\Upsilon,
$$
where $\lambda\in\mathbb{Q}$ such that $\lambda\geqslant 0$, and
$\Upsilon$ is an effective $\mathbb{Q}$-divisor such that
$C\not\subset\mathrm{Supp}(\Upsilon)\not\supset L$. Then
$$
1=L\cdot D=\mu L\cdot C+\lambda L\cdot
L+L\cdot\Upsilon\geqslant \mu L\cdot C+\lambda L\cdot L=\left\{%
\aligned
&2\mu-\lambda\ \text{if}\ C\cap\mathrm{Sing}\big(X\big)=\varnothing\,\\%
&3\mu/2-\lambda/2\ \text{if}\ C\cap\mathrm{Sing}\big(X\big)\ne\varnothing,\\%
\endaligned\right.
$$
which implies that $\mu\le 7/6<4/3$, because $\lambda\leqslant 4/3$ (see
the case when $\mathrm{deg}(C)=1$). But $\mu>4/3$, which gives a
contradiction.
\end{proof}

Let $\pi\colon U\to X$ be a blow up of $P$, and let $E$ be the
$\pi$-exceptional curve. Then
$$
\bar{D}\qlineq \pi^{*}\big(D\big)+\mathrm{mult}_{P}\big(D\big)E,
$$
where $\mathrm{mult}_{P}(D)\geqslant 1/\omega$ and
$\bar{D}$ is a proper transform of $D$ on the surface $U$. The log
pair
$$
\Big(U,\ \omega\bar{D}+\big(\omega\mathrm{mult}_{P}\big(D\big)-1\big)E\Big)%
$$
is not log canonical at some point $Q\in E$. Then either
$\mathrm{mult}_{P}(D)\geqslant 2/\omega$, or
\begin{equation}
\label{equation:blow-up}
\mathrm{mult}_{Q}\big(\bar{D}\big)+\mathrm{mult}_{P}\big(D\big)\geqslant 2\slash\omega>8\slash 3,%
\end{equation}
because the divisor
$\omega\bar{D}+(\omega\mathrm{mult}_{P}(D)-1)E$ is effective.

Let $T$ be the unique hyperplane section of $X$ that is singular
at $P$. We may assume that
$$
\mathrm{Supp}\big(T\big)\not\subseteq\mathrm{Supp}\big(D\big)
$$
by Remark~\ref{remark:convexity}, because $(X,\omega T)$ is log
canonical. The following cases are possible:
\begin{itemize}
\item the curve $T$ is irreducible;%
\item the curve $T$ is a union of a line and an irreducible conic;%
\item the curve $T$ consists of $3$ lines.%
\end{itemize}
Hence $T$ is reduced. Note that $\mathrm{mult}_{P}(T)=2$ since $P$ is not
an Eckardt point.
We exclude these cases one by one.

\begin{lemma}
\label{lemma:dp3-no-lines} The curve $T$ is reducible.
\end{lemma}

\begin{proof}
Suppose that $T$ is irreducible. Then there is a commutative
diagram
$$
\xymatrix{
&U\ar@{->}[dl]_{\pi}\ar@{->}[dr]^{\psi}&\\%
X\ar@{-->}[rr]_{\rho}&&\mathbb{P}^{2}}
$$ %
where $\psi$ is a double cover branched over a quartic curve, and
$\rho$ is the projection from $P\in X$.

Let $\bar{T}$ be the proper transform of $T$ on the surface $U$.
Suppose that $Q\in\bar{T}$. Then
$$
3-2\mathrm{mult}_{P}\big(D\big)=\bar{T}\cdot\bar{D}\geqslant\mathrm{mult}_{Q}\big(\bar{T}\big)\mathrm{mult}_{Q}\big(\bar{D}\big)>\mathrm{mult}_{Q}\big(\bar{T}\big)\Big(8/3-\mathrm{mult}_{P}\big(D\big)\Big)\geqslant 8/3-\mathrm{mult}_{P}\big(D\big),%
$$
which implies that $\mathrm{mult}_{P}(D)\leqslant 1/3$. But
$\mathrm{mult}_{P}(D)>4/3$. Thus, we see that $Q\not\in\bar{T}$.

Let $\tau\in\mathrm{Aut}(U)$ be an involution\footnote{The
involution $\tau$ induces an involution in $\mathrm{Bir}(X)$ that
is called a Geiser involution.} induced by~$\psi$. It follows from
\cite{Ma67} that
$$
\tau^{*}\Big(\pi^{*}\big(-K_{X}\big)\Big)\sim \pi^{*}\big(-2K_{X}\big)-3E,%
$$
and  $\tau(\bar{T})=E$. Put $\breve{Q}=\pi\circ\tau(Q)$. Then
$\breve{Q}\ne P$, because $Q\not\in\bar{T}$.

Let $H$ be the hyperplane section of $X$ that is singular at
$\breve{Q}$. Then $T\ne H$,
because $P\ne \breve{Q}$ and $T$ is smooth outside of the point
$P$. Hence $P\not\in H$, because otherwise
$$
3=H\cdot T\geqslant\mathrm{mult}_{P}\big(H\big)\mathrm{mult}_{P}\big(T\big)+\mathrm{mult}_{\breve{Q}}\big(H\big)\mathrm{mult}_{\breve{Q}}\big(T\big)\geqslant 4.%
$$

Let $\bar{H}$ be the proper transform of $H$ on the surface $U$.
Put $\bar{R}=\tau(\bar{H})$ and $R=\pi(\bar{R})$. Then
$$
\bar{R}\sim \pi^{*}\big(-2K_{X}\big)-3E,%
$$
ant the curve $\bar{R}$ must be singular at the point $Q$.

Suppose that $R$ is irreducible. Taking into account the possible
singularities of $\bar{R}$, we see that
$$
\left(X,\ \frac{3}{8} R\right)
$$
is log canonical. Thus, we may assume that
$R\not\subseteq\mathrm{Supp}(D)$ by Remark~\ref{remark:convexity}.
Then
$$
6-3\mathrm{mult}_{P}\big(D\big)=\bar{R}\cdot\bar{D}\geqslant \mathrm{mult}_{Q}(\bar{R})\mathrm{mult}_{Q}\big(\bar{D}\big)>2\Big(8/3-\mathrm{mult}_{P}\big(D\big)\Big),%
$$
which implies that $\mathrm{mult}_{P}(D)<2/3$. But
$\mathrm{mult}_{P}(D)>4/3$. The curve  $R$ must be reducible.

The curves $R$ and $H$ are reducible. So, there is a line
$L\subset X$ such that $P\not\in L\ni\breve{Q}$.

Let $\bar{L}$ be the proper transform of $L$ on the surface $U$.
Put $\bar{Z}=\tau(\bar{L})$. Then $\bar{L}\cdot E=0$ and
$$
\bar{L}\cdot\bar{T}=\bar{L}\cdot\pi^{*}(-K_{X})=1,%
$$
which implies that $\bar{Z}\cdot E=1$ and
$\bar{Z}\cdot\pi^{*}\big(-K_{X}\big)=2$. We have $Q\in\bar{Z}$.
Then
$$
2-\mathrm{mult}_{P}\big(D\big)=\bar{Z}\cdot\bar{D}\geqslant\mathrm{mult}_{Q}\big(\bar{D}\big)>8/3-\mathrm{mult}_{P}\big(D\big)>2-\mathrm{mult}_{P}\big(D\big)%
$$
in the case when $\bar{Z}\not\subseteq\mathrm{Supp}(\bar{D})$.
Hence, we see that $\bar{Z}\subseteq\mathrm{Supp}(\bar{D})$.

Put $Z=\pi(\bar{Z})$. Then $Z$ is a conic such that $P\in Z$ and
$$
-K_{X}\sim L+Z,
$$
which means that $L\cup Z$ is~cut~out by the plane in
$\mathbb{P}^{3}$ that passes through $Z$. Put
$$
D=\eps Z+\Upsilon,
$$
where $\eps\in\mathbb{Q}$ such that $\eps\geqslant 0$, and
$\Upsilon$ is an effective $\mathbb{Q}$-divisor such that
$Z\not\subset\mathrm{Supp}(\Upsilon)$.

We may assume that $L\not\subseteq\mathrm{Supp}(\Upsilon)$ by
Remark~\ref{remark:convexity}.~Then
$$
1=L\cdot D=\eps Z\cdot L+L\cdot\Upsilon\geqslant \eps
Z\cdot L=\left\{%
\aligned
&2\eps\ \text{if}\ Z\cap\mathrm{Sing}\big(X\big)=\varnothing,\\%
&3\eps/2\ \text{if}\ Z\cap\mathrm{Sing}\big(X\big)\ne\varnothing,\\%
\endaligned\right.
$$
which implies that $\eps\leqslant 2/3$.

Let $\bar{\Upsilon}$ be the proper transform of $\Upsilon$ on the
surface $U$. Then the log pair
$$
\Big(U,\
\eps\omega\bar{Z}+\omega\bar{\Upsilon}+\big(\omega\mathrm{mult}_{P}\big(D\big)-1\big)E\Big)
$$
is not log canonical at $Q\in\bar{Z}$. Then
$$
\omega\bar{\Upsilon}\cdot\bar{Z}+\Big(\omega\mathrm{mult}_{P}\big(D\big)-1\Big)=\Big(\omega\bar{\Upsilon}+\big(\omega\mathrm{mult}_{P}\big(D\big)-1\big)E\Big)\cdot\bar{Z}>1
$$
by Lemma~\ref{lemma:adjunction}, because $\eps\leqslant 2/3$.
In particular, we see that
$$
8/3-\mathrm{mult}_{P}\big(D\big)<\bar{Z}\cdot\bar{\Upsilon}=2-\mathrm{mult}_{P}\big(D\big)-\eps\bar{Z}\cdot\bar{Z}=\left\{%
\aligned
&2-\mathrm{mult}_{P}\big(D\big)+\eps\ \text{if}\ Z\cap\mathrm{Sing}\big(X\big)=\varnothing,\\%
&2-\mathrm{mult}_{P}\big(D\big)+\eps/2\ \text{if}\ Z\cap\mathrm{Sing}\big(X\big)\ne\varnothing,\\%
\endaligned\right.
$$
which implies that $\eps>2/3$. But $\eps\leqslant 2/3$.
\end{proof}

Therefore, there is a line $L_{1}\subset X$ such that $P\in
L_{1}$. Put
$$
D=m_{1}L_{1}+\Omega,
$$
where $m_{1}\in\mathbb{Q}$ such that
$\Omega$ is an effective $\mathbb{Q}$-divisor such that
$L_{1}\not\subset\in\mathrm{Supp}(\Omega)$.~Then
$$
4\slash 3<1\slash\omega<\Omega\cdot L_{1}=1-m_{1}L_{1}\cdot L_{1}=\left\{%
\aligned
&1+m_{1}\ \text{if}\ L_{1}\cap\mathrm{Sing}\big(X\big)=\varnothing,\\%
&1+m_{1}/2\ \text{if}\ L_{1}\cap\mathrm{Sing}\big(X\big)\ne\varnothing.\\%
\endaligned\right.
$$

\begin{corollary}
\label{corollary:cubic-m1} The following inequality holds:
$$
m_{1}>\left\{%
\aligned
&1/3\ \mathrm{if}\ L_{1}\cap\mathrm{Sing}\big(X\big)=\varnothing,\\%
&2/3\ \mathrm{if}\ L_{1}\cap\mathrm{Sing}\big(X\big)\ne\varnothing.\\%
\endaligned\right.%
$$
\end{corollary}

\begin{remark}
\label{remark:six-lines} Suppose that $X$ is singular. Put
$O=\mathrm{Sing}(X)$. It follows from \cite{BW79} that
$$
O=\Gamma_{1}\cap\Gamma_{2}\cap \Gamma_{3}\cap \Gamma_{4}\cap\Gamma_{5}\cap\Gamma_{6},%
$$
where $\Gamma_{1},\ldots,\Gamma_{6}$ are different lines on the
surface $X\subset\mathbb{P}^{3}$. The equivalence
$$
-2K_{X}\sim\sum_{i=1}^{6}\Gamma_{i}
$$
holds. Suppose that $L_{1}=\Gamma_{1}$. Let
$\Pi_{2},\ldots,\Pi_{6}\subset \mathbb{P}^{3}$ be planes such
that
$$
L_{1}\subset\Pi_{i}\supset\Gamma_{i},
$$
and let $\Lambda_{2},\ldots,\Lambda_{6}$ be lines on the surface
$X$ such that
$$
L_{1}\cup\Gamma_{i}\cup\Lambda_{i}=\Pi_{i}\cap X\subset X\subset\mathbb{P}^{3},%
$$
which implies that $-K_{X}\sim L_{1}+\Gamma_{i}+\Lambda_{i}$. Then
$$
-5K_{X}\sim 4L_{1}+\sum_{i=2}^{6}\Lambda_{i}+\Big(L_{1}+\sum_{i=2}^{6}\Gamma_{i}\Big)\sim 4L_{1}+\sum_{i=2}^{6}\Lambda_{i}-2K_{X},%
$$
which implies that $-3K_{X}\sim 4L_{1}+\sum_{i=2}^{6}\Lambda_{i}$.
But the log pair
$$
\left(X,\ L_{1}+\frac{\sum_{i=2}^{6}\Lambda_{i}}{3}\right)
$$
is log canonical at the point $P$. Thus, we may assume that
$$
\mathrm{Supp}\left(\sum_{i=2}^{6}\Lambda_{i}\right)\not\subseteq\mathrm{Supp}\big(D\big)
$$
thanks to Remark~\ref{remark:convexity}, because
$L_{1}\subseteq\mathrm{Supp}(D)$. Then there is $\Lambda_{k}$ such
that
$$
1=D\cdot\Lambda_{k}=\Big(m_{1}L_{1}+\Omega\Big)\cdot\Lambda_{k}=m_{1}+\Omega\cdot\Lambda_{k}\geqslant
m_{1},
$$
because $O\not\in\Lambda_{k}$. Thus, we may assume that
$m_{1}\leqslant 1$ if $L_{1}\cap\mathrm{Sing}(X)\ne\varnothing$.
\end{remark}

Arguing as in the proof of Lemma~\ref{lemma:Cheltsov-Pukhlikov},
we see that $m_{1}\leqslant 1$ if
$L_{1}\cap\mathrm{Sing}(X)=\varnothing$.

\begin{lemma}
\label{lemma:dp3-single-line} There is a line $L_{2}\subset X$
such that $L_{1}\ne L_{2}$ and $P\in L_{2}$.
\end{lemma}

\begin{proof} Suppose that there is no line $L_{2}\subset X$ such that $L_{1}\ne L_{2}$ and
$P\in L_{2}$. Then
$$
T=L_{1}+C,
$$
where $C$ is an irreducible conic on the surface
$X\subset\mathbb{P}^{3}$ such that $P\in C$.

It follows from Remark~\ref{remark:convexity} that we may assume
that $C\not\subseteq\mathrm{Supp}(\Omega)$, because $m_{1}\ne 0$.

Let $\bar{L}_{1}$ and $\bar{C}$ be the proper transforms of
$L_{1}$ and $C$ on the surface $U$, respectively. Then
$$
\bar{D}\qlineq m_{1}\bar{L}_{1}+\bar{\Omega}\qlineq \pi^{*}\Big(m_{1} L_{1}+\Omega\Big)-\Big(m_{1}+\mathrm{mult}_{P}\big(\Omega\big)\Big)E\qlineq\pi^{*}\big(D\big)-\mathrm{mult}_{P}\big(D\big)E,%
$$
where $\bar{\Omega}$ is the proper transform of the divisor
$\Omega$ on the surface $U$. We have
$$
0\leqslant\bar{C}\cdot\bar{\Omega}=2-\mathrm{mult}_{P}\big(D\big)+m_{1}\bar{C}\cdot
\bar{L}<2/3-m_{1}\bar{C}\cdot \bar{L}_{1}=\left\{%
\aligned
&2/3-m_{1},\ \text{if}\ L_{1}\cap\mathrm{Sing}\big(X\big)=\varnothing,\\%
&2/3-m_{1}/2,\ \text{if}\ L_{1}\cap\mathrm{Sing}\big(X\big)\ne\varnothing,\\%
\endaligned\right.
$$
which implies that $m_{1}<2/3$ if
$L_{1}\cap\mathrm{Sing}(X)=\varnothing$. It follows from
inequality~\ref{equation:blow-up} that
$$
\mathrm{mult}_{Q}\big(\bar{\Omega}\big)>8/3-\mathrm{mult}_{P}\big(\Omega\big)-m_{1}\Big(1+\mathrm{mult}_{Q}\big(\bar{L}_{1}\big)\Big).
$$

Suppose that $Q\in\bar{L}_{1}$. Then it follows from
Lemma~\ref{lemma:adjunction} that
$$
8/3<\bar{L}_{1}\cdot\Big(\bar{\Omega}+\big(\mathrm{mult}_{P}\big(\Omega\big)+m_{1}\big)E\Big)=1-m_{1}\bar{L}_{1}\cdot\bar{L}_{1}=\left\{%
\aligned
&1+2m_{1},\ \text{if}\ L_{1}\cap\mathrm{Sing}\big(X\big)=\varnothing,\\%
&1+3m_{1}/2,\ \text{if}\ L_{1}\cap\mathrm{Sing}\big(X\big)\ne\varnothing,\\%
\endaligned\right.
$$
which is impossible, because $m_{1}\leqslant 1$ if
$L_{1}\cap\mathrm{Sing}(X)\ne\varnothing$ by
Remark~\ref{remark:six-lines}.

We see that $Q\not\in\bar{L}_{1}$. Suppose that $Q\in\bar{C}$.
Then
$$
2-\mathrm{mult}_{P}\big(\Omega\big)-m_{1}-m_{1}\bar{C}\cdot\bar{L}{1}=\bar{C}\cdot\bar{\Omega}>8/3-\mathrm{mult}_{P}\big(\Omega\big)-m_{1},%
$$
which is impossible, because
$m_{1}\bar{C}\cdot\bar{L}_{1}\geqslant 0$. Hence, we see that
$Q\not\in\bar{C}$.

There is a commutative diagram
$$
\xymatrix{
&U\ar@{->}[d]_{\pi}\ar@{->}[rr]^{\zeta}&&W\ar@{->}[d]^{\psi}\\%
&X\ar@{-->}[rr]_{\rho}&&\mathbb{P}^{2}&}
$$ %
where $\zeta$ is a birational morphism that contracts the curve
$\bar{L}_{1}$, the morphism $\psi$ is a double cover branched over
a quartic curve, and $\rho$ is a linear projection from the point
$P\in X$.

Let $\tau$ be the birational involution of $U$ induced
by~$\psi$.~Then
\begin{itemize}
\item the involution $\tau$ is biregular $\iff$ $L_{1}\cap\mathrm{Sing}(X)=\varnothing$,%
\item the involution $\tau$ acts biregularly on $U\setminus\bar{L}_{1}$ if $L_{1}\cap\mathrm{Sing}(X)\ne\varnothing$,%
\item it follows from the construction of $\tau$ that $\tau(E)=\bar{C}$,%
\item if $L_{1}\cap\mathrm{Sing}(X)=\varnothing$, then
$$
\tau^{*}\big(\bar{L}_{1}\big)\sim \bar{L}_{1},\ \tau^{*}\big(E\big)\sim \bar{C},\ \tau^{*}\Big(\pi^{*}\big(-K_{X}\big)\Big)\sim \pi^{*}\big(-2K_{X}\big)-3E-\bar{L}_{1}.%
$$
\end{itemize}

Let $H$ be the hyperplane section of $X$ that is singular at
$\pi\circ\tau(Q)\in C$. Then $P\not\in H$, because $C$ is smooth.
Let $\bar{H}$ be the proper transform of $H$ on the surface $U$.
Then
$$
\bar{L}_{1}\not\subseteq\mathrm{Supp}\big(\bar{H}\big)\not\supseteq\bar{C},
$$
and we can put $\bar{R}=\tau(\bar{H})$ and $R=\pi(\bar{R})$. Then
$\bar{R}$ is singular at the point $Q$, and
$$
\bar{R}\sim \pi^{*}\big(-2K_{X}\big)-3E-\bar{L}_{1},%
$$
because $R$ does not pass through a singular point of $X$ if
$\mathrm{Sing}(X)\ne\varnothing$.

Suppose that $R$ is irreducible. Then $R+L_{1}\sim -2K_{X}$, but
the log pair
$$
\left(X,\ \frac{3}{8}\Big(R+L_{1}\Big)\right)
$$
is log canonical. Thus, we may assume that
$R\not\subseteq\mathrm{Supp}(D)$ by Remark~\ref{remark:convexity}.
Then
$$
5-2\Big(m_{1}+\mathrm{mult}_{P}\big(\Omega\big)\Big)+m_{1}\Big(1+\bar{L}_{1}\cdot\bar{L}_{1}\Big)=\bar{R}\cdot\bar{\Omega}\geqslant 2\mathrm{mult}_{Q}\big(\bar{\Omega}\big)>2\Big(8/3-m_{1}-\mathrm{mult}_{P}\big(\Omega\big)\Big),%
$$
which implies that $m_{1}<0$. The curve $R$ must be reducible.

There is a line $L\subset X$ such that $P\not\in L$ and
$\pi\circ\tau(Q)\in L$. Then
$$
L\cap L_{1}=\varnothing,
$$
because $\pi\circ\tau(Q)\in C$ and $(C+L_{1})\cdot L=T\cdot L=1$.
Thus, there is unique conic $Z\subset X$ such that
$-K_{X}\sim L+Z$
and $P\in Z$. Then $Z$ is irreducible and $P=Z\cap L_{1}$, because
$(L+Z)\cdot L_{1}=1$.

Let $\bar{L}$ and $\bar{Z}$ be the proper transform of $L$ and $Z$
on the surface $U$, respectively. Then
$$
\bar{L}\cdot\bar{C}=\bar{Z}\cdot E=1,\
\bar{L}_{1}\cdot\bar{Z}=\bar{L}\cdot E=\bar{L}\cdot\bar{L}_{1}=0,\ \bar{Z}\cdot\bar{Z}=1-\bar{L}\cdot\bar{Z},\ \bar{L}\cdot\bar{Z}=\left\{%
\aligned
&2\ \text{if}\ L\cap\mathrm{Sing}\big(X\big)=\varnothing,\\%
&3/2\ \text{if}\ L\cap\mathrm{Sing}\big(X\big)\ne\varnothing.\\%
\endaligned\right.
$$

We have $\tau(\bar{Z})=\bar{L}$. Then $Q\in\bar{Z}$. Suppose that
$\bar{Z}\not\subseteq\mathrm{Supp}(\bar{\Omega})$. Then
$$
2-m_{1}-\mathrm{mult}_{P}\big(\Omega\big)=\bar{Z}\cdot\bar{\Omega}>8/3-m_{1}-\mathrm{mult}_{P}\big(\Omega\big),
$$
which is impossible. Thus, we see that
$\bar{Z}\subseteq\mathrm{Supp}(\bar{\Omega})$. But the log pair
$$
\Big(X,\ \omega (L+Z)\Big)
$$
is log canonical at the point $P$. Hence, we may assume that
$\bar{L}\not\subseteq\mathrm{Supp}(\bar{\Omega})$ by
Remark~\ref{remark:convexity}. Put
$$
D=\eps Z+m_{1}L_{1}+\Upsilon,
$$
where $\Upsilon$ is an effective $\mathbb{Q}$-divisor such that
$Z\not\subset\mathrm{Supp}(\Upsilon)\not\supset L_{1}$.~Then
$$
1=L\cdot D=\eps L\cdot Z+m_{1}L\cdot
L_{1}+L\cdot\Upsilon=\eps L\cdot Z+L\cdot\Upsilon\geqslant \eps L\cdot Z=\left\{%
\aligned
&2\eps\ \text{if}\ L\cap\mathrm{Sing}\big(X\big)=\varnothing,\\%
&3\eps/2\ \text{if}\ L\cap\mathrm{Sing}\big(X\big)\ne\varnothing,\\%
\endaligned\right.%
$$
which implies that $\eps\leqslant 2/3$. But
$\bar{Z}\cap\bar{L}_{1}=\varnothing$. Then it follows from
Lemma~\ref{lemma:adjunction} that
$$
2-\mathrm{mult}_{P}\big(D\big)-\eps\bar{Z}\cdot\bar{Z}=\bar{Z}\cdot\bar{\Upsilon}>8/3-\mathrm{mult}_{P}\big(D\big),%
$$
where $\bar{\Upsilon}$ is a proper transform of $\Upsilon$ on the
surface $U$. We deduce that  $\eps>2/3$. But
$\eps\leqslant 2/3$.
\end{proof}

Therefore, we see that $T=L_{1}+L_{2}+L_{3}$, where $L_{3}$ is a
line such that $P\not\in L_{3}$. Put
$$
D=m_{1}L_{1}+m_{2}L_{2}+\Delta,
$$
where $\Delta$ is an effective $\mathbb{Q}$-divisor such that
$L_{2}\not\subseteq\mathrm{Supp}(\Delta)\not\supseteq L_{2}$.

The inequalities $m_{1}>1/3$ and $m_{2}>1/3$ hold by
Corollary~\ref{corollary:cubic-m1}. We may assume~that
$L_{3}\not\subseteq\mathrm{Supp}\big(\Delta\big)$
by Remark~\ref{remark:convexity}.
If the singular point of $X$ (provided that there exists one) is contained
in either $L_1$ or $L_2$, we may assume without loss of generality
that it is contained in $L_1$. Then
$L_{3}\cdot L_{2}=1$ and $L_{3}\cdot L_{1}=1/2$
in the case when $L_{1}\cap\mathrm{Sing}(X)\ne\varnothing$, and
$$
L_{3}\cdot L_{2}=L_{3}\cdot L_{1}=1
$$
in the case when $L_{1}\cap\mathrm{Sing}(X)=\varnothing$. Then
$1-m_{1}L_{1}\cdot L_{3}-m_{2}=L_{3}\cdot\Delta\geqslant 0$.

Let $\bar{L}_{1}$ and $\bar{L}_{3}$ be the proper transforms of
$L_{1}$ and $L_{2}$ on the surface $U$, respectively. Then
$$
m_{1}\bar{L}_{1}+m_{2}\bar{L}_{2}+\bar{\Delta}\qlineq \pi^{*}\Big(m_{1}L_{1}+m_{2}L_{2}+\Delta\Big)-\Big(m_{1}+m_{2}+\mathrm{mult}_{P}\big(\Delta\big)\Big)E,%
$$
where $\bar{\Delta}$ is the proper transform of $\Delta$ on the
surface $U$. The inequality~\ref{equation:blow-up} implies that
\begin{equation}
\label{equation:two-lines}
\mathrm{mult}_{Q}\big(\bar{\Delta}\big)>8/3-\mathrm{mult}_{P}\big(\Delta\big)-m_{1}\Big(1+\mathrm{mult}_{Q}\big(\bar{L}_{1}\big)\Big)-m_{1}\Big(1+\mathrm{mult}_{Q}\big(\bar{L}_{2}\big)\Big).
\end{equation}

\begin{lemma}
\label{lemma:cubic-L2} The curve $\bar{L}_{2}$ does not contain
the point $Q$.
\end{lemma}

\begin{proof}
Suppose that $Q\in\bar{L}_{2}$. Then
$$
1-\mathrm{mult}_{P}\big(\Delta\big)-m_{1}+m_{2}=\bar{L}_{2}\cdot\bar{\Delta}>8/3-\mathrm{mult}_{P}\big(\Delta\big)-m_{1}-m_{2}%
$$
by Lemma~\ref{lemma:adjunction}. Thus, we have $m_{2}>5/6$. It
follows from Lemma~\ref{lemma:adjunction} that
$$
1-m_{2}-m_{1}L_{1}\cdot L_{1}=\Delta\cdot L_{1}>4/3-m_{2},
$$
but $L_{1}\cdot L_{1}=-1$ if
$L_{1}\cap\mathrm{Sing}(X)=\varnothing$, and $L_{1}\cdot
L_{1}=-1/2$ if $L_{1}\cap\mathrm{Sing}(X)\ne\varnothing$. Then
$$
m_{1}>\left\{%
\aligned
&1/3\ \text{if}\ L_{1}\cap\mathrm{Sing}\big(X\big)=\varnothing,\\%
&2/3\ \text{if}\ L_{1}\cap\mathrm{Sing}\big(X\big)\ne\varnothing,\\%
\endaligned\right.%
$$
by Corollary~\ref{corollary:cubic-m1}, which is impossible because
$m_{2}>5/6$ and
$$1>m_{1}L_{1}\cdot L_{3}+m_{2}.$$
\end{proof}

\begin{lemma}
\label{lemma:cubic-L1} The curve $\bar{L}_{1}$ does not contain
the point $Q$.
\end{lemma}

\begin{proof}
Suppose that $Q\in\bar{L}_{1}$. Arguing as in the proof of
Lemma~\ref{lemma:cubic-L2}, we see that
$$
L_{1}\cap\mathrm{Sing}\big(X\big)\ne\varnothing,
$$
which implies that $\bar{L}_{1}\cdot\bar{L}_{1}=-1/2$. Then
$m_{1}>10/9$, because
$$
1+3m_{1}/2=\bar{L}_{2}\cdot\Big(\bar{\Delta}+\big(\mathrm{mult}_{P}\big(\Delta\big)-m_{1}-m_{2}\big)E\Big)>8/3
$$
by Lemma~\ref{lemma:adjunction}. But $m_{1}\leqslant 1$ by
Remark~\ref{remark:six-lines}.
\end{proof}

Therefore, we see that $\bar{L}_{1}\not\ni Q\not\in\bar{L}_{2}$.
There is a commutative diagram
$$
\xymatrix{
&U\ar@{->}[d]_{\pi}\ar@{->}[rr]^{\zeta}&&W\ar@{->}[d]^{\psi}\\%
&X\ar@{-->}[rr]_{\rho}&&\mathbb{P}^{2}&}
$$ %
where $\zeta$ is a birational morphism that contracts the curves
$\bar{L}_{1}$ and $\bar{L}_{2}$, the morphism $\psi$ is a double
cover branched over a quartic curve, and $\rho$ is the projection
from the point $P$.

Let $\tau$ be the birational involution of $U$ induced
by~$\psi$.~Then
\begin{itemize}
\item the involution $\tau$ is biregular $\iff$ $L_{1}\cap\mathrm{Sing}(X)=\varnothing$,%
\item the involution $\tau$ acts biregularly on $U\setminus\bar{L}_{1}$ if $L_{1}\cap\mathrm{Sing}(X)\ne\varnothing$,%
\item the equality $\tau(\bar{L}_{2})=\bar{L}_{2}$ holds,%
\item if $L_{1}\cap\mathrm{Sing}(X)=\varnothing$, then
$\tau(\bar{L}_{1})=\bar{L}_{1}$ and
$$
\tau^{*}\Big(\pi^{*}\big(-K_{X}\big)\Big)\sim \pi^{*}\big(-2K_{X}\big)-3E-\bar{L}_{1}-\bar{L}_{2}.%
$$
\end{itemize}

Let $\bar{L}_{3}$ be a proper transform of $L_{3}$ on the surface
$U$. Then $\tau(E)=\bar{L}_{3}$ and
$$
L_{1}\cup L_{2}\not\ni\pi\circ\tau\big(Q\big)\in L_{3}.
$$

\begin{lemma}
\label{lemma:cubic-no-second-line} The line $L_{3}$ is the only
line on $X$ that passes through the point $\pi\circ\tau(Q)$.
\end{lemma}

\begin{proof}
Suppose that there is a line $L\subset X$ such that $L\ne L_{3}$
and $\pi\circ\nlb\tau(Q)\in\nlb L$.  Then
$$
L\cap L_{1}=L\cap L_{2}=\varnothing,
$$
because $\pi\circ\tau(Q)\in L_{3}$ and $(L_{1}+L_{2}+L_{3})\cdot
L=1$. Thus, there is unique conic $Z\subset X$ such that
$-K_{X}\sim L+Z$
and $P\in Z$. Then $Z$ is irreducible, because $P\not\in L$ and
$P$ is not an Eckardt point.

Let $\bar{L}$ and $\bar{Z}$ be the proper transform of $L$ and $Z$
on the surface $U$, respectively. Then
$$
\bar{L}\cdot\bar{L}_{3}=\bar{Z}\cdot E=1,\ \bar{Z}\cdot\bar{Z}=1-\bar{L}\cdot\bar{Z},\ \bar{L}\cdot\bar{Z}=\left\{%
\aligned
&2\ \text{if}\ L\cap\mathrm{Sing}\big(X\big)=\varnothing,\\%
&3/2\ \text{if}\ L\cap\mathrm{Sing}\big(X\big)\ne\varnothing,\\%
\endaligned\right.
$$
and $\bar{L}_{1}\cdot\bar{Z}=\bar{L}_{2}\cdot\bar{Z}=\bar{L}\cdot
E=\bar{L}\cdot\bar{L}_{1}=\bar{L}\cdot\bar{L}_{2}=0$.

We have $\tau(\bar{Z})=\bar{L}$. Then $Q\in\bar{Z}$, which implies
that $\bar{Z}\subseteq\mathrm{Supp}(\bar{\Delta})$, because
$$
2-\mathrm{mult}_{P}\big(\Delta\big)-m_{1}-m_{2}=\bar{Z}\cdot\bar{\Omega}>8/3-\mathrm{mult}_{P}\big(\Delta\big)-m_{1}-m_{2}
$$
in the case when
$\bar{Z}\not\subseteq\mathrm{Supp}(\bar{\Delta})$. On the other
hand, the log pair
$$
\Big(X,\ \omega (L+Z)\Big)
$$
is log canonical at the point $P$. Hence, we may assume that
$\bar{L}\not\subseteq\mathrm{Supp}(\bar{\Delta})$ by
Remark~\ref{remark:convexity}. Put
$$
D=\eps Z+m_{1}L_{1}+m_{2}L_{2}+\Upsilon,
$$
where $\Upsilon$ is an effective $\mathbb{Q}$-divisor such that
$Z\not\subseteq\mathrm{Supp}(\Upsilon)$.~Then
$$
1=L\cdot D=\eps L\cdot Z+m_{1}L\cdot
L_{1}+L\cdot\Upsilon=\eps L\cdot Z+L\cdot\Upsilon\geqslant \eps L\cdot Z=\left\{%
\aligned
&2\eps\ \text{if}\ L\cap\mathrm{Sing}\big(X\big)=\varnothing,\\%
&3\eps/2\ \text{if}\ L\cap\mathrm{Sing}\big(X\big)\ne\varnothing,\\%
\endaligned\right.%
$$
which implies that $\eps\leqslant 2/3$. But
$\bar{Z}\cap\bar{L}_{1}=\varnothing$. Then it follows from
Lemma~\ref{lemma:adjunction} that
$$
2-\mathrm{mult}_{P}\big(D\big)-\eps\bar{Z}\cdot\bar{Z}=\bar{Z}\cdot\bar{\Upsilon}>8/3-\mathrm{mult}_{P}\big(D\big),%
$$
where $\bar{\Upsilon}$ is a proper transform of $\Upsilon$ on the
surface $U$. We deduce that  $\eps>2/3$. But
$\eps\leqslant 2/3$.
\end{proof}

Therefore, there is an unique irreducible conic $C\subset X$ such
that
$$
-K_{X}\sim L_{3}+C
$$
and $\pi\circ\tau(Q)\in C$. Then $C+L_{3}$ is a hyperplane section
of $X$ that is singular at $\pi\circ\tau(Q)$.

Let $\bar{C}$ be the proper transform of $C$ on the surface $U$.
Put $\bar{Z}=\tau(\bar{C})$ and $Z=\pi(\bar{Z})$.

\begin{lemma}
\label{lemma:cubic-final} One has
$L_{1}\cap\mathrm{Sing}(X)\ne\varnothing$.
\end{lemma}

\begin{proof}
Suppose that $L_{1}\cap\mathrm{Sing}(X)=\varnothing$. Then
$$
C\cap L_{1}=C\cap L_{2}=\varnothing,
$$
because $(L_{1}+L_{2}+L_{3})\cdot C=L_{3}\cdot C=2$. One can
easily check that
$$
\bar{Z}\sim \pi^{*}\big(-2K_{X}\big)-4E-\bar{L}_{1}-\bar{L}_{2},
$$
and $Z$ is singular at $P$.  Then $-2K_{X}\sim Z+L_{1}+L_{2}$. But
the log pair
$$
\left(U,\ \frac{1}{2}\Big(Z+L_{1}+L_{2}\Big)\right)
$$
is log canonical at $P$. Thus, we may assume that
$Z\not\subseteq\mathrm{Supp}(D)$ by Remark~\ref{remark:convexity}.

We have $Q\in\bar{Z}$ and $\bar{Z}\cdot E=2$. Then it follows from
the inequality~\ref{equation:blow-up} that
$$
4-2\mathrm{mult}_{P}\big(D\big)=\bar{Z}\cdot\bar{D}\geqslant\mathrm{mult}_{Q}\big(\bar{D}\big)>8/3-\mathrm{mult}_{P}\big(D\big),%
$$
which implies that $\mathrm{mult}_{P}(D)<4/3$. But
$\mathrm{mult}_{P}(D)>4/3$.
\end{proof}

Thus, we see that $L_{1}\cap
L_{3}=\mathrm{Sing}(X)\ne\varnothing$. Then $L_{1}\cap L_{2}\in
C$, which implies that
$$
\bar{Z}\sim \pi^{*}\big(-2K_{X}\big)-4E-2\bar{L}_{1}-\bar{L}_{2},
$$
and $Z$ is smooth rational cubic curve.  Then $-2K_{X}\sim
Z+2L_{1}+L_{2}$. But the log pair
$$
\left(U,\ \frac{1}{2}\Big(Z+2L_{1}+L_{2}\Big)\right)
$$
is log canonical at $P$. Thus, we may assume that
$Z\not\subseteq\mathrm{Supp}(D)$ by Remark~\ref{remark:convexity}.

We have $Q\in\bar{Z}$ and $\bar{Z}\cdot E=\bar{L}_{1}=1$. Then it
follows from the inequality~\ref{equation:blow-up} that
$$
3-\mathrm{mult}_{P}\big(\Delta\big)-2m_{1}-m_{2}=\bar{Z}\cdot\bar{\Delta}\geqslant\mathrm{mult}_{Q}\big(\bar{\Delta}\big)>8/3-\mathrm{mult}_{P}\big(\Delta\big)-m_{1}-m_{2},%
$$
which implies that $m_{1}<1/3$. But $m_{1}>2/3$ by
Corollary~\ref{corollary:cubic-m1}.

The obtained contradiction completes the proof
Theorem~\ref{theorem:cubic-surfaces}.

\section{Del Pezzo surfaces}
\label{section:del-Pezzo}

Let $X$ be a del Pezzo surface that has at most canonical
singularities, let $O$ be a point of the~surface $X$, and let
$B_{X}$ be an effective $\mathbb{Q}$-divisor on the surface $X$.
Suppose that
\begin{itemize}
\item the point $O$ is either smooth or an ordinary double point of $X$,%
\item the surface $X$ is smooth outside the point $O\in X$.%
\end{itemize}

\begin{lemma}
\label{lemma:singular-Del-Pezzo-2} Suppose that $\Sing(X)=O$, $K_{X}^{2}=2$ and
the equivalence
$$
B_{X}\qlineq -\mu K_{X}
$$
holds, where $0<\mu<2/3$. Then
$\mathbb{LCS}(X,\mu B_{X})=\varnothing$.
\end{lemma}

\begin{proof}
Suppose that $\mathbb{LCS}(X,\mu B_{X})\neq\varnothing$. There is a
curve $\mathbb{P}^{1}\cong\nlb L\subset\nlb X$ such that
$$
\mathrm{LCS}\Big(X, \mu B_{X}\Big)\not\subseteq L
$$
the equality $L\cdot L=-1$ holds, and
$L\cap\mathrm{Sing}(X)=\varnothing$. Thus, there is a birational
morphism $\pi\colon X\to S$ that contracts the curve $L$. Then
$$
\mathbb{LCS}\Big(S,\ \mu\pi\big(B_{X}\big)\Big)\ne\varnothing
$$
due to the choice of the curve $L\subset X$. But
$-K_{S}\qlineq\nlb\pi(B_{X})$, and $S\subset\mathbb{P}^{3}$ is a
cubic surface that has at most one ordinary double point, which is
impossible (see Examples~\ref{example:singular-cubics}
and~\ref{example:del-Pezzos}).
\end{proof}

\begin{lemma}
\label{lemma:del-Pezzo-quintic} Suppose that $\mathrm{Sing}(X)=\varnothing$,
$K_{X}^{2}=5$, the equivalence
$$
B_{X}\qlineq -\mu K_{X}
$$
holds, where $\mu\in\mathbb{Q}$ is such that $0<\mu<2/3$.
Assume that $\mathbb{LCS}(X, B_X)\neq\varnothing$.
Then
\begin{itemize}
\item either the set $\mathbb{LCS}(X, B_{X})$ contains a curve,%
\item or there are a curve $\mathbb{P}^{1}\cong L\subset X$ and a point $P\in
L$ such that $L\cdot L=-1$ and
$$
\mathrm{LCS}\Big(X,\ B_{X}\Big)=P.
$$
\end{itemize}
\end{lemma}

\begin{proof}
Suppose that $\mathbb{LCS}(X, B_{X})$ contains no curves. Then it follows from
Theorem~\ref{theorem:connectedness} that
$$
\mathrm{LCS}\Big(X,\ B_{X}\Big)=P,
$$
where $P\in X$ is a point. We may assume that $P$ does not lie on any curve
$\mathbb{P}^{1}\cong L\subset X$~such~that the equality $L\cdot L=-1$ holds.
Then there is a birational morphism
$\phi\colon X\longrightarrow\mathbb{P}^{2}$
such that $\phi$ is an isomorphism in a neighborhood of the point $P$. Note
that
$$
\phi\big(P\big)\in\LCS\Big(\mathbb{P}^{2},\ \phi\big(B_{X}\big)\Big)\ne\varnothing,%
$$
the set $\mathbb{LCS}(\mathbb{P}^{2},\phi(B_{X}))$ contains no curves, and
$$
\phi\big(B_{X}\big)\qlineq -\mu K_{\mathbb{P}^{2}}.
$$
Since $\mu<2/3$, the latter is impossible by
Lemma~\ref{lemma:plane}.
\end{proof}

\begin{example}
\label{example:del-Pezzo-quintic} Suppose that $O=\mathrm{Sing}(X)$ and
$K_{X}^{2}=5$. Let $\alpha\colon V\to X$ be a blow up of $O$,~and let $E$ be
the exceptional divisor of $\alpha$. Then there is a birational morphism
$\omega\colon V\to \mathbb{P}^{2}$ such that
\begin{itemize}
\item the morphism $\omega$ contracts the curves
$E_{1}$, $E_{2}$, $E_{3}$, $E_{4}$,%
\item the curve $\omega(E)$ is a line in $\mathbb{P}^{2}$ that contains
$\omega(E_{1})$, $\omega(E_{2})$, $\omega(E_{3})$, but $\omega(E)\not\ni
\omega(E_{4})$.
\end{itemize}

Let $Z$ be a line in
$\mathbb{P}^{2}$ such that $\omega(E_{1})\in Z\ni \omega(E_{4})$.
Then
$$
2E+\bar{Z}+2E_{1}+E_{2}+E_{3}\sim -K_{V},
$$
where $\bar{Z}\subset V$ is a proper transform of $Z$. One has
$$
\mathrm{lct}\Big(X,\
\alpha\big(\bar{Z}+2\alpha\big(E_{1}\big)+\alpha\big(E_{2}\big)+\alpha\big(E_{3}\big)\Big)=\frac{1}{2},%
$$
which implies $\mathrm{lct}(X)\leqslant 1/2$. Suppose that
$-K_{X}\qlineq 2B_{X}$, but $(X,B_{X})$ is not log canonical. Then
$$
K_{V}+B_{V}+mE\qlineq\alpha^{*}\Big(K_{X}+B_{X}\Big),
$$
for some $m\geqslant 0$, and $B_{V}$ is
the proper transform of $B_{X}$ on the surface $V$. Then
$$
\Big(V,\ B_{V}+mE\Big)
$$
is not log canonical at some point $P\in V$. There is a birational
morphism $\pi\colon V\to U$ such that
\begin{itemize}
\item the surface $U$ is a smooth del Pezzo surface with $K_{U}^{2}=6$,%
\item the morphism $\pi$ is an isomorphism in a neighborhood of $P\in X$,%
\end{itemize}
which implies that $(U,\pi(B_{V})+m\pi(E))$ is not log canonical
at $\pi(P)$. But
$$
\pi\big(B_{V}\big)+m\pi\big(E\big)\qlineq-\frac{1}{2}K_{U},
$$
which is impossible, because $\mathrm{lct}(U)=1/2$ (see
Example~\ref{example:del-Pezzos}). We see that
$\mathrm{lct}(X)=1/2$.
\end{example}

\begin{example}
\label{example:del-Pezzo-quartic} Suppose that $K_{X}^{2}=4$.
Arguing as in Example~\ref{example:del-Pezzo-quintic}, we see that
the equality
$$
\mathrm{lct}\big(X\big)=\left\{%
\begin{aligned}
&1/2\ \text{when}\ O=\mathrm{Sing}\big(X\big),\\
&2/3\ \text{when}\ \mathrm{Sing}\big(X\big)=\varnothing,
\end{aligned}
\right.
$$
holds. Take $\lambda\leqslant 1$. Suppose that
$$
B_{X}\qlineq -K_{X},
$$
and $(X,\lambda B_{X})$ is not log canonical at some point $P\in
X\setminus O$. There is a commutative diagram
$$
\xymatrix{
&V\ar@{->}[dl]_{\alpha}\ar@{->}[dr]^{\beta}&\\%
X\ar@{-->}[rr]_{\psi}&&U}
$$
where $U$ is a cubic surface in $\mathbb{P}^{3}$ that has
canonical singularities, the morphism $\alpha$ is a blow up of the
point $P$, the morphism $\beta$ is birational, and $\psi$ is a
projection from the point $P\in X$. Then
$$
K_{V}+\lambda
B_{V}+\Big(\lambda\mathrm{mult}_{P}\big(B_{X}\big)-1\Big)E\qlineq\alpha^{*}\Big(K_{X}+\lambda
B_{X}\Big),
$$
where $E$ is the exceptional divisor of $\alpha$, and $B_{V}$ is
the proper transform of $B_{X}$. Note that
$$
\Big(V,\ \lambda
B_{V}+\Big(\lambda\mathrm{mult}_{P}\big(B_{X}\big)-1\Big)E\Big)
$$
is not log canonical at some point $Q\in E$ and
$\mathrm{mult}_{P}(B_{X})>1/\lambda$. Then the log pair
$$
\Big(V,\ \lambda
B_{V}+\Big(\lambda\mathrm{mult}_{P}\big(B_{X}\big)-\lambda\Big)E\Big)
$$
is not log canonical at the point $Q\in E$ as well. But the
equivalences
$$
B_{V}+\Big(\mathrm{mult}_{P}\big(B_{X}\big)-1\Big)E\qlineq-K_{V}+\alpha^{*}\Big(K_{X}+B_{X}\Big)\qlineq -K_{V},%
$$
hold. Suppose that $P$ is not contained in any line on the surface
$X$. Then
\begin{itemize}
\item the morphism $\beta\colon V\to U$ is an isomorphism,%
\item the cubic surface $U$ is smooth outside of the point $\psi(O)$,%
\item the point $\psi(O)$ is at most ordinary double point of the surface $U$,%
\end{itemize}
which implies that $\lambda>2/3$ (see
Example~\ref{example:singular-cubics}).

Suppose that $\lambda=3/4$. Then the point
$$
\psi\big(Q\big)\in U\subset\mathbb{P}^{3}
$$
must be an Eckardt point (see
Definition~\ref{definition:Eckardt-point}) of the surface $U$ by
Theorem~\ref{theorem:cubic-surfaces}. But
$$
\beta\big(E\big)\subset U\subset\mathbb{P}^{3}
$$
is a line. So, there are two conics $C_{1}\ne C_{2}$ contained in $X$ such that
$P=C_{1}\cap C_{2}$ and
$$
C_{1}+C_{2}\sim -K_{X}.
$$
\end{example}

\begin{lemma}
\label{lemma:singular-del-Pezzo-sextic} Suppose that $O=\mathrm{Sing}(X)$ and
$K_{X}^{2}=6$ such that there is a diagram
$$
\xymatrix{
&V\ar@{->}[ld]_\alpha\ar@{->}[rd]^\beta&\\
X&&\P^2}
$$
where $\beta$ is a blow up of three points $P_1, P_2, P_3\in\P^2$ lying on a
line $L\subset\mathbb{P}^{2}$, and $\alpha$ is a birational morphism that
contracts an irreducible curve $\bar{L}$ to the point $O$ such that
$\beta(\bar{L})=L$. Then
$$
\mathrm{LCS}\Big(X,\ \lambda B_{X}\Big)=O
$$
in the case when $\mathrm{LCS}(X,\lambda B_{X})\ne\varnothing$,
$B_{X}\qlineq -K_{X}$ and $\lambda<1/2$.
\end{lemma}

\begin{proof}
Suppose that $B_{X}\qlineq -K_{X}$ and
$$
\varnothing\ne\mathrm{LCS}\Big(X,\ \lambda B_{X}\Big)\ne O,
$$
let $M\subset\P^2$ be a general line, and let $\bar{M}\subset V$ be its proper
transform. Then
$$
-K_X\sim 2\alpha\big(\bar{M}\big)
$$
and $O\in \alpha(\bar{M})$. Thus, the set $\mathbb{LCS}(X, \lambda B_{X})$
contains a curve, because otherwise the locus
$$
\LCS\Big(X,\ \lambda B_{X} +\alpha\big(\bar{M}\big)\Big)
$$
would be disconnected, which is impossible by
Theorem~\ref{theorem:connectedness}.

Let $C\subset X$ be an irreducible curve such that $C\subseteq\LCS(X, \lambda
B_{X})$. Then
$$
B_{X}=\eps C+\Omega,
$$
where $\eps>2$, and $\Omega$ is an
effective $\mathbb{Q}$-divisor such that $C\not\subset\mathrm{Supp}(\Omega)$.

Let $\Gamma_{i}\subset X$ be a proper transform of a general line in
$\mathbb{P}^{2}$ that passes through  $P_{i}$. Then
$$
O\not\in\Gamma_{1}\cup\Gamma_{2}\cup\Gamma_{3}
$$
and $-K_{X}\cdot\Gamma_{1}=-K_{X}\cdot\Gamma_{2}=-K_{X}\cdot\Gamma_{3}=2$. But
$$
-K_{X}\qlineq\Gamma_{1}+\Gamma_{2}+\Gamma_{3},
$$
which implies that there is $m\in\{1,2,3\}$ such that $C\cdot\Gamma_{m}\ne 0$.
Then
$$
2=B_{X}\cdot\Gamma_{m}=\Big(\eps C+\Omega\Big)\cdot\Gamma_{m}\geqslant\eps C\cdot\Gamma_{m}\geqslant\eps>2,%
$$
because $\Gamma_{m}\not\subset\mathrm{Supp}(B_{X})$. The obtained contradiction
completes the proof.
\end{proof}

\begin{remark}
\label{remark:del-Pezzo-sextic} Suppose that $O=\mathrm{Sing}(X)$
and $K_{X}^{2}=6$. Let $\alpha\colon V\to X$
be a blow up of the point $O\in X$, and let $E$ be the exceptional divisor of
$\alpha$. Then
$$
K_{V}+B_{V}+mE\qlineq\alpha^{*}\Big(K_{X}+B_{X}\Big)
$$
for some $m\geqslant 0$, and $B_{V}$ is the proper
transform of $B_{X}$ on  $V$.
Note that $\lct(X)\le 1/3$.
Suppose that $\lct(X)<1/3$, i.\,e. there exists an effective $\Q$-divisor
$B_{X}\qlineq -K_{X}$, such that the log pair
$(X, 1/3 B_{X})$ is not log canonical. Then the log pair
$$
\Big(V, \frac{1}{3}\big(B_{V}+ mE\big)\Big)
$$
is not log canonical at some point $P\in V$. There is a birational morphism
$\pi\colon V\longrightarrow U$
such that $U$ is either $\mathbb{F}_{1}$ or
$\mathbb{P}^{1}\times\mathbb{P}^{1}$ and $\pi$ is an isomorphism in a
neighborhood of $P\in X$. Then the log pair
$$
\Big(U, \frac{1}{3}\big(\pi(B_{V})+m\pi(E)\big)\Big)
$$
is not log canonical
at the point $\pi(P)$. But we know that
$$
-K_{U}\qlineq\pi\big(B_{V}\big)+m\pi\big(E\big),
$$
so the latter contradicts Example~\ref{example:del-Pezzos}. Hence
$\mathrm{lct}(X)=1/3$.
\end{remark}

\begin{lemma}
\label{lemma:quadric-cone-far-from-vertex} Suppose that
$X\cong\mathbb{P}(1,1,2)$ and $B_{X}\qlineq -K_X$, but there is a
point $P\in X$ such that
$$
O\ne P\in\mathrm{LCS}\Big(X,\ \lambda B_{X}\Big)
$$
for some $\lambda<1/2$. Take
$L\in|\mathcal{O}_{\mathbb{P}(1,1,2)}(1)|$ such that $P\in L$. Then
$L\subseteq\mathrm{LCS}(X, \lambda B_{X})$.
\end{lemma}

\begin{proof}
Suppose that there is a curve $\Gamma\in\LCS(X, \lambda B_{X})$ such that
$P\in\Gamma\ne L$. Then
$$
B_{X}=\mu\Gamma+\Omega,
$$
where $\mu>2$, and $\Omega$ is an effective
$\mathbb{Q}$-divisor such that $\Gamma\not\subset\mathrm{Supp}(\Omega)$.
Hence
$$
\mu\Gamma+\Omega\qlineq 4L
$$
and $\Gamma\sim mL$, where $m\in\mathbb{Z}_{>0}$. But $m\geqslant
2$, because $P\in\Gamma\ne L$, which is a contradiction.

Suppose that $L\not\subseteq\mathrm{LCS}(X, \lambda B_{X})$. Then it follows
from Theorem~\ref{theorem:connectedness} that
$$
\mathrm{LCS}\Big(X, \lambda B_{X}\Big)=P,%
$$
because we proved that $\mathbb{LCS}(X, \lambda B_{X})$ contains
no curves passing through $P\in X$.

Let $C\in|\mathcal{O}_{\mathbb{P}(1,1,2)}(1)|$ be a general curve. Then
$$
\mathrm{LCS}\Big(X, \lambda B_{X}+C\Big)=P\cup C,%
$$
which is impossible by Theorem~\ref{theorem:connectedness}.
\end{proof}

\begin{lemma}
\label{lemma:F1} Suppose that $X\cong\mathbb{F}_{1}$. Then there
are $0\leqslant\mu\in\mathbb{Q}\ni\lambda\geqslant 0$ such that
$$
B_{X}\qlineq \mu C+\lambda L,
$$
where $C$ and $L$ are irreducible curves on the surface $X$ such
that
$$
C\cdot C=-1,\ C\cdot L=1
$$
and $L\cdot L=0$. Suppose that $\mu<1$ and $\lambda<1$. Then
$\mathbb{LCS}(X, B_{X})=\varnothing$.
\end{lemma}

\begin{proof}
The set $\mathbb{LCS}(X, B_{X})$ contains no curves, because $L$
and $C$ generate the~cone of effective divisors of the surface
$X$. Suppose that $\mathbb{LCS}(X, B_{X})$ contains a point $O\in
X$. Then
$$
K_{X}+B_{X}+\Big(\big(1-\mu\big)C+\big(2-\lambda\big)L\Big)\qlineq -\big(L+C\big),%
$$
because  $-K_{X}\qlineq 2C+3L$. But it follows from
Theorem~\ref{theorem:Shokurov-vanishing} that the map
$$
0=H^{0}\Big(\mathcal{O}_{X}\big(-L-C\big)\Big)\longrightarrow
 H^{0}\Big(\mathcal{O}_{\mathcal{L}(X,\,B_{X})}\Big)\ne 0%
$$
is surjective, because the divisor $(1-\mu)C+(2-\lambda)L$ is
ample.
\end{proof}

\begin{lemma}
\label{lemma:del-Pezzo-septic}  Suppose that
$\mathrm{Sing}(X)=\varnothing$ and $K_{X}^{2}=7$. Then
$$
L_{1}\cdot L_{1}=L_{2}\cdot L_{2}=L_{3}\cdot L_{3}=-1,\ L_{1}\cdot L_{2}=L_{2}\cdot L_{3}=1,\ L_{1}\cdot L_{3}=0%
$$
where $L_{1}$, $L_{2}$, $L_{3}$ are exceptional curves on $X$.
Suppose that $\mathrm{LCS}(X, B_{X})\ne\varnothing$ and
$$
B_{X}\qlineq -\mu K_{X},
$$
where $\mu<1/2$. Then
$\mathrm{LCS}(X, B_{X})=L_{2}$.
\end{lemma}

\begin{proof}
Let $P$ be a point in $\mathrm{LCS}(X, B_{X})$. Then $P\in L_{2}$,
because $\mathrm{lct}(\mathbb{P}^{1}\times\mathbb{P}^{1})=1/2$,
and there is a birational morphism
$X\to\mathbb{P}^{1}\times\mathbb{P}^{1}$ that contracts only the
curve $L_{2}$.

Suppose that $\mathrm{LCS}(X, B_{X})\ne L_{2}$. Then
$\mathrm{LCS}(X, B_{X})=P$ by Theorem~\ref{theorem:connectedness}.

We may assume that $P\not\in L_3$. There is a birational morphism
$\phi\colon X\to\mathbb{P}^{2}$
that contracts the curves $L_{1}$ and $L_{3}$. Let $C_{1}$ and
$C_{3}$ be proper transforms on $X$ of sufficiently general lines
in $\mathbb{P}^{2}$ that pass through the points $\phi(L_{1})$ and
$\phi(L_{3})$, respectively. Then
$$
-K_{X}\sim C_{1}+2C_{3}+L_{3},
$$
and $C_{1}\not\ni P\not\in C_{2}$. Therefore, we see that
$$
C_{2}\cup P\subseteq\mathrm{LCS}\left(X,\ \lambda
D+\frac{1}{2}\Big(C_{1}+2C_{2}+L_{3}\Big)\right)\subseteq C_{2}\cup
P\cup L_3,%
$$
which is impossible by Theorem~\ref{theorem:connectedness},
because $P\not\in L_{3}$.
\end{proof}

\begin{lemma}
\label{lemma:quadric-cone} Suppose that $O=\mathrm{Sing}(X)$, the equality
$K_{X}^{2}=7$ holds, the equivalence
$$
B_{X}\qlineq C+\frac{4}{3}L
$$
holds, where $L\cong\mathbb{P}^{1}\cong C$ are curves on the surface $X$ such
that
$$
L\cdot L=-1/2,\ C\cdot C=-1,\ C\cdot L=1,
$$
but the log pair $(X, B_{X})$ is not log canonical at some point $P\in C$.
Then~$P\in L$.
\end{lemma}

\begin{proof}
Let $S$ be a quadratic cone in $\mathbb{P}^{3}$. Then $S\cong\mathbb{P}(1,1,2)$
and there is a birational morphism
$$
\phi\colon X\longrightarrow S\subset\mathbb{P}^{3}
$$
that contracts the curve $C$ to a smooth point $Q\in S$. Then
$Q\in\phi(L)\in|\mathcal{O}_{\mathbb{P}(1,1,2)}(1)|$.

Suppose that $P\not\in L$. Then it follows from Remark~\ref{remark:convexity}
that to complete the~proof we may assume that either
$C\not\subset\mathrm{Supp}(B_{X})$ or $L\not\subset\mathrm{Supp}(B_{X})$,
because the log pair
$$
\left(X,\ C+\frac{4}{3}L\right)
$$
is log canonical in the point $P\in X$. Suppose that
$C\not\subset\mathrm{Supp}(B_{X})$. Then
$$
\frac{1}{3}=B_{X}\cdot C\ge\mathrm{mult}_{P}\big(B_{X}\big)>1,
$$
which is impossible. Therefore, we see that $C\subset\mathrm{Supp}(B_{X})$.
Then  $L\not\subset\mathrm{Supp}(B_{X})$. Put
$$
B_{X}=\eps C+\Omega,
$$
where $\Omega$ is an
effective $\mathbb{Q}$-divisor such that $C\not\subset\mathrm{Supp}(\Omega)$.
Then
$$
\frac{1}{3}=B_{X}\cdot L=\eps+\Omega\cdot L\geqslant \eps,
$$
which implies that $\eps\leqslant 1/3$. Then
$$
1<\Omega\cdot C=1/3+\eps\leqslant 2/3
$$
by Lemma~\ref{lemma:adjunction}, which
is a contradiction.
\end{proof}

\section{Toric varieties}
\label{section:toric}

The purpose of this section is to prove Lemma~\ref{lemma:toric}
(cf. \cite{BaSe99}, \cite{So05}).

Let $N=\mathbb{Z}^{n}$ be a lattice of rank $n$, and let
$M=\mathrm{Hom}(N,\mathbb{Z})$ be the dual lattice. Put
$M_{\mathbb{R}}=M\otimes_\mathbb{Z}\mathbb{R}$ and
$N_\mathbb{R}=N\otimes_\mathbb{Z}\mathbb{R}$. Let $X$ be a toric
variety defined by a complete fan $\Sigma\subset N_{\mathbb{R}}$,
let
$$
\Delta_1=\big\{v_1,\ldots, v_m\big\}
$$
be a set of generators of one-dimensional cones of the fan
$\Sigma$. Put
$$
\Delta=\Big\{w\in M\ \Big\vert\ \big\langle w,
v_i\big\rangle\geqslant -1\ \text{for all}\ i=1,
 \ldots, m\Big\}.%
$$

Put $T=(\mathbb{C}^{*})^{n}\subset\Aut(X)$. Let $\mathcal{N}$ be
the normalizer of $T$ in $\mathrm{Aut}(X)$ and
$\mathcal{W}=\mathcal{N}/T$.

\begin{lemma}
\label{lemma:toric} Let $G\subset\mathcal{W}$ be a subgroup.
Suppose that $X$ is $\mathbb{Q}$-factorial. Then
$$
\mathrm{lct}\Big(X,\ G\Big)=\frac{1}{1+\mathrm{max}\Big\{\big\langle w, v\big\rangle\ \big\vert\ w\in\Delta^G,\ v\in\Delta_1\Big\}},%
$$
where $\Delta^G$ is the set of the points in $\Delta$ that are
fixed by the group $G$.
\end{lemma}

\begin{proof}
Put $\mu=1+\mathrm{max}\{\langle w, v\rangle\ \vert\
w\in\Delta^G,\ v\in\Delta_1\}$. Then $\mu\in\mathbb{Q}$ is the
largest number~such~that
$$
-K_X\sim_{\mathbb{Q}}\mu R+H,
$$
where $R$ is an integral $T\rtimes G$-invariant effective divisor,
and $H$ is an effective $\mathbb{Q}$-divisor. Hence
$$
\mathrm{lct}\big(X, G\big)\leqslant\frac{1}{\mu}.
$$

Suppose that $\mathrm{lct}(X, G)<1/\mu$. Then there is an
effective $G$-invariant $\mathbb{Q}$-divisor
$D\qlineq -K_{X}$
such~that the log pair $(X,\lambda D)$ is not log canonical for
some rational $\lambda<1/\mu$.

There exists a family $\{D_t \mid t\in\mathbb{C}\}$ of
$G$-invariant effective $\Q$-divisors such that
\begin{itemize}
\item the equivalence $D_t\qlineq D$ holds for every $t\in\mathbb{C}$,%
\item the equality $D_1=D$ holds,%
\item for every $t\neq 0$ there is $\phi_{t}\in\mathrm{Aut}(X)$
such that
$$
D_t=\phi_{t}\big(D\big)\cong D,
$$
\item the divisor $D_0$ is $T$-invariant,%
\end{itemize}
which implies that $(X, \lambda D_0)$ is not log canonical (see
\cite{DeKo01}).

On the other hand, the divisor $D_{0}$ does not have components
with multiplicity greater than~$\mu$, which implies that  $(X,
\lambda D_{0})$ is log canonical (see \cite{Ful93}), which is a
contradiction.
\end{proof}

\begin{corollary}\label{corollary:projectivization}
Let $X=\P_{\P^n}(\O_{\P^n}\oplus\O_{\P^n}(-a_1)\oplus\ldots\oplus
\O_{\P^n}(-a_k))$,
$a_i\ge 0$ for $i=1, \ldots, k$.
Then
$$
\lct(X)=\frac{1}{1+\mathrm{max}\Big\{k,
 n+\sum_{i=1}^{k}a_i\Big\}}.%
$$
\end{corollary}
\begin{proof}
Note that $X$ is a toric variety, and $\Delta_1$ consists of the
following vectors:
\begin{gather*}
(\overbrace{1, 0, \ldots, 0}^k, \overbrace{0, 0, \ldots, 0}^n),
\ldots,
(0, \ldots, 0, 1, 0, 0, \ldots, 0),\\
(-1, \ldots, -1, 0, 0, \ldots, 0), \\
(0, 0, \ldots, 0, 1, 0, \ldots, 0), \ldots,
(0, 0, \ldots, 0, 0, \ldots, 0, 1),\\
(-a_1, \ldots, -a_k, -1, \ldots, -1),
\end{gather*}
which implies the required assertion by Lemma~\ref{lemma:toric}.
\end{proof}

Applying Corollary~\ref{corollary:projectivization}, we obtain the
following result.

\begin{corollary}\label{corollary:2-33-2-35-2-36}
In the notation of section~\ref{section:intro} one has
$\mathrm{lct}(X)=1/4$, whenever $\gimel(X)\in\{2.33, 2.35\}$, and
$\mathrm{lct}(X)=1/5$, whenever $\gimel(X)=2.36$.
\end{corollary}

On the other hand, straightforward computations using
Lemma~\ref{lemma:toric} imply the following result.

\begin{corollary}\label{corollary:other-toric}
In the notation of section~\ref{section:intro} one has
$$
\mathrm{lct}\big(X\big)=\left\{%
\begin{aligned}
&1/3\ \text{whenever}\ \gimel(X)\in\{3.25, 3.31, 4.9, 4.11, 5.2\},\\
&1/4\ \text{whenever}\ \gimel(X)\in\{3.26, 3.30, 4.12\},\\
&1/5\ \text{whenever}\ \gimel(X)=3.29.
\end{aligned}
\right.
$$
\end{corollary}

\begin{remark}
\label{remark:symmetric-Fanos} Suppose that the toric variety $X$
is symmetric, i.e., $\Delta^{\mathcal{W}}=\{0\}$ (see
\cite{BaSe99}). Then it follows from Lemma~\ref{lemma:toric} that
$\lct(X, \mathcal{W})=1$.  Note that the equality $\lct(X,
\mathcal{W})=1$ is proved in \cite{BaSe99} and \cite{So05} under
an additional assumption that $X$ is smooth.
\end{remark}

\section{Del Pezzo threefolds}
\label{section:del-Pezzos}

We use the assumptions and notation of Theorem~\ref{theorem:main}.
Suppose that $-K_{X}\sim 2H$, where $H$ is a Cartier divisor that
is indivisible in $\mathrm{Pic}(X)$.

The purpose of this section is to prove the following result.

\begin{theorem}\label{theorem:del-Pezzo}
The equality $\lct(X)=1/2$ holds, unless $\gimel(X)=2.35$ when
$\lct(X)=1/4$.
\end{theorem}

It follows from Theorems~3.1.14 and~3.3.1 in \cite{IsPr99} that
$$
\gimel(X)\in\Big\{1.11, 1.12, 1.13, 1.14, 1.15, 2.32, 2.35, 3.27\Big\},%
$$
and by~\cite{Ch01b} and \cite{Ch08a} (see also
Lemma~\ref{lemma:double-covers}) one has $\lct(X)=1/2$ if
$\gimel(X)\in\{1.12, 1.13\}$.
It follows from Lemma~\ref{lemma:lct-product} that $\lct(X)=1/2$
when $\gimel(X)=3.27$.
Lemma~\ref{lemma:toric} implies that $\lct(X)=1/4$ when $\gimel(X)=2.35$.

%
%
%
%

We are left with the cases
$$
\gimel(X)\in\Big\{1.11, 1.14, 1.15, 2.32\Big\},
$$
while the inequality $\mathrm{lct}(X)\leqslant 1/2$ is obvious,
because $|H|$ is not empty.

\begin{lemma}
\label{lemma:P2-P2} Suppose that $\gimel(X)=2.32$. Then
$\lct(X)=1/2$.
\end{lemma}

\begin{proof}
We may suppose that $\mathrm{lct}(X)<1/2$. Then there is an
effective $\mathbb{Q}$-divisor $D\qlineq H$  such~that the log
pair $(X,\lambda D)$ is not log canonical for some positive
rational number $\lambda<1$.

The threefold  $X$ is a divisor on $\P^2\times\P^2$ of bi-degree
$(1, 1)$. There are two natural $\P^1$-bundles $\pi_1\colon
X\to\P^2$ and $\pi_2\colon X\to\P^2$, and applying
Theorem~\ref{theorem:Hwang} to $\pi_1$ and~$\pi_2$ we immediately
obtain a contradiction.
\end{proof}

\begin{remark}
\label{remark:rho-1-common} Suppose that $\mathrm{Pic}(X)=\Z[H]$,
and there is an effective $\mathbb{Q}$-divisor $D\qlineq H$
such~that the log pair $(X,\lambda D)$ is not log canonical for
some positive rational number $\lambda<1$. Put
$$
D=\eps S+\Omega\qlineq H,
$$
where $S$ is an irreducible surface and $\Omega$ is an effective
$\mathbb{Q}$-divisor such that
$$
\mathrm{Supp}\big(\Omega\big)\not\supset S.
$$
Then $\eps\leqslant 1$, because $\mathrm{Pic}(X)=\Z[H]$, which
implies that the set $\mathbb{LCS}(X,\lambda D)$ contains no
surfaces. Moreover, for any choice of $H\in |H|$ the locus
$$
\mathrm{LCS}\Big(X,\ \lambda D+H\Big)
$$
is connected by Theorem~\ref{theorem:connectedness}. Let $H$ be a
general surface in $|H|$. Since $\LCS(X, \lambda D+H)$ is
connected, one obtains that the locus $\LCS(X, \lambda D+ H)$ has
no isolated zero-dimensional components outside the base locus of
the linear system $|H|$. Note that $|H|$ has no base points at
all, unless $\gimel(X)=1.11$ when $\mathrm{Bs}|H|$ consists of a
single point $O$. Note that in the latter case $O\not\in\LCS(X,
\lambda D)$, since $X$ is covered by the curves of anticanonical
degree $2$ passing through $O$. Hence the locus $\LCS(X, \lambda
D)$ never has isolated zero-dimensional components; in particular,
it contains an (irreducible) curve $C$. Put $D\vert_{H}=\bar{D}$.
Then
$$
-K_{H}\sim H\big\vert_{H}\qlineq \bar{D},
$$
but $(H, \lambda\bar{D})$ is not log canonical in every point of
the intersection $H\cap C$. The locus
$\mathrm{LCS}(H, \lambda\bar{D})$
is connected by Theorem~\ref{theorem:connectedness}. But the
scheme $\mathcal{L}(H, \lambda\bar{D})$ is zero-dimensional. We
see that
$$
H\cdot C=\big|H\cap C\big|=1,
$$
and the locus $\mathrm{LCS}(X,\lambda D)$ contains no curves
besides the irreducible curve $C$.
\end{remark}

\begin{lemma}
\label{lemma:V4} Let $\gimel(X)=1.14$. Then $\lct(X)=1/2$.
\end{lemma}

\begin{proof}
We may suppose that $\mathrm{lct}(X)<1/2$. Then there is an
effective $\mathbb{Q}$-divisor $D\qlineq H$  such~that the log
pair $(X,\lambda D)$ is not log canonical for some positive
rational number $\lambda<1$.

The linear system $|H|$ induces an embedding
$X\subset\mathbb{P}^{5}$ such that $X$ is a complete intersection
of two quadrics. Then the locus $\mathrm{LCS}(X,\lambda D)$
consists of a single line $C\subset X\subset\mathbb{P}^{5}$ by
Remark~\ref{remark:rho-1-common}.

It follows from \cite[Proposition~3.4.1]{IsPr99} that there is a
commutative diagram
$$
\xymatrix{
&&V\ar@{->}[dl]_{\alpha}\ar@{->}[dr]^{\beta}&&&\\%
&X\ar@{-->}[rr]_{\psi}&&\mathbb{P}^{3}&}
$$
where $\psi$ is a projection from $C$, the morphism $\alpha$ is a
blow up of the line $C$, and $\beta$ is a blow up of a smooth
curve $Z\subset\mathbb{P}^{3}$ of degree $5$ and genus $2$.

Let $S$ be the exceptional divisor of $\beta$, and let $L$ be a
fiber of the morphism $\beta$ over a general point of the curve
$Z$. Put $\bar{S}=\alpha(S)$ and $\bar{L}=\alpha(L)$. Then
$\bar{S}\sim 2H$,
the curve $\bar{L}$ is a line, and $\mathrm{mult}_{C}(\bar{S})=3$.
But the log pair $(X, 1/2\bar{S})$ is log canonical, which implies
that we may assume that $\mathrm{Supp}(D)\not\supset\bar{S}$ by
Remark~\ref{remark:convexity}. Then
$$
1=\bar{L}\cdot D\geqslant\mathrm{mult}_{C}\big(D\big)>1/\lambda>1,
$$
which is a contradiction.
\end{proof}

\begin{remark}
\label{remark:singular-V4} Let $V\subset\mathbb{P}^{5}$ be a
complete intersection of two quadric hypersurfaces that has
isolated singularities, and let $B_{V}$ be an effective
$\mathbb{Q}$-divisor on $V$ such that $B_{V}\qlineq -K_{V}$~and
$$
\mathrm{LCS}\Big(V, \mu B_{V}\Big)\ne\varnothing,
$$
where $\mu<1/2$. Arguing as in the proof of Lemma~\ref{lemma:V4},
we see that
$$
\mathrm{LCS}\Big(V, \mu B_{V}\Big)\subseteq L,
$$
where $L\subset V$ is a line such that
$L\cap\Sing(V)\ne\varnothing$.
\end{remark}

\begin{lemma}
\label{lemma:V5} Suppose that $\gimel(X)=1.15$. Then
$\lct(X)=1/2$.
\end{lemma}

\begin{proof}
Analogous to that of Lemma~\ref{lemma:V4}.
\end{proof}

\begin{lemma}
Suppose that $\gimel(X)=1.11$. Then $\lct(X)=1/2$.
\end{lemma}

\begin{proof}
We may suppose that $\mathrm{lct}(X)<1/2$. Then there is an
effective $\mathbb{Q}$-divisor $D\qlineq H$  such~that the log
pair $(X,\lambda D)$ is not log canonical for some positive
rational number $\lambda<1$.

Recall that the threefold $X$ can be given by an equation
$$
w^{2}=t^{3}+t^{2}f_{2}\big(x,y,z\big)+tf_{4}\big(x,y,z\big)+f_{6}\big(x,y,z\big)\subset\mathbb{P}\big(1,1,1,2,3\big)\cong\mathrm{Proj}\Big(\mathbb{C}\big[x,y,z,t,w\big]\Big),
$$
where $\mathrm{wt}(x)=\mathrm{wt}(y)=\mathrm{wt}(z)=1$,
$\mathrm{wt}(t)=2$, $\mathrm{wt}(w)=3$, and $f_{i}$ is a
polynomial of degree $i$.

The locus $\mathrm{LCS}(X,\lambda D)$ consists of a single curve
$C\subset X$ such that $H\cdot C=1$ by
Remark~\ref{remark:rho-1-common}.

Let $\psi\colon X\dasharrow\mathbb{P}^{2}$ be the natural
projection. Then $\psi$ is not defined in a point $O$ that is cut
out by $x=y=z=0$. The curve $C$ does not contain the point $O$,
because otherwise we get
$$
1=\Gamma\cdot D\geqslant\mathrm{mult}_{O}\big(D\big)\mathrm{mult}_{O}\big(\Gamma\big)\geqslant\mathrm{mult}_{C}\big(D\big)>1/\lambda>1,%
$$
where $\Gamma$ is a general fiber of the projection $\psi$. Thus,
we see that $\psi(C)\subset\P^2$ is a line.

Let $S$ be the unique surface in $|H|$ such that $C\subset S$. Let
$L$ be a sufficiently general fiber of the rational map $\psi$
that intersects the curve $C$. Then $L\subset\Supp(D)$, since
otherwise
$$1=D\cdot L\ge \mult_C(D)>1/\lambda>1.$$

We may assume that $D=S$ by Remark~\ref{remark:convexity}. Then
$S$ has a cuspidal singularity along $C$.

We may assume that the surface $S$ is cut out on $X$ by the
equation $x=0$, and we may assume that the curve $C$ is given by
$w=t=x=0$. Then $S$ is given by
$$
w^{2}=t^{3}+t^{2}f_{2}\big(0,y,z\big)+tf_{4}\big(0,y,z\big)\subset\mathbb{P}\big(1,1,2,3\big)\cong\mathrm{Proj}\Big(\mathbb{C}\big[y,z,t,w\big]\Big),
$$
and $f_{6}(x,y,z)=xf_{5}(x,y,z)$, where $f_{5}(x,y,z)$ is a
homogeneous polynomial of degree $5$.

Since the surface $S$ is singular along the curve $C$, one has
$$
f_{4}\big(x,y,z\big)=xf_{3}\big(x,y,z\big),
$$
where $f_{3}(x,y,z)$ is a homogeneous polynomial of degree $3$.
Then every point of the set
$$
x=f_{5}\big(x,y,z\big)=t=w=0\subset\mathbb{P}\big(1,1,1,2,3\big)
$$
must be singular on $X$, which is a contradiction, because $X$ is
smooth.
\end{proof}

The assertion of Theorem~\ref{theorem:del-Pezzo} is completely
proved.

\section{Fano threefolds with $\rho=2$}
\label{section:rho-2}

We use the assumptions and notation introduced in
section~\ref{section:intro}.

\begin{lemma}\label{lemma:2-1-2-3}
Suppose that $\gimel(X)=2.1$ or $\gimel(X)=2.3$. Then
$\lct(X)=1/2$.
\end{lemma}
\begin{proof}
There is a birational morphism $\alpha\colon X\to V$ that
contracts a surface $E\subset X$ to a smooth elliptic curve
$C\subset V$, where  $V$ is one of the following Fano threefolds:
\begin{itemize}
\item smooth hypersurface in $\mathbb{P}(1,1,1,2,3)$ of degree $6$;%
\item smooth hypersurface in $\mathbb{P}(1,1,1,1,2)$ of degree $4$.%
\end{itemize}

The curve $C$ is contained in a surface $H\subset V$ such that
$$
\mathrm{Pic}\big(V\big)=\mathbb{Z}\big[H\big]
$$
and $-K_{X}\sim 2H$. Then $C$ is a complete intersection of two
surfaces in $|H|$, and
$$
-K_{X}\sim 2\bar{H}+E,
$$
where $E$ is the exceptional divisor of the birational morphism
$\alpha$, and $\bar{H}$ is a proper transform of the surface $H$
on the threefold $X$. In particular, the inequality $\lct(X)\le
1/2$ holds.

We suppose that $\mathrm{lct}(X)<1/2$. Then there exists an
effective $\mathbb{Q}$-divisor $D\qlineq -K_{X}$ such that the log
pair $(X,\lambda D)$ is not log canonical for some positive
rational number $\lambda<1/2$. Then
$$
\varnothing\ne\mathrm{LCS}\Big(X,\lambda D\Big)\subseteq E,
$$
since $\mathrm{lct}(V)=1/2$ by Theorem~\ref{theorem:del-Pezzo} and
$\alpha(D)\qlineq 2H\sim -K_{V}$.

Put $k=H\cdot C$. Then $k=H^{3}\in\{1,2\}$. Note that
$$
\mathcal{N}_{C\slash
 V}\cong\mathcal{O}_{C}\Big(H\big\vert_{C}\Big)\oplus\mathcal{O}_{C}\Big(H\big\vert_{C}\Big),%
$$
which implies that $E\cong C\times\mathbb{P}^{1}$. Let $Z\cong C$
and $L\cong\mathbb{P}^{1}$ be curves on $E$ such that
$$
Z\cdot Z=L\cdot L=0
$$
and $Z\cdot L=1$. Then $\alpha^{*}(H)\vert_{E}\sim kL$, and since
$$
-2Z\sim K_E\sim \Big(K_X+E\Big)\Big\vert_{E}\sim
 \Big(2E-2\alpha^{*}\big(H\big)\Big)\Big\vert_{E}\sim -2kL+2E\Big\vert_E,%
$$
we see that $E\vert_E\sim -Z+kL$. Put
$$
D=\mu E+\Omega,
$$
where $\Omega$ is an effective $\mathbb{Q}$-divisor on $X$ such
that $E\not\subset\mathrm{Supp}(\Omega)$.

The pair $(X, E+\lambda\Omega)$ is not log canonical in the
neighborhood of $E$. Hence, the log pair
$$
\Big(E,\ \lambda\Omega\Big\vert_{E}\Big)
$$
is also not log canonical by Theorem~\ref{theorem:adjunction}. But
$$
\Omega\Big\vert_{E}\qlineq\Big(-K_X-\mu E\Big)\Big\vert_{E}\qlineq
\Big(2\alpha^{*}\big(H\big)-\big(1+\mu\big)E\Big)\Big\vert_{E}\qlineq
\big(1+\mu\big)Z+k\big(1-\mu\big)L,%
$$
and $0\le\lambda k(1-\mu)\le 1$, which contradicts
Lemma~\ref{lemma:elliptic-times-P1}.
\end{proof}

\begin{lemma}
\label{lemma:2-4} Suppose that $\gimel(X)=2.4$ and $X$ is general.
Then $\mathrm{lct}(X)=3/4$.
\end{lemma}
\begin{proof}
There is a commutative diagram
$$
\xymatrix{
&&X\ar@{->}[dl]_{\alpha}\ar@{->}[dr]^{\beta}&&&\\%
&\mathbb{P}^3\ar@{-->}[rr]_{\psi}&&\mathbb{P}^{1}&}
$$
where $\psi$ is a rational map, $\alpha$ is a blow~up of a smooth
curve $C\subset\mathbb{P}^{3}$ such that
$$
C=H_{1}\cdot H_{2}
$$
for some $H_{1}, H_2\in |\mathcal{O}_{\mathbb{P}^{3}}(3)|$, and
$\beta$ is a fibration into cubic surfaces.

Let $\mathcal{P}$ be a pencil in
$|\mathcal{O}_{\mathbb{P}^{3}}(3)|$  generated by $H_{1}$ and
$H_{2}$. Then $\psi$ is given by $\mathcal{P}$.

We assume that $X$ satisfies the following generality conditions:
\begin{itemize}
\item every surface in $\mathcal{P}$ has at most one ordinary double point;%
\item the curve $C$ contains no Eckardt points\footnote{
Note that $C$ does not contain singular points of the surfaces in
$\mathcal{P}$ since $C$ is a complete intersection of two surfaces from
$\mathcal{P}$.} (see Definition~\ref{definition:Eckardt-point}) of any surface in $\mathcal{P}$.%
\end{itemize}

Let $E$  be the exceptional divisor of the birational morphism
$\alpha$. Then
$$
\frac{4}{3}\bar{H}_{1}+\frac{1}{3}E\qlineq \frac{4}{3}\bar{H}_{2}+\frac{1}{3}E\qlineq -K_{X},%
$$
where $\bar{H}_{i}$ is a proper transform of $H_{i}$ on the
threefold $X$. We see that $\mathrm{lct}(X)\leqslant 3/4$.

We suppose that $\mathrm{lct}(X)<3/4$. Then there exists an
effective $\mathbb{Q}$-divisor $D\qlineq -K_{X}$ such that the log
pair $(X,\lambda D)$ is not log canonical for some positive
rational number $\lambda<3/4$.

Suppose that the set $\mathbb{LCS}(X,\lambda D)$ contains a
(irreducible) surface $S\subset X$. Then
$$
D=\eps S+\Delta,
$$
where $\eps\geqslant 1/\lambda$, and $\Delta$ is an effective
$\mathbb{Q}$-divisor such that
$S\not\subset\mathrm{Supp}(\Delta)$. Then
$$
\Big(\bar{H}_{1},\ D\Big\vert_{\bar{H}_{1}}\Big)
$$
is not log canonical by Remark~\ref{remark:hyperplane-reduction}
if $S\cap\bar{H}_1\ne\varnothing$. But
$$
D\Big\vert_{\bar{H}_{1}}\qlineq -K_{\bar{H}_{1}}.
$$
We can choose $\bar{H}_{1}$ to be a smooth cubic surface in
$\mathbb{P}^{3}$. Thus, it follows from
Theorem~\ref{theorem:cubic-surfaces} that
$$
S\cap\bar{H}_{1}=\varnothing,
$$
which implies that $S\sim\bar{H}_{1}$. Thus, we see that
$\alpha(S)$ is a cubic surface in $\mathcal{P}$. Then
$$
\eps\alpha\big(S\big)+\alpha\big(\Delta\big)\qlineq\mathcal{O}_{\mathbb{P}^{3}}\big(4\big),
$$
which is impossible, because $\eps\geqslant 1/\lambda>4/3$.

Let $F$ be a fiber of $\beta$ such that
$F\cap\mathrm{LCS}(X,\lambda D)\ne\varnothing$. Put
$$
D=\mu F+\Omega,
$$
where $\Omega$ is an effective $\mathbb{Q}$-divisor such that
$F\not\subset\mathrm{Supp}(\Omega)$. Then the log pair
$(F,\lambda\Omega\vert_F)$ is not log canonical by
Theorem~\ref{theorem:adjunction}, because $\lambda\mu<1$.  It
follows from Theorem~\ref{theorem:cubic-surfaces} that
$$
\mathrm{LCS}\Big(F,\ \lambda\Omega\Big\vert_{F}\Big)=O,
$$
where either $O$ is an Eckardt point of the surface $F$, or
$O=\mathrm{Sing}(F)$. By Theorem~\ref{theorem:connectedness}
$$
\mathrm{LCS}\Big(X,\ \lambda D\Big)=\mathrm{LCS}\Big(X,\ \lambda\mu F+\lambda\Omega D\Big)=O,%
$$
because it follows from Theorem~\ref{theorem:adjunction} that $(X,
F+\lambda\Omega D)$ is not log canonical at $O$ and is log
canonical in a punctured neighborhood of $O$. But $O\not\in E$ by
our generality assumptions. Then
$$
\alpha\big(O\big)\subset\mathrm{LCS}\Big(\mathbb{P}^{3},\ \lambda\alpha\big(D\big)\Big)\subseteq\alpha(O)\cup C,%
$$
where $\alpha(O)\not\in C$. But $\lambda<3/4$, which contradicts
Lemma~\ref{lemma:plane}.
\end{proof}

\begin{lemma}
\label{lemma:2-5-and-2-10-and-2-14} Suppose that
$\gimel(X)\in\{2.5, 2.10, 2.14\}$ and $X$ is general. Then
$\mathrm{lct}(X)=1/2$.
\end{lemma}

\begin{proof}
There is a commutative diagram
$$
\xymatrix{
&&X\ar@{->}[dl]_{\alpha}\ar@{->}[dr]^{\beta}&&&\\%
&V\ar@{-->}[rr]_{\psi}&&\mathbb{P}^{1}&}
$$
where $V$ is a smooth Fano threefold such that $-K_{V}\sim 2H$ for
some $H\in\mathrm{Pic}(V)$ and
$$
\gimel(V)\in\big\{1.13,\ 1.14,\ 1.15\big\},
$$
the morphism $\alpha$ is a blow~up of a smooth curve $C\subset V$
such that
$$
C=H_{1}\cdot H_{2}
$$
for some $H_{1}, H_{2}\in |H|$, $H_1\neq H_2$, the morphism
$\beta$ is a del Pezzo fibration, and $\psi$ is a linear
projection.

Let $E$  be the exceptional divisor of the birational morphism
$\alpha$. Then
$$
2\bar{H}_{1}+E\sim 2\bar{H}_{2}+E\sim -K_{X},
$$
where $\bar{H}_{i}$ is a proper transform of $H_{i}$ on the
threefold $X$. We see that $\mathrm{lct}(X)\leqslant 1/2$.

We suppose that $\mathrm{lct}(X)<1/2$. Then there exists an
effective $\mathbb{Q}$-divisor $D\qlineq -K_{X}$ such that the~log
pair $(X,\lambda D)$ is not log canonical for some positive
rational number $\lambda<1/2$. Then
$$
\varnothing\ne\mathrm{LCS}\Big(X,\lambda D\Big)\subseteq E,
$$
because $\alpha(D)\qlineq -K_V$ and $\mathrm{lct}(V)=1/2$ by
Theorem~\ref{theorem:del-Pezzo}.

We assume that the threefold $X$ satisfies the following
generality condition: every fiber of the del Pezzo fibration
$\beta$ has at most one singular point that is an ordinary double
point.

Let $F$ be a fiber of $\beta$ such that
$F\cap\mathrm{LCS}(X,\lambda D)\ne\varnothing$. Put
$$
D=\mu F+\Omega,
$$
where $\Omega$ is an effective $\mathbb{Q}$-divisor on $X$ such
that $F\not\subset\mathrm{Supp}(\Omega)$. Then
$$
\alpha\big(D\big)=\mu\alpha\big(F\big)+\alpha\big(\Omega\big)\qlineq 2\alpha\big(F\big)\qlineq -K_{V},%
$$
which implies that $\mu\leqslant 2$. Then
$(F,\lambda\Omega\vert_F)$ is also not log canonical by
Theorem~\ref{theorem:adjunction}. But
$$
\Omega\Big\vert_{F}\qlineq -K_{F},
$$
which implies that $\mathrm{lct}(F)\leqslant \lambda<1/2$. But $F$
has at most one ordinary double point and
$$
K_{F}^{2}=H^{3}\leqslant 5,
$$
which implies that $\mathrm{lct}(F)\geqslant 1/2$ (see
Examples~\ref{example:del-Pezzos}, \ref{example:singular-cubics},
\ref{example:del-Pezzo-quintic}~and~\ref{example:del-Pezzo-quartic}),
which is a contradiction.
\end{proof}

\begin{lemma}
\label{lemma:2-8} Suppose that $\gimel(X)=2.8$ and $X$ is general.
Then $\mathrm{lct}(X)=1/2$.
\end{lemma}

\begin{proof}
Let $O\in\mathbb{P}^{3}$ be a point, and let $\alpha\colon
V_{7}\to\mathbb{P}^{3}$ be a blow up of the point $O$. Then
$$
V_{7}\cong\mathbb{P}\Big(\mathcal{O}_{\mathbb{P}^{2}}\oplus\mathcal{O}_{\mathbb{P}^{2}}\big(1\big)\Big),
$$
and there is a $\mathbb{P}^{1}$-bundle $\pi\colon
V_{7}\to\mathbb{P}^{2}$. Let $E$ be the exceptional divisor of
$\alpha$. Then $E$ is a section~of~$\pi$.

There is a quartic surface $R\subset\mathbb{P}^{3}$ such that
$\mathrm{Sing}(R)=O$, the point $O$ is an isolated double point of
the surface $R$, and there is a commutative diagram
$$
\xymatrix{
&&&&X\ar@{->}[dl]_{\eta}\ar@/_1pc/@{->}[dlll]_{\beta}\ar@{->}[dd]^{\phi}\\%
&V_{2}\ar@{->}[rd]_{\omega}&&V_{7}\ar@{->}[dl]_{\alpha}\ar@{->}[dr]^{\pi}\\%
&&\mathbb{P}^{3}\ar@{-->}^{\psi}[rr]&&\mathbb{P}^{2}}
$$
where $\omega$ is a double cover branched in $R$, the morphism
$\eta$ is a double cover branched in the proper transform of $R$,
the morphism $\beta$ is a birational morphism that contracts a
surface $\bar{E}$ such that $\eta(\bar{E})=E$ to the singular
point of $V_{2}$ and
$$
\omega\Big(\mathrm{Sing}\big(V_{2}\big)\Big)=O,
$$
the map $\psi$~is~a~projection from the point $O$, and $\phi$ is a
conic bundle.

We assume that $X$ satisfies the following mild generality
condition: the point $O$ is an \emph{ordinary} double point of the
surface $R$. Then $
\bar{E}\cong\mathbb{P}^{1}\times\mathbb{P}^{1}$.

Let $\bar{H}$ be the proper transform on $X$ of the general plane
in $\mathbb{P}^{3}$ that passes through $O$. Then
$$
-K_{X}\sim 2\bar{H}+\bar{E},
$$
which implies that $\mathrm{lct}(X)\leqslant 1/2$.

We suppose that $\mathrm{lct}(X)<1/2$. Then there exists an
effective $\mathbb{Q}$-divisor $D\qlineq -K_{X}$ such that the log
pair $(X,\lambda D)$ is not log canonical for some positive
rational number $\lambda<1/2$.

It follows from Lemma~\ref{lemma:double-covers} that
$\mathrm{LCS}(X,D)\cap\bar{E}\ne\varnothing$. Put
$$
D=\mu\bar{E}+\Omega,
$$
where $\Omega$ is an effective $\mathbb{Q}$-divisor on $X$ such
that $\bar{E}\not\subset\mathrm{Supp}(\Omega)$. Then
$$
2=D\cdot\Gamma=\Big(\mu\bar{E}+\Omega\Big)\cdot\Gamma=2\mu+\Omega\cdot\Gamma\geqslant 2\mu,%
$$
where $\Gamma$ is a general fiber of the conic bundle $\phi$.
Hence $\mu\leqslant 1$. Thus, the log pair
$$
\Big(\bar{E},\ \lambda\Omega\Big\vert_{\bar{E}}\Big)
$$
is not log canonical by Theorem~\ref{theorem:adjunction}, because
$\mathrm{LCS}(X,D)\cap\bar{E}\ne\varnothing$. But
$$
\Omega\Big\vert_{\bar{E}}\qlineq -\frac{1+\mu}{2}K_{\bar{E}},
$$
which is impossible by Lemma~\ref{lemma:elliptic-times-P1}.
\end{proof}

\begin{lemma}
\label{lemma:2-11} Suppose that $\gimel(X)=2.11$ and $X$ is
general. Then $\lct(X)=1/2$.
\end{lemma}

\begin{proof}
Let $V$ be a cubic hypersurface in $\P^4$. Then there is a
commutative diagram
$$
\xymatrix{
&&V\ar@{->}[dl]_{\alpha}\ar@{->}[dr]^{\beta}&&&\\%
&X\ar@{-->}[rr]_{\psi}&&\mathbb{P}^{2}&}
$$
such that $\alpha$ contracts a surface $E\subset V$ to a line
$L\subset X$, the map $\psi$ is a linear projection from the line
$L$, the morphism $\beta$ is a conic bundle.

We assume that $X$ satisfies the following generality condition:
the normal bundle $\mathcal{N}_{L\slash V}$
to the line $L$ on the variety $V$ is isomorphic to
$\mathcal{O}_{L}\oplus\mathcal{O}_{L}$.

Let $H$ be a hyperplane section of $V$ such that $L\subset
H$.~Then
$$
-K_{X}\sim 2\bar{H}+E,
$$
where $\bar{H}\subset X$ is the proper transform of the surface
$H$. In particular, $\lct(X)\le 1/2$.

We suppose that $\mathrm{lct}(X)<1/2$. Then there exists an
effective $\mathbb{Q}$-divisor $D\qlineq -K_{X}$ such that the log
pair $(X,\lambda D)$ is not log canonical for some positive
rational number $\lambda<1/2$. Then
$$
\varnothing\ne\mathrm{LCS}\Big(X,\lambda D\Big)\subseteq E,
$$
since $\mathrm{lct}(V)=1/2$ and $\alpha(D)\qlineq -K_V$. Note that
$E\cong\P^1\times\mathbb{P}^{1}$ by the generality condition.

Let $F\subset E$ be a fiber of the induced projection $E\to L$,
let $Z\subset E$ be a section of this projection such that $Z\cdot
Z=0$. Then $\alpha^{*}(H)\vert_{E}\sim F$ and $E\vert_{E}\sim -Z$,
because
$$
-2Z-2F\sim K_E\sim\Big(K_X+E\Big)\Big\vert_{E}\sim
2\Big(E-\alpha^{*}\big(H\big)\Big)\Big\vert_{E}\sim -2F+2E\Big\vert_{E}.%
$$

Put $D=\mu E+\Omega$,
where $\Omega$ is an effective $\mathbb{Q}$-divisor on $X$ such
that $E\not\subset\mathrm{Supp}(\Omega)$. Then
$$
2=D\cdot\Gamma=\mu E\cdot\Gamma+\Omega\cdot\Gamma\geqslant\mu E\cdot\Gamma=2\mu,%
$$
where $\Gamma$ is a general fiber of the conic bundle $\beta$.
Thus, we see that $\mu\leqslant 1$.

The log pair $(E,\lambda\Omega\vert_E)$ is not log canonical by
Theorem~\ref{theorem:adjunction}. But
$$
\Omega\Big\vert_{E}\qlineq \Big(-K_X-\mu
E\Big)\Big\vert_{E}\qlineq\big(1+\mu\big)Z+2F,
$$
which contradicts Lemma~\ref{lemma:elliptic-times-P1}, because
$\mu\leqslant 1$ and $\lambda<1/2$.
\end{proof}

\begin{lemma}
\label{lemma:2-15} Suppose that $\gimel(X)=2.15$ and $X$ is
general. Then $\mathrm{lct}(X)=1/2$.
\end{lemma}

\begin{proof}
There is a birational morphism $\alpha\colon X\to\mathbb{P}^3$
that contracts a surface $E\subset X$ to a smooth curve
$C\subset\mathbb{P}^{3}$ that is complete intersection of an
irreducible quadric $Q\subset\mathbb{P}^{3}$ and a cubic
$F\subset\mathbb{P}^{3}$.

We assume that $X$ satisfies the following generality condition:
the quadric $Q$ is smooth.

Let $\bar{Q}$ be a proper transform of $Q$ on the threefold $X$.
Then there is a commutative diagram
$$
\xymatrix{
&X\ar@{->}[dl]_{\alpha}\ar@{->}[dr]^{\beta}&\\%
\mathbb{P}^{3}&&V\ar@{-->}^{\gamma}[ll]}
$$
where $V$ is a cubic in $\mathbb{P}^{4}$ that has one ordinary
double point $P\in V$, the morphism~$\beta$~contracts the~surface
$\bar{Q}$ to the point $P$, and $\gamma$ is a linear projection
from the point $P$.

Let $E$ be the exceptional divisor of $\alpha$. Then
$$
-K_{X}\sim 2\bar{Q}+E
$$
and $\beta(E)\subset V$ is a surface that contains all lines on
$V$ that pass through $P$. We see that $\mathrm{lct}(X)\leqslant
1/2$.

We suppose that $\mathrm{lct}(X)<1/2$. Then there exists an
effective $\mathbb{Q}$-divisor $D\qlineq -K_{X}$ such that the log
pair $(X,\lambda D)$ is not log canonical for some positive
rational number $\lambda<1/2$.

It follows from Lemma~\ref{lemma:singular-cubic-threefold} that
either
$$
\varnothing\ne\mathrm{LCS}\Big(X,\ \lambda D\Big)\subseteq\bar{Q},
$$
or the set $\mathrm{LCS}(X,\lambda D)$ contains a fiber of the
natural projection $E\to C$. In both cases
$$
\mathrm{LCS}\Big(X,\ \lambda D\Big)\cap\bar{Q}\ne\varnothing.
$$

We have $\bar{Q}\cong\mathbb{P}^{1}\times\mathbb{P}^{1}$. Put
$$
D=\mu\bar{Q}+\Omega,
$$
where $\Omega$ is an effective $\mathbb{Q}$-divisor on $X$ such
that $\bar{Q}\not\subset\mathrm{Supp}(\Omega)$. Then
$$
\alpha\big(D\big)\qlineq \mu Q+\alpha\big(\Omega\big)\qlineq -K_{\mathbb{P}^{3}},%
$$
which gives $\mu\leqslant 2$. The log pair
$(\bar{Q},\lambda\Omega\vert_{\bar{Q}})$ is not log canonical by
Theorem~\ref{theorem:adjunction}. But
$$
\Omega\Big\vert_{\bar{Q}}\qlineq -\frac{1+\mu}{2}K_{\bar{Q}},
$$
which implies that $\mu>1$ by Lemma~\ref{lemma:elliptic-times-P1}.

It follows from Remark~\ref{remark:convexity} that we may assume
that $E\not\subset\mathrm{Supp}(D)$. Then
$$
1=D\cdot F=\mu \bar{Q}\cdot F+\Omega\cdot F=\mu+ \Omega\cdot F\geqslant \mu,%
$$
where $F$ is a general fiber the natural projection $E\to C$. But
$\mu>1$, which is a contradiction.
\end{proof}

\begin{lemma}
\label{lemma:2-18} Suppose that  $\gimel(X)=2.18$. Then
$\lct(X)=1/2$.
\end{lemma}

\begin{proof}
There is a smooth divisor $B\subset\P^1\times\P^2$ of bidegree
$(2, 2)$ such that the diagram
$$
\xymatrix{
&& X\ar@{->}[d]_{\pi}\ar@{->}[lld]_{\phi_1}\ar@{->}[rrd]^{\phi_2} & \\
\P^1 &&\P^1\times\P^2\ar@{->}[ll]^{\pi_1}\ar@{->}[rr]_{\pi_2}&&
\P^2}
$$
commutes, where $\pi$ is a double cover that is  branched in $B$,
the morphisms $\pi_1$ and $\pi_2$ are natural projections, the
morphism $\phi_1$ is a quadric fibration, and $\phi_2$ is a conic
bundle.

Let $H_1$ be a general fiber of $\pi_1$, and let $H_2$ be a
general surface in $|\pi_2^{*}(\mathcal{O}_{\mathbb{P}^{2}}(1))|$.
Then $B\sim 2H_1+2H_2$.

Let $\bar{H}_1$ be a general fiber of $\phi_1$, and let
$\bar{H}_2$ be a general surface in
$|\phi_2^{*}(\mathcal{O}_{\mathbb{P}^{2}}(1))|$. Then
$$
-K_X\sim \bar{H_1}+2\bar{H_2},
$$
which implies that $\lct(X)\leqslant 1/2$.

We suppose that $\mathrm{lct}(X)<1/2$. Then there exists an
effective $\mathbb{Q}$-divisor $D\qlineq -K_{X}$ such that the log
pair $(X,\lambda D)$ is not log canonical for some positive
rational number $\lambda<1/2$.

Applying Lemma~\ref{lemma:Hwang} to the fibration $\phi_1$ we see
that
$$
\varnothing\neq\LCS\Big(X,\ \lambda D\Big)\subseteq Q,
$$
where $Q$ is a singular fiber of $\phi_1$. Moreover, applying
Theorem~\ref{theorem:Hwang} to the fibration $\phi_2$, we see that
$$
\varnothing\neq\LCS\Big(X,\ \lambda D\Big)\subseteq Q\cap R,
$$
where $R\subset X$ be an irreducible surface that is swept out  by
singular fibers of $\phi_2$. In particular, the set
$\mathbb{LCS}(X, \lambda D)$ contains no surfaces.

Suppose that $\LCS(X, \lambda D)$ is zero-dimensional. Then
$$
\LCS\left(X,\ \lambda
D+\frac{1}{2}\Big(\bar{H}_1+2\bar{H}_2\Big)\right)=\LCS(X,\
\lambda D)\cup
\bar{H}_2,%
$$
which is impossible by Theorem~\ref{theorem:connectedness}.

We see that the set $\mathbb{LCS}(X, \lambda D)$ contains a curve
$\Gamma\subset Q\cap R$. Put
$$
D=\mu Q+\Omega,
$$
where $\Omega$ is an effective $\mathbb{Q}$-divisor such that
$Q\not\subset\mathrm{Supp}(\Omega)$. Then
$$
\Big(Q,\ \lambda\Omega\Big\vert_Q\Big)
$$
is not log canonical along $\Gamma$ by
Theorem~\ref{theorem:adjunction}. But
$$
\Omega\Big\vert_{Q}\qlineq \Big(-K_X-\mu
Q\Big)\Big\vert_{Q}\qlineq -K_{Q},
$$
which implies that $\Gamma$ is a ruling of the cone
$Q\subset\mathbb{P}^{3}$ by
Lemma~\ref{lemma:quadric-cone-far-from-vertex}. Then
$\phi_2(\Gamma)\subset\mathbb{P}^{2}$ is a line, and
$$
\phi_2\big(\Gamma\big)\subseteq\phi_2\big(R\big).
$$
But $\phi_2(R)\subset\P^2$ is a curve of degree $4$. Thus, we see
that
$$
\phi_2\big(R\big)=\phi_2\big(\Gamma\big)\cup Z,
$$
where $Z\subset\mathbb{P}^{2}$ is a reduced cubic curve. Then
$\phi_{2}$ induces a double cover of
$$
\phi_2\big(\Gamma\big)\setminus\Big(\phi_2\big(\Gamma\big)\cap
Z\Big)
$$
that must be unramified (see \cite{Sa80}). But the quartic curve
$\phi_2(R)$ has at most ordinary double points (see
\cite{Sa80}).~Then
$$
\big|\phi_2\big(\Gamma\big)\cap Z|=3,
$$
which is impossible, because $\phi_2(\Gamma)\cong\mathbb{P}^{1}$.
\end{proof}

\begin{lemma}
\label{lemma:2-19} Suppose that $\gimel(X)=2.19$ and $X$ is
general. Then $\mathrm{lct}(X)=1/2$.
\end{lemma}

\begin{proof}
It follows from \cite[Proposition~3.4.1]{IsPr99} that there is a
commutative diagram
$$
\xymatrix{
&&X\ar@{->}[dl]_{\alpha}\ar@{->}[dr]^{\beta}&&&\\%
&V\ar@{-->}[rr]_{\psi}&&\mathbb{P}^{3}&}
$$
where $V$ is a complete intersection of two quadric fourfolds in
$\mathbb{P}^{5}$, the morphism $\alpha$ is a blow~up~of a line
$L\subset V$, the morphism $\beta$ is a blow up of a smooth curve
$C\subset\mathbb{P}^{3}$ of degree $5$ and genus~$2$, and the map
$\psi$ is a linear projection from the line $L$.

Let $E$ and $R$ be the exceptional divisors of $\alpha$ and
$\beta$, respectively. Then
\begin{itemize}
\item the surface $\beta(E)\subset\mathbb{P}^{3}$ is an irreducible quadric,%
\item the surface $\alpha(R)\subset V$ is swept out by lines that intersect the line $L$.%
\end{itemize}

We assume that $X$ satisfies the following generality condition:
the surface $\beta(E)$ is smooth.

Let $H$ be any hyperplane section of $V\subset\mathbb{P}^{5}$ such
that $L\subset H$. Then
$$
2\bar{H}+E\sim R+2E\sim -K_{X},
$$
where $\bar{H}$ is a proper transform of $H$ on the threefold $X$.
We see that $\mathrm{lct}(X)\leqslant 1/2$.

We suppose that $\mathrm{lct}(X)<1/2$. Then there exists an
effective $\mathbb{Q}$-divisor $D\qlineq -K_{X}$ such that the log
pair $(X,\lambda D)$ is not log canonical for some positive
rational number $\lambda<1/2$. Then
$$
\varnothing\ne\mathrm{LCS}\Big(X,\lambda D\Big)\subseteq E\cong\P^1\times\mathbb{P}^{1},%
$$
because $\alpha(D)\qlineq -K_V$ and $\mathrm{lct}(V)=1/2$ by
Theorem~\ref{theorem:del-Pezzo}.

Let $F$ be a fiber of the projection $E\to L$, and let $Z$ be a
section of this projection such that $Z\cdot Z=0$. Then
$\alpha^{*}(H)\vert_{E}\sim F$ and $E\vert_{E}\sim -Z$, because
$$
-2Z-2F\sim K_E\sim\Big(K_X+E\Big)\Big\vert_{E}\sim
2\Big(E-\alpha^{*}\big(H\big)\Big)\Big\vert_{E}\sim 2E\Big\vert_{E}-2F.%
$$

By Remark~\ref{remark:convexity}, we may assume that either
$E\not\subset\mathrm{Supp}(D)$, or $R\not\subset\mathrm{Supp}(D)$,
because the log pair
$$
\Big(X, \lambda\big(R+2E\big)\Big)
$$
is log canonical and $-K_{X}\sim R+2E$. Put
$$
D=\mu E+\Omega,
$$
where $\Omega$ is an effective $\mathbb{Q}$-divisor on $X$ such
that $E\not\subset\mathrm{Supp}(\Omega)$.

Suppose that $\mu\leqslant 1$. Then $(X, E+\lambda\Omega)$ is not
log canonical, which implies that
$$
\Big(E,\ \lambda\Omega\Big\vert_E\Big)
$$
is also not log canonical by Theorem~\ref{theorem:adjunction}. But
$$
\Omega\Big\vert_{E}\qlineq \Big(-K_X-\mu
E\Big)\Big\vert_{E}\qlineq\big(1+\mu\big)Z+2F,
$$
which contradicts Lemma~\ref{lemma:elliptic-times-P1}, because
$\mu\leqslant 1$ and $\lambda<1/2$.

Thus, we see that $\mu>1$. Then we may assume that
$R\not\subset\mathrm{Supp}(D)$.

Let $\Gamma$ be a general fiber of the projection $R\to C$. Then
$\Gamma\not\subset\mathrm{Supp}(D)$ and
$$
1=-K_{X}\cdot \Gamma=\mu E\cdot
\Gamma+\Omega\cdot\Gamma=\mu+\Omega\cdot\Gamma\geqslant\mu,
$$
which is a contradiction.
\end{proof}

\begin{lemma}
\label{lemma:2-23} Suppose that $\gimel(X)=2.23$ and $X$ is
general. Then $\lct(X)=1/3$.
\end{lemma}

\begin{proof}
There is a birational morphism $\alpha\colon X\to Q$ such that
$Q\subset\mathbb{P}^{4}$ is a smooth quadric~threefold, and
$\alpha$ contracts a surface $E\subset X$ to a smooth curve
$C\subset Q$ that is a complete intersection~of a hyperplane
section $H\subset Q$ and a divisor $F\in |\O_Q(2)|$.

We assume that $X$ satisfies the following generality condition:
the quadric surface $H$ is smooth.

Let $\bar{H}$ be a proper transform of $H$ on the threefold $X$.
Then there is a commutative diagram
$$
\xymatrix{
&X\ar@{->}[dl]_{\alpha}\ar@{->}[dr]^{\beta}&\\%
Q&&V\ar@{-->}^{\gamma}[ll]}
$$
where $V$ is a complete intersection of two quadrics in $\P^{5}$
such that $V$ has one ordinary double point $P\in V$, the
morphism~$\beta$ contracts $\bar{H}$ to the point $P$, and
$\gamma$ is a projection from~$P$.

Let $E$ be the exceptional divisor of $\alpha$. Then
$$
-K_{X}\sim 3\bar{H}+2E
$$
and $\beta(E)\subset V$ is a surface that contains all lines that
pass through $P$. In particular, $\lct(X)\le 1/3$.

We suppose that $\lct(X)<1/3$. Then there exists an effective
$\Q$-divisor $D\qlineq -K_{X}$ such that the log pair $(X,\lambda
D)$ is not log canonical for some positive rational number
$\lambda<1/3$.

It follows from Remark~\ref{remark:singular-V4} that either
$$
\varnothing\ne\mathrm{LCS}\Big(X,\ \lambda D\Big)\subseteq\bar{H},
$$
or the set $\mathrm{LCS}(X,\lambda D)$ contains a fiber of the
natural projection $E\to C$. In both cases
$$
\mathrm{LCS}\Big(X,\ \lambda D\Big)\cap\bar{H}\ne\varnothing.
$$

We have $\bar{H}\cong\mathbb{P}^{1}\times\mathbb{P}^{1}$. There is
a non-negative rational number $\mu$ such that
$$
D=\mu\bar{H}+\Omega,
$$
where $\Omega$ is an effective $\Q$-divisor on $X$ such that
$\bar{H}\not\subset\mathrm{Supp}(\Omega)$. Then
$$
\alpha\big(D\big)\qlineq \mu H+\alpha\big(\Omega\big)\qlineq -K_{Q},%
$$
which gives $\mu\leqslant 3$. The log pair
$(\bar{H},\lambda\Omega\vert_{\bar{H}})$ is not log canonical by
Theorem~\ref{theorem:adjunction}. But
$$
\Omega\Big\vert_{\bar{H}}\qlineq -\frac{1+\mu}{2}K_{\bar{H}},
$$
which implies that $\mu>1$ by Lemma~\ref{lemma:elliptic-times-P1}.

It follows from Remark~\ref{remark:convexity} that we may assume
that $E\not\subset\mathrm{Supp}(D)$, because the log pair
$$
\Big(X, \lambda\big(3\bar{H}+2E\big)\Big)
$$
is log canonical. Let $F$ be a general fiber the natural
projection $E\to C$. Then
$$
1=D\cdot F=\mu \bar{H}\cdot F+\Omega\cdot F=\mu+ \Omega\cdot F\geqslant \mu,%
$$
which is a contradiction, because $\mu>1$.
\end{proof}

\begin{lemma}
\label{lemma:2-24} Suppose that $\gimel(X)=2.24$ and  $X$ is
general. Then $\mathrm{lct}(X)=1/2$.
\end{lemma}

\begin{proof}
The threefold $X$ is a divisor on $\P^2\times\P^2$ of bidegree
$(1, 2)$. Let $H_{i}$ be a surface in
$$
\Big|\pi_{i}^*\Big(\mathcal{O}_{\mathbb{P}^{2}}\big(1\big)\Big)\Big|,
$$
where $\pi_i\colon X\to\mathbb{P}^{2}$ is a projection of $X$ onto
the $i$-th factor of $\P^2\times\P^2$, $i\in\{1, 2\}$. Then
$$
-K_X\sim 2H_1+H_2,
$$
which implies that $\lct(X)\leqslant 1/2$. Note that $\pi_1$ is a
conic bundle, and $\pi_2$ is a $\mathbb{P}^{1}$-bundle.

Let $\Delta\subset\P^2$ be the degeneration curve of the conic
bundle $\pi_{1}$. Then $\Delta$ is a cubic curve.

We suppose that $X$ satisfies the following generality condition:
the curve $\Delta$ is irreducible.

Assume that $\mathrm{lct}(X)<1/2$. Then there exists an effective
$\mathbb{Q}$-divisor $D\qlineq -K_{X}$ such that the log pair
$(X,\lambda D)$ is not log canonical for some positive rational
number $\lambda<1/2$.

Suppose that the set $\mathbb{LCS}(X,\lambda D)$ contains a
surface $S\subset X$. Then
$$
D=\mu S+\Omega,
$$
where $\Omega$ is an effective $\mathbb{Q}$-divisor such that
$S\not\subset\mathrm{Supp}(\Omega)$,
and $\mu>1/\lambda$. Let $F_i$ be a general fiber of $\pi_i$,
$i\in\{1, 2\}$. Then
$$
2=D\cdot F_{i}=\mu S\cdot F_{i}+\Omega\cdot F_{i}\geqslant\mu
S\cdot F_{i},
$$
but either $S\cdot F_{1}\geqslant 1$ or $S\cdot F_{2}\geqslant 1$.
Thus, we see that $\mu\leqslant 2$, which is a contradiction.

By Theorem~\ref{theorem:Hwang} and
Theorem~\ref{theorem:connectedness}, there is a fiber $\Gamma_{2}$
of the $\mathbb{P}^{1}$-bundle $\pi_{2}$ such that
$$
\varnothing\ne\mathrm{LCS}\Big(X,\ \lambda D\Big)=\Gamma_{2},
$$
because the set $\mathbb{LCS}(X,\lambda D)$ contains no surfaces.

Applying Theorem~\ref{theorem:Hwang} to the conic bundle $\pi_1$,
we see that
$$
\pi_1\big(\Gamma_2\big)\subset\Delta,
$$
which is impossible, because $\Delta\subset\P^2$ is an irreducible
cubic curve and $\pi_{1}(\Gamma_{2})\subset\P^2$ is a line.
\end{proof}

\begin{lemma}
\label{lemma:2-25} Suppose that $\gimel(X)=2.25$. Then
$\mathrm{lct}(X)=1/2$.
\end{lemma}

\begin{proof}
Recall that $X$ is a blow up $\alpha\colon X\to\mathbb{P}^3$ along
a normal elliptic curve $C$ of degree~$4$.

Let $Q\subset\P^3$ be a general quadric containing $C$, and let
$\bar{Q}\subset X$ be a proper transform of $Q$. Then
$$
-K_{X}\sim 2\bar{Q}+E,
$$
where $E$ is the exceptional divisor of $\alpha$. In particular,
$\mathrm{lct}(X)\leqslant 1/2$.

We suppose that $\mathrm{lct}(X)<1/2$. Then there exists an
effective $\mathbb{Q}$-divisor $D\qlineq -K_{X}$ such that the log
pair $(X,\lambda D)$ is not log canonical for some positive
rational number $\lambda<1/2$.

Note that the linear system $|\bar{Q}|$ defines a quadric
fibration
$$
\phi\colon X\longrightarrow\mathbb{P}^1,
$$
such that every fiber of $\phi$ is irreducible. Then the log pair
$(X,\lambda D)$ is log canonical along every nonsingular fiber
$\tilde{Q}$ of the fibration $\phi$ by Theorem~\ref{theorem:Hwang}
since $\mathrm{lct}(\tilde{Q})=1/2$ (see
Example~\ref{example:del-Pezzos}).

The locus $\mathrm{LCS}(X, \lambda D)$ does not contain any fiber
of $\phi$, because $\alpha(D)\qlineq 2Q$ and every fiber of $\phi$
is irreducible. Therefore, we see that
$\mathrm{dim}(\mathrm{LCS}(X, \lambda D))\leqslant 1$.

Let $Z$ be an element in $\mathbb{LCS}(X, \lambda D)$. There is a
singular fiber $\bar{Q}_{1}$ of the fibration $\phi$ such that
$Z\subset \bar{Q}_{1}$.
Note that $\phi$ has $4$ singular fibers and each of them is a
proper transform of a quadric cone in $\mathbb{P}^3$ with vertex
outside $C$.

Let $\bar{Q}_2$ be a singular fiber of $\phi$ such that
$\bar{Q}_{1}\ne\bar{Q}_{2}$, let $\bar{H}$ be a proper transform
of a general plane in $\mathbb{P}^{3}$ that is tangent to the cone
$\alpha(\bar{Q}_2)\subset\mathbb{P}^{3}$ along one of its rulings
$L\subset\alpha(\bar{Q}_2)$, and let $\bar{R}$ be a proper
transform of a very general plane in $\P^3$. Put
$$
\Delta=\lambda D +\frac{1}{2}\Big(\big(1+\eps\big)\bar{Q}_2+\big(2-\eps\big)\bar{H}+3\eps R\Big)%
$$
for some positive rational number $\eps<1-2\lambda$. Then
$$
\Delta\qlineq
-\Big(\lambda+\frac{1}{2}\big(1+\eps\big)\Big)K_X\qlineq
 -\frac{1+\eps+2\lambda}{2}K_X,%
$$
which implies that $-(K_{X}+\Delta)$ is ample.

Let $\bar{L}$ be a proper transform on $X$ of the line $L$. Then
$$
Z\cup\bar{L}\subset\mathrm{LCS}\big(X,\Delta\big)\subset\bar{Q}_{1}\cup\bar{Q}_{2},
$$
which is impossible by Theorem~\ref{theorem:connectedness},
because  $-(K_{X}+\Delta)$ is ample.
\end{proof}

\begin{lemma}
\label{lemma:2-26} Suppose that $\gimel(X)=2.26$ and $X$ is
general. Then $\mathrm{lct}(X)=1/2$.
\end{lemma}

\begin{proof}
Let $V$ be a smooth Fano threefold such that $-K_{V}\sim 2H$ and
$$
\mathrm{Pic}\big(V\big)=\mathbb{Z}\big[H\big],
$$
where $H$ is a Cartier divisor such that $H^{3}=5$ (i.\,e.
$\gimel(V)=1.15$). Then $|H|$ induces an embedding
$X\subset\mathbb{P}^{6}$.

It follows from \cite[Proposition~3.4.1]{IsPr99} that there is a
line $L\subset V\subset\mathbb{P}^{6}$ such that there is a
commutative diagram
$$
\xymatrix{
&&X\ar@{->}[dl]_{\alpha}\ar@{->}[dr]^{\beta}&&&\\%
&V\ar@{-->}[rr]_{\psi}&&Q&}
$$
where $Q$ is a smooth quadric hypersurface in $\mathbb{P}^{4}$,
the morphism $\alpha$ is a blow up of the line $L\subset V$, the
morphism $\beta$ is a blow up of a twisted cubic curve
$\mathbb{P}^{1}\cong C\subset Q$, and $\psi$ is a projection from
the~line~$L$.

Let $S$ be the exceptional divisor of the morphism $\beta$. Put
$\bar{S}=\alpha(S)$. Then $\bar{S}\sim H$, and $\bar{S}$ is
singular along the line $L$. Let $E$ be the exceptional divisor of
the morphism $\alpha$. Then
$$
\beta\big(E\big)\sim\mathcal{O}_{\mathbb{P}^{4}}\big(1\big)\Big\vert_{Q},
$$
which implies that $\beta(E)$ is an irreducible quadric surface.

We suppose that $X$ satisfies the following generality condition:
the surface $\beta(E)$ is smooth.

The equivalence $-K_{X}\sim 2S+3E$ holds. Moreover, the log pair
$$
\left(X,\ \frac{1}{3}\Big(2S+3E\Big)\right)
$$
is log canonical but not log terminal. Thus, we see that
$\lct(X)\leqslant 1/3$.

We suppose that $\mathrm{lct}(X)<1/3$. Then there exists an
effective $\mathbb{Q}$-divisor $D\qlineq -K_{X}$ such that the log
pair $(X,\lambda D)$ is not log canonical for some positive
rational number $\lambda<1/3$. Then
$$
\varnothing\ne\mathrm{LCS}\Big(X,\ \lambda D\Big)\subseteq E,
$$
because $\alpha(D)\qlineq -K_V$ and $\mathrm{lct}(V)=1/2$ by
Theorem~\ref{theorem:del-Pezzo}.

Note that $E\cong\mathbb{P}^{1}\times\mathbb{P}^{1}$ by our
generality condition. Let $F$ be a fiber of the projection $E\to
L$, and let $Z$ be a section of this projection such that $Z\cdot
Z=0$. Then $\alpha^{*}(H)\vert_{E}\sim F$ and $E\vert_{E}\sim -Z$,
because
$$
-2Z-2F\sim K_E\sim\Big(K_X+E\Big)\Big\vert_{E}\sim
2\Big(E-\alpha^{*}\big(H\big)\Big)\Big\vert_{E}\sim 2E\Big\vert_{E}-2F.%
$$

By Remark~\ref{remark:convexity}, we may assume that either
$E\not\subset\mathrm{Supp}(D)$, or $S\not\subset\mathrm{Supp}(D)$.
Put
$$
D=\mu E+\Omega,
$$
where $\Omega$ is an effective $\mathbb{Q}$-divisor on $X$ such
that $E\not\subset\mathrm{Supp}(\Omega)$.

Suppose that $\mu\leqslant 2$. Then $(X, E+\lambda\Omega)$ is not
log canonical, which implies that
$$
\Big(E,\ \lambda\Omega\Big\vert_E\Big)
$$
is also not log canonical by Theorem~\ref{theorem:adjunction}. But
$$
\Omega\Big\vert_{E}\qlineq \Big(-K_X-\mu E\Big)\Big\vert_{E}\qlineq\big(1+\mu\big)Z+2F,%
$$
which contradicts Lemma~\ref{lemma:elliptic-times-P1}, because
$\mu\leqslant 2$ and $\lambda<1/3$.

Thus, we see that $\mu>2$. Then we may assume that
$S\not\subset\mathrm{Supp}(D)$.

Let $\Gamma$ be a general fiber of the projection $S\to C$. Then
$\Gamma\not\subset\mathrm{Supp}(D)$ and
$$
1=-K_{X}\cdot \Gamma=\mu E\cdot
\Gamma+\Omega\cdot\Gamma=\mu+\Omega\cdot\Gamma\geqslant\mu,
$$
which is a contradiction.
\end{proof}

\begin{lemma}
\label{lemma:2-27} Suppose that $\gimel(X)=2.27$. Then
$\mathrm{lct}(X)=1/2$.
\end{lemma}

\begin{proof}
There is a morphism $\alpha\colon X\to\mathbb{P}^3$ contracting a
surface $E$ to a twisted cubic curve~$C\subset \mathbb{P}^{3}$,
and $X\cong\mathbb{P}(\mathcal{E})$, where $\mathcal{E}$ is a
stable rank two vector bundle  on $\mathbb{P}^{2}$  with
$c_{1}(\mathcal{E})=0$ and $c_{1}(\mathcal{E})=2$ such that the
sequence
$$
0\longrightarrow\mathcal{O}_{\mathbb{P}_{2}}\big(-1\big)\oplus\mathcal{O}_{\mathbb{P}_{2}}\big(-1\big)\longrightarrow\mathcal{O}_{\mathbb{P}_{2}}\oplus\mathcal{O}_{\mathbb{P}_{2}}\oplus \mathcal{O}_{\mathbb{P}_{2}}\oplus\mathcal{O}_{\mathbb{P}_{2}}\longrightarrow \mathcal{E}\otimes\mathcal{O}_{\mathbb{P}_{2}}\big(1\big)\longrightarrow 0%
$$
is exact (see \cite[Application~1]{SzWi90}).

Let $Q\subset\mathbb{P}^{3}$ be a general quadric containing $C$,
and let $\bar{Q}\subset X$ be a proper transform of~$Q$. Then
$$
-K_{X}\sim 2\bar{Q}+E,
$$
where $E$ is the exceptional divisor of $\alpha$. In particular,
we see that $\mathrm{lct}(X)\leqslant 1/2$.

We suppose that $\lct(X)<1/2$. Then there exists an effective
$\mathbb{Q}$-divisor $D\qlineq -K_{X}$ such that the log pair
$(X,\lambda D)$ is not log canonical for some positive rational
number $\lambda<1/2$.

Suppose that the set $\mathbb{LCS}(X,\lambda D)$ contains a
surface $S\subset X$. Put
$$
D=\mu F+\Omega,
$$
where $\mu\geqslant 1/\lambda$ and $\Omega$ is an effective
$\mathbb{Q}$-divisor such that
$F\not\subset\mathrm{Supp}(\Omega)$.

Let $\phi\colon X\to \mathbb{P}^{2}$ be the natural
$\mathbb{P}^{1}$-bundle. Then
$$
2=D\cdot \Gamma=\mu F\cdot \Gamma+\Omega\cdot \Gamma=\mu F\cdot\Gamma+\Omega\cdot F\geqslant \mu F\cdot\Gamma,%
$$
where $\Gamma$ is a general fiber of $\phi$. Thus, we see that $F$
is swept out by the fibers of $\phi$. Then
$$
\alpha\big(F\big)\sim\mathcal{O}_{\mathbb{P}^{3}}\big(d\big)
$$
and $d\geqslant 2$. But $\alpha(D)\qlineq
\mu\alpha(F)+\alpha(\Omega)\qlineq\mathcal{O}_{\mathbb{P}^{3}}(4)$,
which is a contradiction.

We wee that the locus $\mathrm{LCS}(X,\lambda D)$ contains no
surfaces. Applying Theorem~\ref{theorem:Hwang} to $(X,\lambda D)$
and $\phi$, we see that
$$
L\subseteq\mathrm{LCS}\Big(X,\ \lambda D\Big),
$$
where $L$ is a fiber of $\phi$. Note that $\alpha(L)$ is a secant
line of the curve $C\subset\P^3$. One has
$$
\alpha\big(L\big)\subseteq\mathrm{LCS}\Big(\mathbb{P}^{3}, \lambda\alpha\big(D\big)\Big)\subseteq \alpha\Big(\mathrm{LCS}\big(X,\ \lambda D\big)\Big)\cup C,%
$$
which is impossible by Lemma~\ref{lemma:P3}.
\end{proof}

\begin{lemma}
\label{lemma:2-28} Suppose that $\gimel(X)=2.28$. Then
$\mathrm{lct}(X)=1/4$.
\end{lemma}
\begin{proof}
There is a blow up $\alpha\colon X\to \mathbb{P}^3$ along a plane
cubic curve $C\subset\mathbb{P}^{3}$. One has
$$
-K_{X}\sim 4G+3E,
$$
where $E$ is the exceptional divisor of $\alpha$ and $G$ is a
proper transform of the plane in $\mathbb{P}^{3}$ that contains
the curve $C$. In particular, we see that the inequality
$\mathrm{lct}(X)\leqslant 1/4$ holds.

We suppose that $\lct(X)<1/4$. Then there exists an effective
$\mathbb{Q}$-divisor $D\qlineq -K_X$ such that the log pair $(X,
\lambda D)$ is not log canonical for some rational number
$\lambda<1/4$. One has
$$
\varnothing\ne\mathrm{LCS}\Big(X,\ \lambda D\Big)\subseteq E,
$$
since $\mathrm{lct}(\mathbb{P}^{4})=1/4$. Computing the
intersections with a strict transform of a general line in $\P^3$
intersecting the curve $C$, one obtains that $\LCS(X, \lambda D)$
does not contain the divisor $E$. Moreover,  any curve
$\Gamma\in\mathbb{LCS}(X, \lambda D)$ must be a fiber of the
natural projection
$$
\psi\colon E\longrightarrow C
$$
by Lemma~\ref{lemma:rational-tree}. Therefore, we see that either
the locus $\mathrm{LCS}(X, \lambda D)$ consists of a single point,
or the locus $\mathrm{LCS}(X, \lambda D)$ consists of a single
fiber of the projection $\psi$ by
Theorem~\ref{theorem:connectedness}.

Let $R$ be a sufficiently  general cone in $\mathbb{P}^3$ over the
curve $C$, and let $H$ be a sufficiently general plane in
$\mathbb{P}^{3}$ that passes through the point $\Sing(R)$. Then
$$
\mathrm{LCS}\left(X,\ \lambda
D+\frac{3}{4}\Big(\bar{R}+\bar{H}\Big)\right)=\mathrm{LCS}\Big(X,\
 \lambda D\Big)\bigcup\Sing\big(\bar{R}\big),%
$$
where $\bar{R}$ and $\bar{H}$ are proper transforms of $R$ and $H$
on the threefold $X$. But the divisor
$$
-\left(K_{X}+\lambda
D+\frac{3}{4}\Big(\bar{R}+\bar{H}\Big)\right)\qlineq
 \big(\lambda-1/4\big)K_{X}%
$$
is ample, which contradicts Theorem~\ref{theorem:connectedness}.
\end{proof}

\begin{lemma}
\label{lemma:2-29} Suppose that $\gimel(X)=2.29$. Then
$\mathrm{lct}(X)=1/3$.
\end{lemma}

\begin{proof}
There is a birational morphism $\alpha\colon X\to Q$ such that $Q$
is a smooth quadric hypersurface, and $\alpha$ is a blow up along
a smooth conic $C\subset Q$.

Let $H$ be a general hyperplane section of
$Q\subset\mathbb{P}^{4}$ that contains $C$, and let $\bar{H}$ be a
proper transform of the surface $H$ on the threefold $X$. Then
$$
-K_{X}\sim 3\bar{H}+2E,
$$
where $E$ is the exceptional divisor of $\alpha$. In particular,
the inequality $\mathrm{lct}(X)\leqslant 1/3$ holds.

We suppose that $\mathrm{lct}(X)<1/3$. Then there exists an
effective $\mathbb{Q}$-divisor $D\qlineq -K_{X}$ such that the log
pair $(X,\lambda D)$ is not log canonical for some rational
$\lambda<1/3$. Then
$$
\varnothing\ne\mathrm{LCS}\Big(X,\lambda D\Big)\subseteq E,
$$
since $\mathrm{lct}(Q)=1/3$ (see
Example~\ref{example:hypersurface-index-big}) and
$\alpha(D)\qlineq -K_{Q}$.

The linear system $|\bar{H}|$ has no base points and defines a
morphism $\beta\colon X\to\mathbb{P}^{1}$,
whose general fiber is a smooth quadric surface. Then the log pair
$(X,\lambda D)$ is log canonical along the smooth fibers of
$\beta$ by Theorem~\ref{theorem:Hwang} (see
Example~\ref{example:del-Pezzos}).

It follows from Theorem~\ref{theorem:connectedness} that there is
a singular fiber $S\sim\bar{H}$ of the morphism $\beta$ such that
$$
\varnothing\ne\mathrm{LCS}\Big(X,\lambda D\Big)\subseteq E\cap S,
$$
and $\alpha(S)\subset\mathbb{P}^3$ is a quadratic cone. Put
$\Gamma=E\cap S$. Then $\Gamma$ is an irreducible conic, the log
pair
$$
\left(X,\ S+\frac{2}{3}E\right)
$$
has log canonical singularities, and $3S+2E\qlineq D$. Therefore,
it follows from Remark~\ref{remark:convexity} that to complete the
proof we may assume that either $S\not\subset\mathrm{Supp}(D)$ or
$E\not\subset\mathrm{Supp}(D)$.

Intersecting the divisor $D$ with a strict transform of a general
ruling of the cone $\alpha(S)\subset\mathbb{P}^{3}$ and with a
general fiber of the projection $E\to C$, we see that
$$
\Gamma\not\subseteq\mathrm{LCS}\Big(X,\lambda D\Big),
$$
which implies that $\mathrm{LCS}(X,\lambda D)$ consists of a
single point $O\in\Gamma$ by Theorem~\ref{theorem:connectedness}.

Let $R$ be a general (not passing through $O$) surface in
$|\alpha^{*}(H)|$. Then
$$
\mathrm{LCS}\left(X,\ \lambda D+\frac{1}{2}\Big(\bar{H}+2R\Big)\right)=R\cup O,%
$$
which is impossible by Theorem~\ref{theorem:connectedness}, since
$-K_{X}\sim\bar{H}+2R\qlineq D$ and $\lambda<1/3$.
\end{proof}

\begin{lemma}
\label{lemma:2-30} Suppose that $\gimel(X)=2.30$. Then
$\mathrm{lct}(X)=1/4$.
\end{lemma}

\begin{proof}
There is a commutative diagram
$$
\xymatrix{
&X\ar@{->}[dl]_{\alpha}\ar@{->}[dr]^{\beta}&\\%
\mathbb{P}^{3}&&Q\ar@{-->}^{\gamma}[ll]}
$$
where $Q$ is a smooth quadric threefold in $\mathbb{P}^{4}$, the
morphism $\alpha$ is a blow up of a smooth conic
$C\subset\mathbb{P}^{3}$, the morphism $\beta$ is a blow up of a
point, and $\gamma$ is a projection from a point.

Let $G$ be a proper transform on the variety $X$ of the unique
plane in $\mathbb{P}^{3}$ that contains the conic $C$. Then the
surface $G$ is contracted by the morphism $\beta$, and
$$
-K_{X}\sim 4G+3E,
$$
where $E$ is the exceptional divisor of the blow up $\alpha$.
Thus, we see that $\mathrm{lct}(X)\leqslant 1/4$.

We suppose that $\mathrm{lct}(X)<1/4$. Then there exists an
effective $\mathbb{Q}$-divisor $D\qlineq -K_{X}$ such that the~log
pair $(X,\lambda D)$ is not log canonical for some rational
$\lambda<1/4$. Then
$$
\varnothing\ne\mathrm{LCS}\Big(X,\lambda D\Big)\subseteq E\cap G,
$$
because $\mathrm{lct}(\mathbb{P}^{4})=1/4$ and
$\mathrm{lct}(Q)=1/3$.

We may assume that either $G\not\subset\mathrm{Supp}(X)$ or
$E\not\subset\mathrm{Supp}(X)$ by Remark~\ref{remark:convexity}.

Intersecting $D$ with lines in $G\cong\mathbb{P}^{2}$ and with
fibers of the projection $E\to C$, we see that
$$
\mathrm{LCS}\Big(X,\lambda D\Big)\subsetneq E\cap G
$$
which implies that there is a point $O\in E\cap G$ such that
$\mathrm{LCS}(X,\lambda D)=O$ by
Theorem~\ref{theorem:connectedness}.

Let $R$ be a general surface in $|\alpha^{*}(H)|$ and $F$ a
general surface in $|\alpha^{*}(2H)-E|$. Then
$$
\mathrm{LCS}\left(X,\ \lambda D+\frac{1}{2}\Big(F+2R\Big)\right)=R\cup O,%
$$
which is impossible by Theorem~\ref{theorem:connectedness} since
$-K_{X}\sim F+2R\qlineq D$ and $\lambda<1/4$.
\end{proof}

\begin{lemma}
\label{lemma:2-31} Suppose that $\gimel(X)=2.31$. Then
$\mathrm{lct}(X)=1/3$.
\end{lemma}

\begin{proof}
There is a birational morphism $\alpha\colon X\to Q$ such that $Q$
is a smooth quadric hypersurface, and $\alpha$ is a blow up of the
quadric $Q$ along a line $L\subset Q$.

Let $H$ be a sufficiently general hyperplane section of the
quadric $Q\subset\mathbb{P}^{4}$ that passes through the line $L$,
and let $\bar{H}$ be a proper transform of the surface $H$ on the
threefold $X$. Then
$$
-K_{X}\sim 3\bar{H}+2E,
$$
where $E$ is the exceptional divisor of $\alpha$. In particular,
$\mathrm{lct}(X)\leqslant 1/3$.

We suppose that $\mathrm{lct}(X)<1/3$. Then there exists an
effective $\mathbb{Q}$-divisor $D\qlineq -K_{X}$ such that the log
pair $(X,\lambda D)$ is not log canonical for some rational
$\lambda<1/3$. Then
$$
\varnothing\ne\mathrm{LCS}\Big(X,\lambda D\Big)\subseteq E,
$$
since $\mathrm{lct}(Q)=1/3$ and $\alpha(D)\qlineq -K_{Q}$.

The linear system $|\bar{H}|$ defines a $\mathbb{P}^1$-bundle
$\phi\colon X\to\mathbb{P}^2$
such that the induced morphism
$E\cong\mathbb{F}_{1}\to\mathbb{P}^{2}$ contracts an irreducible
curve $Z\subset E$. One has
$$
\mathrm{LCS}\Big(X,\lambda D\Big)=Z\subset E
$$
by Theorem~\ref{theorem:Hwang}. Put
$$
D=\mu E+\Omega,
$$
where $\Omega$ is an effective $\mathbb{Q}$-divisor on $X$ such
that $E\not\subset\mathrm{Supp}(\Omega)$. Then
$$
2=D\cdot F=\mu E\cdot F+\Omega\cdot F=\mu+ \Omega\cdot F\geqslant \mu,%
$$
where $F$ is a general fiber of $\phi$. Note that the log pair
$$
\Big(X,\ E+\lambda\Omega\Big)
$$
is not log canonical, because $\lambda<1/3$. Then
$(E,\lambda\Omega\vert_{E})$ is not log canonical by
Theorem~\ref{theorem:adjunction}.

Let $C$ be a fiber of the natural projection $E\to L$. Then
$$
\Omega\Big\vert_{E}\qlineq 3C+\big(1+\mu\big)Z,
$$
which implies that $(E, \lambda\Omega\vert_{E})$ is log canonical
by Lemma~\ref{lemma:F1}, which is a contradiction.
\end{proof}


\section{Fano threefolds with $\rho=3$}
\label{section:rho-3}

We use the assumptions and notation introduced in
section~\ref{section:intro}.

\begin{lemma}
\label{lemma:3-1} Suppose that $\gimel(X)=3.1$ and $X$ is general.
Then $\lct(X)=3/4$.
\end{lemma}

\begin{proof}
There is a double cover
$$
\omega\colon X\longrightarrow \mathbb{P}^1\times\mathbb{P}^1\times\mathbb{P}^1%
$$
branched over a divisor of tridegree $(2, 2, 2)$. The projection
$$
\mathbb{P}^1\times\mathbb{P}^1\times\mathbb{P}^1\longrightarrow\mathbb{P}^1
$$
onto the $i$-th factor induces a morphism $\pi_{i}\colon
X\to\mathbb{P}^1$, whose  fibers are del Pezzo surfaces of
degree~$4$.

Let $R_{1}$ be a singular fiber of the fibration $\pi_{1}$, let
$Q$ be a singular point of the surface $R_{1}$, and let $R_{2}$
and $R_{3}$ be fibers of $\pi_{2}$ and $\pi_{3}$ such that
$$
R_{2}\ni Q\in R_{3},
$$
respectively. Then $\mathrm{mult}_{Q}(R_{1}+R_{2}+R_{3})=4$, which
implies that the log pair
$$
\left(X,\ \frac{3}{4}\Big(R_{1}+R_{2}+R_{3}\Big)\right)
$$
is not log terminal at $Q$. But $-K_{X}\sim R_{1}+R_{2}+R_{3}$.
Thus, we see that $\mathrm{lct}(X)\leqslant 3/4$.


We suppose that $\mathrm{lct}(X)<3/4$. Then there exists an
effective $\mathbb{Q}$-divisor $D\qlineq -K_{X}$ such that the~log
pair $(X,\lambda D)$ is not log canonical at some point $P\in X$
for some rational number~$\lambda<3/4$.

Let $S_{i}$ be a fiber of $\pi_{i}$ such that $P\in S_{i}$.
If $X$ is general we may assume (after a possible renumeration) that
\begin{itemize}
\item the surface $S_{1}$ is smooth at the point $P$,%
\item the singularities of the surface $S_{1}$ consist of at most one ordinary double point,%
\item for every smooth curve $L\subset S_{1}$ such that $-K_{S_{1}}\cdot L=1$, we have $P\not\in L$,%
\item for every smooth curves $C_{1}\subset S_{1}\supset C_{2}$
such that
$$
-K_{S_{1}}\cdot C_{1}=-K_{S_{1}}\cdot C_{2}=2
$$
and $C_{1}+C_{2}\sim -K_{S_{1}}$, we have $P\ne C_{1}\cap C_{2}$.%
\end{itemize}

The surface $S_{1}$ is a del Pezzo surface of degree $4$.  One has
$$
D=\mu S_{1}+\Omega,
$$
where $\Omega$ is an effective $\mathbb{Q}$-divisor on $X$ such
that $S_{1}\not\subset\mathrm{Supp}(\Omega)$.

Let $\phi\colon X\to\mathbb{P}^{1}\times\mathbb{P}^{1}$ be a
natural conic bundle induced by the linear system
$$
\Big|S_{2}+S_{3}\Big|,
$$
and let $\Gamma$ be a general fiber of the conic bundle $\phi$.
Then
$$
2=D\cdot\Gamma=\mu S_{1}\cdot\Gamma+\Omega\cdot\Gamma=2\mu+\Omega\cdot \Gamma\geqslant 2\mu,%
$$
which implies that $\mu\leqslant 1$. Then $(X,
S_{1}+\lambda\Omega)$ is not canonical at the point $P$. Hence
$$
\Big(S_{1},\ \lambda\Omega\Big\vert_{S_{1}}\Big)
$$
is not log canonical at the point $P$ by
Theorem~\ref{theorem:adjunction}. But
$$
\Omega\Big\vert_{S_{1}}\qlineq D\Big\vert_{S_{1}}\qlineq
-K_{S_{k}},
$$
which is impossible (see Example~\ref{example:del-Pezzo-quartic}
and mind the generality assumption).
\end{proof}

\begin{lemma}
\label{lemma:3-2} Suppose that $\gimel(X)=3.2$ and $X$ is general.
Then $\mathrm{lct}(X)=1/2$.
\end{lemma}
\begin{proof}
The threefold $X$ is a primitive Fano threefold (see
\cite[Definition~1.3]{MoMu83}). Put
$$
U=\mathbb{P}\Big(\mathcal{O}_{\P^1\times\P^1}\oplus\mathcal{O}_{\P^1\times\P^1}\big(-1,-1\big)\oplus\mathcal{O}_{\P^1\times\P^1}\big(-1,-1\big)\Big),%
$$
let $\pi\colon U\to\P^1\times\P^1$ be a natural projection, and
let $L$ be a tautological line bundle on $U$. Then
$$
X\in\Big|2L+\pi^*\Big(\mathcal{O}_{\P^1\times\P^1}\big(2,
3\big)\Big)\Big|.
$$

Let us show that $\lct(X)\le 1/2$. Let $E_1$ and $E_2$ be surfaces
in $X$ such that
$$
\pi\big(E_{1}\big)\subset\P^1\times\P^1\supset\pi\big(E_{2}\big)
$$
are divisors on $\P^1\times\P^1$ of bi-degree $(1,0)$ and $(0,1)$,
respectively. Then
$$
-K_X\sim L\Big\vert_{X}+2E_1+E_2,
$$
which implies that $\lct(X)\le 1/2$.


We suppose that $\mathrm{lct}(X)<1/2$. Then there exists an
effective $\mathbb{Q}$-divisor $D\qlineq -K_{X}$ such that the~log
pair $(X,\lambda D)$ is not log canonical at some point $P\in X$
for some rational number~$\lambda<1/2$.

It follows from \cite[Proposition~3.8]{JaPeRa07} that there is a
commutative diagram
$$
\xymatrix{
U_{1}\ar@{->}[dd]_{\psi_{1}}\ar@{->}[rr]^{\gamma_{1}}&&V&&U_{2}\ar@{->}[dd]^{\psi_{2}}\ar@{->}[ll]_{\gamma_{2}} \\
&&X\ar@{->}[d]_{\omega}\ar@{->}[lld]_{\phi_1}\ar@{->}[rrd]^{\phi_2}\ar@{->}[u]_{\alpha}\ar@{->}[ull]^{\beta_{1}}\ar@{->}[urr]_{\beta_{2}}&& \\
\P^1 &&\P^1\times\P^1\ar@{->}[ll]^{\pi_1}\ar@{->}[rr]_{\pi_2}&& \P^1}%
$$
where $V$ is a Fano threefold that has one ordinary double point
$O\in V$ such that
$$
\mathrm{Pic}\big(V\big)=\mathbb{Z}\big[-K_{V}\big]
$$
and $-K_{V}^{3}=16$, the morphism $\alpha$ contracts a unique
surface
$$
\mathbb{P}^{1}\times \mathbb{P}^{1}\cong S\subset X
$$
such that $S\sim L\vert_{X}$ to the point $O\in V$, the morphism
$\beta_{i}$ contracts $S$ to a smooth rational curve, the~morphism
$\gamma_{i}$ contracts the curve $\beta_{i}(S)$ to the point $O\in
V$ so that the rational~map
$$
\gamma_{2}\circ\gamma_{1}^{-1}\colon U_{1}\dasharrow U_{2}
$$
is a flop in $\beta_{1}(S)\cong\mathbb{P}^{1}$,
the morphism $\psi_2$ is a quadric fibration, and the morphisms
$\psi_{1}$, $\phi_{1}$, $\phi_{2}$ are fibrations whose fibers are del Pezzo
surfaces of degree $4$, $3$ and $6$, respectively.
The morphisms $\pi_1$ and $\pi_2$ are natural projections, and
$\omega=\pi\vert_{X}$. Then
$$
\mathrm{Cl}\big(V\big)=\mathbb{Z}\big[\alpha\big(E_{1}\big)\big]\oplus\mathbb{Z}\big[\alpha\big(E_{2}\big)\big],
$$
and $\omega$ is a conic bundle. The curve $\beta_{1}(S)$ is a
section of $\psi_{1}$, and $\beta_{2}(S)$ is a $2$-section of
$\psi_{2}$.

We assume that the threefold $X$ satisfies the following mild
generality condition: every singular fiber of the del Pezzo
fibration $\phi_2$ has at most $\mathbb{A}_1$ singularities.

Applying Lemma~\ref{lemma:Hwang} to the fibration $\phi_1$, we see
that
$$
\varnothing\neq\LCS\Big(X,\ \lambda D\Big)\subseteq S_{1}
$$
where $S_{1}$ is a singular fiber of the del Pezzo fibration
$\phi_1$, because the global log canonical threshold of a smooth
del Pezzo surface of degree $6$ is equal to $1/2$ by
Example~\ref{example:del-Pezzos}.

Applying Lemma~\ref{lemma:Hwang} to $\phi_2$, we obtain a
contradiction by Example~\ref{example:singular-cubics}.
\end{proof}

\begin{lemma}
\label{lemma:3-3} Suppose that $\gimel(X)=3.3$ and $X$ is general.
Then $\lct(X)=2/3$.
\end{lemma}

\begin{proof}
The threefold $X$ is a divisor on
$\mathbb{P}^1\times\mathbb{P}^1\times\mathbb{P}^2$ of tridegree
$(1, 1, 2)$. In particular,
$$
-K_{X}\sim
\pi_{1}^{*}\Big(\mathcal{O}_{\mathbb{P}^{1}}\big(1\big)\Big)+\pi_{2}^{*}\Big(\mathcal{O}_{\mathbb{P}^{1}}\big(1\big)\Big)+\phi^{*}\Big(\mathcal{O}_{\mathbb{P}^{2}}\big(1\big)\Big),
$$
where $\pi_{1}\colon X\to\mathbb{P}^{1}$ and $\pi_{1}\colon
X\to\mathbb{P}^{1}$ are  fibrations into del Pezzo surfaces of
degree $4$ that are induced by the projections of the variety
$\mathbb{P}^1\times\mathbb{P}^1\times\mathbb{P}^2$ onto its first
and second factor, respectively, and $\phi\colon X\to\mathbb{P}^2$
is conic bundle that is induced by the projection
$\mathbb{P}^1\times\mathbb{P}^1\times\mathbb{P}^2\to
\mathbb{P}^2$.

Let $\alpha_{2}\colon X\to\mathbb{P}^{1}\times\mathbb{P}^{2}$ be a
birational morphism induced by the linear system
$$
\Big|\pi_{2}^{*}\Big(\mathcal{O}_{\mathbb{P}^{1}}\big(1\big)\Big)+\phi^{*}\Big(\mathcal{O}_{\mathbb{P}^{2}}\big(1\big)\Big)\Big|,
$$
let $H_{i}\in|\pi_{i}^{*}(\mathcal{O}_{\mathbb{P}^{1}}(1))|$ and
$R\in|\phi^{*}(\mathcal{O}_{\mathbb{P}^{2}}(1))|$ be general
surfaces. Then
$$
H_{1}\sim H_{2}+2R-E_{2},
$$
where $E_{2}$ is the exceptional divisor of the birational
morphism $\alpha_{2}$. Hence
$$
-K_{X}\sim
H_1+H_2+R\qlineq\frac{3}{2}H_{1}+\frac{1}{2}H_{2}+\frac{1}{2}E_{2},
$$
which implies that $\mathrm{lct}(X)\leqslant 2/3$.


We suppose that $\mathrm{lct}(X)<2/3$. Then there exists an
effective $\mathbb{Q}$-divisor $D\qlineq -K_{X}$ such that the~log
pair $(X,\lambda D)$ is not log canonical at some point $P\in X$
for some rational number~$\lambda<2/3$.

Let $S_{i}$ be a fiber of $\pi_{i}$ such that $P\in S_{i}$.
If $X$ is general we may assume (after a possible renumeration) that
\begin{itemize}
\item the surface $S_{1}$ is smooth at the point $P$,%
\item the singularities of the surface $S_{1}$ consist of at most one ordinary double point,%
\item for every smooth curve $L\subset S_{1}$ such that $-K_{S_{1}}\cdot L=1$, we have $P\not\in L$ if $\mathrm{Sing}(S_{1})\ne\varnothing$.%
\end{itemize}

Put $D=\mu S_{1}+\Omega$, where $\Omega$ is an effective
$\mathbb{Q}$-divisor such that
$S_{1}\not\subset\mathrm{Supp}(\Omega)$. Then
$$
\Big(H_{2},\ \lambda\mu
S_{1}\Big\vert_{H_{2}}+\lambda\Omega\Big\vert_{H_{2}}\Big)
$$
is log canonical because $\mathrm{lct}(H_{2})=2/3$. Thus, we see
that $\mu\leqslant 1/\lambda$. Hence
$$
\Big(S_{1},\ \lambda\Omega\Big\vert_{S_{1}}\Big)
$$
is not log canonical at the point $P$ by
Theorem~\ref{theorem:adjunction}. But
$$
\Omega\Big\vert_{S_{1}}\qlineq-K_{S_{1}},
$$
which is impossible (see Example~\ref{example:del-Pezzo-quartic}).
\end{proof}

\begin{lemma}
\label{lemma:3-4} Suppose that $\gimel(X)=3.4$. Then
$\lct(X)=1/2$.
\end{lemma}

\begin{proof}
Let $O$ be a point in $\mathbb{P}^{2}$. Then there is a
commutative diagram
$$
\xymatrix{
&&X\ar@{->}[d]_{\alpha}\ar@/_1pc/@{->}[lldd]_{\eta_{1}}\ar@/_1pc/@{->}[rrd]^{\eta_{2}}\ar@/^1pc/@{->}[rrrrd]^{\phi}&&\\
&&V\ar@{->}[d]_{\omega}\ar@{->}[lld]_{\gamma_{1}}\ar@{->}[rrd]^{\gamma_{2}}&&\mathbb{F}_{1}\ar@{->}[d]^{\beta}\ar@{->}[rr]_{\upsilon}&&\mathbb{P}^1\\
\mathbb{P}^{1}&&\mathbb{P}^{1}\times\mathbb{P}^{2}\ar@{->}[ll]^{\pi_{1}}\ar@{->}[rr]_{\pi_{2}}&&\mathbb{P}^{2}&&}
$$
such that $\pi_{i}$ and $\upsilon$ are natural projections,
$\omega$ is a double cover branched over a divisor
$B\subset\P^1\times\P^2$
of bi-degree $(2, 2)$, the morphism $\gamma_{1}$ is a fibration
into quadrics, $\gamma_{2}$ and $\eta_{2}$ are conic~bundles,
the~morphism $\beta$ is a blow up of the point $O$, the morphism
$\alpha$ is a blow up of a smooth curve that is a fiber of
$\gamma_{2}$ over the point $O$, the morphism $\eta_{1}$ is a
fibration into del Pezzo surfaces~of~degree~$6$, and $\phi$ is a
fibration into del Pezzo surfaces of degree $4$.

Let $H$ be a general fiber of $\eta_{1}$, and let $S$ be a general
fiber of $\phi$. Then
$$
-K_{X}\sim H+2S+E,
$$
where $E$ is the exceptional divisor of $\alpha$. Thus, we see
that $\mathrm{lct}(X)\leqslant 1/2$.

We suppose that $\mathrm{lct}(X)<1/2$. Then there exists an
effective $\mathbb{Q}$-divisor $D\qlineq -K_{X}$ such that the~log
pair $(X,\lambda D)$ is not log canonical for some positive
rational number $\lambda<1/2$. Then
$$
\varnothing\ne\mathrm{LCS}\Big(X,\ \lambda D\Big)\subseteq E,
$$
because $\alpha(D)\qlineq -K_V$ and $\mathrm{lct}(V)=1/2$ by
Lemma~\ref{lemma:2-18}.

Let $\Gamma$ be a fiber of $\eta_{2}$ such that
$\Gamma\cap\mathrm{LCS}(X,\lambda D)\ne\varnothing$. Then
$$
\Gamma\subseteq\mathrm{LCS}\Big(X,\ \lambda D\Big)\subseteq E,
$$
by Theorem~\ref{theorem:Hwang}. Hence $(H,\lambda D\vert_{H})$ is
not log canonical at the points $H\cap\Gamma$. But
$$
D\Big\vert_{H}\qlineq -K_{X}\Big\vert_{H}\sim -K_{H}
$$
and $\mathrm{lct}(H)=1/2$, because $H$ is a del Pezzo surface of
degree $6$, which is a contradiction.
\end{proof}

\begin{lemma}
\label{lemma:3-5} Suppose that $\gimel(X)=3.5$ and $X$ is general.
Then $\lct(X)=1/2$.
\end{lemma}

\begin{proof}
There is a birational morphism $\alpha\colon X\to\P^1\times\P^2$
that contracts a  surface $E\subset X$~to~a~curve $C\subset
\P^1\times\P^2$ of bidegree $(5, 2)$. Let
$\pi_{1}\colon\P^1\times\P^2\to\P^1$ and
$\pi_{2}\colon\P^1\times\P^2\to\P^2$ be natural projections. There
is
$$
Q\in\Big|\pi_{2}^{*}\Big(\mathcal{O}_{\mathbb{P}^{1}}\big(2\big)\Big)\Big|
$$
such that $C\subset Q$. Let $H_1$ be a general fiber of $\pi_1$,
let $H_2$ be a surface in
$|\pi_{2}^{*}(\mathcal{O}_{\mathbb{P}^{1}}(1))|$. We have
$$
-K_X\sim 2\bar{H_1}+\bar{H_2}+\bar{Q},
$$
where $\bar{H_1}, \bar{H_2},\bar{Q}\subset X$ are proper
transforms of $H_1, H_2, Q$, respectively. In particular,
$\lct(X)\le 1/2$.

We suppose that $X$ satisfies the following generality condition:
every fiber $F$ of $\pi_1\circ\alpha$ is singular at most at one
ordinary double point.

Assume that $\lct(X)<1/2$. Then there exists an effective
$\Q$-divisor $D\qlineq -K_{X}$ such that the log pair $(X,\lambda
D)$ is not log canonical for some positive rational number
$\lambda<1/2$.

Let $S$ be an irreducible surface on the threefold $X$. Put
$$
D=\mu S+\Omega,
$$
where $\Omega$ is an effective $\mathbb{Q}$-divisor  such that
$S\not\subset\mathrm{Supp}(\Omega)$. Then
$$
\left(\bar{H}_{1},\ \frac{1}{2}\Big(\mu S+\Omega\Big)\Big\vert_{\bar{H}_{1}}\right)%
$$
is log canonical (see Example~\ref{example:del-Pezzos}). Thus,
either $\mu\leqslant 2$, or $S$ is a fiber of
$\pi_{1}\circ\alpha$.

Let $\Gamma\cong\mathbb{P}^{1}$ be a general fiber of the conic
bundle $\pi_{2}\circ\alpha$. Then
$$
2=D\cdot\Gamma=\mu S\cdot\Gamma+\Omega\cdot\Gamma\geqslant\mu S\cdot\Gamma,%
$$
which implies that  $\mu\leqslant 2$ in the case when $S$ is a
fiber of $\pi_{1}\circ\alpha$.

We see that the set $\mathbb{LCS}(X, \lambda D)$ contains no
surfaces. Now, applying Lemma~\ref{lemma:Hwang} to
$\pi_{1}\circ\alpha$, we obtain a contradiction with
Example~\ref{example:del-Pezzo-quartic}.
\end{proof}

\begin{lemma}\label{lemma:3-6}
Suppose that $\gimel(X)=3.6$ and $X$ is general. Then
$\lct(X)=1/2$.
\end{lemma}

\begin{proof}
Let $\eps\colon V\to\mathbb{P}^{3}$ be a blow up of a line
$L\subset\mathbb{P}^{3}$. Then
$$
V\cong\mathbb{P}\Big(\mathcal{O}_{\mathbb{P}^{1}}\oplus\mathcal{O}_{\mathbb{P}^{1}}\oplus\mathcal{O}_{\mathbb{P}^{1}}\big(1\big)\Big)
$$
and there is a natural $\mathbb{P}^{2}$-bundle $\eta\colon
V\to\mathbb{P}^{1}$. There is a smooth elliptic curve
$C\subset\mathbb{P}^{3}$ of degree~$4$ such that $L\cap
C=\varnothing$ and there is a commutative diagram
$$
\xymatrix{
Y\ar@{->}[d]_{\eps}&&X\ar@{->}[ll]_{\gamma}\ar@{->}[d]^{\beta}\ar@{->}[rr]^{\phi}&&\P^1\\
\P^3&&V\ar@{->}[ll]^{\delta}\ar@{->}[rru]_{\eta}&&}
$$
where $\delta$ is a blow up of $C$,  the morphism $\beta$ is a
blow up of the proper transform of the line $L$, the~morphism
$\gamma$ is a blow up of the proper transform of the curve $C$,
and $\phi$ is a~del~Pezzo~fibration.

We suppose that $X$ satisfies the following generality condition:
for every fiber $F$ of $\phi$, the~surface $F$ has at most one
singular point that is an ordinary double point of the surface
$F$.

Let $E$ and $G$ be the exceptional surfaces of $\beta$ and
$\gamma$, respectively, let $H\subset\P^3$ be a general plane that
passes through $L$, and let $Q\subset\P^3$ a quadric surface that
passes through $C$. Then
$$
-K_{X}\sim 2\bar{H}+\bar{Q}+E,
$$
where $\bar{H}\subset X\supset\bar{Q}$ are proper transforms of
$H$ and $Q$, respectively. We have $\lct(X)\le 1/2$.

We suppose that $\mathrm{lct}(X)<1/2$. Then there exists an
effective $\mathbb{Q}$-divisor $D\qlineq -K_{X}$ such that the log
pair $(X,\lambda D)$ is not log canonical for some positive
rational number $\lambda<1/2$.

It follows from Lemma~\ref{lemma:2-25} that $\lct(V)=1/2$.
Therefore, we see that
$$
\varnothing\ne\mathrm{LCS}\Big(X,\ \lambda D\Big)\subseteq G.
$$
Note that every fiber of the fibration $\phi$ is a del Pezzo
surface of degree $5$ that has at most one ordinary double point.
Thus, applying Lemma~\ref{lemma:Hwang} to $\phi$, we obtain a
contradiction with Example~\ref{example:del-Pezzo-quintic}.
\end{proof}

\begin{lemma}
\label{lemma:3-7} Suppose that $\gimel(X)=3.7$ and $X$ is general.
Then $\lct(X)=1/2$.
\end{lemma}

\begin{proof}
Let $W$ be a divisor on $\P^2\times\P^2$ of bi-degree $(1, 1)$.
Then $-K_{W}\sim 2H$, where $H$ is a Cartier divisor on $W$. There
is a commutative diagram
$$
\xymatrix{
&X\ar@{->}[ld]_\beta\ar@{->}[dd]_\alpha\ar@{->}[rd]^\gamma\ar@{->}`[rr]`[ddd]^{\omega}[ddd]&&\\
\P^1\times\P^2\ar@{->}[d]_{\phi}&&\P^1\times\P^2\ar@{->}[d]^{\psi}&\\
\P^2&W\ar@{-->}[d]_{\rho}\ar@{->}[l]_{\xi}\ar@{->}[r]^{\zeta}&\P^2&\\
&\mathbb{P}^{1}&& }
$$
where $\phi$ and $\psi$ are natural projections, $\alpha$ is a
blow up of a smooth curve $C\subset W$ such that
$$
C=H_{1}\cap H_{2},
$$
where $H_{1}\ne H_{2}$ are surfaces in $|H|$, the map $\rho$ is
induced by the pencil generated by $H_{1}$~and~$H_{2}$, the
morphism $\omega$ is a del Pezzo fibration of degree $6$,  the
morphisms $\zeta$ and $\xi$ are $\P^1$-bundles, while~$\beta$~and
$\gamma$ contract surfaces $\bar{M}_1\subset X\supset\bar{M}_{2}$
such that
$\phi\circ\beta(\bar{M}_{1})=\xi(C)$~and~$\psi\circ\gamma(\bar{M}_{2})=
\zeta(C)$.

Note that $\lct(X)\le 1/2$, because
$$
-K_{X}\sim 2\bar{H}_{1}+E,
$$
where $\bar{H}_{1}\subset X$ is the proper transform of $H_{1}$,
and $E$ is the exceptional surface of $\alpha$.

We suppose that $X$ satisfies the following generality condition:
all singular fibers of the~fibration $\omega$ satisfy the
hypotheses of Lemma~\ref{lemma:singular-del-Pezzo-sextic}.

Assume that $\mathrm{lct}(X)<1/2$. Then there exists an effective
$\mathbb{Q}$-divisor $D\qlineq -K_{X}$ such that the log pair
$(X,\lambda D)$ is not log canonical for some positive rational
number $\lambda<1/2$. Then
$$
\varnothing\ne\mathrm{LCS}\Big(X,\ \lambda D\Big)\subseteq E,
$$
because $\lct(W)=1/2$ by Theorem~\ref{theorem:del-Pezzo}. Applying
Lemma~\ref{lemma:Hwang}, we see that
$$
\varnothing\ne\mathrm{LCS}\Big(X,\ \lambda D\Big)\subseteq E\cap
F,
$$
where $F$ is a singular fiber of $\omega$. Note that $F$ is a del
Pezzo surface of degree $6$. Put
$$
D=\mu F+\Omega,
$$
where $\Omega$ is an effective $\mathbb{Q}$-divisor  such that
$F\not\subset\mathrm{Supp}(\Omega)$. Then
$$
\Omega\Big\vert_{F}\qlineq D\Big\vert_{F}\qlineq -K_{F},
$$
and the surface $F$ is smooth along the curve $E\cap F$. But the
log pair $(F,\lambda\Omega\vert_{F})$ is~not~log~canonical at some
point $P\in E\cap F$ by Theorem~\ref{theorem:adjunction}, which is
impossible by Lemma~\ref{lemma:singular-del-Pezzo-sextic}.
\end{proof}

\begin{remark}
\label{remark:3-7} Let us use the notation and assumptions of the
proof of Lemma~\ref{lemma:3-7}. Then
$$
\varnothing\ne\mathrm{LCS}\Big(X,\ \lambda D\Big)\subseteq E\cap F,%
$$
where $F$ is a singular fiber of the fibration $\omega$. Applying
Theorem~\ref{theorem:Hwang} to $\phi$ and $\psi$, we see that
$$
\varnothing\ne\mathrm{LCS}\Big(X,\ \lambda D\Big)\subseteq E\cap F\cap \bar{M_1}\cap\bar{M_2},%
$$
by Lemma~\ref{lemma:P1xP2}. Regardless to how singular $F$ is, if
the threefold $X$ is sufficiently general, then
$$
E\cap F\cap \bar{M_1}\cap\bar{M_2}=\varnothing,
$$
which implies that an alternative generality condition can be used
in Lemma~\ref{lemma:3-7}.
\end{remark}

\begin{lemma}
\label{lemma:3-8} Suppose that $\gimel(X)=3.8$ and $X$ is general.
Then $\lct(X)=1/2$.
\end{lemma}

\begin{proof}
Let $\pi_{1}\colon\F_1\times\P^2\to \F_1$ and
$\pi_{2}\colon\F_1\times\P^2\to\P^2$ be natural projections. Then
$$
X\in\Big|\big(\alpha\circ\pi_1\big)^*\Big(\O_{\P^2}\big(1\big)\Big)\otimes \pi_2^*\Big(\O_{\P^2}\big(2\big)\Big)\Big|,%
$$
where $\alpha\colon\F_1\to\P^2$ is a blow up of a point. Let $H$
be a surface in $|\pi_2^*(\O_{\P^2}(1))|$. Then
$$
-K_X\sim E+2L+H,
$$
where $E\subset X\supset L$ are irreducible surfaces such that
$\pi_{1}(E)\subset\mathbb{F}_{1}$ is the exceptional curve of
$\alpha$, and $\pi_{1}(L)\subset\mathbb{F}_{1}$ is a fiber of the
natural projection $\mathbb{F}_{1}\to\mathbb{P}^{2}$. We have
$\lct(X)\le 1/2$.

The projection $\pi_{1}$ induces a fibration $\phi\colon X\to\P^1$
into del Pezzo surfaces~of~degree~$5$.

We suppose that $X$ satisfies the following generality condition:
for every fiber $F$ of $\phi$, the~surface $F$ has at most one
singular point that is an ordinary double point of the surface
$F$.

Assume that $\mathrm{lct}(X)<1/2$. Then there exists an effective
$\mathbb{Q}$-divisor $D\qlineq -K_{X}$ such that the log pair
$(X,\lambda D)$ is not log canonical for some positive rational
number $\lambda<1/2$.

Applying Lemma~\ref{lemma:Hwang} to the morphism $\phi$, we obtain
a contradiction with Example~\ref{example:del-Pezzo-quintic}.
\end{proof}

\begin{lemma}
\label{lemma:3-9} Suppose that $\gimel(X)=3.9$. Then
$\lct(X)=1/3$.
\end{lemma}

\begin{proof}
Let $O_{1}\in V_{1}\cong V_{2}\ni O_{2}$ be singular points of
$V_{1}\cong V_{2}\cong\mathbb{P}(1,1,1,2)$, respectively, let
$$
O_{1}\not\in
S_{1}\in\Big|\mathcal{O}_{\mathbb{P}(1,1,1,2)}(2)\Big|
$$
be a smooth surface, and let $C_{1}\subset
S_{1}\cong\mathbb{P}^{2}$ be a smooth quartic curve.
Then~the~diagram
$$
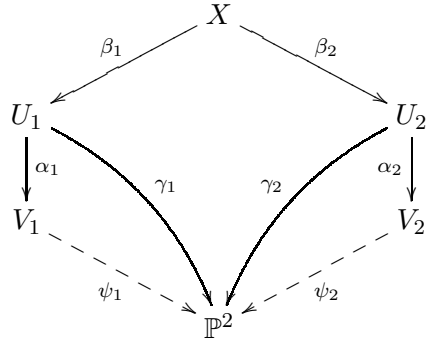
\xymatrix{
&&&&X\ar@{->}[dll]_{\beta_{1}}\ar@{->}[drr]^{\beta_{2}}&&&&\\
&&U_{1}\ar@{->}[d]^{\alpha_{1}}\ar@/^1pc/@{->}[ddrr]^{\gamma_{1}}&&&&U_{2}\ar@{->}[d]_{\alpha_{2}}\ar@/_1pc/@{->}[ddll]_{\gamma_{2}}&&\\
&&V_{1}\ar@{-->}[drr]_{\psi_{1}}&&&&V_{2}\ar@{-->}[dll]^{\psi_{2}}&&\\
&&&&\mathbb{P}^{2}&&&&}
$$
commutes, where $\psi_{i}$ is a natural projection, $\alpha_{i}$
is a blow up of the point $O_{i}$ with weights $(1,1,1)$, the
morphism $\gamma_{i}$~is~a~$\mathbb{P}^{1}$-bundle, and
$\beta_{i}$ is a birational morphism that contracts a surface
$$
\mathbb{P}^{1}\times C_{1}\cong G_{i}\subset X
$$
to a smooth curve $C_{1}\cong C_{i}\subset U_{i}$.

Let $E_{i}\subset X$ be the proper transform of the exceptional
divisor of $\alpha_{i}$. Then
$$
S_{1}=\alpha_{1}\circ\beta_{1}\big(E_{2}\big)\subset V_{1}\cong\mathbb{P}\big(1,1,1,2\big)\cong V_{2}\supset\alpha_{2}\circ\beta_{2}\big(E_{1}\big)%
$$
are surfaces in $|\mathcal{O}_{\mathbb{P}(1,1,1,2)}(2)|$ that
contain the curves $C_{1}$ and $C_{2}$, respectively. On the other
hand,
$$
\alpha_{1}\circ\beta_{1}\big(G_{2}\big)\subset V_{1}\cong\mathbb{P}\big(1,1,1,2\big)\cong V_{2}\supset\alpha_{2}\circ\beta_{2}\big(G_{1}\big)%
$$
are surfaces in $|\mathcal{O}_{\mathbb{P}(1,1,1,2)}(4)|$ that
contain $O_{1}\cup C_{1}$ and $O_{2}\cup C_{2}$, respectively.

Let $\bar{H}\subset X$ be the proper transform of a general
surface in $|\mathcal{O}_{\mathbb{P}(1,1,1,2)}(1)|$.~Then
$$
-K_{X}\sim 3\bar{H}+E_{2}+E_{1},
$$
which gives $\mathrm{lct}(X)\leqslant 1/3$.

Suppose that $\mathrm{lct}(X)<1/3$. Then there is an effective
$\mathbb{Q}$-divisor
$$
D\qlineq -K_{X}\qlineq
\frac{5}{2}\Big(G_{1}+G_{2}\Big)-5\Big(E_{1}+E_{2}\Big)
$$
such that the log pair $(X,\lambda D)$ is not log canonical for
some $\lambda<1/3$. Put
$$
D=\mu_{1}E_{1}+\mu_{2}E_{2}+\Omega,
$$
where $\Omega$ is an effective $\mathbb{Q}$-divisor on $X$ such
that
$$
E_{1}\not\subseteq\mathrm{Supp}\big(\Omega\big)\not\supseteq
E_{2}.
$$

Let $\Gamma$ be a general fiber of the conic bundle
$\gamma_{1}\circ\beta_{1}$. Then
$$
2=\Gamma\cdot D=\Gamma\cdot\Big(\mu_{1}E_{1}+\mu_{2}E_{2}+\Omega\Big)=\mu_{1}+\mu_{2}+\Gamma\cdot\Omega\geqslant\mu_{1}+\mu_{2},%
$$
and without loss of generality we may assume that
$\mu_{1}\leqslant\mu_{2}$. Then $\mu_{1}\leqslant 1$.

Suppose that there is a surface $S\in\mathbb{LCS}(X,\lambda D)$.
Then $S\ne E_{1}$ and $S\ne G_{1}$, because
$\alpha_{2}\circ\beta_{2}(G_{1})\in|\mathcal{O}_{\mathbb{P}(1,1,1,2)}(4)|$
and
$\alpha_{2}\circ\beta_{2}(D)\in|\mathcal{O}_{\mathbb{P}(1,1,1,2)}(5)|$.
Hence $S\cap E_{1}\ne\varnothing$. But
$$
-\frac{1}{3}K_{E_{1}}\qlineq D\Big\vert_{E_{1}}=-\frac{2\mu_{1}}{3}K_{E_{1}}+\Omega\Big\vert_{E_{1}}%
$$
and $E_{1}\cong\mathbb{P}^{2}$, which is impossible by
Theorem~\ref{theorem:adjunction}, because
$\lambda<1/3=\mathrm{lct}(\P^2)$.

We see that the set $\mathbb{LCS}(X,\lambda D)$ contains no
surfaces. Let $P\in\mathrm{LCS}(X,\lambda D)$ be a point. Suppose
that $P\not\in G_{1}$. Let $Z$ be a fiber of $\gamma_{1}$ such
that $\beta_{1}(P)\in Z$. Then
$$
Z\subseteq\mathrm{LCS}\Big(U_{1},\
\lambda\beta_{1}\big(D\big)\Big)
$$
by Theorem~\ref{theorem:Hwang}. Put
$\bar{E}_{1}=\beta_{1}(E_{1})$. Then we have
$$
Z\cap \bar{E}_{1}\in\mathrm{LCS}\Big(\bar{E}_{1},\ \lambda\Omega\Big\vert_{\bar{E}_{1}}\Big)%
$$
by Theorem~\ref{theorem:adjunction}, which is impossible by
Lemma~\ref{lemma:plane}, because $\mu_{1}\leqslant 1$. Hence
$\mathrm{LCS}(X, \lambda D)\subsetneq G_{1}$.

Suppose that $\mathrm{LCS}(X, \lambda D)\subseteq G_{1}\cap
G_{2}$. Then
$$
\Big|\mathrm{LCS}\Big(X,\ \lambda D\Big)\Big|=1
$$
by Lemma~\ref{lemma:rational-tree} and
Theorem~\ref{theorem:connectedness}. One has
$$
\mathrm{LCS}\Big(X,\ \lambda D\Big)\cup\bar{H}\subseteq \mathrm{LCS}\left(X,\ \lambda D+\frac{1}{3}\Big(E_{2}+E_{2}\Big)+\bar{H}\right)\subset\mathrm{LCS}\Big(X,\ \lambda D\Big)\cup\bar{H}\cup E_{1}\cup E_{1}, %
$$
which contradicts Theorem~\ref{theorem:connectedness}, because
$\bar{H}$ is a general surface in
$|(\beta_{1}\circ\gamma_{1})^{*}(\mathcal{O}_{\mathbb{P}^{2}}(1))|$
and
$$
\lambda D+\frac{1}{3}\Big(E_{2}+E_{2}\Big)+\bar{H}\qlineq \big(\lambda-1/3\big)K_{X}.%
$$

Thus, we see that $G_{1}\supsetneq\mathbb{LCS}(X,\lambda
D)\not\subseteq G_{1}\cap G_{2}$. Then
$$
\varnothing\ne\mathrm{LCS}\Big(U_{2},\ \lambda\beta_{2}\big(D\big)\Big)\subsetneq\beta_{2}\big(G_{1}\big),%
$$
and it follows from Theorems~\ref{theorem:connectedness}
and~\ref{theorem:Hwang} that there is a fibre $L$ of the fibration
$\gamma_{2}$ such that
$$
\mathrm{LCS}\Big(U_{2},\ \lambda\beta_{2}\big(D\big)\Big)=L.
$$

Let $B$ be a general surface in
$|\alpha_{2}^{*}(\mathcal{O}_{\mathbb{P}(1,1,1,2)}(2))|$. Then
$\beta_{2}(D)\vert_{B}\qlineq\mathcal{O}_{\mathbb{P}^{2}}(5)$ and
$B\cong\mathbb{P}^{2}$.~But
$$
\mathrm{LCS}\Big(B,\
\lambda\beta_{2}\big(D\big)\Big\vert_{B}\Big)=L\cap B
$$
and $|L\cap B|=1$, which is impossible by Lemma~\ref{lemma:plane},
because $\lambda<1/3$.
\end{proof}

\begin{lemma}
\label{lemma:3-10} Suppose that $\gimel(X)=3.10$. Then
$\mathrm{lct}(X)=1/2$.
\end{lemma}

\begin{proof}
Let $Q\subset\mathbb{P}^{4}$ be a smooth quadric hypersurface.
Let $C_{1}\subset Q\supset C_{2}$ be disjoint (irreducible) conics. Then
there is a commutative diagram
$$
\xymatrix{
&&&&X\ar@{->}[lldd]^{\beta_{2}}\ar@{->}[rrdd]_{\beta_{1}}\ar@{->}[lllldd]_{\phi_{1}}\ar@{->}[rrrrdd]^{\phi_{2}}&&&&\\
&&&&&&&&\\
\P^1&&Y_{1}\ar@{->}[rr]_{\alpha_{1}}\ar@{->}[ll]^{\psi_{1}}&&Q&&Y_{2}\ar@{->}[ll]^{\alpha_{2}}\ar@{->}[rr]_{\psi_{2}}&&\P^1}
$$
where the morphism $\alpha_{i}$ is a blow up along the conic
$C_{i}$, the morphism $\beta_{i}$ is a blow up along the~proper
transform~of the conic $C_{i}$, the morphism $\psi_{i}$ is a
natural fibration into quadric surfaces,
and~$\phi_{i}$~is~fibration, whose general fiber is isomorphic to
a smooth del Pezzo surfaces of degree $6$.

Let $E_{i}$ be the exceptional divisor of the morphism
$\beta_{i}$, and let $H_i$ be a sufficiently general hyperplane
section of the quadric $Q$ that passes through the conic $C_i$.
Then
$$
-K_{X}\sim \bar{H_1}+2\bar{H_2}+E_2,
$$
where $\bar{H}_i\subset X$ is the proper transform of the surface
$H_i$. We see that $\lct(X)\le 1/2$.

We suppose that $\mathrm{lct}(X)<1/2$. Then there exists an
effective $\mathbb{Q}$-divisor $D\qlineq -K_{X}$ such that the log
pair $(X,\lambda D)$ is not log canonical for some positive
rational number $\lambda<1/2$.

Using Example~\ref{example:del-Pezzos} and applying
Lemma~\ref{theorem:Hwang}, we see that
$$
\varnothing\ne\LCS\Big(X,\ \lambda D\Big)\subseteq S_{1}\cap S_{2},%
$$
where $S_{i}$ is a singular fiber of $\phi_{i}$. Hence, the set
$\mathbb{LCS}(X, \lambda D)$ contains no surfaces.

It follows from Theorem~\ref{theorem:connectedness} that either
$\LCS(X, \lambda D)$ is a point in $E_1\cup E_{2}$, or
$$
\LCS\Big(X,\ \lambda D\Big)\cap\Big(X\setminus\big(E_1\cup E_2\big)\Big)\ne\varnothing,%
$$
which implies that we may assume that $\LCS(X, \lambda D)$ is a
point in $E_1$ by Lemma~\ref{lemma:quadric-curves}.

Since $\beta_{2}$ is an isomorphism on $X\setminus E_2$, we see
that
$$
P\in\LCS\Big(Y_{1},\ \lambda\beta_{2}\big(D\big)\Big)\subset P\cup\beta_{2}\big(E_2\big)%
$$
for some point $P\in E_1$. Then $\LCS(Y_{1},
\lambda\beta_{2}(D))=P$ by Theorem~\ref{theorem:connectedness},
because $P\not\in\beta_{2}(E_{2})$.

Let $H$ be a general hyperplane section of the quadric $Q$. Then
$$
-K_{Y_{1}}\sim\tilde{H}_1+2\tilde{H}\qlineq\beta_{2}\big(D\big),
$$
where $\tilde{H}\subset Y_{1}\supset\tilde{H}_1$ are proper
transforms of $H$ and $H_1$, respectively. But
$$
\LCS\left(Y_{1},\ \lambda\beta_{2}\big(D\big)+\frac{1}{2}\Big(\tilde{H}_1+2\tilde{H}\Big)\right)=P\cup\tilde{H},%
$$
which is impossible by Theorem~\ref{theorem:connectedness},
because $\lambda<1/2$.
\end{proof}

\begin{lemma}
\label{lemma:3-11} Suppose that $\gimel(X)=3.11$. Then $\mathrm{lct}(X)=1/2$.%
\end{lemma}

\begin{proof}
Let $O\in\P^3$ be a point, let
$\delta\colon V_{7}\to\mathbb{P}^{3}$
be a blow up of the point $O$, and let $E$ be the exceptional
divisor of $\delta$. Then
$$
V_{7}\cong\mathbb{P}\Big(\mathcal{O}_{\mathbb{P}^{2}}\oplus\mathcal{O}_{\mathbb{P}^{2}}\big(1\big)\Big),
$$
there is a natural $\mathbb{P}^{1}$-bundle $\eta\colon
V_{7}\to\mathbb{P}^{2}$, and $E$ is a section of $\eta$. There is a
linearly normal elliptic curve $O\in C|subset\P^3$ such that the
diagram
$$
\xymatrix{
U\ar@/_/@{->}[ddr]_{\omega}\ar@{->}[rr]^{\gamma}&&\mathbb{P}^{3}&&V_{7}\ar@{->}[ll]_{\delta}\ar@{->}[d]^{\eta}\\%
&&X\ar@{->}[ull]^{\alpha}\ar@{->}[urr]_{\beta}\ar@{->}[dr]^{\upsilon}\ar@{->}[dl]_{\phi}&&\mathbb{P}^{2}\\%
&\mathbb{P}^{1}&&\mathbb{P}^{1}\times\mathbb{P}^{2}\ar@{->}[ll]^{\pi_{1}}\ar@/_/@{->}[ur]_{\pi_{2}}&}%
$$
commutes, where $\pi_{1}$ and $\pi_{2}$ are natural projections,
the morphism $\gamma$ contracts a surface
$$
C\times\mathbb{P}^{1}\cong G\subset U
$$
to the curve $C$, the~morphism $\alpha$ is a blow up of the fiber
of the morphism $\gamma$ over the point~$O\in\mathbb{P}^{3}$, the
morphism $\beta$ is a blow up of the proper transform of $C$, the
morphism $\omega$~is~a~fibration into quadric surfaces, the
morphism $\phi$ is a fibration into del Pezzo surfaces of degree
$7$, and $\upsilon$ contracts~a~surface
$$
C\times\mathbb{P}^{1}\cong F\subset X
$$
to an elliptic curve $Z\subset\mathbb{P}^{1}\times\mathbb{P}^{2}$
such that $-K_{\mathbb{P}^{1}\times\mathbb{P}^{2}}\cdot Z=13$ and
$Z\cong C$.

Let $H_{1}$ be a general fiber of $\phi$, and let $H_{2}$ be a
general surface in
$|(\eta\circ\beta)^{*}(\mathcal{O}_{\mathbb{P}^{2}}(1))|$. Then
$$
-K_{X}\sim H_{1}+2H_{2},
$$
which implies that $\mathrm{lct}(X)\leqslant 1/2$.

We suppose that $\mathrm{lct}(X)<1/2$. Then there exists an
effective $\mathbb{Q}$-divisor $D\qlineq -K_{X}$ such that the log
pair $(X,\lambda D)$ is not log canonical for some positive
rational number $\lambda<1/2$. Note that
$$
\varnothing\ne\mathrm{LCS}\Big(X,\ \lambda D\Big)\subseteq\bar{E},%
$$
where $\bar{E}$ is the exceptional divisor of $\alpha$, because
$\mathrm{lct}(U)=1/2$ by Lemma~\ref{lemma:2-25}.

Let $\Gamma\cong\mathbb{P}^{2}$ be the general fiber of
$\pi_{2}\circ\upsilon$. Then
$$
2=-K_{X}\cdot\Gamma=D\cdot\Gamma=2\bar{E}\cdot\Gamma,
$$
which implies that $\bar{E}\not\subset\mathrm{LCS}(X,\lambda D)$.
Applying Lemma~\ref{lemma:Hwang} to the log pair
$$
\Big(V_{7},\ \lambda \beta\big(D\big)\Big)%
$$
we see that $\mathrm{LCS}(X,\lambda D)\subseteq\bar{E}\cap G$.
Applying Lemma~\ref{lemma:P1xP2} to the log pair
$$
\Big(\mathbb{P}^{1}\times\mathbb{P}^{2},\ \lambda \upsilon\big(D\big)\Big)%
$$
we see that $\mathrm{LCS}(X,\lambda D)=\bar{E}\cap F\cap G$, where
$|\bar{E}\cap F\cap G|=1$. Hence
$$
\mathrm{LCS}\Big(X,\ \lambda D+H_{2}\Big)=\mathrm{LCS}\Big(X,\ \lambda D\Big)\cup H_{2},%
$$
and $H_{2}\cap\mathrm{LCS}(X,\lambda D)=\varnothing$. But the
divisor
$$
-\Big(K_{X}+\lambda D+H_{2}\Big)=\left(\lambda-\frac{1}{2}\right)K_{X}+\frac{1}{2}H_{1}%
$$
is ample, which is impossible by
Theorem~\ref{theorem:connectedness}.
\end{proof}

\begin{lemma}
\label{lemma:3-12} Suppose that $\gimel(X)=3.12$. Then
$\mathrm{lct}(X)=1/2$.
\end{lemma}

\begin{proof}
Let $\eps\colon V\to\mathbb{P}^{3}$ be a blow up of a line
$L\subset\mathbb{P}^{3}$. There is a natural
$\mathbb{P}^{2}$-bundle $\eta\colon V\to\mathbb{P}^{1}$, there is
a smooth rational cubic curve $C\subset\mathbb{P}^{3}$ such that
$L\cap C=\varnothing$, and the diagram
$$
\xymatrix{
\P^1&&&\P^1\times\P^2\ar@{->}[rrd]^{\pi_2}\ar@{->}[lll]_{\pi_{1}}&&\\
&&X\ar@{->}[lld]^{\beta}\ar@{->}[rd]_{\gamma}\ar@{->}[llu]^{\phi}\ar@{->}[ru]_{\omega}&&&\P^2\\
V\ar@{->}[rrd]_\eps\ar@{->}[uu]^{\eta}&&&Y\ar@{->}[ld]^\alpha\ar@{->}[rru]_{\psi}&&\\
&&\P^3&&& }
$$
commutes, where $\alpha$ and $\beta$ are blow ups of the curve $C$
and its proper transform, respectively, the~morphism $\gamma$ is a
blow up of the proper transform of the line $L$, the morphism
$\psi$ is a $\P^1$-bundle, the morphism $\omega$ is a birational
contraction of a surface $F\subset X$ to a curve such that
$$
C\cup L\subset\alpha\circ\gamma\big(F\big)\subset\mathbb{P}^{3},
$$
and $\alpha\circ\gamma(F)$ consists of secant lines of
$C\subset\P^3$ that intersect $L$, the morphism $\phi$ is a
fibration into del Pezzo surfaces of degree $6$, the morphisms
$\pi_1$ and $\pi_2$ are natural projections.

Let $E$ and $G$ be exceptional divisors of $\beta$ and $\gamma$,
respectively, let $Q\subset\P^3$ be a general quadric surface that
passes through $C$, let $H\subset\P^3$ be a general plane that
passes through $L$. Then
$$
-K_{X}\sim \bar{Q}+2\bar{H}+G,
$$
where $\bar{Q}\subset X\supset\bar{H}$ are proper transforms of
$Q\subset\P^3\supset H$, respectively. In particular,
$\lct(X)\le\nolinebreak 1/2$.

We suppose that $\lct(X)<1/2$. Then there exists an effective
$\Q$-divisor $D\qlineq -K_{X}$ such that the log pair $(X,\lambda
D)$ is not log canonical for some positive rational number
$\lambda<1/2$. Note that
$$
\varnothing\neq\LCS\Big(X,\ \lambda D\Big)\subset G,
$$
since $\lct(Y)=1/2$ by Lemma~\ref{lemma:2-27}. Applying
Theorem~\ref{theorem:Hwang} to $\phi$ we see that
$$
\varnothing\neq\LCS\Big(X,\ \lambda D\Big)\subset G\cap S_{\phi},
$$
where $S_{\phi}$ is a singular fiber of the del Pezzo fibration
$\phi$ (see Example~\ref{example:del-Pezzos}). Then we see that
$$
\varnothing\neq\LCS\Big(X,\ \lambda D\Big)\subset G\cap
S_{\phi}\cap F,
$$
by applying Theorem~\ref{theorem:Hwang} to the log pair
$(\mathbb{P}^{1}\times\mathbb{P}^{2},\lambda\omega(D))$ and to the
$\mathbb{P}^{1}$-bundle $\pi_2$.

Let $Z_{1}\cong\P^1$ be a section of the natural projection
$$
\P^1\times\P^1\cong G\longrightarrow L\cong\P^1
$$
such that $Z_1\cdot Z_1=0$, and let $Z_2$ a fiber of this
projection. Then
$$
F\Big\vert_{G}\sim Z_{1}+3Z_{2}
$$
and $S_{\phi}\vert_{G}\sim Z_{1}$. The curve $F\cap G$ is
irreducible. Thus, we see that
$$
\Big|G\cap F\cap S_{\phi}\Big|<+\infty,
$$
which implies that the set $\LCS(X, \lambda D)$ consists of a
single point $P\in G$ by Theorem~\ref{theorem:connectedness}.

The log pair $(V, \lambda\beta(D))$ is not log canonical. Since
$\beta$ is an isomorphism on $X\setminus E$, we see that
$$
\beta\big(P\big)\in\LCS\Big(V,\ \lambda\beta\big(D\big)\Big)\subseteq \beta\big(P\big)\cup\beta\big(E\big),%
$$
which implies that  $\LCS(V, \lambda\beta(D))=\beta(P)$ by
Theorem~\ref{theorem:connectedness}. Let $H\subset\P^3$ be a
general plane.~Then
$$
\LCS\left(V,\ \lambda\beta\big(D\big)+\frac{1}{2}\Big(\tilde{H}_1+3\tilde{H}\Big)\right)=\beta(P)\cup\tilde{H},%
$$
where $\tilde{H}\subset V\supset\tilde{H}_1$ are proper transforms
of $H\subset\P^3\supset H_1$, respectively. But
$$
-K_V\sim \tilde{H}_1+3\tilde{H}\qlineq\beta\big(D\big),
$$
which contradicts Theorem~\ref{theorem:connectedness}, because
$\lambda<1/2$.
\end{proof}

\begin{lemma}
\label{lemma:3-14} Suppose that $\gimel(X)=3.14$. Then
$\lct(X)=1/2$.
\end{lemma}

\begin{proof}
Let $P\in\mathbb{P}^{3}$ be a point, and let $\alpha\colon
V_{7}\to\mathbb{P}^{3}$ be a blow up of the point $P$. Then there
is a natural $\mathbb{P}^{1}$-bundle $\pi\colon
V_{7}\to\mathbb{P}^{2}$.

Let $\zeta\colon Z\to\mathbb{P}(1,1,1,2)$ be a blow up of the
singular point of $\mathbb{P}(1,1,1,2)$. Then
$$
Z\cong\mathbb{P}\Big(\mathcal{O}_{\mathbb{P}^{2}}\oplus\mathcal{O}_{\mathbb{P}^{2}}\big(2\big)\Big),
$$
and there is a natural $\mathbb{P}^{1}$-bundle $\phi\colon
Z\to\mathbb{P}^{2}$.

There is a plane $\Pi\subset\mathbb{P}^{3}$ and a smooth cubic
curve $C\subset\Pi$ such that $P\not\in\Pi$ and the diagram
$$
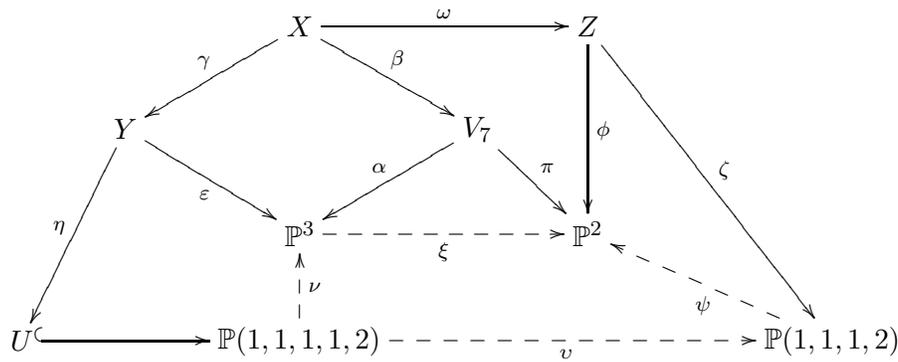
\xymatrix{
&&&X\ar@{->}[ld]_{\gamma}\ar@{->}[rd]^{\beta}\ar@{->}[rr]^{\omega}&&Z\ar@{->}[dd]^{\phi}\ar@{->}[rrddd]^{\zeta}&&\\
&&Y\ar@{->}[rd]_{\eps}\ar@{->}[ldd]_{\eta}&&V_{7}\ar@{->}[ld]_{\alpha}\ar@{->}[rd]^{\pi}& \\
&&&\P^3\ar@{-->}[rr]_{\xi}&&\P^2\\
&U\ar@{^{(}->}[rr]&&\mathbb{P}(1,1,1,1,2)\ar@{-->}[rrrr]_{\upsilon}\ar@{-->}[u]_{\nu}&&&&\mathbb{P}(1,1,1,2)\ar@{-->}[llu]^{\psi}}
$$
commutes (see \cite[Example~3.6]{Stef96}), where we have the
following notation:
\begin{itemize}
\item the morphism $\eps$ is a blow up of the curve $C$;%
\item the threefold $U$ is a cubic hypersurface in $\mathbb{P}(1,1,1,1,2)$;%
\item the rational map $\xi$ is a projection from the point $P$;%
\item the morphism~$\gamma$~is~a~blow~up of the point that dominates $P$;%
\item the~morphism~$\beta$~is~a~blow~up of the proper transform of the curve $C$;%
\item the morphism $\eta$ contracts the proper transform of $\Pi$ to the point $\mathrm{Sing}(U)$,%
\item the~morphism $\omega$ contracts a surface $R\subset X$ to a
curve~such~that
$$
\beta\circ\alpha\big(R\big)\subset\mathbb{P}^{3}
$$
is a cone over the curve $C$ whose vertex is the point $P$;%
\item the rational maps $\psi$ and $\nu$ are natural projections;%
\item the rational map $\upsilon$ is a linear projection from a
point.
\end{itemize}

Let $E$ and $G$ be exceptional divisors of $\gamma$ and $\beta$,
respectively, and let $\bar{H}\subset X$ be a proper transform of
a general plane in $\P^3$ that passes through the point $P$. Then
$$
-K_{X}\sim \bar{\Pi}+3\bar{H}+G,
$$
where $\bar{\Pi}\subset X$ is the proper transform of the plane
$\Pi$. Thus, we see that $\lct(X)\leqslant 1/3$.

We suppose that $\mathrm{lct}(X)<1/3$. Then there exists an
effective $\mathbb{Q}$-divisor $D\qlineq -K_{X}$ such that the log
pair $(X,\lambda D)$ is not log canonical for some positive
rational number $\lambda<1/3$.

Let $\bar{L}\subset X$ be a proper transform of a general line in
$\P^3$ that intersects the curve $C$. Then
$$
D\cdot\bar{L}=\bar{\Pi}\cdot\bar{L}+3\bar{H}\cdot\bar{L}+G\cdot\bar{L}=3\bar{H}\cdot\bar{L}=3,
$$
which implies that $\mathbb{LCS}(X, \lambda D)$ contains no
surfaces except possibly $\bar{\Pi}$ and $E$.

Let $\Gamma$ be a general fiber of $\pi\circ\beta$. Then
$$
D\cdot\Gamma=\bar{\Pi}\cdot\Gamma+3\bar{H}\cdot\Gamma+G\cdot\Gamma=\bar{\Pi}\cdot\Gamma+G\cdot\Gamma=2,
$$
which implies that $\mathbb{LCS}(X, \lambda D)$ does not contain
$\bar{\Pi}$ and $E$. Thus, by Lemma~\ref{lemma:P3}, we have
$$
\varnothing\ne\LCS\Big(X,\ \lambda D\Big)\subsetneq E\cup G.
$$

Suppose that $\LCS(X,\lambda D)\subseteq E$. Then
$$
\varnothing\ne\LCS\Big(V_{7},\ \lambda \beta\big(D\big)\Big)\subseteq \beta\big(E\big),%
$$
which contradicts  Theorem~\ref{theorem:Hwang}, because $\beta(E)$
is a section of $\pi$. We see that $\LCS(X,\lambda D)\subsetneq
G$.

Applying Theorem~\ref{theorem:Hwang} to $(Z, \lambda\omega(D))$
and $\phi$ and Theorem~\ref{theorem:connectedness} to $(X,\lambda
D)$, we see that
$$
\varnothing\ne\LCS\Big(X,\ \lambda D\Big)\subseteq F,
$$
where $F$ is a fiber of the natural projection $G\to \beta(G)$.
Then
$$
\varnothing\ne\LCS\Big(Y,\ \lambda\gamma\big(D\big)\Big)\subseteq\gamma(F),%
$$
where $\gamma(F)$ is a fiber of the blow up $\eps$ over a point in
the curve $C$.

Let $S\subset\P^3$ be a general cone over the curve $C$, and let
$O\in C$ be an inflection point such that
$$
\eps\circ\gamma\big(F\big)\ne O.
$$
Let $L\subset S$ be a line that passes through the point $O$, and
let $H\subset\P^3$ be a plane that is tangent to the cone $S$
along the line $L$. Since $O$ is an inflection point of the curve
$C$, the equality
$$
\mathrm{mult}_{L}\Big(S\cdot H\Big)=3
$$
holds. Let $\breve{S}$, $\breve{H}$ and $\breve{L}$ be the proper
transforms of $S$, $H$ and $L$ on the threefold $Y$. Then
$$
\LCS\left(Y,\ \lambda\gamma(D)+\frac{2}{3}\Big(\breve{S}+\breve{H}\Big)\right)=\LCS\Big(Y,\ \lambda\gamma\big(D\big)\Big)\cup \breve{L}%
$$
due to generality in the choice of $S$. But $-K_Y\sim
\breve{S}+\breve{H}$, which is impossible by
Theorem~\ref{theorem:connectedness}.
\end{proof}

\begin{lemma}
\label{lemma:3-15} Suppose that  $\gimel(X)=3.15$. Then
$\mathrm{lct}(X)=1/2$.
\end{lemma}

\begin{proof}
Let $Q\subset\mathbb{P}^{4}$ be a smooth quadric hypersurface, let
$C\subset Q$ be a smooth conic, and let $\eps\colon V\to Q$ be a
blow up of the conic $C\subset Q$. Then there is a natural
morphism $\eta\colon V\to\mathbb{P}^{1}$
induced by the projection $Q\dasharrow\mathbb{P}^{1}$ from the
two-dimensional linear subspace in $\mathbb{P}^{4}$ that contains
the conic $C\subset Q$. Then a general fiber of $\eta$ is a smooth
quadric surface in $\mathbb{P}^{3}$.

Take a line $L\subset Q$ such that $L\cap C=\varnothing$; then
there is a commutative diagram
$$
\xymatrix{
\P^1&&&\P^1\times\P^2\ar@{->}[rrd]^{\pi_2}\ar@{->}[lll]_{\pi_{1}}&&\\
&&X\ar@{->}[lld]^{\beta}\ar@{->}[rd]_{\gamma}\ar@{->}[llu]^{\phi}\ar@{->}[ru]_{\omega}&&&\P^2\\
V\ar@{->}[rrd]_\eps\ar@{->}[uu]^{\eta}&&&Y\ar@{->}[ld]^\alpha\ar@{->}[rru]_{\psi}&&\\
&&Q&&& }
$$
where $\alpha$ and $\beta$ are blow ups of the line $L\subset Q$
and its proper transform, respectively, the~morphism $\gamma$ is a
blow up of the proper transform of the conic $C$, the morphism
$\psi$ is a $\P^1$-bundle, the morphism $\omega$ is a birational
contraction of a surface $F\subset X$ to a curve such that
$$
C\cup L\subset\alpha\circ\gamma\big(F\big)\subset Q,
$$
and $\alpha\circ\gamma(F)$ consists of all lines in $Q\subset\P^4$
that intersect $L$ and $C$, the morphism $\phi$ is a fibration
into del Pezzo surfaces of degree $7$, the morphisms $\pi_1$ and
$\pi_2$ are natural projections.

Let $E_1$ and $E_{2}$ be exceptional surfaces of $\beta$ and
$\gamma$, respectively, let $H_{1}, H_2\subset Q$ be general
hyperplane sections that pass through the curves $L$ and $C$,
respectively. We have
$$
-K_{X}\sim \bar{H}_1+2\bar{H}_2+E_2\sim \bar{H}_2+2\bar{H}_1+E_1,
$$
where $\bar{H}_1\subset X\supset\bar{H}_2$ are proper transforms
of $H_1\subset Q\supset H_2$, respectively. In particular,
$\lct(X)\le 1/2$.

We suppose that $\mathrm{lct}(X)<1/2$. Then there exists an
effective $\mathbb{Q}$-divisor $D\qlineq -K_{X}$ such that the log
pair $(X,\lambda D)$ is not log canonical for some positive
rational number $\lambda<1/2$.

Let $S$ be an irreducible surface on the threefold $X$. Put
$$
D=\mu S+\Omega,
$$
where $\Omega$ is an effective $\mathbb{Q}$-divisor  such that
$S\not\subset\mathrm{Supp}(\Omega)$. Then
$$
\mathrm{LCS}\left(\bar{H}_{2},\ \frac{1}{2}\Big(\mu
S+\Omega\Big)\Big\vert_{\bar{H}_{2}}\right)\subset E_{1}\cap\bar{H}_{2}%
$$
by Lemma~\ref{lemma:del-Pezzo-septic}. Thus, if $\mu\leqslant 2$,
then either $S=E_{1}$, or $S$ is a fiber of $\phi$.

Let $\Gamma\cong\mathbb{P}^{1}$ be a general fiber of the conic
bundle $\psi\circ\gamma$. Then
$$
2=D\cdot\Gamma=\mu S\cdot\Gamma+\Omega\cdot\Gamma\geqslant\mu S\cdot\Gamma,%
$$
which implies that  $\mu\leqslant 2$ in the case when either
$S=E_{1}$, or $S$ is a fiber of $\phi$.

Therefore, we see that $\mathbb{LCS}(X, \lambda D)$ does not
contain surfaces.

Applying Theorem~\ref{theorem:Hwang} to the log pair
$(Y,\lambda\gamma(D))$ and $\psi$, we see that
$$
\varnothing\ne\LCS\Big(X,\ \lambda D\Big)\subsetneq
E_2\cup\bar{L},
$$
where $\mathbb{P}^1\cong\bar{L}\subset X$ is a curve such that
$\gamma(\bar{L})$ is a fiber of the conic bundle $\psi$.

Suppose that $\bar{L}\not\subset E_{1}$ and
$\bar{L}\subset\mathrm{LCS}(X, \lambda D)$. Then
$$
\alpha\circ\gamma\big(\bar{L}\big)\subseteq\LCS\Big(Q,\ \lambda \alpha\circ\gamma\big(D\big)\Big)\subseteq\alpha\circ\gamma\big(\bar{L}\big)\cup C\cup L,%
$$
which is impossible by Lemma~\ref{lemma:quadric-curves}. Hence by
Theorem~\ref{theorem:connectedness} we see that
\begin{itemize}
\item either $\LCS(X, \lambda D)\subsetneq E_2$,%
\item or $\LCS(X, \lambda D)\subseteq\bar{L}$ and $\bar{L}\subset
E_{1}$.
\end{itemize}

We may assume that $\bar{L}\subset E_{1}$. Note that
$E_{1}\cong\mathbb{F}_{1}$. One has $\bar{L}\cdot\bar{L}=-1$ on
the surface $E_{1}$.

Applying Lemma~\ref{lemma:P1xP2} to the log pair $(\P^1\times\P^2,
\lambda\omega(D))$, we see that $\LCS(X, \lambda D)\subset F$,
because
$$
\omega\big(D\big)\qlineq -K_{\P^1\times\P^2}
$$
and $\lambda<1/2$. Applying Lemma~\ref{lemma:Hwang} to the log
pair $(V, \lambda\beta(D))$ and the fibration $\eta$, we see that
$$
\varnothing\ne\LCS\Big(X,\ \lambda D\Big)\subsetneq E_1\cup
S_{\phi},
$$
where $S_{\phi}$ is a singular fiber of $\phi$, because
$\mathrm{lct}(\P^1\times\P^1)=1/2$ (see
Example~\ref{example:del-Pezzos}).

We have $F\cap\bar{L}=\varnothing$ and $|F\cap\bar{S_{\phi}}\cap
E_2|<+\infty$. Thus, there is point $P\in E_{2}$ such that
$$
\LCS\Big(X,\ \lambda D\Big)=P\in E_{2}
$$
by Theorem~\ref{theorem:connectedness}. But
$\beta(E_{1})\cap\beta(P)=\varnothing$. Thus, it follows from
Theorem~\ref{theorem:connectedness} that
$$
\LCS\Big(V,\ \lambda\beta\big(D\big)\Big)=\beta\big(P\big).
$$

Let $\tilde{H}_1\subset V\supset\tilde{H}_2$ be the proper
transforms of $H_1\subset Q\supset H_2$, respectively. Then
$$
-K_V\sim \tilde{H}_2+2\tilde{H}_1\qlineq\beta\big(D\big),
$$
but it follows from the generality of $H_{1}$ and $H_{2}$ that
$$
\LCS\left(V,\
\lambda\beta\big(D\big)+\frac{1}{2}\Big(\tilde{H}_2+2\tilde{H}_1\Big)\right)=\beta\big(P\big)\cup\tilde{H}_1,
$$
which is impossible by Theorem~\ref{theorem:connectedness},
because $\lambda<1/2$.
\end{proof}

\begin{lemma}
\label{lemma:3-16} Suppose that  $\gimel(X)=3.16$. Then
$\mathrm{lct}(X)=1/2$.
\end{lemma}

\begin{proof} Let $\mathbb{P}^{1}\cong C\subset\mathbb{P}^{3}$ be a twisted
cubic curve, let $O\in C$ be a point. There is a commutative
diagram
$$
\xymatrix{
\mathbb{P}\big(\mathcal{E}\big)&U\ar@{=}[l]\ar@/_/@{->}[ddr]_{\omega}\ar@{->}[rr]^{\gamma}&&\mathbb{P}^{3}&&V_{7}\ar@{->}[ll]_{\delta}\ar@{->}[d]^{\eta}\ar@{=}[r]&\mathbb{P}\Big(\mathcal{O}_{\mathbb{P}^{2}}\oplus\mathcal{O}_{\mathbb{P}^{2}}\big(1\big)\Big)\\%
&&&X\ar@{->}[ull]^{\alpha}\ar@{->}[urr]_{\beta}\ar@{->}[dr]^{\upsilon}&&\mathbb{P}^{2}&\\%
&&\mathbb{P}^{2}&&W\ar@{->}[ll]^{\pi_{1}}\ar@/_/@{->}[ur]_{\pi_{2}}&&}%
$$
where $\mathcal{E}$ is a stable rank two vector bundle on
$\mathbb{P}^{2}$ (see the proof of Lemma~\ref{lemma:2-27}), and we
have the following notation:
\begin{itemize}
\item the morphism $\delta$ is a blow up of the point $O$;%
\item the morphism $\gamma$ contracts a surface $G\subset U$ to the curve $C\subset\mathbb{P}^{3}$;%
\item the~morphism $\alpha$ contracts a surface $E\cong\mathbb{F}_{1}$ to the fiber of $\gamma$ over the point~$O\in\mathbb{P}^{3}$;%
\item the morphism $\beta$ is a blow up of the proper transform of the curve $C$;%
\item the variety $W$ is a smooth divisor on $\mathbb{P}^{2}\times\mathbb{P}^{2}$ of bi-degree $(1,1)$;%
\item the morphisms $\pi_{1}$ and $\pi_{2}$ are natural projections;%
\item the morphisms $\omega$ and $\eta$ are natural $\mathbb{P}^{1}$-bundles;%
\item the morphism $\upsilon$ contracts~a~surface $F\subset X$ to
a curve
$$
\mathbb{P}^{1}\cong Z\subset W
$$
such that $\omega\circ\alpha(E)=\pi_{1}(Z)$ and $\eta\circ\beta(G)=\pi_{2}(Z)$.%
\end{itemize}

Take general surfaces
$H_{1}\in|(\omega\circ\alpha)^{*}(\mathcal{O}_{\mathbb{P}^{2}}(1))|$
and
$H_{2}\in|(\eta\circ\beta)^{*}(\mathcal{O}_{\mathbb{P}^{2}}(1))|$.~Then
$$
-K_{X}\sim H_{1}+2H_{2},
$$
which implies that $\mathrm{lct}(X)\leqslant 1/2$.

We suppose that $\mathrm{lct}(X)<1/2$. Then there exists an
effective $\mathbb{Q}$-divisor $D\qlineq -K_{X}$ such that the log
pair $(X,\lambda D)$ is not log canonical for some positive
rational number $\lambda<1/2$. Note that
$$
\varnothing\ne\mathrm{LCS}\Big(X,\ \lambda D\Big)\subseteq E\cap F,%
$$
because $\mathrm{lct}(U)=1/2$ by Lemma~\ref{lemma:2-25} and
$\mathrm{lct}(W)=1/2$ by Theorem~\ref{theorem:del-Pezzo}.

Applying Lemma~\ref{lemma:V7} to the log pair $(V_{7},\lambda
\beta(D))$, we see that
$$
\mathrm{LCS}\Big(X,\lambda D\Big)=E\cap F\cap G,
$$
where $|E\cap F\cap G|=1$. Then
$$
\mathrm{LCS}\Big(X,\ \lambda D+H_{2}\Big)=\mathrm{LCS}\Big(X,\ \lambda D\Big)\cup H_{2},%
$$
where $H_{2}\cap\mathrm{LCS}(X,\lambda D)=\varnothing$. But the
divisor
$$
-\Big(K_{X}+\lambda D+H_{2}\Big)\qlineq\left(\lambda-\frac{1}{2}\right)K_{X}+\frac{1}{2}H_{1}%
$$
is ample, which is impossible by
Theorem~\ref{theorem:connectedness}.
\end{proof}

\begin{lemma}
\label{lemma:3-17} Suppose that $\gimel(X)=3.17$. Then
$\mathrm{lct}(X)=1/2$.
\end{lemma}
\begin{proof}
The threefold $X$ is a divisor in $\P^1\times\P^1\times\P^2$ of
tri-degree $(1, 1, 1)$. Take general surfaces
$$
H_{1}\in\Big|\pi_{1}^{*}\Big(\mathcal{O}_{\mathbb{P}^{1}}\big(1\big)\Big)\Big|,\ H_{2}\in\Big|\pi_{2}^{*}\Big(\mathcal{O}_{\mathbb{P}^{1}}\big(1\big)\Big)\Big|,\ H_{3}\in\Big|\pi_{3}^{*}\Big(\mathcal{O}_{\mathbb{P}^{2}}\big(1\big)\Big)\Big|,%
$$
where $\pi_i$ is a natural projection of the threefold $X$ onto
the $i$-th factor of $\P^1\times\P^1\times\P^2$. Then
$$
-K_X\sim H_1+H_2+2H_3,
$$
which implies that $\lct(X)\le 1/2$. There is a commutative
diagram
$$
\xymatrix{
\mathbb{P}^{1}&&\mathbb{P}^{1}\times\mathbb{P}^{1}\ar@{->}[ll]_{\upsilon_{1}}\ar@{->}[rr]^{\upsilon_{2}}&&\mathbb{P}^{1}\\%
&&X\ar@{->}[u]_{\zeta}\ar@{->}[dll]_{\alpha_{1}}\ar@{->}[ull]_{\pi_{1}}\ar@{->}[d]^{\pi_3}\ar@{->}[drr]^{\alpha_{2}}\ar@{->}[urr]^{\pi_{2}}&&\\%
\mathbb{P}^{1}\times\mathbb{P}^{2}\ar@{->}[rr]_{\omega_{1}}\ar@{->}[uu]^{\eta_{1}}&&\mathbb{P}^{2}&&\mathbb{P}^{1}\times\mathbb{P}^{2}\ar@{->}[ll]^{\omega_{2}}\ar@{->}[uu]_{\eta_{2}}}
$$
where $\omega_{i}$, $\eta_{i}$ and $\upsilon_{i}$ are natural
projections, $\zeta$ is a $\mathbb{P}^{1}$-bundle, and
$\alpha_{i}$~is~a~birational morphism that contracts a surface
$E_{i}\subset X$ to a smooth curve
$C_{i}\subset\mathbb{P}^{1}\times\mathbb{P}^{2}$ such that
$\omega_{1}(C_{1})=\omega_{2}(C_{2})$ is a (irreducible) conic.

Note that $E_{2}\sim H_{1}+H_{3}-H_{2}$ and $E_{1}\sim
H_{2}+H_{3}-H_{1}$.

We suppose that $\mathrm{lct}(X)<1/2$. Then there exists an
effective $\mathbb{Q}$-divisor $D\qlineq -K_{X}$ such that the log
pair $(X,\lambda D)$ is not log canonical for some positive
rational number $\lambda<1/2$.

Suppose that the set $\mathbb{LCS}(X, \lambda D)$ contains a
(irreducible) surface $S\subset X$. Put
$$
D=\mu S+\Omega,
$$
where $\mu\ge 1/\lambda$ and $\Omega$ is an effective
$\mathbb{Q}$-divisor such that
$S\not\subset\mathrm{Supp}(\Omega)$. Then
$$
2=D\cdot\Gamma=\mu S\cdot\Gamma+\Omega\cdot\Gamma\geqslant\mu S\cdot\Gamma,%
$$
where $\Gamma\cong\mathbb{P}^{1}$ is a general fiber of $\zeta$.
Hence $S\cdot\Gamma=0$, which implies that $E_{2}\ne S\ne E_{1}$.
One also has
$$
2=D\cdot\Delta=\mu S\cdot\Delta+\Omega\cdot\Delta\geqslant\mu S\cdot\Delta,%
$$
where $\Delta\cong\mathbb{P}^{1}$ is a general fiber of the conic
bundle $\pi_{2}$. Hence $S\cdot\Delta=0$, which implies that
$$
S\in\Big|\pi_{3}^{*}\Big(\mathcal{O}_{\mathbb{P}^{2}}\big(m\big)\Big)\Big|
$$
for some $m\in\mathbb{Z}_{>0}$, because $E_{2}\ne S\ne E_{1}$ and
$S$ is an irreducible surface. Then
$$
0=S\cdot\Gamma=m\ne 0,
$$
which is a contradiction. Thus, we see that the set
$\mathbb{LCS}(X, \lambda D)$ contains no surfaces.

Applying Theorem~\ref{theorem:Hwang} to $\zeta$ and using
Theorem~\ref{theorem:connectedness}, we see that
$$
\mathrm{LCS}\Big(X,\ \lambda D\Big)=F\cong\mathbb{P}^{1},
$$
where $F$ is a fiber of the $\mathbb{P}^{1}$-bundle $\zeta$.
Applying Theorem~\ref{theorem:Hwang} to the conic bundle
$\pi_3$,~we~see~that every fiber of the conic bundle $\pi_3$ that
intersects $F$ must be reducible. This means that
$$
\pi_3\big(F\big)\subset\omega_{1}\big(C_{1}\big)=
\omega_{2}\big(C_{2}\big)\subset\P^2,
$$
which is impossible, because $\pi_{3}(F)$ is a line, and
$\omega_{1}(C_{1})=\omega_{2}(C_{2})$~is~an irreducible conic.
\end{proof}

\begin{lemma}
\label{lemma:3-18} Suppose that $\gimel(X)=3.18$. Then
$\mathrm{lct}(X)=1/3$.
\end{lemma}

\begin{proof}
Let $Q\subset\mathbb{P}^{4}$ be a smooth quadric hypersurface,
$C\subset Q$ an irreducible conic, and  $O\in C$ a point. Then there is a
commutative diagram
$$
\xymatrix{
Y\ar@/^1pc/@{->}[drrrrrr]^{\eta}\ar@/_1pc/@{->}[dddrr]_{\tau}&&&&&&&&\\
&&X\ar@{->}[d]_{\gamma}\ar@{->}[ull]_{\sigma}\ar@{->}[rr]^{\beta}&&V\ar@{->}[d]^{\alpha}\ar@{->}[rr]^{\omega}&&\mathbb{P}^{1}\\
&&U\ar@{->}[d]_{\upsilon}\ar@{->}[rr]_{\zeta}&&Q\ar@{-->}[dll]^{\psi}\ar@{-->}[urr]_{\phi}&&\\
&&\P^3\ar@/_2.5pc/@{-->}[rrrruu]_{\xi}&&&&}
$$
where $\zeta$ is a blow up of the point $O$, the morphisms
$\alpha$ and $\gamma$ are blow ups of the conic $C$ and its proper
transform, respectively, $\beta$ is a blow up of the fiber of the
morphism $\alpha$ over~the~point~$O$, the map $\psi$ is a
projection from~$O$, the map $\phi$ is induced by the projection
from the two-dimen\-si\-onal linear subspace that contains the
conic $C$, the morphism $\tau$ is a blow up of the line $\psi(C)$,
the morphism $\upsilon$ is a blow up of an irreducible conic
$Z\subset\mathbb{P}^{3}$ such that
$$
\psi\big(C\big)\cap Z\ne\varnothing,
$$
and $Z$ and $\psi(C)$ are not contained in one plane, the morphism
$\sigma$ is a blow up of the proper transform of the conic $Z$,
the map $\xi$ is a projection from $\psi(C)$, the morphism $\eta$
is a $\mathbb{P}^{1}$-bundle, and $\omega$ is a fibration into
quadric surfaces.

Let $\bar{H}$ be a general fiber of $\omega\circ\beta$. Then
$\bar{H}$ is a del Pezzo surface such that $K_{\bar{H}}^{2}=7$,
and
$$
-K_{X}\sim 3\bar{H}+2E+G,
$$
where $G$ and $E$ are the exceptional divisors of $\beta$ and
$\gamma$, respectively. In particular, $\mathrm{lct}(X)\leqslant
1/3$.

We suppose that $\mathrm{lct}(X)<1/3$. Then there exists an
effective $\mathbb{Q}$-divisor $D\qlineq -K_{X}$ such that the log
pair $(X,\lambda D)$ is not log canonical for some positive
rational number $\lambda<1/3$. Note that
$$
\varnothing\ne\mathrm{LCS}\Big(X,\lambda D\Big)\subseteq G,
$$
since $\mathrm{lct}(V)=1/3$ by Lemma~\ref{lemma:2-29} and
$\beta(D)\qlineq -K_{V}$.

Applying Lemma~\ref{lemma:Hwang} to the del Pezzo fibration
$\omega\circ\beta$ and using Theorem~\ref{theorem:connectedness},
we see that there is a unique singular fiber $S$ of the fibration
$\omega\circ\beta$ such that
$$
\varnothing\ne\mathrm{LCS}\Big(X,\lambda D\Big)\subseteq G\cap S,
$$
because the equality $\mathrm{lct}(\bar{H})=1/3$ holds (see
Example~\ref{example:del-Pezzos}).

Let $P\in G\cap S$ be an arbitrary point of the locus
$\mathrm{LCS}(X,\lambda D)$. Put
$$
D=\mu S+\Omega,
$$
where $\Omega$ is an effective $\mathbb{Q}$-divisor such that
$S\not\subset\mathrm{Supp}(\Omega)$. Then
$$
P\in\mathrm{LCS}\Big(S,\ \lambda \Omega\Big\vert_{S}\Big)%
$$
by Theorem~\ref{theorem:adjunction}.

We can identify the surface $\beta(S)$ with an irreducible quadric
cone in $\mathbb{P}^3$. Note that $G\cap S$ is an exceptional
curve on $S$, so that there is a unique ruling of the cone
$\beta(S)$ intersecting the curve $\beta(G)$. Let $L\subset S$ be
a proper transform of this ruling. Then $L\cap G\ne\varnothing$
(moreover, $|L\cap G\cap S|=1$), while $L\cap E=\varnothing$.
Hence $P=L\cap G$ by Lemma~\ref{lemma:quadric-cone}. We see that
$\mathrm{LCS}(X,\lambda D)=P$. One has
$$
\bar{H}\cup P\subseteq\mathrm{LCS}\left(X,\ \lambda D+\bar{H}+\frac{2}{3}E\right)\subseteq\bar{H}\cup P\cup E,%
$$
because $\bar{H}$ is a general fiber of the fibration
$\omega\circ\beta$. Therefore, the locus
$$
\varnothing\neq\mathrm{LCS}
\left(X,\ \lambda D+\bar{H}+\frac{2}{3}E\right)\subset X%
$$
must be disconnected, because $P\not\in\bar{H}$ and $P\not\in E$.
But
$$
-\left(K_{X}+\lambda
D+\bar{H}+\frac{2}{3}E\right)\qlineq\bar{H}+\frac{2}{3}\Big(E+G\Big)+\big(\lambda-1/3)K_{X}
$$
is an ample divisor, which is impossible by
Theorem~\ref{theorem:connectedness}.
\end{proof}

The proof of Lemma~\ref{lemma:3-18} implies the following
corollary.

\begin{corollary}
\label{corollary:4-4-and-5-1} Suppose that $\gimel(X)=4.4$ or
$\gimel(X)=5.1$. Then $\mathrm{lct}(X)=1/3$.
\end{corollary}

\begin{lemma}
\label{lemma:3-19} Suppose that $\gimel(X)=3.19$. Then
$\mathrm{lct}(X)=1/3$.
\end{lemma}

\begin{proof}
Let $Q\subset\mathbb{P}^4$ be a smooth quadric, and let
$L\subset\mathbb{P}^{4}$ be a line such that
$$
L\cap Q=P_{1}\cup P_{2},
$$
where $P_{1}$ and $P_{2}$ are different points. Let $\eta\colon
Q\dasharrow\mathbb{P}^{2}$ be the projection from $L$. The diagram
$$
\xymatrix{
&&X\ar@{->}[ddl]^{\beta_{2}}\ar@{->}[ddr]_{\beta_{1}}\ar@{->}[dddll]_{\gamma_{1}}\ar@{->}[dddrr]^{\gamma_{2}}&&\\
&&&&&&&\\
&U_{1}\ar@{->}[d]_{\delta_{1}}\ar@{->}[r]^{\alpha_{1}}&Q\ar@{-->}[dd]^{\eta}\ar@{-->}[dl]^{\xi_{1}}\ar@{-->}[dr]_{\xi_{2}}&U_{2}\ar@{->}[d]^{\delta_{2}}\ar@{->}[l]_{\alpha_{2}}&&\\
\mathbb{P}\Big(\mathcal{O}_{\mathbb{P}^{2}}\oplus\mathcal{O}_{\mathbb{P}^{2}}\big(1\big)\Big)\ar@{->}[r]^(0.7){\pi_{2}}\ar@{->}[drr]_{\omega_{2}}&\mathbb{P}^{3}\ar@{-->}[dr]^{\zeta_{2}}&&\mathbb{P}^{3}\ar@{-->}[dl]_{\zeta_{1}}&\mathbb{P}\Big(\mathcal{O}_{\mathbb{P}^{2}}\oplus\mathcal{O}_{\mathbb{P}^{2}}\big(1\big)\ar@{->}[l]_(0.7){\pi_{1}}\ar@{->}[dll]^{\omega_{1}}\Big)\\
&&\mathbb{P}^{2}&&}
$$
commutes, where $\alpha_{i}$ is a blow up of the point $P_{i}$,
the morphism $\beta_{i}$ contracts a surface
$$
\mathbb{P}^{2}\cong E_{i}\subset X
$$
to the point that dominates $P_{i}\in Q$, the map $\xi_{i}$ is a
projection from $P_{i}$, the map $\zeta_{i}$ is a projection from
the image of $P_{i}$, the morphism $\delta_{i}$ is a contraction
of a surface
$$
\mathbb{F}_{2}\cong G_{i}\subset U_{i}
$$
to a conic $C_{i}\subset\mathbb{P}^{3}$, the morphism $\pi_{i}$ is
a blow up of the image of $P_{i}$, the morphism $\gamma_{i}$
contracts the proper transform of $G_{i}$ to the proper transform
of $C_{i}$, and $\omega_{i}$ is a natural projection.

The map $\gamma_{1}\circ\gamma_{2}^{-1}$ is an elementary
transformation of a conic bundle (see \cite{Sa80}), and
$$
\delta_{1}\circ\beta_{2}\big(E_{1}\big)\subset\mathbb{P}^{3}\supset\delta_{2}\circ\beta_{1}\big(E_{2}\big)
$$
are planes that contain the conics $C_{1}$ and $C_{2}$,
respectively.

Let $H$ be a general hyperplane section of $Q$ such that $P_1\in
H\ni P_{2}$. Then
$$
-K_{X}\sim 3\bar{H}+E_{1}+E_{2},
$$
where $\bar{H}$ is the proper transform of $H$ on the threefold
$X$. We see that $\mathrm{lct}(X)\leqslant 1/3$.

We suppose that $\mathrm{lct}(X)<1/3$. Then there exists an
effective $\mathbb{Q}$-divisor $D\qlineq -K_{X}$ such that the log
pair $(X,\lambda D)$ is not log canonical for some $\lambda<1/3$.
Note that
$$
\varnothing\ne\mathrm{LCS}\Big(X,\ \lambda D\Big)\subseteq E_{1}\cup E_{2},%
$$
because $\mathrm{lct}(Q)=1/3$. By
Theorem~\ref{theorem:connectedness}, we may assume that
$$
\varnothing\ne\mathrm{LCS}\Big(X,\ \lambda D\Big)\subseteq E_{1}.%
$$

Let $\bar{G}_{2}\subset X$ be a proper transform of $G_{2}$. Then
$\bar{G}_{2}\cap E_{1}=\varnothing$, because
$\alpha_{2}(G_{2})\subset Q$~is~a~quadric cone whose vertex is the
point $P_{2}$, and the line $L$ is not contained in $Q$. Hence
$$
\varnothing\ne\mathrm{LCS}\Big(\mathbb{P}\Big(\mathcal{O}_{\mathbb{P}^{2}}\oplus\mathcal{O}_{\mathbb{P}^{2}}\big(1\big)\Big),\ \lambda\gamma_{2}\big(D\big)\Big)\subseteq \gamma_{2}\big(E_{1}\big),%
$$
where $\gamma_{2}(E_{1})$ is a section of $\omega_{1}$. Applying
Theorem~\ref{theorem:Hwang} to $\omega_{1}$, we obtain a
contradiction.
\end{proof}

\begin{lemma}
\label{lemma:3-20} Suppose that $\gimel(X)=3.20$. Then
$\mathrm{lct}(X)=1/3$.
\end{lemma}

\begin{proof}
Let $Q\subset\mathbb{P}^{4}$ be a smooth quadric threefold, and
let
$$
W\subset\mathbb{P}^{2}\times\mathbb{P}^{2}
$$
be a smooth divisor of bi-degree $(1,1)$. Let
$L_{1}\subset Q\supset L_{2}$ be lines such that~$L_{1}\cap
L_{2}=\varnothing$; then there exists a commutative diagram
$$
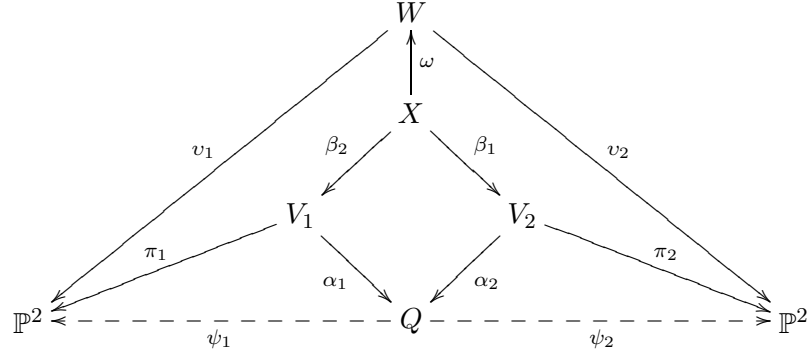
\xymatrix{
&&&&W\ar@{->}[llllddd]_{\upsilon_{1}}\ar@{->}[rrrrddd]^{\upsilon_{2}}&&&&\\
&&&&X\ar@{->}[ld]_{\beta_{2}}\ar@{->}[rd]^{\beta_{1}}\ar@{->}[u]_{\omega}&&&&\\
&&&V_{1}\ar@{->}[llld]_{\pi_{1}}\ar@{->}[rd]_{\alpha_{1}}&&V_{2}\ar@{->}[ld]^{\alpha_{2}}\ar@{->}[rrrd]^{\pi_{2}}&&&\\
\P^2&&&&Q\ar@{-->}[rrrr]_{\psi_{2}}\ar@{-->}[llll]^{\psi_{1}}&&&&\P^2}
$$
where the morphisms $\alpha_{i}$ and $\beta_{i}$ are blow ups of
the line $L_{i}$ and its proper transform, respectively, the
morphism $\omega$ is a blow up of a smooth curve $C\subset W$ of bi-degree
$(1, 1)$, the morphisms
$\upsilon_{i}$ and $\pi_{i}$ are natural $\mathbb{P}^{1}$-bundles,
and the map $\psi_{i}$ is a linear projection from the line
$L_{i}$.

Let $\bar{H}$ be the exceptional divisor of $\omega$, and let
$E_i$ be the exceptional divisor of $\beta_{i}$. Then
$$
-K_{X}\sim 3\bar{H}+2E_1+2E_2,
$$
because $\alpha_{2}\circ\beta_{1}(\bar{H})\subset Q$ is a
hyperplane section that contains $L_{1}$ and $L_{2}$. Hence
$\mathrm{lct}(X)\leqslant 1/3$.

We suppose that $\mathrm{lct}(X)<1/3$. Then there exists an
effective $\Q$-divisor $D\qlineq -K_{X}$ such that the log pair
$(X,\lambda D)$ is not log canonical for some positive rational
number $\lambda<1/3$. Note that
$$
\varnothing\ne\mathrm{LCS}\Big(X,\ \lambda D\Big)\subseteq E_1\cap E_{2}\cap\bar{H}=\varnothing,%
$$
because $\mathrm{lct}(V_{1})=\mathrm{lct}(V_{2})=1/3$ by
Lemma~\ref{lemma:2-31} and $\mathrm{lct}(W)=1/2$ by
Theorem~\ref{theorem:del-Pezzo}, which gives a contradiction.
\end{proof}

\begin{lemma}\label{lemma:3-21}
Suppose that $\gimel(X)=3.21$. Then $\mathrm{lct}(X)=1/3$.
\end{lemma}

\begin{proof}
Let $\pi_{1}\colon \P^1\times\P^2\to\P^1$ and $\pi_{2}\colon
\P^1\times\P^2\to\P^2$ be natural projections. There is a morphism
$$
\alpha\colon X\longrightarrow\P^1\times\P^2
$$
that contracts a surface $E$ to a curve $C$ such that
$\pi_{1}^{*}(\mathcal{O}_{\mathbb{P}^{1}}(1))\cdot C=2$ and
$\pi_{2}^{*}(\mathcal{O}_{\mathbb{P}^{2}}(1))\cdot C=1$.

The curve $\pi_{2}(C)\subset\mathbb{P}^{2}$ is a line. Therefore,
there is a unique surface
$$
H_{2}\in\Big|\pi_{2}^{*}\Big(\mathcal{O}_{\mathbb{P}^{2}}\big(1\big)\Big)\Big|
$$
such that $C\subset H_{2}$. Let $H_{1}$ be a fiber of the
$\mathbb{P}^{2}$-bundle $\pi_{1}$. Then
$$
-K_{X}\sim 2\bar{H_1}+3\bar{H}_{2}+2E,
$$
where $\bar{H}_{i}\subset X$ is a proper transform of the surface
$H_{i}$. In particular, $\lct(X)\leqslant 1/3$.

We suppose that $\mathrm{lct}(X)<1/3$. Then there exists an
effective $\Q$-divisor $D\qlineq -K_{X}$ such that the log pair
$(X,\lambda D)$ is not log canonical for some rational
$\lambda<1/3$. Note that
$$
\LCS\Big(X,\ \lambda D\Big)\subseteq E,
$$
because $\mathrm{lct}(\P^1\times\P^2)=1/3$ by
Lemma~\ref{lemma:lct-P1-product}. There is a commutative diagram
$$
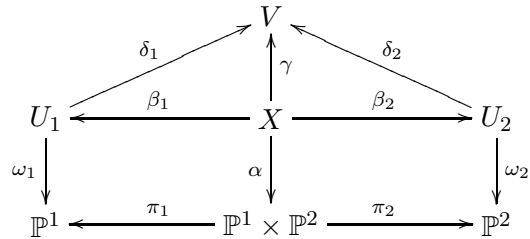
\xymatrix{
&&V&&\\%
U_{1}\ar@{->}[urr]^{\delta_{1}}\ar@{->}[d]_{\omega_{1}}&&X\ar@{->}[u]_{\gamma}\ar@{->}[d]_{\alpha}\ar@{->}[ll]_{\beta_{1}}\ar@{->}[rr]^{\beta_{2}}&&U_{2}\ar@{->}[d]^{\omega_{2}}\ar@{->}[ull]_{\delta_{2}}\\%
\mathbb{P}^{1}&&\mathbb{P}^{1}\times\mathbb{P}^{2}\ar@{->}[ll]_{\pi_{1}}\ar@{->}[rr]^{\pi_{2}}&&\mathbb{P}^{2}}
$$
where $V$ is a Fano threefold of index $2$ with one ordinary
double point $O\in V$ such that $-K_{V}^{3}=40$, the birational
morphism $\beta_{i}$ is a contraction of the surface
$\bar{H}_{2}\cong\mathbb{P}^{1}\times\mathbb{P}^{1}$ to a smooth
rational curve, the~morphism $\delta_{i}$ contracts the curve
$\beta_{i}(\bar{H}_{2})$ to the point $O\in V$ such that the
rational~map
$$
\delta_{2}\circ\delta_{1}^{-1}\colon U_{1}\dasharrow U_{2}
$$
is a standard flop in $\beta_{1}(\bar{H}_{2})\cong\mathbb{P}^{1}$,
the morphism $\omega_{1}$ is a fibration whose general fiber is
$\mathbb{P}^{1}\times\mathbb{P}^{1}$, the morphism $\omega_{2}$ is
a $\mathbb{P}^{1}$-bundle, and $\gamma$ is a birational morphism
such that $\gamma(\bar{H}_{2})=O\in V$.

The variety $V$ is a section of
$\mathrm{Gr}(2,5)\subset\mathbb{P}^9$ by a linear subspace of
codimension $3$. One has
$$
-K_{V}\sim 2\Big(\gamma\big(\bar{H_1}\big)+\gamma\big(E\big)\Big),
$$
and the divisor $\gamma(\bar{H_1})+\gamma(E)$ is very ample. There
is a commutative diagram
$$
\xymatrix{
X\ar@{->}[d]_{\alpha}\ar@{->}[rr]^{\gamma}&&V\ar@{^{(}->}[rr]_{\zeta}&&\mathbb{P}^{6}\ar@{-->}[d]^{\xi}\\%
\mathbb{P}^{1}\times\mathbb{P}^{2}\ar@{^{(}->}[rrrr]_{\eta}&&&&\mathbb{P}^{5}}
$$
such that the embedding $\zeta$ is given by the linear system
$|\gamma(\bar{H_1})+\gamma(E)|$, the map $\xi$ is a linear
projection from the point $O$, the embedding $\eta$ is given by
the linear system $|H_{1}+H_{2}|$.

It follows from \cite[Theorem~3.6]{JaPe06} (see
\cite[Theorem~3.13]{JaPeRa07}) that
$U_{2}\cong\mathbb{P}\big(\mathcal{E}\big)$,
where where $\mathcal{E}$ is a stable rank two vector bundle  on
$\mathbb{P}^{2}$ such that the sequence
$$
0\longrightarrow\mathcal{O}_{\mathbb{P}_{2}}\longrightarrow\mathcal{E}\otimes\mathcal{O}_{\mathbb{P}_{2}}\big(1\big)\longrightarrow \mathcal{I}\otimes\mathcal{O}_{\mathbb{P}_{2}}\big(1\big)\longrightarrow 0%
$$
is exact, where $\mathcal{I}$ is an ideal sheaf of two general
points on $\mathbb{P}^{2}$. One has $c_{1}(\mathcal{E})=-1$ and
$c_{1}(\mathcal{E})=2$. It follows from \cite[Theorem~3.5]{JaPe06}
that
$$
U_{1}\subset\mathbb{P}\Big(\mathcal{O}_{\mathbb{P}^{1}}\oplus\mathcal{O}_{\mathbb{P}^{1}}\big(1\big)\oplus\mathcal{O}_{\mathbb{P}^{1}}\big(1\big)\oplus\mathcal{O}_{\mathbb{P}^{1}}\big(1\big)\Big),
$$
and $U_{1}\in|2T-F|$, where $T$ is a tautological bundle on
$\mathbb{P}(\mathcal{O}_{\mathbb{P}^{1}}\oplus\mathcal{O}_{\mathbb{P}^{1}}(1)\oplus\mathcal{O}_{\mathbb{P}^{1}}(1)\oplus\mathcal{O}_{\mathbb{P}^{1}}(1))$,
and $F$ is a fiber of the projection
$\mathbb{P}(\mathcal{O}_{\mathbb{P}^{1}}\oplus\mathcal{O}_{\mathbb{P}^{1}}(1)\oplus\mathcal{O}_{\mathbb{P}^{1}}(1)\oplus\mathcal{O}_{\mathbb{P}^{1}}(1))\to\mathbb{P}^{1}$.

Either $\bar{H}_{1}$ is a smooth del Pezzo surface such that
$K_{\bar{H}_{1}}^{2}=7$, or
$$
\big|H_{1}\cap C\big|=1,
$$
because $H_{1}\cdot C=2$. Applying Lemma~\ref{lemma:Hwang} to the
morphism $\omega_{1}\circ\beta_{1}$ and the surface $\bar{H}_{1}$,
we see that
\begin{itemize}
\item either $|H_{1}\cap C|=1$,%
\item or $H_{1}\cap\LCS(X,\lambda D)=\varnothing$,
\end{itemize}
because $\mathrm{lct}(\bar{H}_{1})=1/3$ if $\bar{H}_{1}$ is
smooth. So, there is a fiber $L$ of the projection $E\to C$ such
that
$$
\LCS\Big(X,\ \lambda D\Big)\subseteq L
$$
by Theorem~\ref{theorem:connectedness}. Put
$\bar{C}=\bar{H}_{2}\cap E$ and $P=L\cap\bar{C}$. Applying
Theorem~\ref{theorem:Hwang} to $\omega_{2}$ and
$$
\Big(U_{2},\ \lambda\beta_{2}\big(D\big)\Big),
$$
we see that either $\LCS(X,\lambda D)=P$ or $\LCS(X,\lambda D)=L$
by Theorem~\ref{theorem:connectedness}.

Suppose that $\LCS(X,\lambda D)=L$. Then
$$
\LCS\Big(V,\ \lambda\gamma\big(D\big)\Big)=\gamma\big(L\big)\subset V\subset\mathbb{P}^{6},%
$$
where $\gamma(L)\subset V\subset\mathbb{P}^{6}$ is a line, because
$-K_{V}\cdot\gamma(L)=2$ and $-K_{V}\qlineq\gamma(D)$. We have
$\mathrm{Sing}(V)=O\in\gamma(L)$.

Let $S\subset V$ be a general hyperplane section of
$V\subset\mathbb{P}^{6}$ such that $\gamma(L)\subset S$. Then
\begin{itemize}
\item the surface $S$ is a del Pezzo surface such that $K_{S}^{2}=5$,%
\item the point $O$ is an ordinary double point of the surface $S$,%
\item the surface $S$ is smooth outside of the point $O\in\gamma(L)$,%
\item the equivalence $K_{S}\sim\mathcal{O}_{\mathbb{P}^{6}}(1)\vert_{S}$ holds,%
\end{itemize}
which implies that $S$ contains finitely many lines that intersect
the line $\gamma(L)$.

Let $H\subset V$ be a general hyperplane section of
$V\subset\mathbb{P}^{6}$. Put $Q=\gamma(L)\cap H$. Then
$$
\LCS\Big(H,\ \lambda\gamma\big(D\big)\Big\vert_{H}\Big)=Q,%
$$
by Remark~\ref{remark:hyperplane-reduction}, which contradicts
Lemma~\ref{lemma:del-Pezzo-quintic}, because $\lambda<1/3$.

Thus, we see that $\LCS(X,\lambda D)=P\in\bar{C}$. Let $F_{1}$ be
a general fiber of $\pi_{1}$. Then
$$
F_{1}\cap C=P_{1}\cup P_{2}\not\ni\alpha\big(P\big),
$$
where $P_{1}\ne P_{2}$ are two points of the curve $C$. One has
$$
P_{1}\cup P_{2}\subset H_{2}\cap F_{1},
$$
because $C\subset H_{2}$. Let $Z$ be a general line in
$F_{1}\cong\mathbb{P}^{2}$ such that $P_{1}\in Z$. Then there is a
surface
$$
F_{2}\in\Big|\pi_{2}^{*}\Big(\mathcal{O}_{\mathbb{P}^{2}}\big(1\big)\Big)\Big|
$$
such that $Z\subset F_{2}$. Let $\bar{F}_{1}\subset
X\supset\bar{F}_{2}$ be the proper transforms of $F_{1}$ and
$F_{2}$, respectively. Then
$$
P\not\in\bar{F}_{1}\cup\bar{F}_{2}.
$$

Let $\bar{Z}\subset X$ be the proper transform of the curve $Z$.
Then $-K_{X}\cdot\bar{Z}=2$ and
$$
\bar{Z}\subset\bar{F}_{1}\cap\bar{F}_{2},
$$
but $\bar{Z}\cap\bar{H}_{2}=\varnothing$. Thus, the curve
$\gamma(\bar{Z})$ is a line on $V\subset\mathbb{P}^{6}$ such that
$\mathrm{Sing}(V)=O\not\in\gamma(\bar{Z})$.

Let $T$ be a general hyperplane section of the threefold
$V\subset\mathbb{P}^{6}$ such that $\gamma(\bar{Z})\subset T$.
Then
$$
\bar{T}\sim 2\bar{H}_{2}+\bar{H}_{1}+E\sim 2\bar{H}_{2}+\bar{F}_{1}+E\sim 2\bar{F}_{2}+\bar{F}_{1}-E,%
$$
where $\bar{T}$ is the proper transform of the surface $T$ on the
threefold $X$. Hence
$$
\bar{F}_{1}+\bar{F}_{2}+\bar{T}\sim 3\bar{F}_{2}+2\bar{F}_{1}-E\sim 2\bar{H}_{2}+2\bar{H}_{1}+2E\sim -K_{X},%
$$
and applying Theorem~\ref{theorem:connectedness}, we see that the
locus
$$
P\cup\bar{Z}=\LCS\left(X,\ \lambda D+\frac{2}{3}\Big(\bar{F}_{1}+\bar{F}_{2}+\bar{T}\Big)\right)%
$$
must be connected. But $P\not\in\bar{Z}$, which is a
contradiction.
\end{proof}

\begin{lemma}\label{lemma:3-22}
Suppose that $\gimel(X)=3.22$. Then $\mathrm{lct}(X)=1/3$.
\end{lemma}

\begin{proof}
Let $\pi_{1}\colon \P^1\times\P^2\to\P^1$ and $\pi_{2}\colon
\P^1\times\P^2\to\P^2$ be natural projections. There is a morphism
$$
\alpha\colon X\longrightarrow\P^1\times\P^2
$$
that contracts a surface $E$ to a curve $C$ contained in a fiber
$H_{1}$ of $\pi_1$ such that $\pi_{2}(C)$ is a conic.

We have $E\cong\mathbb{F}_{2}$. Let $H_2$ be a general surface in
$|\pi_{2}^{*}(\mathcal{O}_{\mathbb{P}^{2}}(1))|$. The equivalence
$$
-K_{X}\sim 2\bar{H_1}+3\bar{H_2}+E
$$
holds, where $\bar{H}_{i}\subset X$ is a proper transform of the
surface $H_{i}$. Hence $\lct(X)\leqslant 1/3$.

We suppose that $\mathrm{lct}(X)<1/3$. Then there exists an
effective $\Q$-divisor $D\qlineq -K_{X}$ such that the log pair
$(X,\lambda D)$ is not log canonical for some rational
$\lambda<1/3$. Note that
$$
\LCS\Big(X,\lambda D\Big)\subseteq E,
$$
since $\mathrm{lct}(\P^1\times\P^2)=1/3$ by
Lemma~\ref{lemma:lct-P1-product}.

Let $Q$ be the unique surface in
$|\pi_{2}^{*}(\mathcal{O}_{\mathbb{P}^{2}}(2))|$ such that
$C\subset Q$, and let $\bar{Q}\subset X$ be the proper transform
of the surface $Q$. Then $\bar{Q}\cap \bar{H_1}=\varnothing$, and
there is a commutative diagram
$$
\xymatrix{
X\ar@{->}[d]_{\alpha}\ar@{->}[rr]^{\beta}&&\mathbb{P}\Big(\mathcal{O}_{\mathbb{P}^2}\oplus\mathcal{O}_{\mathbb{P}^2}\big(2\big)\Big)\ar@{->}[d]^{\phi}\ar@{->}[drr]^{\gamma}&&\\%
\mathbb{P}^{1}\times\mathbb{P}^{2}\ar@{->}^{\pi_{2}}[rr]&&\mathbb{P}_{2}&&\mathbb{P}\big(1,1,1,2\big)\ar@{-->}[ll]^{\psi}}
$$
such that $\beta$ is a contraction of $\bar{Q}$ to a curve,
$\gamma$ is a contraction of the surface
$\beta(\bar{H}_{1})$~to~a~point, the~morphism $\phi$ is a natural
$\P^1$-bundle, and the map $\psi$ is a natural projection. One has
$$
\gamma\circ\beta\big(D\big)\qlineq
\frac{5\gamma\circ\beta\big(E\big)}{2}\qlineq
 -K_{\mathbb{P}(1,1,1,2)}\qlineq\mathcal{O}_{\mathbb{P}(1,1,1,2)}\big(5\big),%
$$
which implies that $E\not\subseteq\mathrm{LCS}(X,\lambda D)$,
because $\lambda<1/3$.

Applying Theorem~\ref{theorem:Hwang} to $\phi$, we see that there
is a fiber $F$ of the projection $E\to C$ such that
$$
\varnothing\ne \LCS\Big(X,\ \lambda D\Big)\subseteq \Big(E\cap \bar{Q}\Big)\cup F,%
$$
including the possibility that  $\LCS(X, \lambda D)\subset
E\cap\bar{Q}$.

Suppose that $\LCS(X, \lambda D)\subset E\cap\bar{Q}$. Let
$M\subset\P^1\times\P^2$ be a general surface in $|H_1+H_2|$, and
let $\bar{M}\subset X$ be the proper transform of the surface $M$.
Then
$$
\bar{M}\cap\bar{H_1}=L,
$$
where $L$ is a line on $\bar{H_1}\cong\P^2$. Let $R$ be the unique
surface in $|\pi_{2}^{*}(\mathcal{O}_{\mathbb{P}^{2}}(1))|$ such
that $\alpha(L)\subset R$, and let $\bar{R}$ be a proper transform
of the surface $R$ on the threefold $X$. Then
$$
\LCS\Big(X, \lambda D\Big)\cup L\subseteq\LCS\left(X, \lambda
 D+\frac{2}{3}\Big(\bar{M}+\bar{H_1}+\bar{R}+\bar{H_2}\Big)\right)\subseteq\LCS\Big(X, \lambda D\Big)\cup L\cup \bar{H}_{1},%
$$
but $L\cap E\cap\bar{Q}=\bar{Q}\cap\bar{H}_{1}=\varnothing$ and
$-K_{X}\sim\bar{M}+\bar{H_1}+\bar{R}+\bar{H_2}$, which contradicts
Theorem~\ref{theorem:connectedness}.

Therefore, we see that $F\subseteq\mathrm{LCS}(X,\lambda D)$. Put
$\breve{F}=\gamma\circ\beta(F)$ and
$\breve{D}=\gamma\circ\beta(D)$. Then
$$
\breve{F}\subseteq\LCS\Big(\mathbb{P}\big(1,1,1,2\big),\ \lambda
\breve{D}\Big)\subseteq
 \breve{C}\cup\breve{F},%
$$
where
$\breve{C}=\gamma\circ\beta(\bar{Q})\subset\mathbb{P}(1,1,1,2)$ is
a curve such that $\psi(\breve{C})=\pi_{2}(C)$.

Let $S$ be a general surface in
$|\mathcal{O}_{\mathbb{P}(1,1,1,2)}(2)|$. Then
$S\cong\mathbb{P}^{2}$ and
$$
\breve{F}\cap S\subseteq\LCS\Big(S,\ \lambda
\breve{D}\Big\vert_{S}\Big)\subseteq
 \Big(\breve{C}\cup\breve{F}\Big)\cap S;%
$$
but $3D\vert_{S}\qlineq -5K_{S}$, which is impossible by
Lemma~\ref{lemma:plane}.
\end{proof}

\begin{lemma}
\label{lemma:3-23} Suppose that  $\gimel(X)=3.23$. Then
$\mathrm{lct}(X)=1/4$.
\end{lemma}
\begin{proof}
Let $O\in\mathbb{P}^{3}$ be a point, let $C\subset\mathbb{P}^{3}$
be a conic such that~$O\in C$, let $\Pi\subset\mathbb{P}^{3}$ be a
unique plane such that $C\subset\Pi$, and let
$Q\subset\mathbb{P}^{4}$ be a smooth quadric threefold.~Then the
diagram
$$
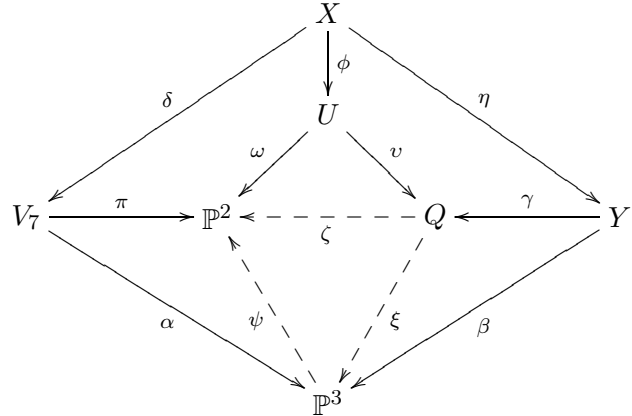
\xymatrix{
&&&X\ar@{->}[llldd]_{\delta}\ar@{->}[d]^{\phi}\ar@{->}[rrrdd]^{\eta}&&&\\
&&&U\ar@{->}[ld]_{\omega}\ar@{->}[rd]^{\upsilon}&&&\\
V_{7}\ar@{->}[ddrrr]_{\alpha}\ar@{->}[rr]^{\pi}&&\mathbb{P}^{2}&&Q\ar@{-->}[ll]^{\zeta}\ar@{-->}[ddl]^{\xi}&&Y\ar@{->}[ddlll]^{\beta}\ar@{->}[ll]_{\gamma}\\
&&&&&&\\
&&&\mathbb{P}^{3}\ar@{-->}[uul]^{\psi}&&&}
$$
commutes, where we have the following notation:
\begin{itemize}
\item the morphism $\alpha$ is a blow up of the point $O$ with an
exceptional divisor $E$;
\item the morphism $\pi$ is a natural $\mathbb{P}^{1}$-bundle;%
\item the morphisms $\beta$ and $\delta$
are blow ups of $C$ and its proper transform, respectively;%
\item the~morphism $\gamma$ contracts the proper transform of the plane $\Pi$ to a point;%
\item the~morphism $\phi$ contracts the proper transform of the plane $\Pi$ to a curve;%
\item  the morphism $\eta$ contracts the proper transform of $E$
to a curve $L\subset Y$ such~that
$$
\gamma\big(\Pi\big)\in\gamma\big(L\big)\subset Q\subset\mathbb{P}^{4}%
$$
and $\gamma(L)$ is a line in $\P^4$;%
\item  the morphism $\omega$ is a natural $\mathbb{P}^{1}$-bundle;%
\item the morphism $\upsilon$ is a blow up of the line $\gamma(L)$;%
\item the maps $\psi$, $\xi$ and $\zeta$ are projections from $O$, $\gamma(\Pi)$ and $\gamma(L)$, respectively.%
\end{itemize}

Note that $E$ is a section of~$\pi$.

Let $\bar{\Pi}\subset X$ be a proper transform of
$\Pi\subset\P^3$. Then $\mathrm{lct}(X)\leqslant 1/4$, because
$$
-K_{X}\sim 4\bar{\Pi}+2\bar{E}+3G,
$$
where $\bar{E}$ and $G$ are exceptional surfaces of $\eta$ and
$\delta$, respectively.

We suppose that $\mathrm{lct}(X)<1/4$. Then there exists an
effective $\mathbb{Q}$-divisor $D\qlineq -K_{X}$ such that the~log
pair $(X,\lambda D)$ is not log canonical for some positive
rational number $\lambda<1/4$. Note that
$$
\varnothing\ne\mathrm{LCS}\Big(X,\ \lambda D\Big)\subseteq\bar{E}\cap\bar{\Pi}\cap G,%
$$
because
$\mathrm{lct}(V_{7})=1/4$ by Theorem~\ref{theorem:del-Pezzo},
$\mathrm{lct}(Y)=1/4$ by Lemma~\ref{lemma:2-30} and
$\mathrm{lct}(U)=1/3$ by Lemma~\ref{lemma:2-31}.

Let $R\subset\mathbb{P}^{3}$ be a general cone over $C$ whose
vertex is $P\in\mathbb{P}^{3}$, let $H_1\subset\P^3$ be a general
plane such that $O\in H_{1}\ni P$, and let $H_2\subset\P^3$ be a
general plane such that $P\in H_{2}$. Then
$$
\bar{R}\sim \big(\alpha\circ\delta\big)^*\big(R\big)-\bar{E}-G,\ \bar{H_1}\sim \big(\alpha\circ\delta\big)^*\big(H_{1}\big)-\bar{E},\ \bar{H_2}\sim\big(\alpha\circ\delta\big)^*\big(H_{2}\big),%
$$
where $\bar{R}$, $\bar{H}_1$, $\bar{H}_2$ are proper transforms of
$R$, $H_{1}$, $H_{2}$ on the threefold $X$, respectively. One has
$$
-K_{X}\sim\bar{Q}+\bar{H}_1+\bar{H}_2,
$$
but it follows from the generality of $R$, $H_{1}$, $H_{2}$ that
the locus
$$
\LCS\left(X,\ \lambda{D}+\frac{3}{4}\Big(\bar{Q}+\bar{H}_1+\bar{H}_2\Big)\right)=\mathrm{LCS}\Big(X,\ \lambda D\Big)\cup P,%
$$
is disconnected, which is impossible by
Theorem~\ref{theorem:connectedness}.
\end{proof}

\begin{lemma}
\label{lemma:3-24} Suppose that $\gimel(X)=3.24$. Then
$\lct(X)=1/3$.
\end{lemma}

\begin{proof}
Let $W$ is a divisor of bi-degree $(1,1)$ on
$\mathbb{P}^{2}\times\mathbb{P}^{2}$. There is a commutative
diagram
$$
\xymatrix{
&&X\ar@{->}[dll]_{\zeta}\ar@{->}[d]^{\pi}\ar@{->}[rr]^{\alpha}&&W\ar@{->}[d]^{\omega_{1}}\\%
\mathbb{P}^{1}&&\mathbb{F}_{1}\ar@{->}^{\xi}[ll]\ar@{->}^{\gamma}[rr]&&\mathbb{P}^{2}}
$$
where $\omega_{1}$ is a natural $\mathbb{P}^{1}$-bundle, the
morphism $\alpha$ contracts a smooth surface
$$
E\cong\mathbb{P}^{1}\times\mathbb{P}^{1}
$$
to a fiber $L$ of $\omega_{1}$, $\gamma$ is a blow up of the point
$\omega_1(L)$, the morphism $\xi$ is a $\mathbb{P}^{1}$-bundle,
and $\zeta$ is a $\mathbb{F}_{1}$-bundle.

Let $\omega_{2}\colon X\to\mathbb{P}^{2}$ be a natural
$\mathbb{P}^{1}$-bundle that is different from $\omega_{1}$. Then
there is a surface
$$
G\in\Big|\omega_{2}^{*}\Big(\mathcal{O}_{\mathbb{P}^{2}}\big(1\big)\Big)\Big|
$$
such that $L\subset G$, because $\omega_{2}(L)$ is a line. Let
$\bar{G}\subset X$ be a proper transform of $G$. Then
$$
-K_{X}\sim 2F+2\bar{G}+3E,
$$
where $E$ is the exceptional divisor of $\alpha$, and $F$ is a
fiber of $\zeta$. We see that $\lct(X)\leqslant 1/3$.

We suppose that $\mathrm{lct}(X)<1/3$. Then there exists an
effective $\mathbb{Q}$-divisor $D\qlineq -K_{X}$ such that the~log
pair $(X,\lambda D)$ is not log canonical for some positive
rational number $\lambda<1/3$. Note that
$$
\varnothing\ne\mathrm{LCS}\Big(X,\lambda D\Big)\subseteq E,
$$
since $\mathrm{lct}(W)=1/2$ by Theorem~\ref{theorem:del-Pezzo}. We
may assume that $F\cap\mathrm{LCS}(X,\lambda D)\ne\varnothing$.
Then
$$
\mathbb{F}_{1}\cong F\subseteq\mathrm{LCS}\Big(X,\lambda
D\Big)\subseteq
 E\cong\mathbb{P}^{1}\times\mathbb{P}^{1}%
$$
by Lemma~\ref{lemma:Hwang}, because $\mathrm{lct}(F)=1/3$ (see
Example~\ref{example:del-Pezzos}), which is a contradiction.
\end{proof}


\section{Fano threefolds with $\rho\ge 4$}
\label{section:rho-4}

We use the assumptions and notation introduced in
section~\ref{section:intro}.

\begin{lemma}
\label{lemma:4-1} Suppose that $\gimel(X)=4.1$. Then
$\mathrm{lct}(X)=1/2$.
\end{lemma}

\begin{proof}
The threefold $X$ is a divisor on
$\mathbb{P}^1\times\mathbb{P}^1\times\mathbb{P}^1\times\mathbb{P}^1$
of multidegree $(1, 1, 1, 1)$. Let
$$
\Big[\big(x_{1}:y_{1}\big),\ \big(x_{2}:y_{2}\big),\ \big(x_{3}:y_{3}\big),\ \big(x_{4}:y_{4}\big)\Big]%
$$
be coordinates on
$\mathbb{P}^1\times\mathbb{P}^1\times\mathbb{P}^1\times\mathbb{P}^1$.
Then $X$ is given by the equation
$$
F\big(x_{1},y_{1},x_{2},y_{2},x_{3},y_{3},x_{4},y_{4}\big)=0,
$$
where $F$ is a of multidegree $(1, 1, 1, 1)$.

Let $\pi_{1}\colon
X\to\mathbb{P}^{1}\times\mathbb{P}^{1}\times\mathbb{P}^{1}$ be a
projection given by
$$
\Big[\big(x_{1}:y_{1}\big),\ \big(x_{2}:y_{2}\big),\
\big(x_{3}:y_{3}\big),\
 \big(x_{4}:y_{4}\big)\Big]\mapsto \Big[\big(x_{2}:y_{2}\big),\ \big(x_{3}:y_{3}\big),\
 \big(x_{4}:y_{4}\big)\Big]\in \mathbb{P}^{1}\times\mathbb{P}^{1}\times\mathbb{P}^{1},%
$$
and let $\pi_{2}$, $\pi_{3}$ and $\pi_{4}\colon
X\to\mathbb{P}^{1}\times\mathbb{P}^{1}\times\mathbb{P}^{1}$ be
projections defined in a similar way. Put
$$
F=x_{1}G\big(x_{2},y_{2},x_{3},y_{3},x_{4},y_{4}\big)+y_{1}H\big(x_{2},y_{2},x_{3},y_{3},x_{4},y_{4}\big),
$$
where $G(x_{2},y_{2},x_{3},y_{3},x_{4},y_{4})$ and
$H(x_{2},y_{2},x_{3},y_{3},x_{4},y_{4})$ are multi-linear forms
that do not depend on $x_{1}$ and $y_{1}$. Then $\pi_{1}$ is a
blow up of a curve
$C_{1}\subset\mathbb{P}^{1}\times\mathbb{P}^{1}\times\mathbb{P}^{1}$
given by the equations
$$
G\big(x_{2},y_{2},x_{3},y_{3},x_{4},y_{4}\big)=H\big(x_{2},y_{2},x_{3},y_{3},x_{4},y_{4}\big)=0,
$$
which define a surface
$E_{1}\subset\mathbb{P}^1\times\mathbb{P}^1\times\mathbb{P}^1\times\mathbb{P}^1$
that is contracted by $\pi_{1}$. The equations
$$
x_{1}=H\big(x_{2},y_{2},x_{3},y_{3},x_{4},y_{4}\big)=0%
$$
define a divisor $H_{1}\subset X$ such that $-K_{X}\sim
2H_{1}+E_{1}$, which implies that $\mathrm{lct}(X)\leqslant 1/2$.

We suppose that $\mathrm{lct}(X)<1/2$. Then there exists an
effective $\mathbb{Q}$-divisor $D\qlineq -K_{X}$ such that the~log
pair $(X,\lambda D)$ is not log canonical for some positive
rational number $\lambda<1/2$.

Let $E_{2}$, $E_{3}$, $E_{4}$ be surfaces in $X$ defined in a way
similar to $E_{1}$. Then
$$
\varnothing\ne\mathrm{LCS}\Big(X,\ \lambda D\Big)\subseteq
E_{1}\cap E_{2}\cap E_{3}\cap E_{4},
$$
because
$\mathrm{lct}(\mathbb{P}^1\times\mathbb{P}^1\times\mathbb{P}^1)=1/2$
by Lemma~\ref{lemma:lct-P1-product}. But
$E_{i}\subset\mathbb{P}^1\times\mathbb{P}^1\times\mathbb{P}^1\times\mathbb{P}^1$
is given by
$$
\frac {\partial
F\big(x_{1},y_{1},x_{2},y_{2},x_{3},y_{3},x_{4},y_{4}\big)}
{\partial x_{i}}=\frac
 {\partial F\big(x_{1},y_{1},x_{2},y_{2},x_{3},y_{3},x_{4},y_{4}\big)}
 {\partial y_{i}}=0,%
$$
which implies that the intersection $E_{1}\cap E_{2}\cap E_{3}\cap
E_{4}$ is given by the equations
$$
\frac {\partial F} {\partial x_{1}}=\frac {\partial F} {\partial
y_{1}}= \frac {\partial F}
 {\partial x_{2}}=\frac {\partial F} {\partial y_{2}}=\frac {\partial F} {\partial x_{3}}=\frac
 {\partial
 F} {\partial y_{3}}=\frac {\partial F} {\partial x_{4}}=\frac {\partial F}
 {\partial y_{4}}=0.%
$$
Hence $E_{1}\cap E_{2}\cap E_{3}\cap
E_{4}=\mathrm{Sing}(X)=\varnothing$, and $\mathrm{LCS}(X,\ \lambda
D)=\varnothing$.
\end{proof}

\begin{lemma}
\label{lemma:4-2} Suppose that $\gimel(X)=4.2$. Then
$\lct(X)=1/2$.
\end{lemma}

\begin{proof}
Let $Q_{1}\subset\mathbb{P}^4\supset Q_{2}$ be quadric cones,
whose vertices are $O_{1}\in \mathbb{P}^4\ni O_{2}$, respectively.
Let
$$
O_{1}\not\in S_{1}\subset Q_{1}\subset\mathbb{P}^{4}
$$
be a hyperplane section of $Q_{1}$, and let $C_{1}\subset
|-K_{S_{1}}|$ be a smooth elliptic curve. Then~the~diagram
$$
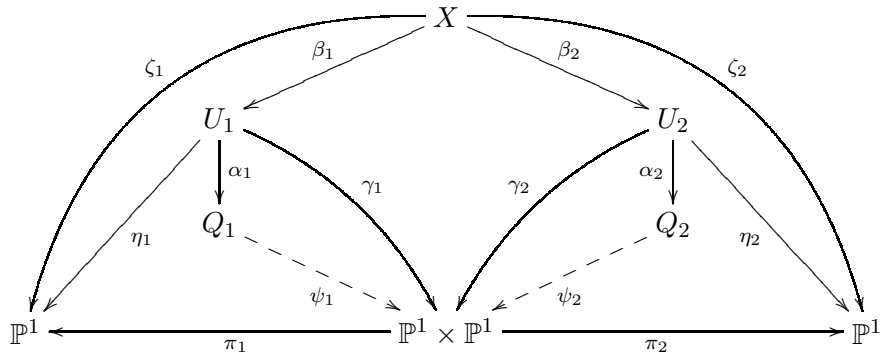
\xymatrix{
&&&&X\ar@{->}[dll]_{\beta_{1}}\ar@{->}[drr]^{\beta_{2}}\ar@/^3.5pc/@{->}[dddrrrr]^{\zeta_{2}}\ar@/_3.5pc/@{->}[dddllll]_{\zeta_{1}}&&&&\\
&&U_{1}\ar@{->}[d]^{\alpha_{1}}\ar@/^1pc/@{->}[ddrr]^{\gamma_{1}}\ar@{->}[ddll]^{\eta_{1}}&&&&U_{2}\ar@{->}[d]_{\alpha_{2}}\ar@/_1pc/@{->}[ddll]_{\gamma_{2}}\ar@{->}[ddrr]_{\eta_{2}}&&\\
&&Q_{1}\ar@{-->}[drr]_{\psi_{1}}&&&&Q_{2}\ar@{-->}[dll]^{\psi_{2}}&&\\
\mathbb{P}^{1}&&&&\mathbb{P}^{1}\times\mathbb{P}^{1}\ar@{->}[llll]^{\pi_{1}}\ar@{->}[rrrr]_{\pi_{2}}&&&&\mathbb{P}^{1}}
$$
commutes, where $\pi_{1}\ne\pi_{2}$ are natural projections, the
map $\psi_{i}$ is a projection from $O_{i}\in
Q_{i}\subset\mathbb{P}^{4}$, the~morphism $\alpha_{i}$ is a blow
up of the vertex $O_{i}$, the morphism $\beta_{i}$ contracts a
surface
$$
\mathbb{P}^{1}\times C_{1}\cong G_{i}\subset X
$$
to a curve $C_{1}\cong C_{i}\subset U_{i}$, the morphism
$\eta_{i}$ is an $\mathbb{F}_{1}$-bundle,
$\gamma_{i}$~is~a~$\mathbb{P}^{1}$-bundle, and $\zeta_{i}$ is a
fibration into del Pezzo surfaces of degree $6$ that has $4$
singular fibers.

Let $E_{i}\subset X$ be the proper transform of the exceptional
divisor of $\alpha_{i}$. Then
$$
S_{1}=\alpha_{1}\circ\beta_{1}\big(E_{2}\big)\subset Q_{1}\subset\mathbb{P}^{4}\supset Q_{2}\supset\alpha_{2}\circ\beta_{2}\big(E_{1}\big)%
$$
are hyperplane sections that contain $C_{1}$ and $C_{2}$,
respectively. It is also easy to see that
$$
\alpha_{1}\circ\beta_{1}\big(G_{2}\big)\subset Q_{1}\subset\mathbb{P}^{4}\supset Q_{2}\supset\alpha_{2}\circ\beta_{2}\big(G_{1}\big)%
$$
are the cones in $\mathbb{P}^{4}$ over the curves $C_{1}$ and
$C_{2}$, respectively.

Let $\bar{H}\subset X$ be the proper transform of a hyperplane
section of $Q_{1}\subset\mathbb{P}^{4}$ that contains
$O_{1}$.~Then
$$
-K_{X}\sim 2\bar{H}+E_{2}+E_{1},
$$
which gives $\mathrm{lct}(X)\leqslant 1/2$. Suppose that
$\mathrm{lct}(X)<1/2$. Then there is an effective
$\mathbb{Q}$-divisor
$$
D\qlineq -K_{X}\sim E_{1}+E_{2}+G_{1}+G_{2}
$$
such that the log pair $(X,\lambda D)$ is not log canonical for
some $\lambda<1/2$. Put
$$
D=\mu_{1}E_{1}+\mu_{2}E_{2}+\Omega,
$$
where $\Omega$ is an effective $\mathbb{Q}$-divisor on $X$ such
that
$$
E_{1}\not\subseteq\mathrm{Supp}\big(\Omega\big)\not\supseteq
E_{2}.
$$

Let $\Gamma$ be a general fiber of the conic bundle
$\gamma_{1}\circ\beta_{1}$. Then
$$
2=\Gamma\cdot D=\Gamma\cdot\Big(\mu_{1}E_{1}+\mu_{2}E_{2}+\Omega\Big)=\mu_{1}+\mu_{2}+\Gamma\cdot\Omega\geqslant\mu_{1}+\mu_{2},%
$$
and without loss of generality we may assume that
$\mu_{1}\leqslant\mu_{2}$. Then $\mu_{1}\leqslant 1$.

Suppose that there is a surface $S\in\mathbb{LCS}(X,\lambda D)$.
Then $S\ne E_{1}$. Moreover, we have $S\ne G_{1}$, because
$\alpha_{2}\circ\beta_{2}(G_{1})$ is a quadric surface and
$\lambda<1/2$. Hence $S\cap E_{1}\ne\varnothing$. But
$$
-\frac{1}{2}K_{E_{1}}\qlineq D\Big\vert_{E_{1}}\qlineq
-\frac{\mu_{1}}{2}K_{E_{1}}+\Omega\Big\vert_{E_{1}},%
$$
and $E_{1}\cong\mathbb{P}^{1}\times\mathbb{P}^{1}$, which is
impossible by Theorem~\ref{theorem:adjunction} and
Lemma~\ref{lemma:elliptic-times-P1}.

We see that the set $\mathbb{LCS}(X,\lambda D)$ contains no
surfaces. Let $P\in\mathrm{LCS}(X,\lambda D)$ be a point.

Suppose that $P\not\in G_{1}$. Let $Z$ be a fiber of $\gamma_{1}$
such that $\beta_{1}(P)\in Z$. Then
$$
Z\subseteq\mathrm{LCS}\Big(U_{1},\
\lambda\beta_{1}\big(D\big)\Big)
$$
by Theorem~\ref{theorem:Hwang}. Put
$\bar{E}_{1}=\beta_{1}(E_{1})$. Then we have
$$
Z\cap \bar{E}_{1}\in\mathrm{LCS}\Big(\bar{E}_{1},\ \lambda\Omega\Big\vert_{\bar{E}_{1}}\Big)%
$$
by Theorem~\ref{theorem:adjunction}, which is impossible by
Lemma~\ref{lemma:elliptic-times-P1}, because $\mu_{1}\leqslant 1$.

Thus, we see that $P\in G_{1}$. Let $F_{1}\subset X\supset F_{2}$
be fibers of $\zeta_{1}$ and $\zeta_{2}$ passing through the point
$P$. Then either $F_{1}$ or $F_{2}$ is smooth, because
$\alpha_1(P)\in C_{1}$. But
$$
\mathrm{lct}\big(F_{i}\big)=1/2
$$
in the case when $F_{i}$ is smooth (see
Example~\ref{example:del-Pezzos}), which contradicts
Lemma~\ref{lemma:Hwang}.
\end{proof}

\begin{lemma}
\label{lemma:4-3} Suppose that $\gimel(X)=4.3$. Then
$\lct(X)=1/2$.
\end{lemma}

\begin{proof}
Let $F_{1}\cong F_{2}\cong F_{3}\cong\P^1\times\P^1$ be fibers of
three different projections
$$
\mathbb{P}^1\times\mathbb{P}^1\times\mathbb{P}^1\longrightarrow\mathbb{P}^{1},
$$
respectively. There is a contraction $\alpha\colon X\to
\mathbb{P}^1\times\mathbb{P}^1\times\mathbb{P}^1$ of a surface
$E\subset X$ to a curve
$$
C\subset\mathbb{P}^1\times\mathbb{P}^1\times\mathbb{P}^1
$$
such that $C\cdot F_{1}=C\cdot F_{2}=1$ and $C\cdot F_{3}=2$.
There is a smooth surface
$$
\mathbb{P}^1\times\mathbb{P}^1\cong G\in \big|F_{1}+F_{2}\big|
$$
such that $C\subset G$. In particular, we see that
$$
-K_{X}\sim 2\bar{G}+E+\bar{F}_{3},
$$
where $\bar{F}_{3}$ and $\bar{G}$ are proper transforms of $F_{3}$
and $G$, respectively. Hence $\mathrm{lct}(X)\leqslant 1/2$.

We suppose that $\mathrm{lct}(X)<1/2$. Then there exists an
effective $\mathbb{Q}$-divisor $D\qlineq -K_{X}$ such that the~log
pair $(X,\lambda D)$ is not log canonical for some positive
rational number $\lambda<1/2$. Note that
$$
\varnothing\ne\mathrm{LCS}\Big(X,\lambda D\Big)\subseteq E,%
$$
because
$\mathrm{lct}(\mathbb{P}^1\times\mathbb{P}^1\times\mathbb{P}^1)=1/2$
and
$\alpha(D)\qlineq-K_{\mathbb{P}^1\times\mathbb{P}^1\times\mathbb{P}^1}$.

There is a smooth surface $H\in |3F_{1}+F_{3}\big|$ such that
$C=G\cap H$. Let $\bar{H}$ be a proper transform of the surface
$H$ on the threefold $X$. Then $\bar{H}\cap\bar{G}=\varnothing$
and there is a commutative diagram
$$
\xymatrix{%
&U\ar@{->}[d]_{\phi}&&X\ar@{->}[d]^{\alpha}\ar@{->}[ll]_{\gamma}\ar@{->}[rr]^{\beta}&&V\ar@{->}[d]^{\pi}&&\\
&\mathbb{P}^{1}\times\mathbb{P}^{1}&&\mathbb{P}^{1}\times\mathbb{P}^{1}\times\mathbb{P}^{1}\ar@{->}[ll]_{\zeta}\ar@{->}[rr]^{\xi}&&\mathbb{P}^{1}\times\mathbb{P}^{1}}
$$
such that $\beta$ and $\gamma$ are contractions of the surfaces
$\bar{G}$ and $\bar{H}$ to smooth curves, the~morphisms $\pi$ and
$\phi$ are $\mathbb{P}^{1}$-bundles, the morphisms $\zeta$ and
$\xi$ are projections that are given~by the linear systems
$|F_{1}+F_{2}|$ and $|F_{1}+F_{3}|$, respectively.

It follows from $\bar{H}\cap\bar{G}=\varnothing$ that
\begin{itemize}
\item either the log pair $(V,\lambda\beta(D))$ is not log canonical,%
\item or the log pair $(U,\lambda\gamma(D))$ is not log canonical.
\end{itemize}

Applying Theorem~\ref{theorem:Hwang} to the log pairs
$(V,\lambda\beta(D))$ or $(U,\lambda\gamma(D))$ (and the
fibrations $\pi$ or $\phi$, respectively) and using
Theorem~\ref{theorem:connectedness}, we see that
$$
\mathrm{LCS}\Big(X,\ \lambda D\Big)=\Gamma,
$$
where $\Gamma$ is a fiber of the~natural projection $E\to C$.

We may assume that $\alpha(\Gamma)\in F_{3}$. Let
$\bar{F}_{3}\subset X$ be the proper transform of the surface
$F_{3}$. Put
$$
D=\mu\bar{F}_{3}+\Omega,
$$
where $\Omega$ is an effective $\mathbb{Q}$-divisor on $X$ such
that $\bar{F}_{3}\not\subset\mathrm{Supp}(\Omega)$. Then
$$
\mu F_{3}+\alpha\big(\Omega\big)\qlineq 2\Big(F_{1}+F_{2}+F_{3}\Big),%
$$
which gives $\mu\leqslant 2$. Hence the log pair
$(\bar{F}_{3},\lambda\Omega\vert_{\bar{F}_{3}})$ is not log
canonical along $\Gamma\subset\bar{F}_{3}$ by
Theorem~\ref{theorem:adjunction}. But
$$
\Omega\Big\vert_{\bar{F}_{3}}\qlineq -K_{\bar{F}_{3}},
$$
and $\bar{F}_{3}$ is a del Pezzo surface such that
$K_{\bar{F}_{3}}^{2}=6$ and
\begin{itemize}
\item either $\bar{F}_{3}$ is smooth and $|C\cap F_{3}|=2$;%
\item or $\bar{F}_{3}$ has one ordinary double point and $|C\cap F_{3}|=1$.%
\end{itemize}

We have $\mathrm{lct}(\bar{F}_{3})\leqslant\lambda$. Then
$\bar{F}_{3}$ is singular by Example~\ref{example:del-Pezzos}. It
follows from Lemma~\ref{lemma:singular-del-Pezzo-sextic} that
$$
\mathrm{LCS}\Big(\bar{F}_{3},\
\lambda\Omega\vert_{\bar{F}_{3}}\Big)=\mathrm{Sing}\big(\bar{F}_{3}\big),
$$
but the log pair $(\bar{F}_{3},\lambda\Omega\vert_{\bar{F}_{3}})$
is not log canonical along the whole curve
$\Gamma\subset\bar{F}_{3}$, which is a contradiction.
\end{proof}

\begin{lemma}
\label{lemma:4-5} Suppose that $\gimel(X)=4.5$. Then
$\mathrm{lct}(X)=3/7$.
\end{lemma}

\begin{proof}
Let $Q\subset\mathbb{P}^{4}$ be a quadric cone, let
$V\subset\mathbb{P}^{6}$ be a a section of
$\mathrm{Gr}(2,5)\subset\mathbb{P}^{9}$ by a linear subspace of
dimension $6$ such that $V$ has one ordinary double point. Then
the~diagram
$$
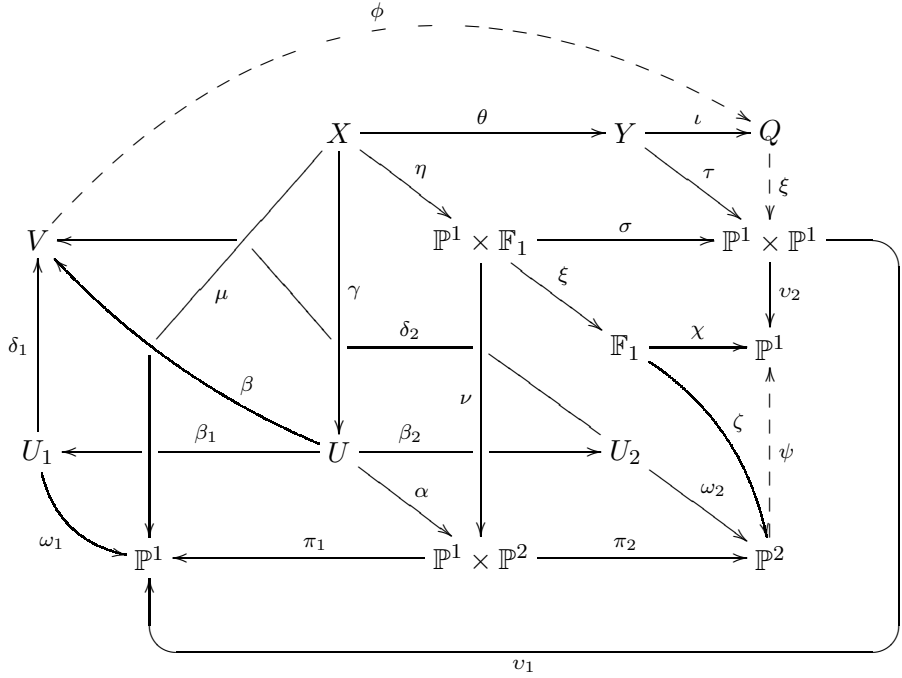
\xymatrix{
&&&X\ar@{->}'[ddll]^(0.7){\mu}[ddddll]\ar@{->}[ddd]^{\gamma}\ar@{->}[rd]^{\eta}\ar@{->}[rr]^{\theta}&&Y\ar@{->}[r]^{\iota}\ar@{->}[rd]^{\tau}&Q\ar@{-->}[d]^{\xi}&&\\%
V\ar@/^5pc/@{-->}[urrrrrr]^{\phi}&&&&\mathbb{P}^{1}\times\mathbb{F}_{1}\ar@{->}[dr]^{\xi}\ar@{->}[ddd]_{\nu}\ar@{->}[rr]^{\sigma}&&\mathbb{P}^{1}\times\mathbb{P}^{1}\ar@{->}[d]^{\upsilon_{2}}\ar@{->}`[rd]`[dddd]`[lllll]^{\upsilon_{1}}[dddlllll]&&\\%
&&&&&\mathbb{F}_{1}\ar@/^1pc/[ddr]^{\zeta}\ar@{->}[r]^{\chi}&\mathbb{P}^{1}&\\%
U_{1}\ar@{->}[uu]^{\delta_{1}}\ar@/_1pc/@{->}[dr]_{\omega_{1}}&&&U\ar@/^0.7pc/@{->}[uulll]_(0.3){\beta}\ar@{->}[dr]^{\alpha}\ar@{->}'[ll]_(0.7){\beta_{1}}[lll]\ar@{->}'[r]^{\beta_{2}}[rr]&&U_{2}\ar@{->}[dr]^{\omega_{2}}\ar@{->}'[ul]'[ull]_{\delta_{2}}'[uulll][uulllll]&&&\\%
&\mathbb{P}^{1}&&&\mathbb{P}^{1}\times\mathbb{P}^{2}\ar@{->}[lll]_{\pi_{1}}\ar@{->}[rr]^{\pi_{2}}&&\mathbb{P}^{2}\ar@{-->}[uu]_{\psi}&\\
&&&&&&&&&}%
$$
commutes (cf. \cite[Lemma~2.6]{Fur04}), where we have the
following notation:
\begin{itemize}
\item the morphisms $\pi_{i}$, $\upsilon_{i}$, $\xi$ and $\chi$ are natural projections;%
\item the morphism $\alpha$ contracts a surface
$\mathbb{F}_{3}\cong E\subset U$ to a curve $C$ such that
$$
\pi_{1}^{*}\Big(\mathcal{O}_{\mathbb{P}^{1}}\big(1\big)\Big)\cdot
C=2,
\ \pi_{2}^{*}\Big(\mathcal{O}_{\mathbb{P}^{2}}\big(1\big)\Big)\cdot C=1;%
$$
\item the morphism $\beta$ contracts a surface $\mathbb{P}^{1}\times\mathbb{P}^{1}\cong\bar{H}_{2}\subset U$ to the singular point of $V$;%
\item the morphism $\beta_{i}$ contracts the surface $\bar{H}_{2}$ to a smooth rational curve;%
\item the~morphism $\delta_{i}$ contracts the curve
$\beta_{i}(\bar{H}_{2})$ to the singular point of $V$ so that
the~map
$$
\delta_{2}\circ\delta_{1}^{-1}\colon U_{1}\dasharrow U_{2}
$$
is a standard flop in the curve $\beta_{1}(\bar{H}_{2})\cong\mathbb{P}^{1}$;%
\item the morphism $\omega_{1}$ is a fibration whose general fiber is $\mathbb{P}^{1}\times\mathbb{P}^{1}$;%
\item the morphisms $\omega_{2}$, $\pi_{2}$, $\xi$, $\sigma$, $\tau$ are $\mathbb{P}^{1}$-bundles;%
\item the morphism $\zeta$ is a blow up of a point $O\in\mathbb{P}^{2}$ such that $O\not\in\pi_{2}(C)$;%
\item the map $\psi$ is a linear projection from the point $O\in\mathbb{P}^{2}$;%
\item the morphism $\nu$ contracts a surface $G\cong\mathbb{P}^{1}\times\mathbb{P}^{1}$ to a curve $L$ such that $\pi_{2}(L)=O$;%
\item the morphism $\gamma$ contracts a surface $\breve{G}$ to a
curve $\bar{L}$ such~that
$$
\alpha\big(\bar{L}\big)=L\subset\mathbb{P}^{1}\times\mathbb{P}^{2}
$$
and the curve $\beta(\bar{L})$ is a line in $V\subset\mathbb{P}^{6}$ such that $\beta(\bar{L})\cap\mathrm{Sing}(V)=\varnothing$;%
\item the morphism $\eta$ contracts a surface $\breve{E}$ to a curve such~that $\nu\circ\eta(\breve{E})=C\subset\mathbb{P}^{1}\times\mathbb{P}^{2}$;%
\item the morphism $\theta$ contracts a surface
$\breve{R}\subset X$ to a curve such that $\breve{R}\ne\breve{E}$
and
$$
\tau\circ\theta\big(\breve{R}\big)=\sigma\circ\eta\big(\breve{E}\big)\subset\mathbb{P}^{1}\times\mathbb{P}^{1};
$$
\item the morphism $\mu$ is a fibration into del Pezzo surfaces of degree $6$;%
\item the morphism $\iota$ contracts the surface $\theta(\breve{H}_{2})$ to the singular point of the quadric $Q$;%
\item the map $\phi$ is a linear projection from the line $\beta(\bar{L})\subset V\subset\mathbb{P}^{6}$.%
\end{itemize}

The curve $\pi_{2}(C)\subset\mathbb{P}^{2}$ is a line. Then
$\alpha(\bar{H}_{2})\in|\pi_{2}^{*}(\mathcal{O}_{\mathbb{P}^{2}}(1))|$
and $C\subset\alpha(\bar{H}_{2})$.

The morphism $\pi_{1}$ induces a double cover $C\to\mathbb{P}^{1}$
branched in two points $Q_{1}\in C\ni Q_{2}$.~Let
$$
T_{i}\in\Big|\pi_{1}^{*}\Big(\mathcal{O}_{\mathbb{P}^{1}}\big(1\big)\Big)\Big|
$$
be the unique surface such that $Q_{i}\in T_{i}$. Let
$\bar{T}_{i}\subset U$ be the proper transform of $T_{i}$. Then
\begin{itemize}
\item the surface $\bar{T}_{i}$ has one ordinary double point,%
\item the surface $\bar{T}_{i}$ is tangent to the surface $E$ along the curve $E\cap\bar{T}_{i}$,%
\item the surface $\bar{T}_{i}$ is a del Pezzo surface such that $K_{\bar{T}_{i}}^{2}=7$.%
\end{itemize}

Let $Z_{i}\subset\mathbb{P}^{2}$ be the unique line such that
$O\in Z\ni \pi_{2}\circ\alpha(Q_{i})$. Then there is a unique
surface
$$
\bar{R}_{i}\in\Big|\big(\pi_{2}\circ\alpha\big)^{*}\Big(\mathcal{O}_{\mathbb{P}^{2}}\big(1\big)\Big)\Big|
$$
such that $Z_{i}\subset\pi_{2}\circ\alpha(\bar{R}_{i})$. One has
$\bar{L}\subset\bar{R}_{i}$ and
$$
-K_{U}\sim 2\bar{H}_{2}+\bar{R}_{i}+2\bar{T}_{i}+E.
$$

Let $\Gamma_{i}$ be a fiber of the projection $E\to C$ over the
point $Q_{i}$. Then $\Gamma_{i}=E\cap\bar{T}_{i}$ and
$$
\Gamma_{i}\subset\mathrm{LCS}\left(U,\
\frac{3}{7}\Big(2\bar{H}_{2}+\bar{R}_{i}+2\bar{T}_{i}+E\Big)\right).
$$

Let $\breve{R}_{i}$ and $\breve{T}_{i}$ be the proper transforms
of $\bar{R}_{i}$ and $\bar{T}_{i}$ on the threefold $X$,
respectively. Then
$$
-K_{X}\sim 2\breve{H}_{2}+\breve{R}_{i}+2\breve{T}_{i}+\breve{E},
$$
because $\bar{L}\subset\bar{R}_{i}$. Let
$\breve{\Gamma}_{i}\subset X$ be the proper transform of the curve
$\Gamma_{i}$. Then the log pair
$$
\left(X,\
\frac{3}{7}\Big(2\breve{H}_{2}+\breve{R}_{i}+2\breve{T}_{i}+\breve{E}\Big)\right)
$$
is log canonical but not log terminal. Thus, we see that
$\mathrm{lct}(X)\leqslant 3/7$.

We suppose that $\mathrm{lct}(X)<3/7$. Then there exists an
effective $\Q$-divisor $D\qlineq -K_{X}$ such that the log pair
$(X,\lambda D)$ is not log canonical for some rational
$\lambda<3/7$.

The surfaces $\breve{T}_{1}$ and $\breve{T}_{2}$ are the only
singular fibers of the fibration $\mu\colon X\to\mathbb{P}^{1}$.
Then
$$
\breve{T}_{i}\not\subseteq\LCS\Big(X,\ \lambda D\Big)\subsetneq\breve{T}_{1}\cup\breve{T}_{2},%
$$
by Lemma~\ref{lemma:Hwang}, because $D\cdot Z=\breve{T}_{1}=2$,
where $Z$ is a general fiber of $\pi_{2}\circ\alpha\circ\gamma$.

We may assume that $\LCS(X,\lambda D)\subseteq\breve{T}_{1}$ by
Theorem~\ref{theorem:connectedness}.

Applying Theorem~\ref{theorem:Hwang} to the log pair
$(\mathbb{P}^{1}\times\mathbb{F}_{1}, \lambda\eta(D))$, we see
that
$$
\varnothing\ne\LCS\Big(X,\ \lambda D\Big)\ne \breve{T}_{1}\cap\breve{G},%
$$
because $G=\eta(\breve{G})$ is a section of the
$\mathbb{P}^{1}$-bundle $\sigma$.

Applying Theorem~\ref{theorem:Hwang} to the log pair
$(\mathbb{P}^{1}\times\mathbb{P}^{2},
\lambda\alpha\circ\gamma(D))$, we see that
$$
\varnothing\ne\LCS\Big(X,\ \lambda D\Big)\subseteq\breve{T}_{1}\cap\breve{E}=\breve{\Gamma}_{1}%
$$
by Theorem~\ref{theorem:connectedness}, because
$\breve{G}\cap\breve{E}=\varnothing$ and $T_{1}$ is a section of
$\pi_{2}$.

Applying Theorem~\ref{theorem:Hwang} to the log pairs $(Y,
\lambda\theta(D))$ and $(U_{2}, \lambda\beta_{2}\circ\gamma(D))$
(and the fibrations $\tau$ and $\omega_2$) we see that
$$
\varnothing\ne\LCS\Big(X,\ \lambda D\Big)=\breve{\Gamma}_{1},%
$$
because $\breve{R}\cap\breve{H}_{2}=\varnothing$.  Put
$\bar{D}=\gamma(D)$. Then $\LCS(U,
\lambda\bar{D})=\Gamma_{1}$. Put
$$
\bar{D}=\eps\bar{H}_{2}+\Omega,
$$
where $\Omega$ is an effective $\mathbb{Q}$-divisor such that
$\bar{H}_{2}\not\subseteq\mathrm{Supp}(\Omega)$. Then
$$
\Omega\Big\vert_{\bar{H}_{2}}\qlineq
-\frac{\big(1+\eps\big)}{2}K_{\bar{H}_{2}},
$$
and the log pair $(\bar{H}_{2},\lambda\Omega\vert_{\bar{H}_{2}})$
is not log canonical by Theorem~\ref{theorem:adjunction}. The
latter implies that
$$
\frac{3}{7}\cdot\frac{1+\eps}{2}>\lambda \frac{1+\eps}{2}>1/2
$$
by Lemma~\ref{lemma:elliptic-times-P1}, and hence $\eps>4/3$.

We may assume that either $E\not\subseteq\mathrm{Supp}(\bar{D})$
or $\bar{T}_{1}\not\subseteq\mathrm{Supp}(\bar{D})$ by
Remark~\ref{remark:convexity}.

Suppose that $E\not\subseteq\mathrm{Supp}(\bar{D})$. Let $Z$ be a
general fiber of the projection $E\to C$. Then
$$
1=-K_{U}\cdot Z=\bar{D}\cdot Z=\eps+\Omega\cdot Z\geqslant \eps,
$$
which is a contradiction, because $\eps>4/3$. Thus, we see that
$\bar{T}_{1}\not\subseteq\mathrm{Supp}(\bar{D})$.

Let $\bar{\Delta}\subset\bar{T}_{1}$ be a proper transform of a
general line in $T_{1}\cong\mathbb{P}^{2}$ that passes through
$Q_{1}$. Then
$$
2=-K_{U}\cdot\bar{\Delta}=\bar{D}\cdot\bar{\Delta}\geqslant\mathrm{mult}_{\Gamma_{1}}\big(\bar{D}\big)\geqslant
1/\lambda>7/3,
$$
because $\bar{\Delta}\not\subset\mathrm{Supp}(\bar{D})$ and
$\bar{\Delta}\cap\Gamma_{1}\ne\varnothing$. The obtained
contradiction completes the proof.
\end{proof}

\begin{lemma}
\label{lemma:4-6} Suppose that $\gimel(X)=4.6$. Then
$\mathrm{lct}(X)=1/2$.
\end{lemma}

\begin{proof}
There is a birational morphism $\alpha\colon X\to\mathbb{P}^{3}$
that blows up three disjoint lines $L_{1}$, $L_{2}$, $L_{3}$.

Let $H_{i}$ be the proper transform on $X$ of a general plane in
$\mathbb{P}^{3}$ such that $L_{i}\subset\alpha(H_{i})$. Then
$$
-K_{X}\sim 2H_{1}+E_{1}+H_{2}+H_{3}\sim
2H_{2}+E_{2}+H_{1}+H_{3}\sim 2H_{3}+E_{3}+H_{1}+H_{2},
$$
where $E_{i}$ is the exceptional divisor of $\alpha$ such that
$\alpha(E_{i})=L_{i}$. In particular, we see that
$\mathrm{lct}(X)\leqslant 1/2$.

We suppose that $\mathrm{lct}(X)<1/2$. Then there exists an
effective $\mathbb{Q}$-divisor $D\qlineq -K_{X}$ such that the~log
pair $(X,\lambda D)$ is not log canonical for some positive
rational number $\lambda<1/2$.

The surface $H_{i}$ is a smooth del Pezzo surface such that
$K_{H_{i}}^{2}=7$, the linear system $|H_{i}|$ has no base points
and induces a smooth morphism $\phi_{i}\colon X\to\mathbb{P}^{1}$,
whose fibers are isomorphic to $H_{i}$.

Suppose that $|\mathrm{LCS}(X,\lambda D)|<+\infty$. We may assume
that  $\mathrm{LCS}(X,\lambda D)\not\subseteq E_{1}$. Then the set
$$
\mathrm{LCS}\left(X,\ \lambda D+H_{1}+\frac{1}{2}E_{1}\right)
$$
is disconnected, which is impossible by
Theorem~\ref{theorem:connectedness}, because
$H_{2}+H_{3}+(\lambda-1/2)K_{X}$ is ample.

We may assume that $H_{1}\cap\mathrm{LCS}(X,\lambda
D)\ne\varnothing$. Then
$$
\varnothing\ne H_{1}\cap\mathrm{LCS}\Big(X,\ \lambda D\Big)\subseteq\mathrm{LCS}\Big(H_{1},\ \lambda D\Big\vert_{H_{1}}\Big)%
$$
by Remark~\ref{remark:hyperplane-reduction}. Put
$C_{2}=E_{2}\vert_{H_{1}}$ and $C_{3}=E_{3}\vert_{H_{1}}$. Then
$$
C_{2}\cdot C_{2}=C_{3}\cdot C_{3}=-1,
$$
and there is a unique curve $\mathbb{P}^{1}\cong C\subset H_{1}$
such that $C\cdot C_{2}=C\cdot C_{3}=1$ and $C\cdot C=-1$. Note
that
$$
\mathrm{LCS}\Big(H_{1},\ \lambda D\Big\vert_{H_{1}}\Big)=C
$$
by Lemma~\ref{lemma:del-Pezzo-septic}.

There is a unique smooth quadric $Q\subset\mathbb{P}^{3}$ that
contains $L_{1}$, $L_{2}$, $L_{3}$. Note that
$$
\bar{Q}\cap H_{1}=C,
$$
where $\bar{Q}\subset X$ is a proper transform of the surface $Q$.

There is a morphism $\sigma\colon
X\to\mathbb{P}^{1}\times\mathbb{P}^{1}\times\mathbb{P}^{1}$
contracting $\bar{Q}$ to a curve of tri-degree $(1,1,1)$. Since
$\bar{Q}\cap H_{1}=C$, one obtains (see
Remark~\ref{remark:hyperplane-reduction}) that
$$
\mathrm{LCS}\Big(X,\ \lambda D\Big)\supset\bar{Q},
$$
and hence $\LCS(X, \lambda D)=\bar{Q}$, because
$\mathrm{lct}(\mathbb{P}^{1}\times\mathbb{P}^{1}\times\mathbb{P}^{1})=1/2$.
Put
$$
D=\mu \bar{Q}+\Omega,
$$
where $\mu\geqslant 1/\lambda>2$, and $\Omega$ is an effective
$\mathbb{Q}$-divisor such that
$\bar{Q}\not\subset\mathrm{Supp}(\Omega)$. Then
$$
\alpha\big(D\big)=\mu Q+\alpha\big(\Omega\big),
$$
which is impossible, because $\alpha(D)\qlineq
2Q\sim-K_{\mathbb{P}^{3}}$ and $\mu>2$.
\end{proof}

\begin{lemma}
\label{lemma:4-7} Suppose that $\gimel(X)=4.7$. Then
$\lct(X)=1/2$.
\end{lemma}

\begin{proof}
There is a birational morphism $\alpha\colon X\to W$ such that
\begin{itemize}
\item the variety $W$ is a smooth divisor of bi-degree $(1,1)$ on $\mathbb{P}^{2}\times\mathbb{P}^{2}$;%
\item the morphism $\alpha$ contracts two (irreducible)
surfaces $E_{1}\ne E_{2}$ to two disjoint curves $L_1$ and $L_2$;%
\item the curves $L_i$
are fibers of one natural $\mathbb{P}^{1}$-bundle $W\to\P^2$.%
\end{itemize}

There is a surface $H\subset W$ such that $-K_{X}\sim 2H$ and
$L_{1}\subset H\supset L_{2}$. Then
$$
-K_{X}\sim 2\bar{H}+E_1+E_2,
$$
where $\bar{H}$ is a proper  transform of $H$ on the threefold
$X$. In particular, $\mathrm{lct}(X)\leqslant 1/2$.

We suppose that $\mathrm{lct}(X)<1/2$. Then there exists an
effective $\mathbb{Q}$-divisor $D\qlineq -K_{X}$ such that the log
pair $(X,\lambda D)$ is not log canonical for some $\lambda<1/2$.
Then
$$
\varnothing\ne\mathrm{LCS}\Big(X,\lambda D\Big)\subseteq E_1\cup
E_2,
$$
since $\mathrm{lct}(W)=1/2$ by Theorem~\ref{theorem:del-Pezzo} and
$\alpha(D)\qlineq -K_{W}$.

We may assume that $\LCS(X, \lambda D)\cap E_1\neq\varnothing$.
Let $\beta\colon X\to Y$ be a contraction of $E_2$. Then
$$
\mathbb{LCS}\Big(Y,\ \lambda\beta\big(D\big)\Big)\ne\varnothing
$$
and $\beta(D)\qlineq -K_Y$, which contradicts
Lemma~\ref{lemma:3-24}.
\end{proof}

\begin{lemma}
\label{lemma:4-8} Suppose that $\gimel(X)=4.8$. Then
$\mathrm{lct}(X)=1/3$.
\end{lemma}

\begin{proof}
There is blow up $\alpha\colon X\to
\mathbb{P}^1\times\mathbb{P}^1\times\mathbb{P}^1$ of a curve
$C\subset\mathbb{P}^1\times\mathbb{P}^1\times\mathbb{P}^1$ such
that~$C\subset F_{1}$~and
$$
C\cdot F_{2}=C\cdot F_{3}=1,
$$
where $F_{i}$ is a fiber of the projection of
$\mathbb{P}^1\times\mathbb{P}^1\times\mathbb{P}^1$~to its $i$-th
factor. There is a surface
$$
\mathbb{P}^1\times\mathbb{P}^1\cong G\in \big|F_{2}+F_{3}\big|
$$
such that $C\subset G$. Let $E$ be the exceptional divisor of
$\alpha$. Then
$$
-K_{X}\sim 2\bar{F}_{1}+2\bar{G}+3E,
$$
where $\bar{F}_{1}$ and $\bar{G}$ are proper transforms of $F_{1}$
and $G$, respectively. In particular, $\mathrm{lct}(X)\leqslant
1/3$.

We suppose that $\mathrm{lct}(X)<1/3$. Then there exists an
effective $\mathbb{Q}$-divisor $D\qlineq -K_{X}$ such that the~log
pair $(X,\lambda D)$ is not log canonical for some positive
rational number $\lambda<1/3$. Note that
$$
\varnothing\ne\mathrm{LCS}\Big(X,\lambda D\Big)\subseteq E,
$$
because
$\mathrm{lct}(\mathbb{P}^1\times\mathbb{P}^1\times\mathbb{P}^1)=1/2$
and
$\alpha(D)\qlineq-K_{\mathbb{P}^1\times\mathbb{P}^1\times\mathbb{P}^1}$.

Let $Q$ be a quadric cone in $\mathbb{P}^{4}$. Then there is a
commutative diagram
$$
\xymatrix{%
&&&&X\ar@{->}[dlll]_{\alpha}\ar@{->}[d]_{\beta}\ar@{->}[rrrd]^{\gamma}&&&&&\\
&\mathbb{P}^{1}\times\mathbb{P}^{1}\times\mathbb{P}^{1}\ar@{->}[drr]_{\phi}&&&V\ar@{->}[dl]_{\pi}\ar@{->}[rrd]^{\delta}&&&U\ar@{->}[dl]^{\xi}\\%
&&&\mathbb{P}^{1}\times\mathbb{P}^{1}&&&Q\ar@{-->}[lll]^{\psi}&}
$$
where we have the following notations:
\begin{itemize}
\item $V$ is a variety with $\gimel(V)=3.31$;%
\item the morphism $\beta$ is a contraction of the surface $\bar{G}$ to a curve;%
\item the morphism $\gamma$ is a contraction of $\bar{F}_{1}\cong\mathbb{P}^{1}\times\mathbb{P}^{1}$ to an ordinary double point;%
\item the~morphism $\delta$ is a blow up of the vertex of the
quadric cone
$Q\subset\mathbb{P}^{4}$;%
\item the morphism $\xi$~is~a~blow~up of a~smooth conic in $Q$;%
\item the map $\psi$ is a~projection from the~vertex of the cone $Q$;%
\item the morphism $\phi$ is a projection that is given by
$|F_{2}+F_{3}|$, i.\,e. the projection of
$\P^1\times\P^1\times\P^1$ onto the product of the
last two factors;%
\item the morphism $\pi$ is a natural $\mathbb{P}^{1}$-bundle.%
\end{itemize}

It follows from Corollary~\ref{corollary:other-toric} that
$\mathrm{lct}(V)=1/3$. On the other hand, $\mathrm{lct}(U)=1/3$ by
Lemma~\ref{lemma:2-29-singular}. Hence
$$
\varnothing\ne\mathrm{LCS}\Big(X,\ \lambda D\Big)\subseteq
E\cap\bar{G}\cap\bar{F}_{1}=\varnothing,
$$
which is a contradiction.
\end{proof}

The following result is implied by
Corollaries~\ref{corollary:other-toric} and
\ref{corollary:4-4-and-5-1}, Lemma~\ref{lemma:lct-product} and
Example~\ref{example:del-Pezzos}.

\begin{corollary}
\label{corollary:rho-5-6-7-8-9-10} Suppose that $\rho\geqslant 5$.
Then
$$
\mathrm{lct}\big(X\big)=\left\{%
\aligned
&1/3\ \text{whenever}\ \gimel(X)\in\{5.1, 5.2\},\\%
&1/2\ \text{in the remaining cases}.\\%
\endaligned\right.%
$$
\end{corollary}

\begin{lemma}
\label{lemma:4-13} Suppose that $\gimel(X)=4.13$ and $X$ is
general. Then $\lct(X)=1/2$.
\end{lemma}
\begin{proof}
Let $F_{1}\cong F_{2}\cong F_{3}\cong\P^1\times\P^1$ be fibers of
three different projections
$$
\mathbb{P}^1\times\mathbb{P}^1\times\mathbb{P}^1\longrightarrow\mathbb{P}^{1},
$$
respectively. There is a contraction $\alpha\colon X\to
\mathbb{P}^1\times\mathbb{P}^1\times\mathbb{P}^1$ of a surface
$E\subset X$ to a curve
$$
C\subset\mathbb{P}^1\times\mathbb{P}^1\times\mathbb{P}^1
$$
such that $C\cdot F_{1}=C\cdot F_{2}=1$ and $C\cdot F_{3}=3$. Then
there is a smooth surface
$$
\mathbb{P}^1\times\mathbb{P}^1\cong G\in \big|F_{1}+F_{2}\big|
$$
such that $C\subset G$. In particular, we see that
$$
-K_{X}\sim 2\bar{G}+E+2\bar{F}_{3},
$$
where $\bar{F}_{3}$ and $\bar{G}$ are proper transforms of $F_{3}$
and $G$, respectively. Hence $\mathrm{lct}(X)\leqslant 1/2$.

We suppose that $\mathrm{lct}(X)<1/2$. Then there exists an
effective $\mathbb{Q}$-divisor $D\qlineq -K_{X}$ such that the~log
pair $(X,\lambda D)$ is not log canonical for some positive
rational number $\lambda<1/2$. Note that
$$
\varnothing\ne\mathrm{LCS}\Big(X,\lambda D\Big)\subseteq E\cong\mathbb{F}_{4},%
$$
because
$\mathrm{lct}(\mathbb{P}^1\times\mathbb{P}^1\times\mathbb{P}^1)=1/2$
and
$\alpha(D)\qlineq-K_{\mathbb{P}^1\times\mathbb{P}^1\times\mathbb{P}^1}$.

There are smooth surfaces $H_{1}\in |3F_{1}+F_{3}\big|$ and
$H_{2}\in |3F_{2}+F_{3}\big|$ such that
$$
C=G\cdot H_{1}=G\cdot H_{2},
$$
and $H_{1}\cong H_{2}\cong\mathbb{P}^1\times\mathbb{P}^1$. Let
$\bar{H}_{i}$ be a proper transform of $H_{i}$ on the threefold
$X$. Then
$$
\bar{H}_{1}\cap\bar{G}=\bar{H}_{2}\cap\bar{G}=\varnothing.
$$

There is a commutative diagram
$$
\xymatrix{%
&&&&&X\ar@{->}[dd]^{\alpha}\ar@{->}[dlll]_{\gamma_{1}}\ar@{->}[ddll]^{\beta}\ar@{->}[rrd]^{\gamma_{2}}&&&&\\
&&U_{1}\ar@{->}[dddr]_{\phi_{1}}&&&&&U_{2}\ar@{->}[dddl]^{\phi_{2}}\\%
&&&V\ar@/_0.37pc/[ddrr]_{\pi}&&\mathbb{P}^{1}\times\mathbb{P}^{1}\times\mathbb{P}^{1}\ar@{->}[dd]_{\zeta}\ar@{->}[rdd]^{\xi_{2}}\ar@{->}'[ld]^{\xi_{1}}[lldd]&&\\%
&&&&&&\\
&&&\mathbb{P}^{1}\times\mathbb{P}^{1}&&\mathbb{P}^{1}\times\mathbb{P}^{1}&\mathbb{P}^{1}\times\mathbb{P}^{1}&}
$$
such that $\beta$ and $\gamma_{i}$ are contractions of the
surfaces $\bar{G}$ and $\bar{H}_{i}$ to a smooth curves,
the~morphisms $\pi$ and $\phi_{i}$ are $\mathbb{P}^{1}$-bundles,
the morphisms $\zeta$ and $\xi_{i}$ are projections that are
given~by the linear systems $|F_{1}+F_{2}|$ and $|F_{i}+F_{3}|$,
respectively.

It follows from $\bar{H}_{1}\cap\bar{G}=\varnothing$ that
\begin{itemize}
\item either the log pair $(V,\lambda\beta(D))$ is not log canonical,%
\item of the log pair $(U_{1},\lambda\gamma_{1}(D))$ is not log
canonical.
\end{itemize}

Applying Theorem~\ref{theorem:Hwang} to $(V,\lambda\beta(D))$ or
$(U_{1},\lambda\gamma_{1}(D))$ (and the fibration $\pi$ or
$\phi_1$) and using Theorem~\ref{theorem:connectedness}, we see
that
$$
\mathrm{LCS}\Big(X,\ \lambda D\Big)=\Gamma,
$$
where $\Gamma$ is a fiber of the~natural projection $E\to C$.

We may assume that $\alpha(\Gamma)\in F_{3}$. Let
$\bar{F}_{3}\subset X$ be the proper transform of the surface
$F_{3}$. Put
$$
D=\mu\bar{F}_{3}+\Omega,
$$
where $\Omega$ is an effective $\mathbb{Q}$-divisor on $X$ such
that $\bar{F}_{3}\not\subset\mathrm{Supp}(\Omega)$. Then
$$
\mu F_{3}+\alpha\big(\Omega\big)\qlineq 2\Big(F_{1}+F_{2}+F_{3}\Big),%
$$
which gives $\mu\leqslant 2$. The log pair
$(\bar{F}_{3},\lambda\Omega\vert_{\bar{F}_{3}})$ is not log
canonical along $\Gamma\subset\bar{F}_{3}$ by
Theorem~\ref{theorem:adjunction}. One has
$$
\Omega\Big\vert_{\bar{F}_{3}}\qlineq -K_{\bar{F}_{3}},
$$
and $\bar{F}_{3}$ is a del Pezzo surface such that
$K_{\bar{F}_{3}}^{2}=5$. Note that $\bar{F}_{3}$ may be singular.
Namely, we have
$$
\mathrm{Sing}\big(\bar{F}_{3}\big)=\varnothing\iff\big|C\cap F_{3}\big|=F_{3}\cdot C=3,%
$$
and $\mathrm{Sing}(\bar{F}_{3})\subset\Gamma$. The following cases
are possible:
\begin{itemize}
\item the surface $\bar{F}_{3}$ is smooth and $|C\cap F_{3}|=3$;%
\item the surface $\bar{F}_{3}$ has one ordinary double point and $|C\cap F_{3}|=2$;%
\item the surface $\bar{F}_{3}$ has a singular point of type
$\mathbb{A}_{2}$ and $|C\cap
 F_{3}|=1$.%
\end{itemize}

We have $\mathrm{lct}(\bar{F}_{3})\leqslant\lambda<1/2$. Thus, it
follows from Examples~\ref{example:del-Pezzos} and
\ref{example:del-Pezzo-quintic} that $|C\cap F_{3}|=1$, which  is
impossible if the threefold $X$ is sufficiently general.
\end{proof}

\section{Upper bounds}
\label{section:bounds}

We use the assumptions and notation introduced in
section~\ref{section:intro}. The purpose of this section is to
find upper bounds for the global log canonical thresholds of the
varieties $X$ with
$$
\gimel\big(X\big)\in\Big\{1.1,1.2,\ldots,1.17,2.1,\ldots,2.36,3.1,\ldots,3.31,4.1,\ldots,4.13,5.1,\ldots,5.7,5.8\Big\}.
$$

\begin{lemma}
\label{lemma:1-8} Suppose that $\gimel(X)=1.8$. Then
$\mathrm{lct}(X)\leqslant 6/7$.
\end{lemma}

\begin{proof}
The linear system $|-K_{X}|$ does not have base points and induces
an embedding $X\subset\mathbb{P}^{10}$, and the threefold $X$
contains a line $L\subset X$ (see \cite{Sho79}).

It follows from \cite[Theorem~4.3.3]{IsPr99} that there is a
commutative diagram
$$
\xymatrix{
U\ar@{->}[d]_{\alpha}\ar@{-->}[rr]^{\rho}&&W\ar@{->}[d]^{\beta}\\%
X\ar@{-->}[rr]_{\psi}&&\mathbb{P}^{3}}
$$
where $\alpha$ is a blow up of the line $L$, the map $\rho$ is a
composition of flops, the morphism $\beta$ is a blow up of a
smooth curve of degree $7$ and genus $3$, and $\psi$ is a double
projection from $L$.

Let $S\subset X$ be the proper transform of the exceptional
surface of $\beta$. Then
$$
\mathrm{mult}_{L}\big(S\big)=7
$$
and $S\sim -3K_{X}$, which implies that $\mathrm{lct}(X)\leqslant
6/7$.
\end{proof}

\begin{lemma}
\label{lemma:1-9} Suppose that $\gimel(X)=1.9$. Then
$\mathrm{lct}(X)\leqslant 4/5$.
\end{lemma}

\begin{proof}
The linear system $|-K_{X}|$ does not have base points and induces
an embedding $X\subset\mathbb{P}^{11}$, and the threefold $X$
contains a line $L\subset X$ (see \cite{Sho79}).

It follows from \cite[Theorem~4.3.3]{IsPr99} that there is a
commutative diagram
$$
\xymatrix{
U\ar@{->}[d]_{\alpha}\ar@{-->}[rr]^{\rho}&&W\ar@{->}[d]^{\beta}\\%
X\ar@{-->}[rr]_{\psi}&&Q}
$$
where $Q\subset\mathbb{P}^{4}$ is a smooth quadric threefold,
$\alpha$ is a blow up of the line $L$, the map
$\rho$~is~a~composition of flops,  the morphism $\beta$ is a blow
up along a smooth curve of degree $7$ and genus $2$,
and~$\psi$~is~a~double projection from the line $L$.

Let $S\subset X$ be the proper transform of the exceptional
surface of $\beta$. Then
$$
\mathrm{mult}_{L}\big(S\big)=5
$$
and  $S\sim -2K_{X}$, which implies that $\mathrm{lct}(X)\leqslant
4/5$.
\end{proof}

\begin{lemma}
\label{lemma:1-10} Suppose that $\gimel(X)=1.10$. Then
$\mathrm{lct}(X)\leqslant 2/3$.
\end{lemma}

\begin{proof}
The linear system $|-K_{X}|$ does not have base points and induces
an embedding $X\subset\mathbb{P}^{13}$, and the threefold $X$
contains a line $L\subset X$ (see \cite{Sho79}).

It follows from \cite[Theorem~4.3.3]{IsPr99} that the diagram
$$
\xymatrix{
U\ar@{->}[d]_{\alpha}\ar@{-->}[rr]^{\rho}&&W\ar@{->}[d]^{\beta}\\%
X\ar@{-->}[rr]_{\psi}&&V_{5}}
$$
commutes, where $V_{5}$ is a smooth section of
$\mathrm{Gr}(2,5)\subset\mathbb{P}^9$ by a linear subspace of
dimension~$6$, the morphism $\alpha$ is a blow up of the line $L$,
the map $\rho$ is a composition of flops, the morphism $\beta$ is
a blow up of a smooth rational curve of degree $5$, and $\psi$ is
a double projection from $L$.

Let $S\subset X$ be the proper transform of the exceptional
surface of $\beta$. Then
$$
\mathrm{mult}_{L}\big(S\big)=3
$$
and  $S\sim -K_{X}$, which implies that $\mathrm{lct}(X)\leqslant
2/3$.
\end{proof}

\begin{lemma}
\label{lemma:2-2} Suppose that $\gimel(X)=2.2$. Then $\lct(X)\le
13/14$.
\end{lemma}

\begin{proof}
There is a smooth divisor $B\subset\P^1\times\P^2$  of bi-degree
$(2, 4)$ such that the diagram
$$
\xymatrix{
&& X\ar@{->}[d]_{\pi}\ar@{->}[lld]_{\phi_1}\ar@{->}[rrd]^{\phi_2} & \\
\P^1 &&\P^1\times\P^2\ar@{->}[ll]^{\pi_1}\ar@{->}[rr]_{\pi_2}&& \P^2}%
$$
commutes, where $\pi$ is a double cover branched along $B$, the
morphisms $\pi_1$ and $\pi_2$ are natural projections, $\phi_1$ is
a fibration into del Pezzo surfaces of degree $2$, and $\phi_2$ is
a conic bundle.

Let $H_1$ be a general fiber of $\phi_1$. Put
$\bar{H}_{1}=\pi(H_{1})$. Then the intersection
$$
C=\bar{H}_1\cap B\subset\bar{H}_{1}\cong\mathbb{P}^{2}
$$
is a smooth quartic curve.

There is a point $P\in C$ such that
$$
\mathrm{mult}_{P}\Big(C\cdot L\Big)\geqslant 3,
$$
where $L\subset\bar{H}_{1}\cong\mathbb{P}^{2}$ is a line that is
tangent to $C$ at the point $P$.

The curve $\pi_{2}(L)$ is a line. Thus, there is a unique surface
$$
H_{2}\in\Big|\phi_{2}^{*}\Big(\mathcal{O}_{\mathbb{P}^{2}}\big(1\big)\Big)\Big|
$$
such that $\phi_{2}(H_{2})=\pi_{2}(L)$. Hence $-K_X\sim H_1+H_2$.

Let us show that $\mathrm{lct}(X,H_{1}+H_{2})\leqslant 13/14$. Put
$\bar{H}_{2}=\pi(H_{2})$. Then
$$
\mathbb{LCS}\left(X,\
\frac{13}{14}\Big(H_{1}+H_{2}\Big)\right)\ne\varnothing\iff\mathbb{LCS}\left(\P^1\times\P^2,\ \frac{1}{2}B+\frac{13}{14}\Big(\bar{H}_{1}+\bar{H}_{2}\Big)\right)\ne\varnothing%
$$
by \cite[Proposition~3.16]{Ko97}. Let $\alpha\colon
V\to\P^1\times\P^2$ be a blow up of the curve $C$. Then
$$
K_{V}+\frac{1}{2}\tilde{B}+\frac{13}{14}\Big(\tilde{H}_{1}+\tilde{H}_{2}\Big)+\frac{3}{7}E\qlineq\alpha^{*}\left(K_{\P^1\times\P^2}+\frac{1}{2}B+\frac{13}{14}\Big(\bar{H}_{1}+\bar{H}_{2}\Big)\right),
$$
where $\tilde{B},\tilde{H}_{1},\tilde{H}_{2}\subset V$ are proper
transforms of $B,\bar{H}_{1},\bar{H}_{2}$, respectively. But the
log pair
$$
\left(V,\ \frac{13}{14}\tilde{H}_{2}+\frac{3}{7}E\right)%
$$
is not log terminal along the fiber $\Gamma$ of the morphism
$\alpha$ such that $\alpha(\Gamma)=P$, because
$$
\mathrm{mult}_{\Gamma}\Big(\tilde{H}_{2}\cdot E\Big)=
\mathrm{mult}_{P}\Big(C\cdot\bar{H}_{2}\Big)
\geqslant\mathrm{mult}_{P}\Big(C\cdot L\Big)\geqslant 3%
$$
due to the choice of the fiber $H_1$. We see that
$$
\Gamma\subseteq\mathrm{LCS}\left(V,\ \frac{13}{14}\tilde{H}_{2}+\frac{3}{7}E\right)\subseteq\mathrm{LCS}\left(V,\ \frac{1}{2}\tilde{B}+\frac{13}{14}\Big(\tilde{H}_{1}+\tilde{H}_{2}\Big)+\frac{3}{7}E\right),%
$$
which implies that $\mathrm{lct}(X,H_{1}+H_{2})\leqslant 13/14$.
Hence the inequality $\mathrm{lct}(X)\leqslant 13/14$ holds.
\end{proof}

\begin{remark}
It follows from Lemmas~\ref{lemma:Hwang} and
\ref{lemma:singular-Del-Pezzo-2} that $\mathrm{lct}(X)\geqslant
2/3$ if $\gimel(X)=2.2$ and the threefold $X$ satisfies the
following generality condition: any fiber of $\phi_1$ satisfies
the hypotheses of Lemma~\ref{lemma:singular-Del-Pezzo-2}.
\end{remark}

\begin{lemma}
\label{lemma:2-7} Suppose that $\gimel(X)=2.7$. Then
$\mathrm{lct}(X)\leqslant 2/3$.
\end{lemma}

\begin{proof}
There is a commutative diagram
$$
\xymatrix{
&X\ar@{->}[dl]_{\alpha}\ar@{->}[dr]^{\beta}&\\%
Q\ar@{-->}[rr]_{\psi}&&\mathbb{P}^{1}}
$$
where $Q\subset\mathbb{P}^{4}$ is a smooth quadric threefold,
$\alpha$ is a blow up of a smooth curve that is a complete intersection of
two divisors
$$
S_{1}, S_2\in\Big|\mathcal{O}_{\mathbb{P}^{4}}\big(2\big)\Big\vert_{Q}\Big|,%
$$
the morphism $\beta$ is a fibration into del Pezzo surfaces of
degree $4$, and $\psi$~is~a~rational map that is induced by the
pencil generated by the surfaces $S_{1}$ and $S_{2}$. Then
$\mathrm{lct}(X)\leqslant 2/3$, because
$$
-K_{X}\qlineq \frac{3}{2}\bar{S}_{1}+\frac{1}{2}E,
$$
where $\bar{S}_{1}\subset X$ is a proper transform of the surface
$S_{1}$, and $E$ is the exceptional divisor of $\alpha$.
\end{proof}

\begin{lemma}
\label{lemma:2-9} Suppose that $\gimel(X)=2.9$. Then
$\mathrm{lct}(X)\leqslant 3/4$.
\end{lemma}

\begin{proof}
There is a commutative diagram
$$
\xymatrix{
&X\ar@{->}[dl]_{\alpha}\ar@{->}[dr]^{\beta}&\\%
\mathbb{P}^{3}\ar@{-->}[rr]_{\psi}&&\mathbb{P}^{2}}
$$
where $\alpha$ is a blow up of a smooth curve
$C\subset\mathbb{P}^{3}$ of degree $7$ and genus $5$ that is a
scheme-theoretic intersection of cubic surfaces in
$\mathbb{P}^{3}$, the morphism $\beta$ is a conic bundle, and
$\psi$ is a rational map that is given by the linear system of
cubic surfaces that contain $C$. One has
$$
-K_{X}\qlineq \frac{4}{3}S+\frac{1}{3}E,
$$
where $S\in|\beta^{*}(\mathcal{O}_{\mathbb{P}^{2}}(1))|$, and $E$
is the exceptional divisor of $\alpha$. We see that
$\mathrm{lct}(X)\leqslant 3/4$.
\end{proof}

\begin{lemma}
\label{lemma:2-12} Suppose that $\gimel(X)=2.12$. Then
$\mathrm{lct}(X)\leqslant 3/4$.
\end{lemma}

\begin{proof}
There is a commutative diagram
$$
\xymatrix{
&X\ar@{->}[dl]_{\alpha}\ar@{->}[dr]^{\beta}&\\%
\mathbb{P}^{3}\ar@{-->}[rr]_{\psi}&&\mathbb{P}^{3}}
$$
where $\alpha$ and $\beta$ are blow ups of smooth curves
$C\subset\mathbb{P}^{3}$ and $Z\subset\P^3$ of degree $6$ and
genus $3$ that are scheme-theoretic intersections of cubic
surfaces in $\mathbb{P}^{3}$, and $\psi$ is a
birational map that is given by the linear system of cubic
surfaces that contain $C$. Then
$$
-K_{X}\qlineq \frac{4}{3}S+\frac{1}{3}E,
$$
where $S\in|\beta^{*}(\mathcal{O}_{\mathbb{P}^{3}}(1))|$, and $E$
is the exceptional divisor of $\alpha$. We see that
$\mathrm{lct}(X)\leqslant 3/4$.
\end{proof}

\begin{lemma}
\label{lemma:2-13} Suppose that $\gimel(X)=2.13$. Then
$\mathrm{lct}(X)\leqslant 2/3$.
\end{lemma}

\begin{proof}
There is a commutative diagram
$$
\xymatrix{
&X\ar@{->}[dl]_{\alpha}\ar@{->}[dr]^{\beta}&\\%
Q\ar@{-->}[rr]_{\psi}&&\mathbb{P}^{2}}
$$
where $Q\subset\mathbb{P}^{4}$ is a smooth quadric threefold,
$\alpha$ is a blow up of a smooth curve $C\subset Q$~of~degree~$6$
and genus $2$, the morphism $\beta$ is a conic bundle, and $\psi$
is a rational map that is given by the~linear system of surfaces
in $|\mathcal{O}_{\mathbb{P}^{4}}(2)\vert_{Q}|$ that contain the
curve $C$. One has
$$
-K_{X}\qlineq \frac{3}{2}S+\frac{1}{2}E,
$$
where $S\in|\beta^{*}(\mathcal{O}_{\mathbb{P}^{2}}(1))|$, and $E$
is the exceptional divisor of $\alpha$. We see that
$\mathrm{lct}(X)\leqslant 2/3$.
\end{proof}

\begin{lemma}
\label{lemma:2-16} Suppose that $\gimel(X)=2.16$. Then
$\mathrm{lct}(X)\leqslant 1/2$.
\end{lemma}

\begin{proof}
There is a commutative diagram
$$
\xymatrix{
&X\ar@{->}[dl]_{\alpha}\ar@{->}[dr]^{\beta}&\\%
V_{4}\ar@{-->}[rr]_{\psi}&&\mathbb{P}^{2}}
$$
where $V_{4}\subset\mathbb{P}^{5}$ is a smooth complete
intersection of two quadric hypersurfaces, $\alpha$ is a blow up
of an irreducible conic $C\subset V_{4}$, the morphism $\beta$ is
a conic bundle, and $\psi$ is a rational map that is given by
the~linear system of surfaces in
$|\mathcal{O}_{\mathbb{P}^{5}}(1)\vert_{V_{4}}|$ that contain $C$.
One has
$$
-K_{X}\sim 2S+E,
$$
where $S\in|\beta^{*}(\mathcal{O}_{\mathbb{P}^{2}}(1))|$, and $E$
is the exceptional divisor of $\alpha$. We see that
$\mathrm{lct}(X)\leqslant 1/2$.
\end{proof}

\begin{lemma}
\label{lemma:2-17}  Suppose that $\gimel(X)=2.17$. Then
$\mathrm{lct}(X)\leqslant 2/3$.
\end{lemma}

\begin{proof}
There is a commutative diagram
$$
\xymatrix{
&X\ar@{->}[dl]_{\alpha}\ar@{->}[dr]^{\beta}&\\%
Q\ar@{-->}[rr]_{\psi}&&\mathbb{P}^{3}}
$$
where $Q\subset\mathbb{P}^{4}$ is a smooth quadric threefold, the
morphisms $\alpha$ and $\beta$ are blow ups of smooth elliptic
curves $C\subset Q$ and $Z\subset\mathbb{P}^{3}$~of~degree~$5$,
respectively, and the map $\psi$ is given by the~linear system of
surfaces in $|\mathcal{O}_{\mathbb{P}^{4}}(2)\vert_{Q}|$ that
contain the curve $C$. One has
$$
-K_{X}\qlineq \frac{3}{2}S+\frac{1}{2}E,
$$
where $S\in|\beta^{*}(\mathcal{O}_{\mathbb{P}^{3}}(1))|$, and $E$
is the exceptional divisor of $\alpha$. We see that
$\mathrm{lct}(X)\leqslant 2/3$.
\end{proof}

\begin{lemma}
\label{lemma:2-20} Suppose that $\gimel(X)=2.20$. Then
$\mathrm{lct}(X)\leqslant 1/2$.
\end{lemma}

\begin{proof}
There is a commutative diagram
$$
\xymatrix{
&X\ar@{->}[dl]_{\alpha}\ar@{->}[dr]^{\beta}&\\%
V_{5}\ar@{-->}[rr]_{\psi}&&\mathbb{P}^{2}}
$$
where $V_{5}\subset\mathbb{P}^{6}$ is a smooth intersection of
$\mathrm{Gr}(2,5)\subset\mathbb{P}^9$ with a linear subspace of
dimension~$6$, the morphism $\alpha$ is a blow up of a cubic curve
$\mathbb{P}^1\cong C\subset V_{5}$, the morphism $\beta$ is a
conic bundle, and $\psi$ is given by the~linear system of surfaces
in $|\mathcal{O}_{\mathbb{P}^{6}}(1)\vert_{V_{5}}|$ that contain
the curve $C$. One has
$$
-K_{X}\sim 2S+E,
$$
where $S\in|\beta^{*}(\mathcal{O}_{\mathbb{P}^{2}}(1))|$, and $E$
is the exceptional divisor of $\alpha$. We see that
$\mathrm{lct}(X)\leqslant 1/2$.
\end{proof}

\begin{lemma}
\label{lemma:2-21}  Suppose that $\gimel(X)=2.21$. Then
$\mathrm{lct}(X)\leqslant 2/3$.
\end{lemma}

\begin{proof}
There is a commutative diagram
$$
\xymatrix{
&X\ar@{->}[dl]_{\alpha}\ar@{->}[dr]^{\beta}&\\%
Q\ar@{-->}[rr]_{\psi}&&Q}
$$
where $Q\subset\mathbb{P}^{4}$ is a smooth quadric threefold, the
morphisms $\alpha$ and $\beta$ are blow ups of smooth rational
curves $C\subset Q$ and $Z\subset Q$ of degree~$4$,
and $\psi$ is a birational map that is given by the~linear system
of surfaces in $|\mathcal{O}_{\mathbb{P}^{4}}(2)\vert_{Q}|$ that
contain the curve $C$. One has
$$
-K_{X}\qlineq \frac{3}{2}S+\frac{1}{2}E,
$$
where $S\in|\beta^{*}(\mathcal{O}_{\mathbb{P}^{4}}(1))\vert_{Q}|$,
and $E$ is the exceptional divisor of $\alpha$. We see that
$\mathrm{lct}(X)\leqslant 2/3$.
\end{proof}

\begin{lemma}
\label{lemma:2-22} Suppose that $\gimel(X)=2.22$. Then
$\mathrm{lct}(X)\leqslant 1/2$.
\end{lemma}

\begin{proof}
There is a commutative diagram
$$
\xymatrix{
&X\ar@{->}[dl]_{\alpha}\ar@{->}[dr]^{\beta}&\\%
V_{5}\ar@{-->}[rr]_{\psi}&&\mathbb{P}^{3}}
$$
where $V_{5}\subset\mathbb{P}^{6}$ is a smooth intersection of
$\mathrm{Gr}(2,5)\subset\mathbb{P}^9$ with a linear subspace of
dimension~$6$, the~morphisms $\alpha$ and $\beta$ are blow ups of
a conic $C\subset V_{5}$ and a rational (not linearly normal)
quartic $Z\subset\mathbb{P}^{3}$,~respectively, and $\psi$ is given by
the~linear system of surfaces in
$|\mathcal{O}_{\mathbb{P}^{6}}(1)\vert_{V_{5}}|$ that contain the
curve $C$. One has
$$
-K_{X}\sim 2S+E,
$$
where $S\in|\beta^{*}(\mathcal{O}_{\mathbb{P}^{3}}(1))|$, and $E$
is the exceptional divisor of $\alpha$. We see that
$\mathrm{lct}(X)\leqslant 1/2$.
\end{proof}

\begin{lemma}
\label{lemma:3-13} Suppose that $\gimel(X)=3.13$. Then
$\mathrm{lct}(X)\leqslant 1/2$.
\end{lemma}

\begin{proof}
There is a commutative diagram
$$
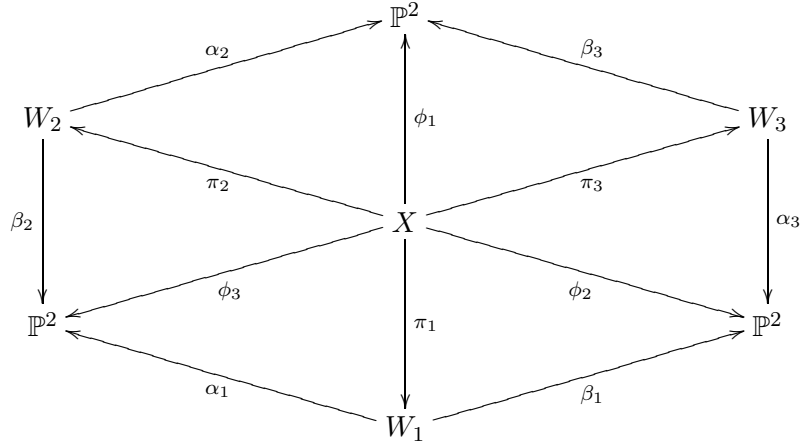
\xymatrix{
&&&&\mathbb{P}^{2}&&&&\\
W_{2}\ar@{->}[dd]_{\beta_{2}}\ar@{->}[urrrr]^{\alpha_{2}}&&&&&&&&W_{3}\ar@{->}[ullll]_{\beta_{3}}\ar@{->}[dd]^{\alpha_{3}}\\
&&&&X\ar@{->}[uu]_{\phi_1}\ar@{->}[drrrr]_{\phi_2}\ar@{->}[dllll]^{\phi_3}
\ar@{->}[dd]^{\pi_{1}}\ar@{->}[ullll]^{\pi_{2}}\ar@{->}[urrrr]_{\pi_{3}}&&&&\\
\mathbb{P}^{2}&&&&&&&&\mathbb{P}^{2}\\
&&&&W_{1}\ar@{->}[ullll]^{\alpha_{1}}\ar@{->}[urrrr]_{\beta_{1}}&&&&}
$$
such that $W_{i}\subset\mathbb{P}^2\times\mathbb{P}^2$ is a
divisor of bi-degree $(1,1)$, the morphisms $\alpha_{i}$ and
$\beta_{i}$ are $\mathbb{P}^{1}$-bundles, $\pi_{i}$ is a blow up
of a smooth curve $C_{i}\subset W_{i}$ of bi-degree $(2, 2)$ such
that
$$
\alpha_{i}\big(C_{i}\big)\subset\mathbb{P}^2\supset\beta_{i}\big(C_{i}\big)
$$
are irreducible conics, and $\phi_i$ is a conic bundle. Let
$E_{i}$ be the exceptional divisor of $\pi_{i}$. Then
$$
-K_{X}\sim 2H_{1}+E_{1}\sim 2H_{2}+E_{2}\sim 2H_{3}+E_{3}\sim
E_{1}+E_{2}+E_{3},
$$
where $H_{i}\in|\phi_{i}^*(\mathcal{O}_{\mathbb{P}^{2}}(1))|$. We
see that $\mathrm{lct}(X)\leqslant 1/2$.
\end{proof}

\begin{remark}
Let us use the notation of the proof of Lemma~\ref{lemma:3-13} and
assume that $\lct(X)<1/2$. Then there is an effective $\Q$-divisor
$D\qlineq -K_X$ such that the log pair $(X, \lambda D)$ is not log
canonical for some $\lambda< 1/2$. Since $\lct(W_i)=1/2$ by
Theorem~\ref{theorem:del-Pezzo}, one has
$$\varnothing\neq\LCS(X, \lambda D)\subset E_1\cap E_2\cap E_3.$$
In particular, by Theorem~\ref{theorem:connectedness} the locus
$\LCS(X, \lambda D)$ consists of a single point $P$; note that $P$
is an intersection $P=F_1\cap F_2\cap F_3$ of three curves $F_i$
such that $F_2\cup F_3$ (resp., $F_1\cup F_3$, $F_1\cup F_2$) is a
reducible fiber of the conic bundle $\phi_1$ (resp., $\phi_2$,
$\phi_3$).
\end{remark}

\appendix

\section{By Jean-Pierre Demailly. On Tian's invariant and log canonical thresholds}%
\label{section:alpha}

The goal of this appendix is to relate log canonical thresholds with
the $\alpha$-invariant introduced by G.\,Tian~\cite{Ti87} for the study of
the existence of K\"ahler--Einstein metrics. The approximation
technique of closed positive $(1,1)$-currents introduced in~\cite{De92} is
used to show that the $\alpha$-invariant actually coincides with the
log canonical threshold.

Algebraic geometers have been aware of this fact after~\cite{DeKo01}
appeared, and several papers have used it consistently in the latter
years (see e.g.~\cite{JoKo01b},~\cite{BoGaKo05}). However, it turns out that
the required result is stated only in a local analytic form
in~\cite{DeKo01},
in a language which may not be easily recognizable by algebraically
minded people. Therefore, we will repair here the lack of a proper
reference by stating and proving the statements required for the
applications to projective varieties, e.g.\ existence of
K\"ahler--Einstein metrics on Fano varieties and Fano orbifolds.

Usually, in these applications, only the case of the anticanonical
line bundle $L=-K_X$ is considered. Here we will consider more generally
the case of an arbitrary line bundle $L$ (or $\Q$-line bundle $L$) on a
complex  manifold $X$, with some additional restrictions which will
be stated later.

Assume that $L$ is equipped with a \textit{singular hermitian metric} $h$
(see e.g.~\cite{De90}). Locally, $L$ admits trivializations
$\theta:L_{|U}\simeq U\times\mathbb{C}$, and on $U$ the metric $h$ is given
by a weight function $\varphi$ such that
$$
\Vert\xi\Vert_h^2 =|\xi|^2e^{-2\varphi(z)} \text{ for all } z\in U,
\xi\in L_z,
$$
when $\xi\in L_z$ is identified with a complex number. We are interested
in the case where $\varphi$ is (at the very least) a locally integrable
function for the Lebesgue measure, since it is then possible to compute the
curvature form
$$
\Theta_{L,h}=\frac{i}{\pi}\ddbar\varphi
$$
in the sense of distributions. We have $\Theta_{L,h}\ge 0$ as a
$(1,1)$-current
if and only if the weights $\varphi$ are plurisubharmonic functions. In
the sequel we will be interested only in that case.

Let us first introduce
the concept of complex singularity exponent for singular hermitian metrics,
following e.g.~\cite{Va82},
\cite{Va83b}, \cite{ArGV85} and~\cite{DeKo01}.

\begin{definition}
\label{definition:A1} If $K$ is a compact subset of $X$, we define
the complex singularity exponent $c_K(h)$ of the metric $h$,
written locally as $h=e^{-2\varphi}$, to be the supremum of all
positive numbers $c$ such that  $h^c=e^{-2c\varphi}$ is integrable
in a neighborhood of every point $z_0\in K$, with respect to the
Lebesgue measure in holomorphic coordinates centered at~$z_0$.
\end{definition}

Now, we introduce a generalized version of Tian's invariant $\alpha$, as
defined in~\cite{Ti87} (see also~\cite{Siu88}).

\begin{definition}
\label{definition:A2} Assume that $X$ is a compact manifold and
that $L$ is a pseudo-effective line bundle, i.\,e. $L$ admits a
singular hermitian metric $h_0$ with $\Theta_{L,h_0}\ge 0$. If $K$
is a compact subset of $X$, we put
$$
\alpha_K(L)=\inf_{\{h,\,\Theta_{L,h}\ge 0\}}c_K(h)
$$
where $h$ runs over all singular hermitian metrics on $L$ such that
$\Theta_{L,h}\ge 0$.
\end{definition}

In algebraic geometry, it is more usual to look instead at linear systems
defined by a family of linearly independent sections
$\sigma_0,\sigma_1,\ldots,\sigma_N\in\nlb
H^0(X,L^{\otimes m})$. We denote by $\Sigma$ the vector subspace generated by
these sections and by
$$
|\Sigma|:=P(\Sigma)\subset|mL|:=P(H^0(X,L^{\otimes m}))
$$
the corresponding linear system. Such an $(N+1)$-tuple of sections
$\sigma=(\sigma_j)_{0\le j\le N}$ defines a
singular hermitian metric $h$ on $L$ by putting in any trivialization
$$
\Vert \xi\Vert^2_h
=\frac{|\xi|^2}{\big(\sum_j|\sigma_j(z)|^2\big)^{1/m}}
=\frac{|\xi|^2}{|\sigma(z)|^{2/m}} \text{ for } \xi\in L_z,
$$
hence $h(z)=|\sigma(z)|^{-2/m}$ with
$$\varphi(z)=\frac{1}{m}\log|\sigma(z)|=\frac{1}{2m}\log\sum_j|\sigma_j(z)|^2$$
as the associated weight function. Therefore, we are interested in the
number $c_K(|\sigma|^{-2/m})$. In the case of a single section
$\sigma_0$ (corresponding to a linear system containing a single divisor),
this is the
same as the log canonical threshold $\lct_K(X, \frac{1}{m}D)$ of the
where $D$ is a divisor corresponding to $\sigma_0$.
We will also use the formal notation
$\lct_K(X, \frac{1}{m}|\Sigma|)$ in the case of a higher dimensional
linear system $|\Sigma|\subset|mL|$.

Now, recall that the line bundle $L$ is
said to be \textit{big} if the Kodaira--Iitaka dimension
$\kappa(L)$ equals $n=\dim_\mathbb{C}(X)$. The main result of this
appendix is

\begin{theorem}
\label{theorem:A3} Let $L$ be a big line bundle on a compact
complex manifold $X$. Then for every compact set $K$ in $X$ we
have
$$
\alpha_K(L)=\inf_{\{h,\,\Theta_{L,h}\ge 0\}} c_K(h)
=\inf_{m\in\Z_{>0}}\inf_{D\in|mL|}
\lct_K\Big(X, \frac{1}{m}D\Big).
$$
\end{theorem}

Observe that the inequality
$$
\inf_{m\in\Z_{>0}}\inf_{D\in|mL|}\lct_K\Big(X, \frac{1}{m}D\Big)
\geq \inf_{\{h,\,\Theta_{L,h}\ge 0\}} c_K(h)
$$
is trivial, since any divisor $D\in|mL|$ gives rise to a singular
hermitian metric~$h$. The converse inequality will follow from the
approximation technique of~\cite{De92} and some elementary analysis.  The
proof is parallel to the proof of \cite[Theorem~4.2]{DeKo01},
although the language used there was somewhat different. In any case,
we use again the crucial concept of multiplier ideal sheaves: if
$h$ is a singular hermitian metric with local plurisubharmonic
weights $\varphi$, the multiplier ideal sheaf $\mathcal{I}(h)\subset\O_X$
(also denoted by $\mathcal{I}(\varphi)$) is the ideal sheaf defined by

\begin{multline*}
\mathcal{I}(h)_z=\left\{\phantom{\int f}
f\in\O_{X,z}\ \mid \text{there exists a neighborhood $V\ni z$,}
\right. \\ \left\{
\text{ such that }\int_V|f(x)|^2e^{-2\varphi(x)}d\lambda(x)<+\infty\right\},
\end{multline*}
where $\lambda$ is the Lebesgue measure. By Nadel
(see~\cite{Na90}), this is a coherent analytic sheaf on $X$.
Theorem~\ref{theorem:A3} has a more precise version which can be
stated as follows.

\begin{theorem}
\label{theorem:A4} Let $L$ be a line bundle on a compact complex
manifold $X$ possessing a singular hermitian metric $h$ with
$\Theta_{L,h}\ge \eps\omega$ for some $\eps>0$ and some smooth
positive definite hermitian $(1,1)$-form $\omega$ on~$X$. For
every real number $m>0$, consider the space
$\mathcal{H}_m=H^0(X,L^{\otimes m}\otimes\mathcal{I}(h^m))$ of
holomorphic sections $\sigma$ of $L^{\otimes m}$ on $X$ such that
$$
\int_X|\sigma|_{h^m}^2dV_\omega
=\int_X|\sigma|^2e^{-2m\varphi}dV_\omega<+\infty,
$$
where $dV_\omega=\frac{1}{m!}\omega^m$ is the hermitian volume form.
Then for $m\gg 1$, $\mathcal{H}_m$
is a non zero finite dimensional Hilbert space
and we consider the closed positive $(1,1)$-current
$$
T_m=\frac{i}{2\pi}\ddbar\Big(\frac{1}{2m}\log\sum_k|g_{m,k}|^2\Big)
=\frac{i}{2\pi}\ddbar\Big(\frac{1}{2m}\log\sum_k|g_{m,k}|_h^2\Big)
+\Theta_{L,h}
$$
where $(g_{m,k})_{1\le k\le N(m)}$ is an orthonormal basis of~$\mathcal{H}_m$.
The following statements hold.
\begin{list}{\labelitemi}{\leftmargin=0.6em}
\item[(i)] For every trivialization  $L_{|U}\simeq U\times\mathbb{C}$
on a cordinate open set $U$ of $X$ and every compact set $K\subset U$,
there are constants $C_1,C_2>0$ independent of $m$ and $\varphi$ such that
$$
\varphi(z)-\frac{C_1}{m}\le
\psi_m(z):=\frac{1}{2m}\log\sum_k|g_{m,k}(z)|^2
\le\sup_{|x-z|<r}\varphi(x)+\frac{1}{m}\log\frac{C_2}{r^n}
$$
for every $z\in K$ and $r\le\frac{1}{2}d(K,\partial U)$. In particular,
$\psi_m$ converges to $\varphi$ pointwise and in $L^1_{\rm loc}$ topology
on~$\Omega$ when $m\to+\infty$, hence $T_m$ converges weakly
to $T=\Theta_{L,h}$.
\smallskip
\item[(ii)] The Lelong numbers $\nu(T,z)=\nu(\varphi,z)$ and
$\nu(T_m,z)=\nu(\psi_m,z)$ are related by
$$
\nu(T,z)-\frac{n}{m}\le\nu(T_m,z)\le
\nu(T,z)\quad\text{ for every $z\in X$.}
$$
\item[(iii)] For every compact set $K\subset X$, the complex
singularity
exponents of the metrics given locally by $h=e^{-2\varphi}$ and
$h_m=e^{-2\psi_m}$ satisfy
$$
c_K(h)^{-1}-\frac{1}{m}\le c_K(h_m)^{-1}\le c_K(h)^{-1}.
$$
\end{list}
\end{theorem}

\begin{proof} The major part of the proof is a small variation
of the arguments already explained in~\cite{De92} (see
also~\cite[Theorem~4.2]{DeKo01}).
We give them here in some detail for the convenience of the reader.
\medskip

\noindent
(i) We note that $\sum|g_{m,k}(z)|^2$ is the square of the norm of
the evaluation linear form $\sigma\mapsto \sigma(z)$ on $\mathcal{H}_m$, hence
$$
\psi_m(z)=\sup_{\sigma\in B(1)}\frac{1}{m}\log|\sigma(z)|
$$
where $B(1)$ is the unit ball of $\mathcal{H}_m$. For
$r\le \frac{1}{2}d(K,\partial\Omega)$, the mean value inequality
applied to the plurisubharmonic function $|\sigma|^2$ implies
\begin{multline*}
|\sigma(z)|^2\le\frac{1}{\pi^nr^{2n}/n!}\int_{|x-z|<r}
|\sigma(x)|^2d\lambda(x)\le \\
\le\frac{1}{\pi^nr^{2n}/n!}\exp\Big(2m\sup_{|x-z|<r}\varphi(x)\Big)
\int_\Omega|\sigma|^2e^{-2m\varphi}d\lambda.
\end{multline*}
If we take the supremum over all $\sigma\in B(1)$ we get
$$
\psi_m(z)\le\sup_{|x-z|<r}\varphi(x)+\frac{1}{2m}
\log\frac{1}{\pi^nr^{2n}/n!}
$$
and the right hand inequality in (i) is proved. Conversely, the
Ohsawa--Takegoshi extension theorem~\cite{OhT87}, \cite{Ohs88}
applied to the $0$-dimensional
subvariety $\{z\}\subset U$ shows that for any $a\in\mathbb{C}$ there is
a holomorphic function $f$ on $U$ such that $f(z)=a$ and
$$
\int_U|f|^2e^{-2m\varphi}d\lambda\le C_3|a|^2e^{-2m\varphi(z)},
$$
where $C_3$ only depends on $n$ and~$\diam(U)$. Now, provided $a$ remains in
a compact set $K\subset U$, we can use a cut-off function $\theta$ with
support in $U$ and equal to $1$ in a neighborhood of $a$, and solve the
$\dbar$-equation $\dbar g=\dbar(\theta f)$ in the $L^2$ space associated
with the
weight $2m\varphi+2(n+1)\log|z-a|$, that is, the singular hermitian
metric $h(z)^m|z-a|^{-2(n+1)}$ on~$L^{\otimes m}$. For this, we apply
the standard Andreotti--Vesentini--H\"ormander $L^2$ estimates
(see e.g.~\cite{De82}
for the required version). This is possible for $m\ge m_0$ thanks to the
hypothesis that $\Theta_{L,h}\ge \varepsilon\omega>0$, even if $X$
is non K\"ahler ($X$ is in any event a Moishezon variety from our
assumptions). The bound $m_0$ depends only on $\varepsilon$ and the
geometry
of a finite covering of $X$ by compact sets $K_j\subset U_j$, where
$U_j$ are coordinate balls (say); it is independent of the point $a$ and
even of the metric~$h$. It follows that $g(a)=0$ and therefore
$\sigma=\theta f-g$ is a holomorphic section of $L^{\otimes m}$ such that
$$
\int_X|\sigma|^2_{h^m}dV_\omega=
\int_X|\sigma|^2e^{-2m\varphi}dV_\omega\le C_4
\int_U|f|^2e^{-2m\varphi}dV_\omega\le C_5
|a|^2e^{-2m\varphi(z)},
$$
in particular,
$\sigma\in \mathcal{H}_m=H^0(X,L^{\otimes m}\otimes\mathcal{I}(h^m))$.
We fix $a$ such that the right hand side of the latter inequality
is~$1$. This gives the inequality
$$
\psi_m(z)\ge\frac{1}{m}\log|a|=\varphi(z)-\frac{\log C_5}{2m}
$$
which is the left hand part of statement (i).
\medskip

\noindent
(ii) The first inequality in (i) implies $\nu(\psi_m,z)\le\nu(\varphi,z)$.
In the opposite direction, we find
$$
\sup_{|x-z|<r}\psi_m(x)\le\sup_{|x-z|<2r}\varphi(x)+
\frac{1}{m}\log\frac{C_2}{r^n}.
$$
Divide by $\log r<0$ and take the limit as $r$ tends to~$0$. The quotient
by $\log r$ of the supremum of a psh function over $B(x,r)$ tends to the
Lelong number at~$x$. Thus we obtain
$$
\nu(\psi_m,x)\ge\nu(\varphi,x)-\frac{n}{m}.
$$
(iii) Again, the first inequality in (i) immediately yields $h_m\le C_6h$,
hence $c_K(h_m)\ge c_K(h)$. For the converse inequality, since we
have $c_{\,\cup K_j}(h)=\min_j c_{K_j}(h)$, we can assume without loss
of generality that $K$ is contained in a trivializing open patch $U$
of~$L$.
Let us take $c<c_K(\psi_m)$. Then, by definition, if
$V\subset X$ is a sufficiently small open neighborhood of $K$,
the H\"older inequality
for the conjugate exponents $p=1+mc^{-1}$ and $q=1+m^{-1}c$
implies, thanks to equality $\frac{1}{p}=\frac{c}{mq}$,
\begin{multline*}
\int_V e^{-2(m/p)\varphi}dV_\omega =\int_V\Big(\sum_{1\le k\le
N(m)}|g_{m,k}|^2e^{-2m\varphi}\Big)^{1/p} \Big(\sum_{1\le k\le
N(m)}|g_{m,k}|^2\Big)^{-c/mq}dV_\omega\le\\
\le\left(\int_X\sum_{1\le k\le
N(m)}|g_{m,k}|^2e^{-2m\varphi}dV_\omega\right)^{1/p} \left(\int_V
\Big(\sum_{1\le k\le N(m)}|g_{m,k}|^2\Big)^{-c/m}dV_\omega
\right)^{1/q}= \\
=N(m)^{1/p} \left(\int_V \Big(\sum_{1\le k\le
N(m)}|g_{m,k}|^2\Big)^{-c/m}dV_\omega \right)^{1/q}<+\infty.
\end{multline*}
From this we infer $c_K(h)\ge m/p$, i.e., $c_K(h)^{-1}\le
p/m=1/m+c^{-1}$.
As $c<c_K(\psi_m)$ was arbitrary, we get $c_K(h)^{-1}\le
1/m+c_K(h_m)^{-1}$
and the inequalities of (iii) are proved.
\end{proof}

\begin{proof}[Proof of Theorem~\ref{theorem:A3}]
Given a big line bundle $L$ on~$X$, there exists a modification
$\mu:\tilde{X}\to X$ of $X$ such that $\tilde{X}$ is projective
and $\mu^*L=\O(A+E)$ where $A$ is an ample divisor and $E$ an
effective divisor with rational coefficients. By pushing forward
by $\mu$ a smooth metric $h_A$ with positive curvature on $A$, we
get a singular hermitian metric $h_1$ on $L$ such that
$$\Theta_{L,h_1}\ge\mu_*\Theta_{A,h_A}\ge \varepsilon\omega$$
on~$X$. Then for any $\delta>0$ and any singular hermitian metric
$h$ on $L$ with $\Theta_{L,h}\ge 0$, the interpolated metric
$h_\delta=h_1^\delta h^{1-\delta}$ satisfies
$\Theta_{L,h_\delta}\ge\delta\varepsilon\omega$. Since $h_1$ is
bounded away from $0$, it follows that $c_K(h)\ge
(1-\delta)c_K(h_\delta)$ by monotonicity. By
Theorem~\ref{theorem:A4}~(iii) applied to $h_\delta$, we infer
$$
c_K(h_\delta)=\lim_{m\to+\infty} c_K(h_{\delta,m}),
$$
and we also have
$$
c_K(h_{\delta,m})\ge \lct_K\Big(\frac{1}{m}D_{\delta,m}\Big)
$$
for any divisor $D_{\delta,m}$ associated with a section $\sigma\in
H^0(X,L^{\otimes m}\otimes\mathcal{I}(h_\delta^m))$, since the metric
$h_{\delta,m}$
is given by $h_{\delta,m}=(\sum_k|g_{m,k}|^2)^{-1/m}$ for an orthornormal
basis
of such sections. This clearly implies
$$
c_K(h)\ge\liminf_{\delta\to 0}\;\liminf_{m\to+\infty}\;
\lct_K\Big(\frac{1}{m}D_{\delta,m}\Big)\ge
\inf_{m\in\Z_{>0}}\inf_{D\in|mL|}
\lct_K\Big(\frac{1}{m}D\Big).\eqno\square
$$

\vskip10pt
In the applications, it is frequent to have a finite or compact group
$G$ of automorphisms of $X$ and to look at $G$-invariant objects, namely
$G$-equivariant metrics on $G$-equivariant line bundles~$L$; in the case
of a reductive algebraic group $G$ we simply consider a compact real form
$G^{\mathbb{R}}$ instead of $G$ itself.

One then gets an $\alpha$ invariant $\alpha_{G,K}(L)$ by
looking only at $G$-equivariant metrics in
Definition~\ref{definition:A2}. All contructions made are then
$G$-equivariant, especially $\mathcal{H}_m\subset |mL|$ is a $G$-invariant
linear system. For every $G$-invariant compact set $K$ in $X$, we thus
infer
$$
\alpha_{G,K}(L)=\inf_{\{h\ \text{\scriptsize is} \
G-\text{\scriptsize equivariant},\,\Theta_{L,h}\ge 0\}}c_K(h)
=\inf_{m\in\Z_{>0}}~~\inf_{|\Sigma|\subset|mL|,~\Sigma^G=\Sigma}~~
\lct_K\Big(\frac{1}{m}|\Sigma|\Big).
\leqno({\rm A}.5)
$$
When $G$ is a finite group, one can pick for $m$ large enough a
$G$-invariant divisor $D_{\delta,m}$ associated with a
$G$-invariant section $\sigma$, possibly after multiplying $m$ by
the order of $G$. One then gets the slightly simpler equality
$$
\alpha_{G,K}(L)
=\inf_{m\in\Z_{>0}}~~\inf_{D\in|mL|^G}~~\lct_K\Big(\frac{1}{m}D\Big).
\leqno({\rm A}.6)
$$
In a similar manner, one can work on an orbifold $X$ rather than on a non
singular variety. The $L^2$ techniques work in this setting with almost
no change ($L^2$ estimates are essentially insensitive  to
singularities, since one can just use an orbifold metric on
the open set of regular points).
\end{proof}

\section{The Big Table}
\label{section:appendix}

This appendix contains the list of nonsingular Fano threefolds. We
follow the notation and the numeration of these in~\cite{IsPr99},
\cite{MoMu81}, \cite{MoMu03}. We also assume the following
conventions. The symbol $V_i$ denotes a smooth Fano threefold such that
$-K_{X}\sim 2H$ and
$\mathrm{Pic}(V_{i})=\mathbb{Z}[H]$,
where $H$ is a Cartier divisor on $V_{i}$, and
$H^{3}=8i\in\{8,16,\ldots,40\}$.
The symbol $W$ denotes a divisor on $\P^2\times\P^2$ of
bidegree $(1,1)$
(or, that is the same, the variety $\P(T_{\P^2})$).
The symbol $V_{7}$ denotes a blow up of $\P^3$ at a point
(or, that is the same, the variety
$\mathbb{P}(\mathcal{O}_{\mathbb{P}^2}\oplus\mathcal{O}_{\mathbb{P}^2}(1))$).
The symbol $Q$ denotes a smooth quadric hypersurface in
$\mathbb{P}^{4}$.
The symbol $S_i$ denotes a smooth del Pezzo surface such
that
$$
K_{S_{i}}^{2}=i\in\big\{1,\ldots,8\big\},
$$
where $S_{8}\not\cong\mathbb{P}^{1}\times\mathbb{P}^{1}$.

The fourth column of Table~\ref{table:Fanos} contains the values
of global log canonical thresholds of the corresponding Fano
varieties. The symbol $\star$ near a number means that $\lct(X)$
is calculated for a general $X$ with a given deformation type. If
we know only the upper bound $\lct(X)\le\alpha$, we write
${}\le\alpha$ instead of the exact value of $\lct(X)$, and the
symbol ? means that we don't know any reasonable upper bound
(apart from a trivial $\lct(X)\le 1$).

\small{
\begin{longtable}{|c|c|p{8.5cm}|c|}
\caption{Smooth Fano threefolds}\label{table:Fanos}\\
\hline $\gimel(X)$ & $-K_X^3$ &  Brief description & $\mathrm{lct}(X)$ \\
\hline $1.1$ & $2$ & a hypersurface in $\mathbb{P}(1,1,1,1,3)$ of degree $6$ & $1\star$\\
\hline $1.2$ & $4$ & a hypersurface in $\mathbb{P}^4$ of degree
$4$ or\hfill\break a double cover of smooth quadric in
$\mathbb{P}^{4}$
branched over a surface of degree $8$ & ? \\
\hline $1.3$ & $6$ & a complete intersection of a quadric and a
cubic in
$\mathbb{P}^{5}$ & ? \\
\hline $1.4$ & $8$ & a complete intersection of three quadrics
$\mathbb{P}^{6}$ & ? \\
\hline $1.5$ & $10$ & a section of
$\mathrm{Gr}(2,5)\subset\mathbb{P}^9$ by
quadric and linear subspace of dimension~$7$ & ? \\
\hline $1.6$ & $12$ & a section of the Hermitian symmetric space
$M=G/P\subset\mathbb{P}^{15}$\hfill\break of type DIII  by a
linear
subspace of dimension~$8$ & ? \\
\hline $1.7$ & $14$ & a section of
$\mathrm{Gr}(2,6)\subset\mathbb{P}^{14}$
by a linear subspace of codimension $5$ & ? \\
\hline $1.8$ & $16$ & a section of the Hermitian symmetric space
$M=G/P\subset \mathbb{P}^{19}$\hfill\break of type CI  by a linear
subspace
of dimension~$10$ & ${}\le 6/7$\\
\hline $1.9$ & $18$ & a section of the $5$-dimensional rational
homogeneous contact\hfill\break manifold
$G_2/P\subset\mathbb{P}^{13}$  by
a linear subspace of dimension~$11$ &  ${}\le 4/5$\\
\hline $1.10$ & $22$ & a zero locus of three sections of the rank
$3$ vector bundle $\bigwedge^2\mathcal{Q}$,\hfill\break where
$\mathcal{Q}$ is
the universal quotient bundle on $\mathrm{Gr}(7,3)$ &  ${}\le 2/3$\\
\hline $1.11$ & $8$ & $V_{1}$ that is a hypersurface in $\mathbb{P}(1,1,1,2,3)$ of degree $6$ & $1/2$\\
\hline $1.12$ & $16$ & $V_{2}$ that is a hypersurface in $\mathbb{P}(1,1,1,1,2)$ of degree $4$ & $1/2$\\
\hline $1.13$ & $24$ & $V_{3}$ that is a hypersurface in $\mathbb{P}^{4}$ of degree $3$ & $1/2$\\
\hline $1.14$ & $32$ & $V_{4}$ that is a complete intersection of two quadrics in $\mathbb{P}^{5}$ & $1/2$\\
\hline $1.15$ & $40$ & $V_{5}$ that is a section of $\mathrm{Gr}(2,5)\subset\mathbb{P}^9$ by linear subspace of  codimension $3$ & $1/2$\\
\hline $1.16$ & $54$ & $Q$ that is a hypersurface in $\mathbb{P}^{4}$ of degree $2$ & $1/3$\\
\hline $1.17$ & $64$ & $\mathbb{P}^{3}$ & $1/4$\\
\hline $2.1$ & $4$ & a blow up of the Fano threefold $V_1$ along an elliptic curve\hfill\break that is an intersection of two divisors from $|-\frac{1}{2}K_{V_1}|$   & $1/2$\\%
\hline $2.2$ & $6$ & a double cover of
$\mathbb{P}^1\times\mathbb{P}^2$
whose branch locus is a divisor of bidegree $(2, 4)$   &  ${}\le 13/14$\\%
\hline $2.3$ & $8$ & the blow up of the Fano threefold $V_2$ along an elliptic curve\hfill\break that is an intersection of two divisors from $|-\frac{1}{2}K_{V_2}|$   & $1/2$\\%
\hline $2.4$ & $10$ & the blow up of $\mathbb{P}^3$ along an intersection of two cubics   & $3/4\star$ \\
\hline $2.5$ & $12$ & the blow up of the threefold $V_3\subset\mathbb{P}^4$ along a plane cubic   & $1/2\star$ \\
\hline $2.6$ & $12$ & a divisor on
$\mathbb{P}^2\times\mathbb{P}^2$ of bidegree $(2, 2)$
or\hfill\break a double cover of $W$ whose branch locus
is a surface in $|-K_W|$  & ? \\
\hline $2.7$ & $14$ & the blow up of $Q$ along the intersection of
two
divisors from $|\mathcal{O}_Q (2)|$   &  ${}\le 2/3$\\
\hline $2.8$ & $14$ & a double cover of $V_7$ whose branch locus
is a surface
in $|-K_{V_7}|$ & $1/2\star$\\
\hline $2.9$ & $16$ & the blow up of $\mathbb{P}^3$ along a curve
of degree $7$ and genus~$5$\hfill\break which is an intersection
of cubics   &
${}\le 3/4$\\
\hline $2.10$ & $16$ & the blow up of $V_4\subset\mathbb{P}^5$
along an elliptic curve\hfill\break which is an intersection of
two hyperplane
sections   &$1/2\star$  \\
\hline $2.11$ & $18$ & the blow up of $V_3$ along a line   & $1/2\star$\\
\hline $2.12$ & $20$ & the blow up of $\mathbb{P}^3$ along a curve
of degree $6$ and genus~$3$\hfill\break which is an intersection
of cubics   &
${}\le 3/4$\\
\hline $2.13$ & $20$ & the blow up of $Q\subset\mathbb{P}^4$ along
a curve
of degree $6$ and genus $2$   &  ${}\le 2/3$\\
\hline $2.14$ & $20$ & the blow up of $V_5\subset\mathbb{P}^6$ along an elliptic curve\hfill\break which is an intersection of two hyperplane sections   &$1/2\star$  \\
\hline $2.15$ & $22$ & the blow up of $\mathbb{P}^3$ along the intersection of a quadric and a cubic surfaces & $1/2\star$ \\
\hline $2.16$ & $22$ & the blow up of $V_4\subset\mathbb{P}^5$
along a
conic   & ${}\le 1/2$ \\
\hline $2.17$ & $24$ & the blow up of $Q\subset\mathbb{P}^4$ along
an
elliptic curve of degree~$5$ &  ${}\le 2/3$\\
\hline $2.18$ & $24$ & a double cover of $\mathbb{P}^1\times\mathbb{P}^2$ whose branch locus is a divisor of bidegree $(2, 2)$   & $1/2$\\
\hline $2.19$ & $26$ & the blow up of $V_4\subset\mathbb{P}^5$ along a line   & $1/2\star$\\
\hline $2.20$ & $26$ & the blow up of $V_5\subset\mathbb{P}^6$
along a
twisted cubic   &  ${}\le 1/2$\\
\hline $2.21$ & $28$ & the blow up of $Q\subset\mathbb{P}^4$ along
a
twisted quartic   &  ${}\le 2/3$\\
\hline $2.22$ & $30$ & the blow up of $V_5\subset\mathbb{P}^6$
along a
conic   &  ${}\le 1/2$\\
\hline $2.23$ & $30$ & the blow up of $Q\subset\mathbb{P}^4$ along a curve of degree $4$ that is an intersection of a surface in $|\mathcal{O}_{\mathbb{P}^{4}}(1)\vert_{Q}|$ and a surface in $|\mathcal{O}_{\mathbb{P}^{4}}(2)\vert_{Q}|$ & $1/3\star$\\
\hline $2.24$ & $30$ & a divisor on $\mathbb{P}^2\times\mathbb{P}^2$ of bidegree $(1, 2)$   & $1/2\star$\\
\hline $2.25$ & $32$ & the blow up of $\mathbb{P}^3$ along an elliptic curve which is an intersection of two quadrics   & $1/2$ \\
\hline $2.26$ & $34$ & the blow up of the threefold $V_5\subset\mathbb{P}^6$ along a line & $1/2\star$\\
\hline $2.27$ & $38$ & the blow up of $\mathbb{P}^3$ along a twisted cubic   & $1/2$\\
\hline $2.28$ & $40$ & the blow up of $\mathbb{P}^3$ along a plane cubic   & $1/4$\\
\hline $2.29$ & $40$ & the blow up of $Q\subset\mathbb{P}^4$ along a conic   & $1/3$\\
\hline $2.30$ & $46$ & the blow up of $\mathbb{P}^3$ along a conic   & $1/4$\\
\hline $2.31$ & $46$ & the blow up of $Q\subset\mathbb{P}^4$ along a line   & $1/3$\\
\hline $2.32$ & $48$ & $W$ that is a divisor on $\mathbb{P}^2\times\mathbb{P}^2$ of bidegree $(1, 1)$  & $1/2$\\
\hline $2.33$ & $54$ & the blow up of $\mathbb{P}^3$ along a line  & $1/4$\\
\hline $2.34$ & $54$ & $\mathbb{P}^1\times\mathbb{P}^2$  & $1/3$\\
\hline $2.35$ & $56$ & $V_7\cong\mathbb{P}(\mathcal{O}_{\mathbb{P}^2}\oplus\mathcal{O}_{\mathbb{P}^2}(1))$  & $1/4$\\
\hline $2.36$ & $62$ & $\mathbb{P}(\mathcal{O}_{\mathbb{P}^2}\oplus\mathcal{O}_{\mathbb{P}^2}(2))$  & $1/5$\\
\hline $3.1$ & $12$ & a double cover of $\mathbb{P}^1\times\mathbb{P}^1\times\mathbb{P}^1$ branched in a divisor of tridegree $(2, 2, 2)$  & $3/4\star$\\
\hline $3.2$ & $14$ & a divisor on a $\mathbb{P}^{2}$-bundle $\mathbb{P}(\mathcal{O}_{\mathbb{P}^1\times\mathbb{P}^1}\oplus\mathcal{O}_{\mathbb{P}^1\times\mathbb{P}^1}(-1,-1)\oplus\mathcal{O}_{\mathbb{P}^1\times\mathbb{P}^1}(-1,-1))$\hfill\break such that $X\in|L^{{}\otimes 2}\otimes\mathcal{O}_{\mathbb{P}^{1}\times\mathbb{P}^{1}}(2,3)|$, where $L$ is the tautological line bundle & $1/2\star$\\
\hline $3.3$ & $18$ & a divisor on $\mathbb{P}^1\times\mathbb{P}^1\times\mathbb{P}^2$ of tridegree  $(1, 1, 2)$  & $2/3\star$\\
\hline $3.4$ & $18$ & the blow up of the Fano threefold $Y$ with $\gimel(Y)=2.18$ along a smooth fiber\hfill\break of the composition $Y\to\mathbb{P}^1\times\mathbb{P}^2\to\mathbb{P}^2$ of the double cover with the projection  & $1/2$\\
\hline $3.5$ & $20$ & the blow up of $\mathbb{P}^1\times\mathbb{P}^2$ along a curve $C$ of bidegree $(5, 2)$\hfill\break such that the composition  $C\hookrightarrow\mathbb{P}^1\times\mathbb{P}^2\to\mathbb{P}^2$ is an embedding  & $1/2\star$\\
\hline $3.6$ & $22$ & the blow up of $\mathbb{P}^3$ along a disjoint union of a line and an elliptic curve of degree~$4$  & $1/2\star$\\
\hline $3.7$ & $24$ & the blow up of the threefold $W$ along an elliptic curve\hfill\break that is an intersection of two  divisors from $|-\frac{1}{2}K_W|$  & $1/2\star$\\
\hline $3.8$ & $24$ & a divisor in $|(\alpha\circ\pi_1)^*(\O_{\P^2}(1))\otimes\pi_2^*(\O_{\P^2}(2))|$, where $\pi_{1}\colon\F_1\times\P^2\to\F_1$\hfill\break and $\pi_{2}\colon\F_1\times\P^2\to\P^2$ are projections, and $\alpha\colon\F_1\to\P^2$ is a blow up of a point& $1/2\star$\\
\hline $3.9$ & $26$ & the blow up of a cone $W_4\subset\mathbb{P}^6$ over the Veronese surface  $R_4\subset\mathbb{P}^5$\hfill\break with center in a disjoint union of the vertex and a quartic on $R_4\cong\mathbb{P}^2$  & $1/3$\\
\hline $3.10$ & $26$ & the blow up of $Q\subset\mathbb{P}^4$ along a disjoint union of two conics & $1/2$ \\
\hline $3.11$ & $28$ & the blow up of the threefold $V_7$ along an elliptic curve\hfill\break that is an intersection of  two divisors from $|-\frac{1}{2}K_{V_7}|$  & $1/2$\\
\hline $3.12$ & $28$ & the blow up of $\mathbb{P}^3$ along a disjoint union of a line and a twisted cubic  & $1/2$\\
\hline $3.13$ & $30$ & the blow up of
$W\subset\mathbb{P}^2\times\mathbb{P}^2$ along a curve $C$ of
bidegree $(2, 2)$\hfill\break such that
$\pi_{1}(C)\subset\mathbb{P}^2$ and
$\pi_{2}(C)\subset\mathbb{P}^{2}$ are irreducible
conics,\hfill\break where $\pi_{1}\colon W\to\mathbb{P}^2$ and
$\pi_{2}\colon W\to\mathbb{P}^2$ are
natural projections &  ${}\le 1/2$\\
\hline $3.14$ & $32$ & the blow up of $\mathbb{P}^3$ along a
disjoint union of a plane cubic curve that\hfill\break is
contained in a  plane
$\Pi\subset\mathbb{P}^{3}$ and a point that is not contained in $\Pi$  & $1/2$\\
\hline $3.15$ & $32$ & the blow up of $Q\subset\mathbb{P}^4$ along a disjoint union of a line and a conic  & $1/2$ \\
\hline $3.16$ & $34$ & the blow up of $V_7$ along a proper
transform via the blow up $\alpha\colon
V_7\to\mathbb{P}^3$\hfill\break of a twisted cubic
passing through the center of the blow up $\alpha$  & $1/2$\\
\hline $3.17$ & $36$ & a divisor on $\mathbb{P}^1\times\mathbb{P}^1\times\mathbb{P}^2$ of tridegree $(1, 1, 1)$  & $1/2$\\
\hline $3.18$ & $36$ & the blow up of $\mathbb{P}^3$ along a disjoint union of a line and a conic & $1/3$\\
\hline $3.19$ & $38$ & the blow up of $Q\subset\mathbb{P}^4$ at two non-collinear points  & $1/3$\\
\hline $3.20$ & $38$ & the blow up of $Q\subset\mathbb{P}^4$ along a disjoint union of two lines & $1/3$\\
\hline $3.21$ & $38$ & the blow up of $\mathbb{P}^1\times\mathbb{P}^2$ along a curve of bidegree  $(2, 1)$  & $1/3$\\
\hline $3.22$ & $40$ & the blow up of $\mathbb{P}^1\times\mathbb{P}^2$ along a conic in a fiber of the projection $\mathbb{P}^{1}\times\mathbb{P}^2\to\mathbb{P}^1$  & $1/3$\\
\hline $3.23$ & $42$ & the blow up of $V_7$ along a proper transform via the blow up $\alpha\colon V_7\to\mathbb{P}^3$\hfill\break of an irreducible conic passing through the center of the blow up $\alpha$  & $1/4$\\
\hline $3.24$ & $42$ & $W\times_{\mathbb{P}^2}\mathbb{F}_1$, where
$W\to\mathbb{P}^2$ is a $\mathbb{P}^1$-bundle and
$\mathbb{F}_1\to\mathbb{P}^2$ is the
blow up  & $1/3$\\
\hline $3.25$ & $44$ & the blow up of $\mathbb{P}^3$ along a disjoint union of two lines & $1/3$\\
\hline $3.26$ & $46$ & the blow up of $\mathbb{P}^3$ with center in a disjoint union of a point and a line  & $1/4$\\
\hline $3.27$ & $48$ & $\mathbb{P}^1\times\mathbb{P}^1\times\mathbb{P}^1$  & $1/2$\\
\hline $3.28$ & $48$ & $\mathbb{P}^1\times\mathbb{F}_1$  & $1/3$\\
\hline $3.29$ & $50$ & the blow up of the Fano threefold $V_7$
along a line in $E\cong\P^2$,\hfill\break where $E$ is the
exceptional divisor of the
blow up $V_7\to\mathbb{P}^3$  & $1/5$ \\
\hline $3.30$ & $50$ & the blow up of $V_7$ along a proper transform via the blow up $\alpha\colon V_7\to\mathbb{P}^3$\hfill\break of a line that passes through the center of the blow up $\alpha$   & $1/4$\\
\hline $3.31$ & $52$ & the blow up of a cone over a smooth quadric in $\mathbb{P}^3$ at the vertex & $1/3$\\
\hline $4.1$ & $24$ & divisor on $\mathbb{P}^1\times\mathbb{P}^1\times\mathbb{P}^1\times\mathbb{P}^1$ of multidegree $(1, 1, 1, 1)$  & $1/2$\\
\hline $4.2$ & $28$ & the blow up of the cone over a smooth quadric $S\subset\nlb\mathbb{P}^3$\hfill\break along a disjoint union of the vertex and an elliptic curve on $S$  & $1/2$\\
\hline $4.3$ & $30$ & the blow up of $\mathbb{P}^1\times\mathbb{P}^1\times\mathbb{P}^1$ along a curve of tridegree $(1, 1, 2)$  & $1/2$\\
\hline $4.4$ & $32$ & the blow up of the smooth Fano threefold $Y$ with $\gimel(Y)=3.19$\hfill\break along the proper transform of a conic on the quadric $Q\subset\P^4$\hfill\break that passes through the both centers of the blow up $Y\to Q$  & $1/3$\\
\hline $4.5$ & $32$ & the blow up of $\mathbb{P}^1\times\mathbb{P}^2$ along a disjoint union of\hfill\break two irreducible curves of bidegree $(2, 1)$ and $(1, 0)$  & $3/7$\\
\hline $4.6$ & $34$ & the blow up of $\mathbb{P}^3$ along a disjoint union of three lines & $1/2$ \\
\hline $4.7$ & $36$ & the blow up of $W\subset\mathbb{P}^2\times\mathbb{P}^2$ along a disjoint union of\hfill\break two curves of bidegree $(0, 1)$ and $(1, 0)$  & $1/2$\\
\hline $4.8$ & $38$ & the blow up of $\mathbb{P}^1\times\mathbb{P}^1\times\mathbb{P}^1$ along a curve of tridegree $(0, 1, 1)$  & $1/3$\\
\hline $4.9$ & $40$ & the blow up of the smooth Fano threefold $Y$
with $\gimel(Y)=3.25$\hfill\break along a curve that is contracted
by the blow
up $Y \to\mathbb{P}^3$  &  $1/3$\\
\hline $4.10$ & $42$ & $\mathbb{P}^1\times S_7$  & $1/3$\\
\hline $4.11$ & $44$ & the blow up of
$\mathbb{P}^1\times\mathbb{F}_1$ along a curve $C\cong\P^1$ such
that $C$ is contained\hfill\break in a fiber
$F\cong\mathbb{F}_{1}$ of the projection
$\P^1\times\mathbb{F}_{1}\to\P^1$ and $C\cdot C=-1$ on $F$ &  $1/3$\\
\hline $4.12$ & $46$ & the blow up of the smooth Fano threefold
$Y$ with $\gimel(Y)=2.33$\hfill\break along two curves that are
contracted by the blow up
$Y\to\mathbb{P}^3$ & $1/4$\\
\hline $4.13$ & $26$ & the blow up of $\mathbb{P}^1\times\mathbb{P}^1\times\mathbb{P}^1$ along a curve of tridegree $(1, 1, 3)$  & $1/2\star$\\
\hline $5.1$ & $28$ & the blow up of the smooth Fano threefold $Y$
with $\gimel(Y)=2.29$\hfill\break along three curves that are
contracted by the
blow up $Y\to Q$  &  $1/3$\\
\hline $5.2$ & $36$ & the blow up of the smooth Fano threefold $Y$
with $\gimel(Y)=3.25$ along\hfill\break two curves $C_{1}\ne
C_{2}$ that are contracted by the blow up $\phi\colon
Y\to\nlb\mathbb{P}^3$\hfill\break  and that are contained in the
same exceptional divisor of the blow up $\phi$  &
$1/3$\\
\hline $5.3$ & $36$ & $\mathbb{P}^1\times S_6$  & $1/2$\\
\hline $5.4$ & $30$ & $\mathbb{P}^1\times S_5$  & $1/2$\\
\hline $5.5$ & $24$ & $\mathbb{P}^1\times S_4$  & $1/2$\\
\hline $5.6$ & $18$ & $\mathbb{P}^1\times S_3$  & $1/2$\\
\hline $5.7$ & $12$ & $\mathbb{P}^1\times S_2$  & $1/2$\\
\hline $5.8$ & $6$ & $\mathbb{P}^1\times S_1$  & $1/2$\\
\hline
\end{longtable}

}


\end{document}